\documentclass{amsart}
\usepackage{amssymb}
\usepackage{amsfonts}
\usepackage{amsmath}
\usepackage{amsthm, amsmath, amssymb, enumerate}
\usepackage{amssymb, enumitem}
\usepackage{bbm}
\usepackage{amscd}
\usepackage[all]{xy}
\usepackage[hidelinks]{hyperref}
\usepackage{amsmath}
\usepackage{amssymb}
\usepackage{amsthm}
\usepackage[sort&compress,numbers]{natbib}
\usepackage[left=2cm,right=2cm,bottom=3cm,top=3cm]{geometry}
\usepackage{tikz}
\usepackage{tikz-cd}
\usepackage{wrapfig}

\newcounter{tmp}
\newtheorem{theorem}{Theorem}[section]
\newtheorem{proposition}[theorem]{Proposition}
\newtheorem{lemma}[theorem]{Lemma}
\newtheorem{corollary}[theorem]{Corollary}
\theoremstyle{definition}
\newtheorem*{claim*}{Claim}
\newtheorem{definition}[theorem]{Definition}

\newtheorem{remark}[theorem]{Remark}

\newtheorem{example}[theorem]{Example}

\AtBeginDocument{   \def\MR#1{}}

\begin{document}
\title{Definable $\mathrm{K}$-homology of separable C*-algebras}
\author{Martino Lupini}
\address{School of Mathematics and Statistics\\
Victoria University of Wellington\\
PO Box 600, 6140 Wellington, New Zealand}
\email{martino.lupini@vuw.ac.nz}
\thanks{The author was partially supported by the Marsden Fund Fast-Start
Grant VUW1816 from the Royal Society Te Ap\={a}rangi.}
\date{\today }

\begin{abstract}
In this paper we show that the $\mathrm{K}$-homology groups of a separable
C*-algebra can be enriched with additional descriptive set-theoretic
information, and regarded as \emph{definable groups}. Using a definable
version of the Universal Coefficient Theorem, we prove that the
corresponding \emph{definable }$\mathrm{K}$-homology is a finer invariant
than the purely algebraic one, even when restricted to the class of UHF
C*-algebras, or to the class of unital commutative C*-algebras whose
spectrum is a $1$-dimensional connected subspace of $\mathbb{R}^{3}$.
\end{abstract}

\maketitle









\section*{Introduction}

Given a compact metrizable space $X$, the group $\mathrm{Ext}\left( X\right) 
$ classifying extensions of the C*-algebra $C\left( X\right) $ by the
C*-algebra $K\left( H\right) $ of compact operators was initially considered
by Brown, Douglas, and Fillmore in their celebrated work \cite%
{brown_extensions_1977}. There, they showed that $\mathrm{Ext}\left(
-\right) $ is indeed a group, and that defining, for a compact metrizable
space $X$, 
\begin{equation*}
\mathrm{\tilde{K}}_{p}\left( X\right) :=\left\{ 
\begin{array}{ll}
\mathrm{Ext}\left( X\right) & \text{if }p\text{ is odd,} \\ 
\mathrm{Ext}\left( \Sigma X\right) & \text{if }p\text{ is even;}%
\end{array}%
\right.
\end{equation*}%
where $\Sigma X$ is the suspension of $X$, yields a (reduced) \emph{homology
theory }that satisfies all the Eilenberg--Steenrod--Milnor axioms for
Steenrod homology, apart from the Dimension Axiom; see also \cite%
{kaminker_theory_1977}. They furthermore observed, building on a previous
insight of Atiyah \cite{atiyah_global_1970}, that such a homology theory can
be seen as the Spanier--Whitehead dual of topological $\mathrm{K}$-theory 
\cite{atiyah_theory_1989}.

More generally, for an arbitrary separable unital C*-algebra $A$, one can
consider a semigroup $\mathrm{Ext}\left( A\right) $ classifying the
essential, unital extensions of $A$ by $K\left( H\right) $. By Voiculescu's
non-commutative Weyl-von Neumann Theorem \cite%
{voiculescu_non-commutative_1976,arveson_notes_1977}, the trivial element of 
$\mathrm{Ext}\left( A\right) $ correspond to the class of \emph{trivial }%
essential, unital extensions. The group $\mathrm{Ext}\left( A\right) ^{-1}$
of invertible elements of $\mathrm{Ext}\left( A\right) $ corresponds to the
essential, unital extensions that are \emph{semi-split}. Thus, by the
Choi--Effros lifting theorem \cite{choi_completely_1976}, $\mathrm{Ext}%
\left( A\right) $ is a group when $A$ is nuclear. One can extend $\mathrm{K}$%
-homology to the category of all separable C*-algebras by setting%
\begin{equation*}
\mathrm{K}^{p}\left( A\right) =\left\{ 
\begin{array}{ll}
\mathrm{Ext}\left( A^{+}\right) ^{-1} & \text{if }p\text{ is odd,} \\ 
\mathrm{Ext}((SA)^{+})^{-1} & \text{if }p\text{ is even;}%
\end{array}%
\right.
\end{equation*}%
where $SX$ is the suspension of $A$ and $A^{+}$ is the unitization of $A$.
This gives a cohomology theory on the category of separable C*-algebras,
which can be recognized as the dual of $\mathrm{K}$-theory via Paschke
duality \cite{paschke_theory_1981,higson_algebra_1995,kaminker_spanier_2017}%
. Kasparov's bivariant functor $\mathrm{KK}\left( -,-\right) $
simultaneously generalizes $\mathrm{K}$-homology and $\mathrm{K}$-theory,
where $\mathrm{K}^{1}\left( A\right) $ is recovered as $\mathrm{KK}\left( A,%
\mathbb{C}\right) $ and $\mathrm{K}^{0}\left( A\right) $ as $\mathrm{KK}%
\left( A,C_{0}\left( \mathbb{R}\right) \right) $.

It was already noticed in the seminal work of Brown, Douglas, and Fillmore 
\cite{brown_extensions_1977,brown_extensions_1973} that the invariant $%
\mathrm{K}^{p}\left( X\right) $ for a compact metrizable space $X$ can be
endowed with more structure than the purely algebraic group structure.
Indeed, one can write $X$ as the inverse limit of a tower $\left(
X_{n}\right) _{n\in \omega }$ of compact polyhedra, and endow $\mathrm{K}%
^{p}\left( X\right) $ with the topology induced by the maps $\mathrm{K}%
^{p}\left( X\right) \rightarrow \mathrm{K}^{p}\left( X_{n}\right) $ for $%
n\in \omega $, where $\mathrm{K}^{p}\left( X_{n}\right) $ is a countable
group endowed with the discrete topology. This gives to $\mathrm{K}%
^{p}\left( X\right) $ the structure of a topological group, which is however
in general not Hausdorff.

The study of $\mathrm{K}^{p}\left( A\right) $ as a topological group for a
separable unital C*-algebra $A$ was later systematically undertaken by
Dadarlat \cite{dadarlat_approximate_2000,dadarlat_topology_2005} and
Schochet \cite{schochet_fine_2001,schochet_fine_2002,schochet_fine_2005}
building on previous work of Salinas \cite{salinas_relative_1992}. (In fact,
they consider more generally Kasparov's $\mathrm{KK}$-groups.) In \cite%
{schochet_fine_2001,dadarlat_topology_2005} several natural topologies on $%
\mathrm{K}^{p}\left( A\right) $, corresponding to different ways to define $%
\mathrm{K}$-homology for separable C*-algebras, are shown to coincide and to
turn $\mathrm{K}^{p}\left( A\right) $ into a \emph{pseudo-Polish }group.
This means that, if $\mathrm{K}_{\infty }^{p}\left( A\right) $ denotes the
closure of zero in $\mathrm{K}^{p}\left( A\right) $, then the quotient of $%
\mathrm{K}^{p}\left( A\right) $ by $\mathrm{K}_{\infty }^{p}\left( A\right) $
is a Polish group. In \cite{schochet_fine_2002}, for a C*-algebra $A$
satisfying the Universal Coefficient Theorem\ (UCT), the topology on $%
\mathrm{K}^{p}\left( A\right) $ is related to the UCT exact sequence, and $%
\mathrm{K}_{\infty }^{p}\left( A\right) $ is shown to be isomorphic to the
group $\mathrm{PExt}\left( \mathrm{K}_{1-p}\left( A\right) ,\mathbb{Z}%
\right) $ classifying \emph{pure }extensions of $\mathrm{K}_{1-p}\left(
A\right) $ by $\mathbb{Z}$. A characterization of $\mathrm{K}_{\infty
}^{p}\left( A\right) $ for an arbitrary separable nuclear C*-algebra $A$ is
obtained in \cite{dadarlat_approximate_2000}. For a separable \emph{%
quasidiagonal }C*-algebra satisfying the UCT, $\mathrm{K}_{\infty
}^{1}\left( A\right) $ is shown to be the subgroup of $\mathrm{K}^{1}\left(
A\right) =\mathrm{Ext}\left( A^{+}\right) $ corresponding to \emph{%
quasidiagonal }extensions of $A^{+}$ by $K\left( H\right) $ \cite%
{schochet_fine_2002}; see also \cite{brown_universal_1984} for the
commutative case. The quotient $\mathrm{K}_{\mathrm{w}}^{p}\left( A\right) $
of $\mathrm{K}^{p}\left( A\right) $ by $\mathrm{K}_{\infty }^{p}\left(
A\right) $ is the group $\mathrm{KL}_{p}\left( A,\mathbb{C}\right) $
introduced by R\o rdam \cite{rordam_classification_1995}. A universal
multicoefficient theorem describing $\mathrm{K}_{\mathrm{w}}^{p}\left(
A\right) $ in terms of the $\mathrm{K}$-groups of $A$ with arbitrary cyclic
groups as coefficients is obtained in \cite{dadarlat_classification_2002}
for all separable nuclear C*-algebras satisfying the UCT; see \cite[Theorem
5.4]{dadarlat_topology_2005}.

In many cases of interest, the topology on $\mathrm{K}^{p}\left( A\right) $
turns out to be trivial, i.e.\ the closure of zero in $\mathrm{K}^{p}\left(
A\right) $ is the whole group. For example, the topology on $\mathrm{K}%
^{1}\left( A\right) $ is trivial when $A$ is a UHF C*-algebra, despite the
fact that $\mathrm{K}^{1}\left( A\right) $ is not trivial, and in fact
uncountable. Similarly, for every $1$-dimensional solenoid $X$, the topology
on $\mathrm{\tilde{K}}_{0}\left( X\right) $ is trivial, although $\mathrm{%
\tilde{K}}_{0}\left( X\right) $ is an uncountable group.

In this paper, we take a different approach and consider the group $\mathrm{K%
}^{p}\left( A\right) $, rather than as a pseudo-Polish topological group, as
a \emph{definable group}. This should be thought of as a group $G$
explicitly defined as the quotient of a Polish space $X$ by a
\textquotedblleft well-behaved\textquotedblright\ equivalence relation $E$,
in such a way that the multiplication and inversion operations in $G$ are
induced by a Borel functions on $X$. This is formally defined in Section \ref%
{Subsection:definable-groups}, where the notion of well-behaved equivalence
relation is made precise. The definition is devised to ensure that the
category of definable groups has good properties, and behaves similarly to
the category of standard Borel groups. A morphism in this category is a 
\emph{definable }group homomorphism, namely a group homomorphism that \emph{%
lifts }to a Borel function between the corresponding Polish spaces.

It has recently become apparent that several homological invariants in
algebra and topology can be seen as functors to the category of definable
groups. The homological invariants $\mathrm{Ext}$ and $\mathrm{lim}^{1}$ are
considered in \cite{bergfalk_ulam_2020}, whereas Steenrod homology and \v{C}%
ech cohomology are considered in \cite%
{bergfalk_cohomology_2020,lupini_definable_2020}. It is shown there that the
definable versions of these invariants are finer than the purely algebraic
versions.

In this paper, we show that, for an arbitrary separable C*-algebra $A$, $%
\mathrm{K}^{p}\left( A\right) $ can be regarded as a definable group.
Furthermore, different descriptions of $\mathrm{K}^{p}\left( A\right) $---in
terms of extensions, Paschke duality, Fredholm modules, and
quasi-homomorphisms---yield naturally \emph{definably }isomorphic definable
groups. For C*-algebras that have a $\mathrm{KK}$-filtration in the sense of
Schochet \cite{schochet_uct_1996}, we show that the definable subgroup $%
\mathrm{K}_{\infty }^{p}\left( A\right) $ of $\mathrm{K}^{p}\left( A\right) $
is \emph{definably }isomorphic to $\mathrm{PExt}\left( \mathrm{K}%
_{1-p}\left( A\right) ,\mathbb{Z}\right) $. The latter is regarded as a
definable group as in \cite[Section 7]{bergfalk_ulam_2020}. (In fact, $%
\mathrm{PExt}\left( \mathrm{K}_{1-p}\left( A\right) ,\mathbb{Z}\right) $ is
the quotient of a Polish group by a Borel Polishable subgroup, and hence a
group with a Polish cover in the parlance of \cite[Section 7]%
{bergfalk_ulam_2020}.)

Using this and the rigidity theorem for $\mathrm{PExt}\left( \Lambda ,%
\mathbb{Z}\right) $ from \cite[Section 7]{bergfalk_ulam_2020} where $\Lambda 
$ is a torsion-free abelian group without finitely-generated direct
summands, we prove that \emph{definable }$\mathrm{K}$-homology provides a
finer invariant than the purely algebraic (or topological) groups $\mathrm{K}%
^{p}\left( A\right) $ for a separable C*-algebra $A$, even when one
restricts to UHF C*-algebras or commutative unital C*-algebras whose
spectrum is a $1$-dimensional subspace of $\mathbb{R}^{3}$.

\begingroup
\setcounter{tmp}{\value{theorem}}
\setcounter{theorem}{0} 
\renewcommand\thetheorem{\Alph{theorem}} 

\begin{theorem}
\label{Theorem:A}The definable $\mathrm{K}^{1}$-group is a complete
invariant for UHF C*-algebras up to stable isomorphism. In contrast, there
exists an uncountable family of pairwise non stably isomorphic UHF
C*-algebras with algebraically isomorphic $\mathrm{K}^{1}$-groups (and
trivial $\mathrm{K}^{0}$-groups).
\end{theorem}

\begin{theorem}
\label{Theorem:B}The definable $\mathrm{\tilde{K}}_{0}$-group is a complete
invariant for $1$-dimensional solenoids up to homeomorphism. In contrast,
there exists an uncountable family of pairwise non homeomorphic $1$%
-dimensional solenoids with algebraically isomorphic $\mathrm{\tilde{K}}_{0}$%
-groups (and trivial $\mathrm{\tilde{K}}_{1}$-groups).
\end{theorem}

The historic evolution in the treatment of $\mathrm{K}$-homology described
above should be compared with the similar evolution in the study of unitary
duals of second countable, locally compact groups or, more generally,
separable C*-algebras. Given a separable C*-algebra $A$, its unitary dual $%
\hat{A}$ is the quotient of the Polish space $\mathrm{Irr}(A)$ of unitary
irreducible representations of $A$ by the relation of unitary equivalence.
This includes as a particular instance the case of second countable, locally
compact groups, by considering the corresponding universal C*-algebras.
While initially $\hat{A}$ was considered as a topological space endowed with
the quotient topology, it was recognized in the seminal work of Mackey,
Glimm, and Effros \cite%
{mackey_borel_1957,glimm_type_1961,effros_transformation_1965} that a more
fruitfuil theory is obtained by considering $\hat{A}$ endowed with the
quotient Borel structure, called the Mackey Borel structure. This led to the
notion of type I C*-algebra, which precisely captures those separable
C*-algebras with the property that the Mackey Borel structure is standard.
It was soon realized that, in the non type I case, the right notion of
\textquotedblleft isomorphism\textquotedblright\ of Macky Borel structures
on duals $\hat{A},\hat{B}$ corresponds to a bijection $\hat{A}\rightarrow 
\hat{B}$ that is induced by a Borel function $\mathrm{Irr}\left( A\right)
\rightarrow \mathrm{Irr}\left( B\right) $. In our terminology from Section %
\ref{Subsection:definable-set}, this corresponds to regarding a unitary dual 
$\hat{A}$ as a \emph{definable set}, where an isomorphism of Mackey Borel
structures on $\hat{A},\hat{B}$ is a \emph{definable bijection} $\hat{A}%
\rightarrow \hat{B}$. For example, this approach is taken by Elliott in \cite%
{elliott_mackey_1977}, where he proved that the unitary duals of any two
separable AF C*-algebras that are not type I are isomorphic in the category
of definable sets. It is a question of Dixmier from 1967 whether the unitary
duals of any two non-type I separable C*-algebras are isomorphic in the
category of definable sets; see \cite%
{thomas_descriptive_2015,farah_dichotomy_2012,kerr_turbulence_2010}. This
problem was recently considered in the case of groups by Thomas, who showed
that the unitary duals of any two countable amenable non-type I groups are
isomorphic in the category of definable sets \cite[Theorem 1.10]%
{thomas_descriptive_2015}. Furthermore, the unitary dual of any countable
group admits a \emph{definable injection }to the unitary dual of the free
group on two generators \cite[Theorem 1.9]{thomas_descriptive_2015}.

The work of Mackey, Glimm, and Effors on unitary representations pioneered
the application of methods from descriptive set theory to C*-algebras. More
recent applications have been obtained by Kechris \cite%
{kechris_descriptive_1998} and Farah--Toms--T\"{o}rnquist \cite%
{farah_turbulence_2014,farah_descriptive_2012}, who studied the problem of
classifying several classes of C*-algebras from the perspective of Borel
complexity theory; see also \cite%
{lupini_unitary_2014,gardella_conjugacy_2016,elliott_isomorphism_2013}.

\endgroup

\setcounter{theorem}{\thetmp} 
The rest of this paper is organized as follows. In Section \ref%
{Section:Polish-spaces} we recall fundamental results from descriptive set
theory about Polish spaces and standard Borel spaces, and make precise the
notions of definable set, and the corresponding notion of definable group.
In Section \ref{Section:strict} we introduce the notion of strict
C*-algebra, which is a (not necessarily norm-separable) C*-algebra whose
unit ball is endowed with a Polish topology induced by bounded seminorms,
called the strict topology, such that the C*-algebra operations are strictly
continuous on the unit ball. The main example we will consider are
multiplier algebras of separable C*-algebras, endowed with their usual
strict topology, as well as Paschke dual algebras of separable C*-algebras.
In Section \ref{Section:K-theory} we study the $\mathrm{K}$-theory of a
strict C*-algebra or, more generally, the quotient of a strict C*-algebra by
a strict ideal, such as a corona algebra or the commutant of a separable
C*-algebra in a corona algebra. We observe that the $\mathrm{K}$-theory
groups of a strict C*-algebra can be regarded as quotients of a Polish space
by an equivalence relation. As such an equivalence relation is not
necessarily well-behaved, they are in general only semidefinable groups,
although we they will be in fact definable groups in the case of Paschke
dual algebras of separable C*-algebras. In Section \ref{Section:Ext}
definable $\mathrm{K}$-homology for separable C*-algebras is introduced, and
shown to be given by definable groups by considering its description in
terms of the $\mathrm{K}$-theory of Pashcke dual algebras. The descriptions
of $\mathrm{K}$-homology due to Cuntz and Kasparov are considered in Section %
\ref{Section:Kasparov}, where they are shown to yield definably isomorphic
groups. In Section \ref{Section:properties} we discuss properties of
definable $\mathrm{K}$-homology, which can be seen as definable versions of
the general properties that an abstract cohomology theory for separable,
nuclear C*-algebras in the sense of Schochet satisfies \cite%
{schochet_topologicalIII_1984}. A definable version of the Universal
Coefficient Theorem of Brown \cite{brown_universal_1984}, later generalized
to $\mathrm{KK}$-groups by Rosenberg and Schochet \cite%
{rosenberg_kunneth_1987}, is considered in Section \ref{Section:UCT}.
Theorem \ref{Theorem:A} is a consequence of the definable UCT, the
classification of AF C*-algebras by $\mathrm{K}$-theory, and the rigidity
result for definable $\mathrm{PExt}$ of torsion-free finite-rank abelian
groups from \cite{bergfalk_ulam_2020}. Finally, Section \ref{Section:spaces}
considers definable $\mathrm{K}$-homology for compact metrizable spaces, and
Theorem \ref{Theorem:B} is obtained applying again the definable UCT and the
rigidity theorem for definable $\mathrm{PExt}$ from \cite{bergfalk_ulam_2020}%
.


\section{Polish spaces and definable groups\label{Section:Polish-spaces}}

In this section we recall some fundamental notions concerning Polish spaces
and Polish groups, as well as standard Borel spaces and standard Borel
groups, as can be found in \cite%
{becker_descriptive_1996,kechris_classical_1995,gao_invariant_2009}. We also
consider the notion of Polish category, which is a category whose hom-sets
are Polish spaces and composition of morphisms is a continuous function, and
establish some of its basic properties. Furthermore, we recall the notion of
idealistic equivalence relation on a standard Borel space and some of its
fundamental properties as established in \cite{kechris_borel_2016,
motto_ros_complexity_2012}. We then define precisely the notion of
(semi)definable set and (semi)definable group.

\subsection{Polish spaces and standard Borel spaces}

A \emph{Polish space} is a second countable topological space whose topology
is induced by a complete metric. A subset of a Polish space $X$ is $%
G_{\delta }$ if and only if it is a Polish space when endowed with the
subspace topology. If $X$ is a Polish space, then the Borel $\sigma $%
-algebra of $X$ is the $\sigma $-algebra generated by the collection of open
sets. By definition, a\ subset of $X$ is \emph{Borel} if it belongs to the
Borel $\sigma $-algebra. If $X,Y$ are Polish spaces, then the product $%
X\times Y$ is a Polish space when endowed with the product topology. More
generally, if $\left( X_{n}\right) _{n\in \omega }$ is a sequence of Polish
spaces, then the product $\prod_{n\in \omega }X_{n}$ is a Polish space when
endowed with the product topology. The class of Polish spaces includes all
locally compact second countable Hausdorff spaces. We denote by $\omega $
the set of natural numbers including $0$. We regard $\omega $ and any other
countable set as a Polish space endowed with the discrete topology. The 
\emph{Baire space }$\omega ^{\omega }$ is the Polish space obtained as the
infinite product of copies of $\omega $.

A \emph{standard Borel space }is a set $X$ endowed with a $\sigma $-algebra
(the Borel $\sigma $-algebra) that comprises the Borel sets with respect to
some Polish topology on $X$. A function between standard Borel spaces is 
\emph{Borel} if it is measurable with respect to the Borel $\sigma $%
-algebras. A subset of a standard Borel space $X$ is \emph{analytic }if it
is the image of a Borel function $f:Z\rightarrow X$ for some standard Borel
space $Z$. This is equivalent to the assertion that there exists a Borel
subset $B\subseteq X\times \omega ^{\omega }$ such that $B=\mathrm{proj}%
_{X}\left( A\right) $ is the projection of $A$ on the first coordinate. A
subset of $X$ is \emph{co-analytic} if its complement is analytic. One has
that a subset of $X$ is Borel if and only if it is both analytic and
co-analytic.

Given standard Borel spaces $X,Y$, we let $X\times Y$ be their product
endowed with the product Borel structure, which is also a standard Borel
space. If $\left( X_{n}\right) _{n\in \omega }$ is a sequence of standard
Borel spaces, then their disjoint union $X$ is a standard Borel space, where
a subset $A$ of $X$ is Borel if and only if $A\cap X_{n}$ is Borel for every 
$n\in \omega $. The product $\prod_{n\in \omega }X_{n}$ is also a standard
Borel space when endowed with the product Borel structure. In the following
proposition, we collect some well-known properties of the category of
standard Borel spaces and Borel functions.

\begin{proposition}
\label{Proposition:SB}Let $\mathbf{SB}$ be the category that has standard
Borel spaces as objects and Borel functions and morphisms.

\begin{enumerate}
\item If $X$ is a standard Borel space and $A\subseteq X$ is a Borel subset,
then $A$ is a standard Borel space when endowed with the induced standard
Borel structure;

\item If $X,Y$ are standard Borel spaces, $f:X\rightarrow Y$ is an injective
Borel function, and $A\subseteq X$ is Borel, then $f\left( A\right) $ is a
Borel subset of $Y$;

\item If $X,Y$ are standard Borel spaces, and $f:X\rightarrow Y$ is a
bijective Borel function, then the inverse function $f^{-1}:Y\rightarrow X$
is Borel;

\item If $X,Y$ are standard Borel spaces, and there exist injective Borel
functions $f:X\rightarrow Y$ and $g:Y\rightarrow X$, then there exists a
Borel bijection $h:X\rightarrow Y$;

\item The category $\mathbf{SB}$ has finite products, finite coproducts,
equalizers, and pullbacks;

\item A Borel function is monic in $\mathbf{SB}$ if and only if it is
injective, and epic in $\mathbf{SB}$ if and only if it is surjective;

\item An inductive sequence of standard Borel spaces and \emph{injective}%
\textrm{\ }Borel functions has a colimit in $\mathbf{SB}$.
\end{enumerate}
\end{proposition}

A \emph{Polish group }is a topological group whose topology is Polish. If $G$
is a Polish group, and $H$ is a closed subgroup of $G$, then $H$ is a Polish
group when endowed with the subspace topology.\ If furthermore $H$ is
normal, then $G/H$ is a Polish group when endowed with the quotient
topology. If $G_{0},G_{1}$ are Polish groups, and $\varphi :G_{0}\rightarrow
G_{1}$ is a Borel function, then $\varphi $ is continuous. In particular, if 
$G$ is a Polish space, then it has a unique Polish group topology that
induces its Borel structure. A subgroup $H$ of a Polish group $G$ is \emph{%
Polishable }if it is Borel and there is a (necessarily unique) Polish group
topology on $H$ that induces the Borel structure on $H$ inherited from $G$.
This is equivalent to the assertion that $H$ is equal to the range of a
continuous group homomorphisms $\hat{G}\rightarrow G$ for some Polish group $%
\hat{G}$. If $G$ is a Polish group, then a Polish $G$-space is a Polish
space $X$ endowed with a continuous action of $G$. A Borel $G$-space is a
standard Borel space $X$ endowed with a Borel action of $G$. Given a Borel $%
G $-space $X$, there exists a Polish topology $\tau $ on $X$ such that $%
\left( X,\tau \right) $ is a Polish $G$-space; see \cite[Theorem 5.2.1]%
{becker_descriptive_1996}.

A\emph{\ standard Borel group} is, simply, a group object in the category of
standard Borel spaces \cite[Section III.6]{mac_lane_categories_1998}.
Explicitly, a standard Borel group is a standard Borel space $G$ that is
also a group, and such that the group operation on $G$ and the function $%
G\rightarrow G$, $x\mapsto x^{-1}$ are Borel; see \cite[Definition 12.23]%
{kechris_classical_1995}. Clearly, every Polish group is, in particular, a
standard Borel group.

The notion of Polish topometric space was introduced and studied in \cite%
{ben_yaacov_grey_2015,ben_yaacov_polish_2013,ben_yaacov_topometric_2008,ben_yaacov_continuous_2010}%
. A\emph{\ topometric space }is a Hausdorff space $X$ endowed with a
topology $\tau $ and a $\left[ 0,\infty \right] $-valued metric $d$ such
that:

\begin{enumerate}
\item the metric-topology is finer than $\tau $;

\item the metric is lower-semicontinuous with respect to $\tau $, i.e.\ for
every $r\geq 0$ the set 
\begin{equation*}
\left\{ \left( a,b\right) \in X\times X:d\left( a,b\right) \leq r\right\}
\end{equation*}%
is $\tau $-closed in $X\times X$.
\end{enumerate}

A \emph{Polish topometric space} is a topometric space such that the
topology $\tau $ is Polish and the metric is complete. A \emph{Polish
topometric group }is a Polish topometric space $\left( G,\tau ,d\right) $
that is also a group, and such that $G$ endowed with the topology $\tau $ is
a Polish group, and the metric $d$ on $G$ is bi-invariant

\subsection{Polish categories\label{Subsection:Polish-categories}}

By definition, we let a Polish category be a category $\mathcal{C}$ enriched
over the category of Polish spaces (regarded as a monoidal category with
respect to binary products). Thus, for each pair of objects $a,b$ of $%
\mathcal{C}$, $\mathcal{C}\left( a,b\right) $ is a Polish space, such that
for objects $a,b,c$, the composition operation $\mathcal{C}\left( b,c\right)
\times \mathcal{C}\left( a,b\right) \rightarrow \mathcal{C}\left( a,c\right) 
$ is continuous.

Suppose that $\mathcal{C}$ is a Polish category. For objects $a,b$ of $%
\mathcal{C}$, define $\mathrm{Iso}_{\mathcal{C}}\left( a,b\right) \subseteq 
\mathcal{C}\left( a,b\right) $ be the set of $\mathcal{C}$-isomorphisms $%
a\rightarrow b$. While $\mathrm{Iso}_{\mathcal{C}}\left( a,b\right) $ is not
necessarily a $G_{\delta }$ subset of $\mathcal{C}\left( a,b\right) $, and
hence not necessarily a Polish space when endowed with the subspace
topology, $\mathrm{Iso}_{\mathcal{C}}\left( a,b\right) $ is endowed with a
canonical Polish topology, defined as follows. For a net $\left( \alpha
_{i}\right) $ in $\mathrm{Iso}_{\mathcal{C}}\left( a,b\right) $ and $\alpha
\in \mathrm{Iso}_{\mathcal{C}}\left( a,b\right) $, set $\alpha
_{i}\rightarrow \alpha $ if and only if $\alpha _{i}\rightarrow \alpha $ in $%
\mathcal{C}\left( a,b\right) $ and $\alpha _{i}^{-1}\rightarrow \alpha ^{-1}$
in $\mathcal{C}\left( b,a\right) $. One can then easily show the following.

\begin{lemma}
\label{Lemma:Iso-Polish}Adopt the notations above. Then $\mathrm{Iso}_{%
\mathcal{C}}\left( a,b\right) $ is a Polish space.
\end{lemma}

It is clear from the definition that, for every object $a$ of $\mathcal{C}$, 
$\mathrm{\mathrm{Aut}}_{\mathcal{C}}\left( a\right) :=\mathrm{Iso}_{\mathcal{%
C}}\left( a,a\right) $ is a Polish group. Furthermore, the canonical (right
and left) actions of $\mathrm{\mathrm{Aut}}_{\mathcal{C}}\left( a\right) $
and $\mathrm{\mathrm{Aut}}_{\mathcal{C}}\left( b\right) $ on $\mathcal{C}%
\left( a,b\right) $ are continuous.

\begin{definition}
\label{Definition:topological-equivalence}Suppose that $\mathcal{C}$ and $%
\mathcal{D}$ are Polish categories, and $F:\mathcal{C}\rightarrow \mathcal{D}
$ is a functor. We say that $F$ is continuous if, for every pair of objects $%
a,b$ of $\mathcal{C}$, the map $\mathcal{C}\left( a,b\right) \rightarrow 
\mathcal{D}\left( F\left( a\right) ,F\left( b\right) \right) $, $f\mapsto
F\left( f\right) $ is continuous. We say that $F$ is a \emph{topological
equivalence} if it is continuous, and there exists a continuous functor $G:%
\mathcal{D}\rightarrow \mathcal{C}$ such that $GF$ is isomorphic to the
identity functor $I_{\mathcal{C}}$, and $FG$ is isomorphic to the identity
functor $I_{\mathcal{D}}$.
\end{definition}

The notion of topological equivalence of categories is the natural analogue
of the notion of equivalence of categories in the context of Polish
categories; see \cite[Section IV.4]{mac_lane_categories_1998}.\ The same
proof as \cite[Section IV.4, Theorem 1]{mac_lane_categories_1998} gives the
following characterization of topological equivalences.

\begin{lemma}
\label{Lemma:topological-equivalence}Suppose that $\mathcal{C}$ and $%
\mathcal{D}$ are Polish categories, and $F:\mathcal{C}\rightarrow \mathcal{D}
$ is a functor. The following assertions are equivalent:

\begin{enumerate}
\item $F$ is a topological equivalence;

\item each object of $\mathcal{D}$ is isomorphic to one of the form $F\left(
a\right) $ for some object $a$ of $\mathcal{C}$, and for each pair of
objects $c,d$ of $\mathcal{C}$, the map $\mathcal{C}\left( c,d\right)
\rightarrow \mathcal{C}\left( F\left( c\right) ,F\left( d\right) \right) $
is a homeomorphism.
\end{enumerate}
\end{lemma}

\subsection{Idealistic equivalence relations}

Suppose that $C$ is a set. A $\sigma $-filter on $C$ is a nonempty family $%
\mathcal{F}$ of subsets of $C$ that is closed under countable intersections,
and such that $\varnothing \notin \mathcal{F}$ and if $A\subseteq B\subseteq
C$ and $A\in \mathcal{F}$ then $B\in \mathcal{F}$. The dual notion is the
one of $\sigma $-ideal. Thus, a nonempty family $\mathcal{I}$ of subsets of $%
C$ is a $\sigma $-ideal if it is closed under countable unions, $C\notin 
\mathcal{I}$, and $A\subseteq B\subseteq C$ and $B\in \mathcal{I}$ imply $%
A\in \mathcal{I}$. Clearly, if $\mathcal{F}$ is a $\sigma $-filter on $C$,
then $\left\{ C\setminus A:A\in \mathcal{F}\right\} $ is a $\sigma $-ideal
on $C$, and vice-versa.\ Thus, one can equivalently formulate notions in
terms of $\sigma $-filters or in terms of $\sigma $-ideals.

If $\mathcal{F}$ is a $\sigma $-filter on $C$, then $\mathcal{F}$ can be
thought of as a notion of \textquotedblleft largeness\textquotedblright\ for
subsets of $C$. Based on this interpretation, we use the \textquotedblleft $%
\sigma $-filter quantifier\textquotedblright\ notation \textquotedblleft $%
\mathcal{F}x$, $x\in A$\textquotedblright\ for a subset $A\subseteq C$ to
express the fact that $A\in \mathcal{F}$. If $P(x)$ is a unary relation for
elements of $C$, \textquotedblleft $\mathcal{F}x$, $P(x)$\textquotedblright\
is the assertion that the set of $x\in C$ that satisfy $P(x)$ belongs to $%
\mathcal{F}$.

\begin{example}
\label{Example:comeager}Suppose that $C$ is a Polish space. A subset $A$ of $%
C$ is \emph{meager} if it is contained in the union of a countable family of
closed nowhere dense sets. By the Baire Category Theorem \cite[Theorem 8.4]%
{kechris_classical_1995}, meager subsets of $C$ form a $\sigma $-ideal $%
\mathcal{I}_{C}$. The corresponding dual $\sigma $-filter is the $\sigma $%
-filter $\mathcal{F}_{C}$ of comeager sets, which are the subsets of $C$
whose complement is meager.
\end{example}

Suppose that $X$ is a standard Borel space. We consider an equivalence
relation $E$ on $X$ as a subset of $X\times X$, endowed with the product
Borel structure. Consistently, we say that $E$ is Borel or analytic,
respectively, if it is a Borel or analytic subset of $X\times X$. In the
following, we will exclusively consider analytic equivalence relations, most
of which will in fact be Borel. For an element $x$ of $X$ we let $\left[ x%
\right] _{E}$ be its corresponding $E$-class.

We now recall the notion of \emph{idealistic }equivalence relation,
initially considered in \cite{kechris_countable_1994}; see also \cite[%
Definition 5.4.9]{gao_invariant_2009} and \cite{kechris_borel_2016}. We will
consider a slightly more generous definition than the one from \cite%
{kechris_countable_1994,gao_invariant_2009,kechris_borel_2016}. The more
restrictive notion is recovered as a particular case by insisting that the
function $s$ in Definition \ref{Definition:idealistic} be the identity
function of $X$. In the following definition, for a subset $A$ of a product
space $X\times Y$ and $x\in X$, we let $A_{x}=\left\{ y\in Y:\left(
x,y\right) \in A\right\} $ be the corresponding \emph{vertical section}.

\begin{definition}
\label{Definition:idealistic}An equivalence relation $E$ on a standard Borel
space $X$ is \emph{idealistic }if there exist a Borel function $%
s:X\rightarrow X$ satisfying $s(x)Ex$ for every $x\in X$, and a function $%
C\mapsto \mathcal{F}_{C}$ that assigns to each $E$-class $C$ a $\sigma $%
-filter $\mathcal{F}_{C}$ of subsets of $C$ such that, for every Borel
subset $A$ of $X\times X$, the set%
\begin{equation*}
A_{s,\mathcal{F}}:=\left\{ x\in X:\mathcal{F}_{[x]_{E}}x^{\prime },\left(
s(x),x^{\prime }\right) \in A\right\} =\left\{ x\in X:A_{s(x)}\in \mathcal{F}%
_{\left[ x\right] _{E}}\right\} \text{.}
\end{equation*}%
is Borel.
\end{definition}

Idealistic equivalence relations arise naturally as orbit equivalence
relations of Polish group actions. Suppose that $G$ is a Polish group and $X$
is a Polish $G$-space. Let $E_{G}^{X}$ be the corresponding \emph{orbit
equivalence relation }on $X$, obtained by setting, for $x,y\in X$, $%
xE_{G}^{X}y$ if and only if there exists $g\in G$ such that $g\cdot x=y$.
Then $E_{G}^{X}$ is an idealistic equivalence relation, as witnessed by the
identity function $s$ on $X$ and the function $C\mapsto \mathcal{F}_{C}$
where $A\in \mathcal{F}_{C}$ if and only if $\mathcal{F}_{G}g$, $g\cdot x\in
A$. (As in Example \ref{Example:comeager}, $\mathcal{F}_{G}$ denotes the $%
\sigma $-filter of comeager subsets of $G$.) In particular, if $G$ is a
Polish group, and $H$ is a Polishable subgroup of $G$, then the coset
equivalence relation $E_{H}^{G}$ of $H$ in $G$ is Borel and idealistic.

Suppose that $E$ is an equivalence relation on a standard Borel space $X$. A
Borel selector for $E$ is a Borel function $s:X\rightarrow X$ such that, for 
$x,y\in x$, $xEy$ if and only if $s(x)=s(y)$. If $E$ has a Borel selector,
then $E$ is Borel and idealistic; see \cite[Theorem 5.4.11]%
{gao_invariant_2009}. (Precisely, an equivalence relation has a Borel
selector if and only if it is Borel, idealistic, and \emph{smooth }\cite[%
Definition 5.4.1]{gao_invariant_2009}.)

\subsection{Definable sets\label{Subsection:definable-set}}

Definable sets are a generalization of standard Borel sets, and can be
thought of as sets explicitly presented as the quotient of a standard Borel
space by a \textquotedblleft well-behaved\textquotedblright\ equivalence
relation $E$.

\begin{definition}
A \emph{definable set} $X$ is a pair $(\hat{X},E)$ where $\hat{X}$ is a
standard Borel space and $E$ is a Borel and idealistic equivalence relation
on $\hat{X}$. We think of $(\hat{X},E)$ as a presentation of the quotient
set $X=\hat{X}/E$. Consistently, we also write the definable set $(\hat{X}%
,E) $ as $\hat{X}/E$. A subset $Z$ of $X$ is \emph{Borel }if there is an $E$%
-invariant Borel subset $\hat{Z}$ of $\hat{X}$ such that $Z=\hat{Z}/E$.
\end{definition}

We now define the notion of morphism between definable sets. Let $X=\hat{X}%
/E $ and $Y=\hat{Y}/F$ be definable sets. A \emph{lift }of a function $%
f:X\rightarrow Y$ is a function $\hat{f}:\hat{X}\rightarrow \hat{Y}$ such
that $f\left( \left[ x\right] _{E}\right) =[\hat{f}(x)]_{F}$ for every $x\in 
\hat{X}$.

\begin{definition}
Let $X$ and $Y$ be definable sets. A function $f:X\rightarrow Y$ is \emph{%
Borel-definable }if it has a lift $\hat{f}:\hat{X}\rightarrow \hat{Y}$ that
is a Borel function.
\end{definition}

\begin{remark}
Since Borel-definability is the only notion of definability we will consider
in this paper, we will abbreviate \textquotedblleft
Borel-definable\textquotedblright\ to \textquotedblleft
definable\textquotedblright .
\end{remark}

We consider definable sets as objects of a category $\mathbf{DSet}$, whose
morphisms are the definable functions. We regard a standard Borel space $X$
as a particular instance of definable set $X=\hat{X}/E$ where $X=\hat{X}$
and $E$ is the relation of equality on $X$. This renders the category of
standard Borel spaces a full subcategory of the category of definable sets.

If $X=\hat{X}/E$ and $Y=\hat{Y}/F$ are definable sets, then their product $%
X\times Y$ in $\mathbf{DSet}$ is the definable set $X\times Y:=(\hat{X}%
\times \hat{Y})\left/ \left( E\times F\right) \right. $, $E\times F$ being
the equivalence relation on $\hat{X}\times \hat{Y}$ defined by setting $%
\left( x,y\right) \left( E\times F\right) \left( x^{\prime },y^{\prime
}\right) $ if and only if $xEx^{\prime }$ and $yFy^{\prime }$. (It is easy
to see that $E\times F$ is Borel and idealistic if both $E$ and $F$ are
Borel and idealistic.)

Many of the good properties of standard Borel spaces, including all the ones
listed in Proposition \ref{Proposition:SB}, generalize to definable sets.

\begin{proposition}
\label{Proposition:DSet}Let as above $\mathbf{DSet}$ be the category that
has definable as objects and definable functions as morphisms.

\begin{enumerate}
\item If $X$ is a definable and $A\subseteq X$ is a Borel subset, then $A$
is itself a definable set;

\item If $X,Y$ are definable sets, $f:X\rightarrow Y$ is an injective
definable function, and $A\subseteq X$ a Borel subset, then $f\left(
A\right) $ is a Borel subset of $Y$;

\item If $X,Y$ are definable sets, and $f:X\rightarrow Y$ is a bijective
definable function, then the inverse function $f^{-1}:Y\rightarrow X$ is
definable;

\item If $X,Y$ are definable sets, and there exist injective definable
functions $f:X\rightarrow Y$ and $g:Y\rightarrow X$, then there exists a
definable bijection $h:X\rightarrow Y$;

\item The category $\mathbf{DSet}$ has finite products, finite coproducts,
equalizers, and pullbacks;

\item A definable function is monic in $\mathbf{DSet}$ if and only if it is
injective, and epic in $\mathbf{DSet}$ if and only if it is surjective;

\item An inductive sequence of definable sets and \emph{injective}\textrm{\ }%
definable functions has a colimit in $\mathbf{DSet}$.
\end{enumerate}
\end{proposition}

\begin{proof}
(1) is immediate from the definition. (2) and (3) are consequences of \cite[%
Lemma 3.7]{kechris_borel_2016}, after observing that the same proof there
applies in the case of the more generous notion of idealistic equivalence
relation considered here. (4) is a consequence of (2) and \cite[Proposition
2.3]{motto_ros_complexity_2012}. Finally, (5), (6), and (7) are easily
verified.
\end{proof}

Occasionally we will need to consider quotients $X=\hat{X}/E$ where $\hat{X}$
is a standard Borel space $E$ is an \emph{analytic} equivalence relation on $%
\hat{X}$ that is not Borel and idealistic, or has not yet been shown to be
Borel and idealistic. In this case, we say that $X=\hat{X}/E$ is a \emph{%
semidefinable set}. Clearly, every definable set is, in particular, a
semidefinable set. The notion of Borel subset and definable function are the
same as in the case of definable sets. Thus, if $X=\hat{X}/E$ and $Y=\hat{Y}%
/F$ are semidefinable sets, $Z\subseteq X$ is a subset and $f:X\rightarrow Y$
is a function, then $f$ is \emph{definable} if it has a Borel lift $\hat{f}:%
\hat{X}\rightarrow \hat{Y}$, and $Z$ is Borel if there is a Borel $E$%
-invariant subset $\hat{Z}$ of $\hat{X}$ such that $Z=\hat{Z}/E$. The
category $\mathbf{SemiDSet}$ has semidefinable sets as objects and definable
functions as morphisms. Notice that, in particular, an isomorphism from $X$
to $Y$ in $\mathbf{SemiDSet}$ is a bijection $f:X\rightarrow Y$ such that
both $f$ and the inverse function $f^{-1}:Y\rightarrow X$ are definable.

\begin{lemma}
\label{Lemma:iso-semi}Suppose that $X=\hat{X}/E$ is a definable set, $Y=\hat{%
Y}/F$ is a semidefinable set. If $X$ and $Y$ are isomorphic in $\mathbf{%
SemiDSet}$, then $Y$ is a definable set.
\end{lemma}

\begin{proof}
Suppose that the Borel function $s_{X}:\hat{X}\rightarrow \hat{X}$ and the
assignment $C\mapsto \mathcal{E}_{C}$ witness that $E$ is idealistic. By
assumption, there exists a bijection $f:X\rightarrow Y$ such that $f$ has a
Borel lift $\alpha :\hat{X}\rightarrow \hat{Y}$, and $f^{-1}$ has a Borel
lift $\beta :\hat{Y}\rightarrow \hat{X}$. For $y,y^{\prime }\in \hat{Y}$ we
have that $yFy^{\prime }$ if and only if $\beta (y)E\beta \left( y^{\prime
}\right) $, whence $F$ is Borel. We now show that $F$ is idealistic.

Define an assignment $D\mapsto \mathcal{F}_{D}$ from $F$-classes to $\sigma $%
-filters, by setting $S\in \mathcal{F}_{D}$ if and only if $\alpha
^{-1}\left( S\right) \in \mathcal{E}_{C}$ where $f\left( C\right) =D$.
Consider also the Borel map $s_{Y}:=\alpha \circ s_{X}\circ \beta :\hat{Y}%
\rightarrow \hat{Y}$. Then it is easy to verify that $s_{Y}$ and the
assignment $D\mapsto \mathcal{F}_{D}$ witness that $F$ is idealistic.
\end{proof}

\begin{lemma}
\label{Lemma:Borel-target}Suppose that $X=\hat{X}/E$ and $Y=\hat{Y}/F$ are
semidefinable sets. Assume that there exists a definable bijection $%
f:X\rightarrow Y$ (which is not necessarily an isomorphism in $\mathbf{%
SemiDSet}$). If $E$ is Borel, then $F$ is Borel.
\end{lemma}

\begin{proof}
By assumption $E\subseteq \hat{X}\times \hat{X}$ is Borel, and $F\subseteq 
\hat{Y}\times \hat{Y}$ is analytic. Furthermore, $f$ has a Borel lift $\hat{f%
}:\hat{X}\rightarrow \hat{Y}$. Since $f$ is a bijection, we have that, for $%
y,y^{\prime }\in \hat{Y}$, 
\begin{equation*}
yFy^{\prime }\Leftrightarrow \forall x,x^{\prime }\in \hat{X}\text{, }\left(
(\hat{f}(x)Fy\wedge \hat{f}\left( x^{\prime }\right) Fy^{\prime
})\rightarrow xEx^{\prime }\right) \text{.}
\end{equation*}%
This shows that $F$ is co-analytic. As $F$ is also analytic, we have that $F$
is Borel.
\end{proof}

Lemma 3.7 in \cite{kechris_borel_2016} can be stated as the following
proposition, which generalizes items (2) and (3) in Proposition \ref%
{Proposition:DSet}.

\begin{proposition}[Kechris--Macdonald]
\label{Proposition:KM2}Let $X=\hat{X}/E$ be a definable set, $Y=\hat{Y}/F$
be semidefinable set such that $F$ is Borel, and $f:X\rightarrow Y$ be a
definable function. If $f$ is injective, then the range of $f$ a Borel
subset of $Y$. If $f$ is bijective, then the inverse function $%
f^{-1}:Y\rightarrow X$ is definable.
\end{proposition}

The following result is a consequence of Lemma \ref{Lemma:Borel-target} and
Proposition \ref{Proposition:KM2}.

\begin{corollary}
\label{Corollary:Kechris--MacDonald}Suppose that $X=\hat{X}/E$ is a
definable set, and $Y=\hat{Y}/F$ is a semidefinable set. If $f:X\rightarrow
Y $ is a definable bijection, then $Y$ is a definable set and $f$ is an
isomorphism in $\mathbf{DSet}$.
\end{corollary}

\begin{proof}
By Lemma \ref{Lemma:Borel-target}, $F$ is Borel. Whence, by Proposition \ref%
{Proposition:KM2}, $f$ is an isomorphism in $\mathbf{SemiDSet}$. Since $X$
is a definable set, it follows from Lemma \ref{Lemma:iso-semi} that $Y$ is
also a definable set, and $f$ is an isomorphism in $\mathbf{DSet}$.
\end{proof}

\subsection{Definable groups\label{Subsection:definable-groups}}

A definable group can be simply defined as a group in the category $\mathbf{%
DSet}$ in the sense of \cite[Section III.6]{mac_lane_categories_1998}. Thus,
a definable group is a definable set $G=\hat{G}/E$ that is also a group, and
such that the group operation $G\times G\rightarrow G$ is definable, and the
function $G\rightarrow G$, $x\mapsto x^{-1}$ is also definable. As in the
case of sets, we regard a standard Borel group as a particular instance of
definable group $G=\hat{G}/E$ where $G=\hat{G}$ is a standard Borel group
and $E$ is the relation of equality on $\hat{G}$. Thus, standard Borel
groups form a full subcategory of the category of definable groups.

Naturally, a semidefinable group will be a group in $\mathbf{SemiDSets}$,
i.e.\ a semidefinable set $G=\hat{G}/E$ that is also a group, and such that
the group operation $G\times G\rightarrow G$ is definable, and the function $%
G\rightarrow G$ that maps every element to its inverse is definable.

\begin{lemma}
\label{Lemma:definable-group}If $G=\hat{G}/E$ is a semidefinable group, then
the equivalence relation $E$ is Borel if and only if the identity element of 
$G$, which is the $E$-class $\left[ \ast \right] _{E}$ of some element $\ast 
$ of $\hat{G}$, is a Borel subset of $\hat{G}$.
\end{lemma}

\begin{proof}
Clearly, if $E$ is Borel, then $\left[ \ast \right] _{E}$ is Borel.
Conversely, suppose that $\left[ \ast \right] _{E}$ is Borel. If $m:\hat{G}%
\times \hat{G}\rightarrow \hat{G}$ and $\zeta :\hat{G}\rightarrow \hat{G}$
are Borel lifts of the group operation in $X$ and of the function that maps
each element to its inverse, respectively, then we have that $xEy$ if and
only if $m(x,\zeta (y))\in \lbrack \ast ]_{E}$. This shows that $E$ is Borel.
\end{proof}

\begin{corollary}
\label{Corollary:definable-group}Suppose that $G=\hat{G}/E$ is a
semidefinable group. If $E$ is the orbit equivalence relation of a Borel
action of a Polish group $H$ on the standard Borel space $\hat{G}$, then $G$
is a definable group.
\end{corollary}

\begin{proof}
By \cite[Theorem 5.2.1]{becker_descriptive_1996} one can assume that $\hat{G}
$ is a Polish $H$-space, and $E$ is the orbit equivalence relation of a
continuous $H$-action on $\hat{G}$. By \cite[Proposition 3.1.10]%
{gao_invariant_2009}, every $E$-class is Borel. Therefore $E$ is Borel by
Lemma \ref{Lemma:definable-group}. Furthermore, $E$ is idealistic by \cite[%
Proposition 5.4.10]{gao_invariant_2009}.
\end{proof}

\begin{remark}
\label{Remark:Polish-cover}A particular instance of definable group is
obtained as follows. Suppose that $G$ is a Polish group and $H$ is a Borel
Polishable subgroup. Let $E_{H}^{G}$ be the coset equivalence relation of $H$
in $G$. The quotient group $G/H$ is the quotient of $G$ by the equivalence
relation $E_{H}^{G}$. Since $H$ is Polishable, $E_{H}^{G}$ is the orbit
equivalence relation of a Borel action of a Polish group on $G$. Thus, $%
G/H=G/E_{H}^{G}$ is a definable group by Corollary \ref%
{Corollary:definable-group}. The definable groups obtained in this way are
called \emph{groups with a Polish cover }in \cite{bergfalk_ulam_2020}.
\end{remark}

\section{Strict C*-algebras\label{Section:strict}}

In this section we introduce the notion of strict Banach space and strict
C*-algebra and some of their properties. Briefly, a strict Banach space is a
Banach space whose unit ball is endowed with a Polish topology (called the
strict topology) that is coarser than the norm-topology and induced by a
sequence of bounded seminorms. A suitable semicontinuity requirement relates
the norm and the strict topology.\ A strict C*-algebra is a strict Banach
space that is also a C*-algebra with some suitable continuity requirement
relating the C*-algebra operations and the strict topology. The name is
inspired by the strict topology on the multiplier algebra of a separable
C*-algebra, which will be one of the main examples. Other examples are
Paschke dual algebras of separable C*-algebras.

\subsection{Strict Banach spaces}

Let $X$ be a Banach space. We denote by $\mathrm{\mathrm{Ball}}\left(
X\right) $ its unit ball. A seminorm $p$ on $X$ is bounded if $\left\Vert
p\right\Vert :=\mathrm{sup}_{x\in \mathrm{\mathrm{Ball}}\left( X\right)
}p(x)<\infty $. We say that $p$ is contractive if $\left\Vert p\right\Vert
\leq 1$.

\begin{definition}
\label{Definition:strictB1}A strict Banach space is a Banach space $%
\mathfrak{X}$ such that $\mathrm{\mathrm{Ball}}\left( \mathfrak{X}\right) $
is endowed with a topology (called the strict topology) such that, for some
sequence $\left( p_{n}\right) $ of contractive seminorms on $\mathfrak{X}$,
letting $d$ be the pseudometric on $\mathrm{\mathrm{Ball}}\left( \mathfrak{X}%
\right) $ defined by%
\begin{equation*}
d\left( x,y\right) =\sum_{n\in \omega }2^{-n}p_{n}\left( x-y\right) \text{,}
\end{equation*}%
one has that:

\begin{enumerate}
\item $d$ is a complete metric that induces the strict topology on $\mathrm{%
\mathrm{Ball}}\left( \mathfrak{X}\right) $;

\item $\mathrm{\mathrm{Ball}}\left( \mathfrak{X}\right) $ contains a
countable strictly dense subset;

\item $\left\Vert x\right\Vert =\mathrm{sup}_{n\in \omega }p_{n}(x)$ for
every $x\in \mathfrak{X}$.
\end{enumerate}
\end{definition}

\begin{example}
Suppose that $X$ is a separable Banach space. Then $X$ is a strict Banach
space where the strict topology on $\mathrm{\mathrm{Ball}}\left( X\right) $
is the norm topology.
\end{example}

\begin{example}
Suppose that $Y$ is a separable Banach space, and $Y^{\ast }$ is its Banach
space dual. Then $Y^{\ast }$ is a strict Banach space where the strict
topology on $\mathrm{\mathrm{Ball}}\left( Y^{\ast }\right) $ is the
weak*-topology.
\end{example}

Let $X$ be a seminormed space, and consider the cone $\mathcal{S}\left(
X\right) $ of bounded seminorms on $X$ as a complete metric space, with
respect to the metric defined by $d\left( p,q\right) =\mathrm{sup}_{x\in 
\mathrm{\mathrm{Ball}}\left( X\right) }\left\vert p(x)-q(x)\right\vert $.
For a subset $\mathfrak{S}\subseteq \mathcal{S}\left( X\right) $, we let $%
\sigma \left( X,\mathfrak{S}\right) $ be the topology on $\mathrm{\mathrm{%
Ball}}\left( X\right) $ generated by $\mathfrak{S}$. We denote by \textrm{%
Ball}$\left( \mathfrak{S}\right) $ the set of contractive seminorms in $%
\mathfrak{S}$. If $p\in \mathcal{S}\left( X\right) $, $\mathfrak{S}\subseteq 
\mathcal{S}\left( X\right) $, and $\left( x_{n}\right) $ is a sequence in $%
\mathrm{\mathrm{Ball}}\left( X\right) $, then we say that:

\begin{itemize}
\item $\left( x_{n}\right) _{n\in \omega }$ is $p$-Cauchy if for every $%
\varepsilon >0$ there exists $n_{0}\in \omega $ such that, for $n,m\geq
n_{0} $, $p\left( x_{n}-x_{m}\right) <\varepsilon $;

\item $\left( x_{n}\right) _{n\in \omega }$ is $\mathfrak{S}$-Cauchy if it
is $p$-Cauchy for every $p\in \mathfrak{S}$;

\item $\mathrm{\mathrm{Ball}}\left( X\right) $ is $\mathfrak{S}$-complete
if, for every sequence $\left( x_{n}\right) _{n\in \omega }$ in \textrm{Ball}%
$\left( X\right) $, if $\left( x_{n}\right) _{n\in \omega }$ is $\mathfrak{S}
$-Cauchy, then $\left( x_{n}\right) _{n\in \omega }$ is $\sigma \left( X,%
\mathfrak{S}\right) $-convergent to some element of $\mathrm{\mathrm{Ball}}%
\left( X\right) $.
\end{itemize}

The following lemma is elementary.

\begin{lemma}
\label{Lemma:strict-Cauchy}Suppose that $X$ is a seminormed space. Let $\tau 
$ be a topology on $\mathrm{\mathrm{Ball}}\left( X\right) $. Assume that $%
\mathfrak{S}$ and $\mathfrak{T}$ are two sets of bounded seminorms on $X$
such that the topologies $\sigma \left( X,\mathfrak{S}\right) $ and $\sigma
\left( X,\mathfrak{T}\right) $ on $\mathrm{\mathrm{Ball}}\left( X\right) $
coincide with $\tau $. Then, for a sequence $\left( x_{n}\right) $ in $%
\mathrm{\mathrm{Ball}}\left( X\right) $, $\left( x_{n}\right) $ is $%
\mathfrak{S}$-Cauchy if and only if it is $\mathfrak{T}$-Cauchy. In this
case, we say that $\left( x_{n}\right) $ is $\tau $-Cauchy. It follows that $%
\mathrm{\mathrm{Ball}}\left( X\right) $ is $\mathfrak{S}$-complete if and
only if it is $\mathfrak{T}$-complete. In this case, we say that $\mathrm{%
\mathrm{Ball}}\left( X\right) $ is $\tau $-complete.
\end{lemma}

In view of Lemma \ref{Lemma:strict-Cauchy} one can equivalently define a
strict Banach space as follows.

\begin{definition}
\label{Definition:strictB2}A strict Banach space is a Banach space $%
\mathfrak{X}$ such that $\mathrm{\mathrm{Ball}}\left( \mathfrak{X}\right) $
is endowed with a topology (called the strict topology) such that, for some
separable cone $\mathfrak{S}$ of bounded seminorms on $X$, one has that:

\begin{enumerate}
\item the strict topology on $\mathrm{\mathrm{Ball}}\left( \mathfrak{X}%
\right) $ is the $\sigma \left( \mathfrak{X},\mathfrak{S}\right) $-topology,
and $\mathrm{\mathrm{Ball}}\left( \mathfrak{X}\right) $ is strictly complete;

\item $\mathrm{\mathrm{Ball}}\left( \mathfrak{X}\right) $ contains a
countable strictly dense subset;

\item $\left\Vert x\right\Vert =\mathrm{sup}_{p\in \mathrm{\mathrm{Ball}}%
\left( \mathfrak{S}\right) }p(x)$ for every $x\in \mathfrak{X}$.
\end{enumerate}
\end{definition}

\begin{proposition}
Suppose that $\mathfrak{X}$ is a strict Banach space. Then $\mathrm{\mathrm{%
Ball}}\left( \mathfrak{X}\right) $ is a Polish topometric space when endowed
with the strict topology and the norm-distance.
\end{proposition}

\begin{proof}
By definition, the strict topology on $\mathrm{\mathrm{Ball}}\left( 
\mathfrak{X}\right) $ is Polish. Since the strict topology is induced by
bounded seminorms on $X$, it is coarser than the norm topology. The function 
$\left( x,y\right) \mapsto \left\Vert x-y\right\Vert $ is strictly
lower-semicontinuous, being the supremum of strictly continuous functions.
Since the norm on $X$ is complete, the distance $\left( x,y\right) \mapsto
\left\Vert x-y\right\Vert $ on \textrm{Ball}$\left( \mathfrak{X}\right) $ is
complete.
\end{proof}

Suppose that $\mathfrak{X}$ is a strict Banach space. We extend the strict
topology of \textrm{Ball}$\left( \mathfrak{X}\right) $ to any bounded subset
of $\mathfrak{X}$ by declaring the function%
\begin{equation*}
n\mathrm{\mathrm{Ball}}\left( \mathfrak{X}\right) \rightarrow \mathrm{%
\mathrm{Ball}}\left( \mathfrak{X}\right) \text{, }z\mapsto \frac{1}{n}z
\end{equation*}%
to be a homeomorphism with respect to the strict topology, where%
\begin{equation*}
n\mathrm{\mathrm{Ball}}\left( \mathfrak{X}\right) =\left\{ z\in \mathfrak{X}%
:\left\Vert z\right\Vert \leq n\right\} \text{.}
\end{equation*}
Then we have that addition and scalar multiplication on $\mathfrak{X}$ are
strictly continuous \emph{on bounded sets}, and the norm is strictly
lower-semicontinuous \emph{on bounded sets}. In particular $n\mathrm{\mathrm{%
Ball}}\left( X\right) $ is a strictly closed subspace of $m\mathrm{\mathrm{%
Ball}}\left( X\right) $ for $n\leq m$. Notice that, if $\mathfrak{Y}$ is a
norm-closed subspace of $X$ such that $\mathrm{\mathrm{Ball}}\left( 
\mathfrak{Y}\right) $ is strictly closed in \textrm{Ball}$\left( \mathfrak{X}%
\right) $, then $\mathfrak{Y}$ is a strict Banach space with the induced
norm and the induced strict topology.

\begin{definition}
\label{Definition:Borel}Let $\mathfrak{X}$ be a strict Banach space. The
(standard) Borel structure on $\mathfrak{X}$ is defined by declaring a
subset $A$ of $\mathfrak{X}$ to be Borel if and only if $A\cap n\mathrm{%
\mathrm{Ball}}\left( \mathfrak{X}\right) $ is Borel for every $n\geq 1$.
\end{definition}

Notice that the Borel structure on $\mathfrak{X}$ is standard, as $\mathfrak{%
X}$ is Borel isomorphic to the disjoint union of the standard Borel spaces $%
\left( n+1\right) \mathrm{\mathrm{Ball}}\left( \mathfrak{X}\right) \setminus
n\mathrm{\mathrm{Ball}}\left( \mathfrak{X}\right) $ for $n\geq 1$.

\begin{definition}
\label{Definition:strictL}If $\mathfrak{X}$ and $\mathfrak{Y}$ are strict
Banach spaces. A bounded linear map $T:\mathfrak{X}\rightarrow \mathfrak{Y}$
is \emph{contractive }if $\left\Vert T\right\Vert \leq 1$, and \emph{strict }%
if it is strictly continuous on bounded sets. A bounded seminorm $p$ on $%
\mathfrak{X}$ is \emph{strict }if it is strictly continuous on bounded sets.
\end{definition}

Clearly, strict Banach spaces form a category where the morphisms are the
strict contractive linear maps. Notice that, if $T$ is a strict, bijective,
and isometric linear map $T:\mathfrak{X}\rightarrow \mathfrak{Y}$ between
strict Banach spaces, then the inverse $T^{-1}:\mathfrak{Y}\rightarrow 
\mathfrak{X}$ is not necessarily strict, whence $T$ is not necessarily an
isomorphism in the category of strict Banach spaces. Nonetheless, $T:%
\mathfrak{X}\rightarrow \mathfrak{Y}$ is a \emph{Borel isomorphism}, as both 
$\mathfrak{X}$ and $\mathfrak{Y}$ are standard Borel spaces.

\begin{definition}
\label{Definition:strict-seminorms}Let $\mathfrak{X}$ be a strict Banach
space. Define $\mathcal{S}_{\mathrm{strict}}\left( \mathfrak{X}\right) $ to
be the space of bounded, strict seminorms on $\mathfrak{X}$.
\end{definition}

Notice that $\mathcal{S}_{\mathrm{strict}}\left( \mathfrak{X}\right) $ is a
closed subspace of the complete metric space $\mathcal{S}\left( \mathfrak{X}%
\right) $. A sequence $\left( x_{n}\right) $ in $\mathrm{\mathrm{Ball}}%
\left( \mathfrak{X}\right) $ is strictly convergent if and only if it is $%
\mathcal{S}_{\mathrm{strict}}\left( \mathfrak{X}\right) $-Cauchy. A bounded
linear map $T:\mathfrak{X}\rightarrow Y$ is strict if and only if $p\circ
T\in \mathcal{S}_{\mathrm{strict}}\left( \mathfrak{X}\right) $ for every $%
p\in \mathcal{S}_{\mathrm{strict}}\left( Y\right) $.

\begin{remark}
\label{Remark:non-countable}Suppose that $\mathfrak{X}$ is a strict Banach
space, and $\mathfrak{S}$ is a separable cone of bounded, strict seminorms
on $\mathfrak{X}$ that induces the strict topology on $\mathrm{\mathrm{Ball}}%
\left( \mathfrak{X}\right) $. One can consider the globally defined topology 
$\sigma \left( \mathfrak{X},\mathfrak{S}\right) $ on $\mathfrak{X}$, induced
by all the seminorms in $\mathfrak{S}$. This topology coincides with the
strict topology on $\mathrm{\mathrm{Ball}}\left( \mathfrak{X}\right) $.
However, it is not first countable on the whole of $\mathfrak{X}$, unless $%
\mathfrak{X}$ is a separable Banach space and the strict topology is equal
to the norm topology. Indeed, if the $\sigma \left( \mathfrak{X},\mathfrak{S}%
\right) $-topology on $\mathfrak{X}$ is first-countable, then $\left( 
\mathfrak{X},\sigma \left( \mathfrak{X},\mathfrak{S}\right) \right) $ is a
Frechet space. By the Open Mapping Theorem for Frechet spaces \cite[Theorem
8, page 120]{robertson_topological_1964}, any two comparable Frechet space
topologies must be equal. Thus, $\sigma \left( \mathfrak{X},\mathfrak{S}%
\right) $ equals the norm topology. In particular, the norm-topology on $%
\mathrm{\mathrm{Ball}}\left( \mathfrak{X}\right) $ is equal to the strict
topology, and it has a countable dense subset. Hence the norm-topology on $%
\mathfrak{X}$ is separable.
\end{remark}

For future reference, we record the easily proved observation that a uniform
limit of strictly continuous functions is strictly continuous.

\begin{lemma}
\label{Lemma:uniform-convergence}Suppose that $\mathfrak{X}$ and $\mathfrak{Y%
}$ are strict Banach spaces, and $A\subseteq \mathrm{\mathrm{Ball}}\left( 
\mathfrak{X}\right) $. Suppose that $f:A\rightarrow \mathfrak{Y}$ is an
function. Assume that there exists an sequence $\left( f_{n}\right) $ of
strictly continuous function $f_{n}:A\rightarrow \mathfrak{Y}$ such that%
\begin{equation*}
\mathrm{lim}_{n\rightarrow \infty }\mathrm{sup}_{x\in A}\left\Vert
f_{n}(x)-f(x)\right\Vert =0\mathrm{.}
\end{equation*}%
Then $f$ is strictly continuous.
\end{lemma}

A standard Baire Category argument shows that one can characterize bounded
subsets in terms of bounded, strict seminorms, as follows.

\begin{lemma}
\label{Lemma:bounded}Let $\mathfrak{X}$ be a strict Banach space. If $%
A\subseteq \mathfrak{X}$, then $A$ is bounded if and only if, for every $%
p\in \mathcal{S}_{\mathrm{strict}}\left( \mathfrak{X}\right) $, $p\left(
A\right) $ is a bounded subset of $\mathbb{R}$.
\end{lemma}

A natural way to obtain strict Banach spaces is via pairings.

\begin{definition}
\label{Definition:Banach-pairing}A Banach pairing is a bounded bilinear map $%
\left\langle \cdot ,\cdot \right\rangle :\mathfrak{X}\times Y\rightarrow Z$,
where $\mathfrak{X},Y,Z$ are Banach spaces. Define the $\sigma \left( 
\mathfrak{X},Y\right) $-topology to be the topology on $\mathfrak{X}$
generated by the cone $\mathfrak{S}_{Y}$ of bounded seminorms $x\mapsto
\left\Vert \left\langle x,y\right\rangle \right\Vert $ for $y\in Y$.
\end{definition}

The following lemma is an immediate consequence of the definition of strict
Banach space.

\begin{lemma}
\label{Lemma:Banach-pairing}Suppose that $\left\langle \cdot ,\cdot
\right\rangle :\mathfrak{X}\times Y\rightarrow Z$ is a Banach pairing.\
Assume that:

\begin{itemize}
\item $Y,Z$ are norm-separable Banach spaces;

\item for every $x_{0}\in \mathfrak{X}$,%
\begin{equation*}
\left\Vert x_{0}\right\Vert =\mathrm{sup}_{y\in \mathrm{\mathrm{Ball}}\left(
Y\right) }\left\Vert \left\langle x_{0},y\right\rangle \right\Vert \text{;}
\end{equation*}

\item $\mathrm{\mathrm{Ball}}\left( \mathfrak{X}\right) $ is $\sigma \left( 
\mathfrak{X},Y\right) $-complete;

\item $\mathrm{\mathrm{Ball}}\left( \mathfrak{X}\right) $ has a countable $%
\sigma \left( \mathfrak{X},Y\right) $-dense subset.
\end{itemize}

Then $\mathfrak{X}$ is a strict Banach space where the strict topology on $%
\mathrm{\mathrm{Ball}}\left( \mathfrak{X}\right) $ is the $\sigma \left( 
\mathfrak{X},Y\right) $-topology.
\end{lemma}

Suppose that $X$ is a norm-separable Banach space, and $\mathfrak{Y}$ is a
strict Banach spaces. A linear map $T:X\rightarrow \mathfrak{Y}$ is bounded
if it maps bounded sets to bounded sets or, equivalently,%
\begin{equation*}
\left\Vert T\right\Vert =\mathrm{sup}_{x\in \mathrm{\mathrm{Ball}}\left(
X\right) }\left\Vert T(x)\right\Vert <\infty \text{.}
\end{equation*}%
This defines a norm on the space $L\left( X,\mathfrak{Y}\right) $ of bounded
linear maps $X\rightarrow \mathfrak{Y}$.\ We also define the strict topology
on \textrm{Ball}$\left( L\left( X,\mathfrak{Y}\right) \right) $ to be the
topology of pointwise convergence in the strict topology of \textrm{Ball}$%
\left( \mathfrak{Y}\right) $. Then one can easily show the following.

\begin{proposition}
\label{Proposition:strict-operators}Suppose that $X$ is a norm-separable
Banach space, and $\mathfrak{Y}$ is a strict Banach space. Then $L\left( X,%
\mathfrak{Y}\right) $ is a strict Banach space.
\end{proposition}

\subsection{Strict C*-algebras}

We now introduce the notion of strict C*-algebra. Given a C*-algebra $A$, we
let $A_{\mathrm{sa}}$ be the set of its self-adjoint elements. We also
denote by $M_{n}\left( A\right) $ the C*-algebra of $n\times n$ matrices
over $A$, which can be identified with the tensor product $M_{n}\left( 
\mathbb{C}\right) \otimes A$. We refer to \cite%
{blackadar_operator_2006,davidson_algebras_1996,murphy_algebras_1990,pedersen_algebras_1979}
for fundamental notions and results from the theory of C*-algebras.

\begin{definition}
\label{Definition:strictC}A \emph{strict C*-algebra} is a C*-algebra $%
\mathfrak{A}$ such that, for every $n\geq 1$, $M_{n}\left( \mathfrak{A}%
\right) $ is also a strict Banach space satisfying the following properties:

\begin{enumerate}
\item the *-operation and the multiplication operation on $M_{n}\left( 
\mathfrak{A}\right) $ are strictly continuous on bounded sets;

\item the strict topology on $\mathrm{\mathrm{Ball}}\left( M_{n}\left( 
\mathfrak{A}\right) \right) $ is induced by the inclusion 
\begin{equation*}
\mathrm{\mathrm{Ball}}\left( M_{n}\left( \mathfrak{A}\right) \right)
\subseteq M_{n}\left( \mathrm{\mathrm{Ball}}\left( \mathfrak{A}\right)
\right) \text{,}
\end{equation*}%
where $\mathrm{\mathrm{Ball}}\left( \mathfrak{A}\right) $ is endowed with
the strict topology, and $M_{n}\left( \mathrm{\mathrm{Ball}}\left( \mathfrak{%
A}\right) \right) $ is endowed with the product topology.
\end{enumerate}
\end{definition}

\begin{example}
Suppose that $A$ is a \emph{separable} C*-algebra. Then we have that $A$ is
a strict C*-algebra where, for every $n\geq 1$, $\mathrm{\mathrm{Ball}}%
\left( M_{n}\left( A\right) \right) $ is endowed with the norm-topology.
\end{example}

Suppose that $\mathfrak{A}$ is a strict C*-algebra. Then, for every $n\geq 1$%
, $M_{n}\left( \mathfrak{A}\right) $ is also a strict C*-algebra. If $%
\mathfrak{A}$ is a strict C*-algebra, then we regard $\mathfrak{A}$ as a
standard Borel space with respect to the standard Borel structure induced by
the strict topology on $\mathrm{\mathrm{Ball}}\left( \mathfrak{A}\right) $
as in Definition \ref{Definition:Borel}. We say that a subset of $\mathfrak{A%
}$ is Borel if it is Borel with respect to such a Borel structure. We have
that the Borel structure on $M_{n}\left( \mathfrak{A}\right) $ (as a strict
C*-algebra) coincides with the product Borel structure.

\begin{definition}
\label{Definition:strict-ideal}Suppose that $\mathfrak{A}$ is a strict
unital C*-algebra. A \emph{strict ideal} of $\mathfrak{A}$ is a norm-closed
proper two-sided Borel ideal $\mathfrak{J}$ of $\mathfrak{A}$ that is also a
strict C*-algebra, and such that the inclusion map $\mathfrak{J}\rightarrow 
\mathfrak{A}$ is strict.
\end{definition}

\begin{remark}
In order for $\mathfrak{J}$ to be a strict ideal of $\mathfrak{A}$, we do
not require that $\mathrm{\mathrm{Ball}}\left( \mathfrak{J}\right) $ be
strictly closed in $\mathrm{\mathrm{Ball}}\left( \mathfrak{A}\right) $ nor
that the strict topology on $\mathrm{\mathrm{Ball}}\left( \mathfrak{J}%
\right) $ be the subspace topology induced by the strict topology of $%
\mathrm{\mathrm{Ball}}\left( \mathfrak{A}\right) $.
\end{remark}

\begin{example}
Suppose that $\mathfrak{A}$ is a strict unital C*-algebra and $J\subseteq 
\mathfrak{A}$ is a norm-closed and norm-separable proper two-sided ideal of $%
\mathfrak{A}$. Then $J$ is a strict ideal of $\mathfrak{A}$.
\end{example}

We regard \emph{strict} (unital)\emph{\ }C*-algebras as objects of a
category with \emph{strict }(unital)\emph{\ }*-homomorphisms as morphisms.
(Recall that a bounded linear map is strict if it is strictly continuous on
bounded sets.) If $\mathfrak{A}\subseteq \mathfrak{B}$, then we say that $%
\mathfrak{A}$ is \emph{strictly dense} in $\mathfrak{B}$ if $\mathrm{\mathrm{%
Ball}}\left( \mathfrak{A}\right) $ is dense in $\mathrm{\mathrm{Ball}}\left( 
\mathfrak{B}\right) $ with respect to the strict topology.

It follows from the axioms of a strict C*-algebra that, if $\mathfrak{A}$ is
a strict C*-algebra, and $p\left( x_{1},\ldots ,x_{n}\right) $ is a
*-polynomial, then $p$ defines a function $\mathfrak{A}^{n}\rightarrow 
\mathfrak{A}$ that is strictly continuous on bounded sets. In particular,
the sets of normal, self-adjoint, and positive elements of norm at most $1$
are strictly closed in $\mathrm{\mathrm{Ball}}\left( \mathfrak{A}\right) $.
If $f:[-1,1]^{n}\rightarrow n\mathrm{\mathrm{Ball}}\left( \mathbb{C}\right) $
is a continuous function, then $f$ induces by continuous functional calculus
and Lemma \ref{Lemma:uniform-convergence} a strictly continuous functions $%
\left( x_{1},\ldots ,x_{n}\right) \mapsto f\left( x_{1},\ldots ,x_{n}\right) 
$ from the strictly closed set of $n$-tuples of pairwise commuting
self-adjoint elements in $\mathrm{\mathrm{Ball}}\left( \mathfrak{A}\right) $
to $n\mathrm{\mathrm{Ball}}\left( \mathfrak{A}\right) $. Similarly, if $f:%
\mathrm{\mathrm{Ball}}\left( \mathbb{C}\right) ^{n}\rightarrow k\mathrm{%
\mathrm{Ball}}\left( \mathbb{C}\right) $ is a continuous function, then $f$
induces by continuous functional calculus and Lemma \ref%
{Lemma:uniform-convergence} a strictly continuous function $\left(
x_{1},\ldots ,x_{n}\right) \mapsto f\left( x_{1},\ldots ,x_{n}\right) $ from
the strictly closed set of $n$-tuples of pairwise commuting normal elements
in $\mathrm{\mathrm{Ball}}\left( \mathfrak{A}\right) $ to $k\mathrm{\mathrm{%
Ball}}\left( \mathfrak{A}\right) $.

Suppose that $\mathfrak{A}$ is a strict C*-algebra. Let $\mathrm{Normal}%
\left( \mathfrak{A}\right) $ be the Borel set of normal elements of $%
\mathfrak{A}$. For $a\in \mathrm{Normal}\left( \mathfrak{A}\right) $, the
spectrum $\sigma \left( a\right) $ is a closed subset of $\mathbb{C}$. We
consider the space $\mathrm{Closed}\left( \mathbb{C}\right) $ of closed
subsets of $\mathbb{C}$ as a standard Borel space endowed with the Effros
Borel structure \cite[Section 12.C]{kechris_classical_1995}. If $X$ is a
standard Borel space and $\mathcal{B}$ is a basis of open subsets of $%
\mathbb{C}$, then a function $\Phi :X\rightarrow \mathrm{Closed}\left( 
\mathbb{C}\right) $ is Borel if and only if, for every $U\in \mathcal{B}$, $%
\left\{ x\in X:\Phi (x)\cap U\neq \varnothing \right\} $ is Borel. The proof
of the following lemma is standard; see \cite[Lemma 1.6]%
{simon_operators_1995}.

\begin{lemma}
\label{Lemma:Borel-spectrum}Suppose that $\mathfrak{A}$ is a strict
C*-algebra. The function $\mathrm{Normal}\left( \mathfrak{A}\right)
\rightarrow \mathrm{Closed}\left( \mathbb{C}\right) $, $a\mapsto \sigma
\left( a\right) $ is Borel.
\end{lemma}

\begin{proof}
It suffices to show that the map $\mathrm{Normal}\left( \mathfrak{A}\right)
\cap \mathrm{\mathrm{Ball}}\left( \mathfrak{A}\right) \rightarrow \mathrm{%
Closed}\left( \mathbb{C}\right) $, $a\mapsto \sigma \left( a\right) $ is
Borel. Observe that $\mathbb{C}$ has a basis of open sets of the form $%
U_{f}:=\left\{ x\in \mathbb{C}:f(x)>0\right\} $ where $f:\mathbb{C}%
\rightarrow \lbrack 0,1]$ is a continuous function. For such a continuous
function $f:\mathbb{C}\rightarrow \lbrack 0,1]$, we have that%
\begin{equation*}
\left\{ a\in \mathrm{Normal}\left( \mathfrak{A}\right) \cap \mathrm{\mathrm{%
Ball}}\left( \mathfrak{A}\right) :\sigma \left( a\right) \cap U_{f}\neq
\varnothing \right\} =\left\{ a\in \mathrm{Normal}\left( \mathfrak{A}\right)
\cap \mathrm{\mathrm{Ball}}\left( \mathfrak{A}\right) :f\left( a\right) \neq
0\right\} \text{,}
\end{equation*}%
which is closed in $\mathrm{Normal}\left( \mathfrak{A}\right) \cap \mathrm{%
\mathrm{Ball}}\left( \mathfrak{A}\right) $. This concludes the proof.
\end{proof}

Suppose that $\mathfrak{A}$ is a strict C*-algebra. Fix $r\in \left(
0,1\right) $ and consider the set%
\begin{equation*}
X=\left\{ x\in \mathfrak{A}:\left\Vert 1-x\right\Vert \leq r\right\}
\subseteq 2\mathrm{\mathrm{Ball}}\left( \mathfrak{A}\right) \text{.}
\end{equation*}%
Then, for $x\in X$ we have that $x$ is invertible, $\left\Vert
x^{-1}\right\Vert \leq \frac{1}{1-r}$, and 
\begin{equation*}
x^{-1}=\sum_{n\in \omega }x^{n}\text{.}
\end{equation*}%
It follows from Lemma \ref{Lemma:uniform-convergence} that the function $%
X\rightarrow \frac{1}{1-r}\mathrm{\mathrm{Ball}}\left( \mathfrak{A}\right) $%
, $x\mapsto x^{-1}$ is strictly continuous.

More generally, suppose that $\Omega $ is an open subset of $\mathbb{C}$,
and $f:\Omega \rightarrow \mathbb{C}$ is a holomorphic function. Suppose
that $0\in \Omega $ and $r>0$ is such that $\left\{ z\in \mathbb{C}%
:\left\vert z\right\vert \leq r\right\} \subseteq \Omega $. Then $f$ admits
a Taylor expansion%
\begin{equation*}
f(z)=\sum_{n=0}^{\infty }a_{n}z^{n}
\end{equation*}%
that converges uniformly for $\left\vert z\right\vert \leq r$ \cite[Chapter
5, Theorem 3 and Chapter 2, Theorem 2]{ahlfors_complex_1978}. Fix $b_{0}\in 
\mathfrak{A}$ and set%
\begin{equation*}
X:=\left\{ x\in \mathfrak{A}:\left\Vert x-b_{0}\right\Vert \leq r\right\}
\subseteq \left( 1+\left\Vert b_{0}\right\Vert \right) \mathrm{\mathrm{Ball}}%
\left( \mathfrak{A}\right)
\end{equation*}%
Then for $x\in X$, 
\begin{equation*}
f\left( x-b_{0}\right) :=\sum_{n=0}^{\infty }a_{n}\left( x-b_{0}\right)
^{n}\in \mathfrak{A}\text{;}
\end{equation*}%
see \cite[Lemma 4.1.11]{pedersen_analysis_1989}. Furthermore, the function $%
X\rightarrow c\mathrm{\mathrm{Ball}}\left( \mathfrak{A}\right) $, $x\mapsto
f\left( x-b_{0}\right) $ is strictly continuous on $X$ by Lemma \ref%
{Lemma:uniform-convergence}, where $c=\mathrm{sup}\left\{ \left\vert
f(z)\right\vert :\left\vert z\right\vert \leq r\right\} $.

\subsection{Multiplier algebras\label{Subsection:multiplier}}

Suppose that $A$ is a separable C*-algebra. A \emph{double centralizer} for $%
A$ is a pair $\left( L,R\right) $ of bounded linear maps $L,R:A\rightarrow A$
such that $\left\Vert L\right\Vert =\left\Vert R\right\Vert $ and $%
L(x)y=xR(y)$ for every $x,y\in A$. Let $M\left( A\right) $ be the set of
double centralizers for $A$. Then $M\left( A\right) $ is a C*-algebra with
respect to the operations%
\begin{equation*}
\left( L_{1},R_{1}\right) +\left( L_{2},R_{2}\right) =\left(
L_{1}+L_{2},R_{1}+R_{2}\right)
\end{equation*}%
\begin{equation*}
\left( L_{1},R_{1}\right) \left( L_{2},R_{2}\right) =\left(
L_{1}L_{2},R_{2}R_{1}\right)
\end{equation*}%
\begin{equation*}
\lambda \left( L,R\right) =\left( \lambda L,\lambda R\right)
\end{equation*}%
\begin{equation*}
\left( L,R\right) ^{\ast }=\left( R^{\ast },L^{\ast }\right)
\end{equation*}%
and the norm%
\begin{equation*}
\left\Vert \left( L,R\right) \right\Vert =\left\Vert L\right\Vert
=\left\Vert R\right\Vert
\end{equation*}%
for $\left( L,R\right) ,\left( L_{1},R_{1}\right) ,\left( L_{2},R_{2}\right)
\in M\left( A\right) $ and $\lambda \in \mathbb{C}$. The \emph{strict
topology} on $\mathrm{\mathrm{Ball}}\left( M\left( A\right) \right) $ is the
topology of pointwise convergence, namely the topology induced by the
seminorms%
\begin{equation*}
p_{a}:\left( L,R\right) \mapsto \max \{\left\Vert L\left( a\right)
\right\Vert ,\left\Vert R\left( a\right) \right\Vert \}
\end{equation*}%
for $a\in A$.

An element $a\in A$ can be identified with the multiplier $\left(
L_{a},R_{a}\right) \in M\left( A\right) $ defined by setting $L_{a}(x)=ax$
and $R_{a}(x)=xa$ for $x\in X$. This allows one to regard $A$ as an
essential ideal of $M\left( A\right) $. (An ideal $J$ of a C*-algebra $B$ is
essential if $J^{\bot }:=\left\{ b\in B:bJ=0\right\} $ is zero or,
equivalently, $J$ has nonzero intersection with every nonzero ideal of $B$.)
If $\left( v_{n}\right) _{n\in \omega }$ is an approximate unit for $A$ \cite%
[Definition 1.7.1]{higson_analytic_2000} then, by definition, $\left(
v_{n}\right) $ strictly converges to $1$ in $\mathrm{\mathrm{Ball}}\left(
M\left( A\right) \right) $. In particular, $\mathrm{\mathrm{Ball}}\left(
A\right) $ is strictly dense in $\mathrm{\mathrm{Ball}}\left( M\left(
A\right) \right) $.

If $\left( x_{i}\right) _{i\in \omega }$ is a strictly Cauchy sequence in $%
\mathrm{\mathrm{Ball}}\left( M\left( A\right) \right) $, in the sense that $%
\left( x_{i}\right) _{i\in \omega }$ is $p_{a}$-Cauchy for every $a\in A$,
then setting%
\begin{equation*}
L\left( a\right) :=\mathrm{lim}_{i\rightarrow \infty }x_{i}a
\end{equation*}%
\begin{equation*}
R\left( a\right) :=\mathrm{lim}_{i\rightarrow \infty }ax_{i}
\end{equation*}%
for $a\in A$ defines a double centralizer $\left( L,R\right) \in \mathrm{%
\mathrm{Ball}}\left( M\left( A\right) \right) $ that is the strict limit of $%
\left( x_{i}\right) _{i\in \omega }$ in $\mathrm{\mathrm{Ball}}\left(
M\left( A\right) \right) $. For $n\geq 1$, one can identify $M_{n}\left(
M\left( A\right) \right) $ with $M\left( M_{n}\left( A\right) \right) $ and
consider the corresponding strict topology. From the above remarks and Lemma %
\ref{Lemma:Banach-pairing}, one easily obtains the following; see \cite[%
Chapter 13]{farah_combinatorial_2019} or \cite[Chapter 2]%
{wegge-olsen_theory_1993}.

\begin{proposition}
\label{Proposition:multiplier}Let $A$ be a separable C*-algebra. Then $%
M\left( A\right) $ is a strict unital C*-algebra containing $A$ as a
strictly dense essential strict ideal where, for every $n\geq 1$, the strict
topology on $\mathrm{\mathrm{Ball}}\left( M\left( A\right) \right) $ is as
described above, and $M_{n}\left( M\left( A\right) \right) $ is identified
with $M\left( M_{n}\left( A\right) \right) $.
\end{proposition}

\begin{example}
When $A$ is the algebra $K\left( H\right) $ of compact operators on a
separable Hilbert space, then $M\left( A\right) =B\left( H\right) $ and the
strict topology on $\mathrm{\mathrm{Ball}}\left( B\left( H\right) \right) $
is the strong-* topology \cite[Proposition I.8.6.3]{blackadar_operator_2006}.
\end{example}

\begin{example}
One can also regard $B\left( H\right) $ as the dual space of the Banach
space $\mathfrak{L}^{1}\left( H\right) $ of trace-class operators. This
turns $B\left( H\right) $ into a strict Banach space, where the strict
topology on $\mathrm{\mathrm{Ball}}\left( B\left( H\right) \right) $ is the
weak* topology, which coincides with the weak operator topology \cite[%
Definition I.8.6.2]{blackadar_operator_2006}. As the identity map $\mathrm{%
\mathrm{Ball}}\left( B\left( H\right) \right) \rightarrow \mathrm{\mathrm{%
Ball}}\left( B\left( H\right) \right) $ is strong-*--weak continuous, the
strong-* topology and weak operator topology on $\mathrm{\mathrm{Ball}}%
\left( B\left( H\right) \right) $ define the same standard Borel structure
on $B\left( H\right) $.
\end{example}

One can define as above the strict topology on the whole multiplier algebra $%
M\left( A\right) $ to be the topology of pointwise convergence of double
multipliers. However, this topology on $M\left( A\right) $ is not first
countable whenever $A$ is not unital; see Remark \ref{Remark:non-countable}. 

Suppose that $A$ is a separable C*-algebra, and $X$ is a compact metrizable
space. One can then consider the separable C*-algebra $C\left( X,A\right) $
of continuous functions $X\rightarrow A$. Let also $C_{\beta }\left(
X,M\left( A\right) \right) $ be the C*-algebra of \emph{strictly continuous }%
bounded functions $X\rightarrow M\left( A\right) $. There is an obvious
unital *-homomorphism $C_{\beta }\left( X,M\left( A\right) \right)
\rightarrow M\left( C\left( X,A\right) \right) $, where $C_{\beta }\left(
X,M\left( A\right) \right) $ acts on $C\left( X,A\right) $ by pointwise
multiplication. The unital *-homomorphism $C_{\beta }\left( X,M\left(
A\right) \right) \rightarrow M\left( C\left( X,A\right) \right) $ is in fact
a *-isomorphism \cite[Corollary 3.4]{akemann_multipliers_1973}. We can thus
identify $C_{\beta }\left( X,M\left( A\right) \right) $ with $M\left(
C\left( X,A\right) \right) $ and regard it as a strict C*-algebra. Observe
that, for $t\in X$, the function $\mathrm{\mathrm{Ball}}\left( C_{\beta
}\left( X,M\left( A\right) \right) \right) \rightarrow \mathrm{\mathrm{Ball}}%
\left( M\left( A\right) \right) $, $f\mapsto f\left( t\right) $ is strictly
continuous. We let $C\left( X,M\left( A\right) \right) $ be the C*-algebra
of \emph{norm-continuous }functions $X\rightarrow M\left( A\right) $, which
is a C*-subalgebra of $C_{\beta }\left( X,M\left( A\right) \right) $.

\begin{lemma}
Suppose that $A$ is a separable C*-algebra, and $X$ is a compact metrizable
space. Then $C\left( X,M\left( A\right) \right) $ is a Borel subset of $%
C_{\beta }\left( X,M\left( A\right) \right) $.
\end{lemma}

\begin{proof}
Fix a compatible metric $d$ on $X$, and a countable dense subset $X_{0}$ of $%
X$. Clearly, it suffices to show that $\mathrm{\mathrm{Ball}}\left( C\left(
X,M\left( A\right) \right) \right) $ is a Borel subset of $\mathrm{\mathrm{%
Ball}}\left( C_{\beta }\left( X,M\left( A\right) \right) \right) $. Fix, for
every $k\in \omega $, a finite cover $\left\{ A_{0}^{k},\ldots ,A_{\ell
_{k}-1}^{k}\right\} $ of $X$ consisting of open sets of diameter less than $%
2^{-k}$, and fix elements $t_{i}^{k}\in A_{i}^{k}$ for $i<\ell _{k}$. We
have that a strictly continuous function $f:X\rightarrow \mathrm{\mathrm{Ball%
}}\left( M\left( A\right) \right) $ is norm-continuous if and only if, for
every $n\in \omega $ there exists $k\in \omega $ such that, for every $%
i<\ell _{k}$ and $s\in A_{i}^{k}$, $\left\Vert f\left( s\right) -f\left(
t_{i}^{k}\right) \right\Vert \leq 2^{-k}$. Since $2^{-k}\mathrm{\mathrm{Ball}%
}\left( M\left( A\right) \right) $ is strictly closed and $f$ is strictly
continuous, we have that $f$ is norm-continuous if and only if for every $%
n\in \omega $ there exists $k\in \omega $ such that, for every $i<\ell _{k}$
and for every $s\in A_{i}^{k}\cap X_{0}$, $\left\Vert f\left( s\right)
-f\left( t_{i}^{k}\right) \right\Vert \leq 2^{-k}$. This shows that the set
of norm-continuous functions is Borel.
\end{proof}

\begin{corollary}
\label{Corollary:norm-continuous-paths}Suppose that $A$ is a separable
C*-algebra. Then the set $C\left( [0,1],M\left( A\right) \right) $ of
norm-continuous paths $[0,1]\rightarrow M\left( A\right) $ is a Borel subset
of $C_{\beta }\left( [0,1],M\left( A\right) \right) $.
\end{corollary}

Suppose that $A$ and $B$ are separable C*-algebra. A \emph{morphism }from $A$
to $B$ in the sense of \cite%
{woronowicz_pseudospaces_1980,woronowicz_unbounded_1991,woronowicz_operator_1992,woronowicz_algebras_1995}
is a *-homomorphism $\varphi :A\rightarrow M(B)$ such that $\varphi \left(
A\right) B$ is norm-dense in $B$. (This is called $S$-morphism in \cite[%
Definition 0.2.7]{vallin_algebres_1985} and a nondegenerate *-homomorphism
in \cite{lance_hilbert_1995}.) We recall the well-known fact that there is a
correspondence between morphisms from $A$ to $B$ and strict unital
*-homomorphisms $M\left( A\right) \rightarrow M(B)$; see \cite[Proposition
2.1]{lance_hilbert_1995}.

\begin{lemma}
\label{Lemma:W-morphism}Let $A$ and $B$ be separable C*-algebra.

\begin{itemize}
\item Suppose that $\psi :M(A)\rightarrow M(B)$ is a strict unital
*-homomorphism. Then $\psi |_{A}$ is a morphism from $A$ to $B$.

\item Conversely, if $\varphi $ is a morphism from $A$ to $B$, then $\varphi 
$ extends to a unique strict unital *-homomorphism $\bar{\varphi}%
:M(A)\rightarrow M(B)$. If $\varphi $ is injective, then $\bar{\varphi}$ is
injective.

\item If $\left( e_{n}\right) $ is an approximate unit for $A$, then a
*-homomorphism $\varphi :A\rightarrow M(B)$ is a morphism from $A$ to $B$ if
and only if $\left( \varphi \left( e_{n}\right) \right) $ strictly converges
to $1$.
\end{itemize}
\end{lemma}

A further characterization of morphisms is provided in \cite[Lemme 0.2.6]%
{vallin_algebres_1985} and \cite[Proposition 1.1]{iorio_hopf_1980}. It
follows from Lemma \ref{Lemma:W-morphism} that the composition of morphisms $%
A\rightarrow B$ and $B\rightarrow C$ is meaningful, and it gives a morphism $%
A\rightarrow C$.

Suppose that $A,B$ are separable C*-algebras. A *-homomorphism $\varphi
:A\rightarrow M(B)$ is\emph{\ quasi-unital} \cite[Definition 1.3.13]%
{jensen_elements_1991}\emph{\ }(also called \emph{strict }\cite[page 49]%
{lance_hilbert_1995}) if there exists a projection $p_{\varphi }\in M(B)$,
called the \emph{relative unit }of $\varphi $, such that $\overline{\varphi
(A)B}=p_{\varphi }B$. One has the following generalization of Lemma \ref%
{Lemma:W-morphism}; see \cite[Corollary 5.7]{lance_hilbert_1995}.

\begin{lemma}
\label{Lemma:strict-morphism}Let $A$ and $B$ be separable C*-algebra.

\begin{itemize}
\item Suppose that $\psi :M(A)\rightarrow M(B)$ is a strict *-homomorphism.
Then $\psi |_{A}$ is a quasi-unital *-homomorphism from $A$ to $M(B)$ with
relative unit $\psi \left( 1\right) $.

\item Conversely, if $\varphi $ is a quasi-unital *-homomorphism from $A$ to 
$M(B)$ with relative unit $p_{\varphi }$, then $\varphi $ extends to a
unique strict *-homomorphism $\bar{\varphi}:M(A)\rightarrow M(B)$ with $\bar{%
\varphi}=p_{\varphi }$. If $\varphi $ is injective, then $\bar{\varphi}$ is
injective.

\item If $\left( e_{n}\right) $ is an approximate unit for $A$, then a
*-homomorphism $\varphi :A\rightarrow M(B)$ is quasi-unital if and only if $%
\left( \varphi \left( e_{n}\right) \right) $ is strictly Cauchy.
\end{itemize}
\end{lemma}

We now observe that the category of multiplier algebras of separable
C*-algebras, regarded as a full subcategory of the category of strict unital
C*-algebras, can be regarded as a Polish category; see Section \ref%
{Subsection:Polish-categories}. This means that, for every separable
C*-algebras $A$ and $B$, the set \textrm{Mor}$\left( M(A),M(B)\right) $ of
strict unital *-homomorphisms $M(A)\rightarrow M(B)$ is a Polish space, and
composition of morphisms is a continuous function.

Following \cite{woronowicz_algebras_1995} we consider $\mathrm{Mor}\left(
M(A),M(B)\right) $ as endowed with the topology of pointwise strict
convergence. This is the subspace topology induced by regarding, as in Lemma %
\ref{Lemma:W-morphism}, $\mathrm{Mor}\left( M(A),M(B)\right) $ as a subspace
of $\mathrm{\mathrm{Ball}}\left( L\left( A,M(B)\right) \right) $, where $%
L\left( A,M(B)\right) $ is the space of bounded linear maps from $A$ to $%
M(B) $. (Recall that, if $X$ is a Banach space and $\mathfrak{Y}$ is a
strict Banach space, then the space $L\left( X,\mathfrak{Y}\right) $ of
bounded linear maps $X\rightarrow \mathfrak{Y}$ is a strict Banach space
when $\mathrm{\mathrm{Ball}}\left( L\left( X,\mathfrak{Y}\right) \right) $
is endowed with the topology of pointwise strict convergence; see
Proposition \ref{Proposition:strict-operators}.) As $\mathrm{Mor}\left(
M(A),M(B)\right) $ is a $G_{\delta }$ subset of $\mathrm{\mathrm{Ball}}%
\left( L\left( A,M(B)\right) \right) $, it is a Polish space with the
induced topology. It is easy to see that this turns the category of
muliplier algebras of separable C*-algebras into a Polish category.

If $A,B$ are separable C*-algebras, then the space $\mathrm{Iso}\left(
M(A),M(B)\right) $ of isomorphisms $M(A)\rightarrow M(B)$ in the category of
strict unital C*-algebras endowed with the Polish topology as in Lemma \ref%
{Lemma:Iso-Polish} can be identified, via the correspondence given by Lemma %
\ref{Lemma:W-morphism}, with the space $\mathrm{Iso}\left( A,B\right) $ of
*-isomorphisms $A\rightarrow B$ endowed with the topology of pointwise
norm-convergence.

Consider now the category of locally compact second countable Hausdorff
spaces, where a morphism is simply a continuous map. Given locally compact
second countable Hausdorff spaces $X,Y$, let $\mathrm{\mathrm{Mor}}\left(
X,Y\right) $ be the set of all continuous maps $X\rightarrow Y$. This is
endowed with a Polish topology called the compact-open topology, that has as
subbasis of open sets the sets of the form%
\begin{equation*}
\left( K,U\right) :=\left\{ f\in \mathrm{\mathrm{Mor}}\left( X,Y\right)
:f\left( K\right) \subseteq U\right\}
\end{equation*}%
for a compact subset $K$ of $X$ and an open subset $U$ of $Y$. This turns
the category of locally compact second countable Hausdorff spaces and
continuous maps into a Polish category. We let $\mathrm{Iso}\left(
X,Y\right) \subseteq \mathrm{\mathrm{Mor}}\left( X,Y\right) $ be the set of
homeomorphisms $X\rightarrow Y$.\ The Polish topology induced on $\mathrm{Iso%
}\left( X,Y\right) $ as in Lemma \ref{Lemma:Iso-Polish} was shown in \cite[%
Theorem 5]{arens_topologies_1946}, where it is called the $g$-topology, to
have as subbasis of open sets the sets of then $\left( K,Y\setminus L\right) 
$ where $K,L$ are closed sets and at least one between $K$ and $L$ is
compact. For a locally compact second countable Hausdorff space $X$, let $%
X^{+}$ be its one-point compactification, obtained by adjoining a point at
infinity $\infty _{X}$. Each $f\in \mathrm{Iso}\left( X,Y\right) $ admits a
unique extension to $f^{+}\in \mathrm{Iso}(X^{+},Y^{+})$ that fixes the
point at infinity, in the sense that $f^{+}\left( \infty _{X}\right) =\infty
_{Y}$. By \cite[Theorem 5]{arens_topologies_1946}, the assignment $f\mapsto
f^{+}$ defines a homeomorphism from $\mathrm{Iso}\left( X,Y\right) $ onto
the closed subset of $\mathrm{Iso}(X^{+},Y^{+})$ consisting of the
homeomorphisms that fix the point at infinity.

Given a locally compact second countable Hausdorff space $X$, we let $%
C_{0}\left( X\right) $ be the separable C*-algebra of continuous
complex-valued functions on $X$ vanishing at infinity. Its multiplier
algebra $M\left( C_{0}\left( X\right) \right) $ is the algebra $C_{b}\left(
X\right) $ of \emph{bounded }continuous complex-valued functions on $X$. The
unit ball $\mathrm{\mathrm{Ball}}\left( C_{b}\left( X\right) \right) $ of $%
C_{b}\left( X\right) =M\left( C_{0}\left( X\right) \right) $ endowed with
the strict topology can be identified with the space $\mathrm{\mathrm{Mor}}%
\left( X,\mathrm{\mathrm{Ball}}\left( \mathbb{C}\right) \right) $ of
continuous functions $X\rightarrow \mathrm{\mathrm{Ball}}\left( \mathbb{C}%
\right) $ endowed with the compact-open topology. Every separable \emph{%
commutative }C*-algebra $A$ is isomorphic to $C_{0}(\hat{A})$, where $\hat{A}
$ is the locally compact second countable Hausdorff space of nonzero
homomorphisms $A\rightarrow \mathbb{C}$ (the spectrum of $A$).

A continuous map $f:X\rightarrow Y$ induces a strict unital *-homomorphism $%
C_{b}\left( Y\right) \rightarrow C_{b}\left( X\right) $ given by $\varphi
_{f}:C_{b}\left( Y\right) \rightarrow C_{b}\left( X\right) $, $a\mapsto
a\circ f$. This defines a \emph{fully faithful} contravariant functor from
the category of locally compact second countable Hausdorff spaces to the
category of strict unital C*-algebras. In fact, the assignment $\mathrm{%
\mathrm{Mor}}\left( X,Y\right) \rightarrow \mathrm{\mathrm{Mor}}\left(
C_{b}\left( Y\right) ,C_{b}\left( X\right) \right) $, $f\mapsto \varphi _{f}$
is a homeomorphism, where \textrm{Mor}$\left( X,Y\right) $ is endowed as
above with the compact-open topology and $\mathrm{\mathrm{Mor}}\left(
C_{b}\left( X\right) ,C_{b}\left( Y\right) \right) $ is endowed with the
topology of pointwise strict convergence. Thus, by Lemma \ref%
{Lemma:topological-equivalence}, the assignment $X\rightarrow C_{b}\left(
X\right) $ is a \emph{contravariant topological equivalence }of categories
from the Polish category of locally compact second countable Hausdorff
spaces to the Polish category of multiplier algebras of \emph{commutative }%
separable C*-algebras; see Definition \ref%
{Definition:topological-equivalence}.

\subsection{Essential commutants and Paschke dual algebras}

Suppose that $B$ is a separable C*-algebra, and $C\subseteq M(B)$ is a
separable C*-subalgebra. Define then the \emph{essential commutant} $%
\mathfrak{D}\left( C\right) $ of $C$ in $M(B)$ to be the C*-algebra 
\begin{equation*}
\left\{ x\in M(B):\forall c\in B,\left[ x,c\right] \in B\right\} \text{,}
\end{equation*}%
where $\left[ x,c\right] $ is the commutator $xc-cx$. Define the strict
topology on \textrm{Ball}$\left( \mathfrak{D}\left( C\right) \right) $ to be
the topology generated by the seminorms%
\begin{equation*}
x\mapsto \max \left\{ \left\Vert xb\right\Vert ,\left\Vert bx\right\Vert
,\left\Vert \left[ x,c\right] \right\Vert \right\}
\end{equation*}%
for $c\in C$ and $b\in B$. If $\left( v_{n}\right) _{n\in \omega }$ is a
approximate unit for $B$ that is approximately central for $C$ \cite[%
Definition 3.2.4]{higson_analytic_2000}, then $\left( v_{n}\right) _{n\in
\omega }$ converges strictly to $1$ in \textrm{Ball}$\left( \mathfrak{D}%
(B)\right) $.

We have that $\mathfrak{D}\left( C\right) $ is strictly complete. Indeed,
consider a strictly Cauchy sequence $\left( x_{i}\right) _{i\in \omega }$ in 
\textrm{Ball}$\left( \mathfrak{D}\left( C\right) \right) $. Then we have
that $\left( x_{i}\right) _{i\in \omega }$ converges to some $x\in \mathrm{%
\mathrm{Ball}}\left( M(B)\right) $ in the strict topology of $M(B)$. For
every $c\in C$, the sequence $\left( \left[ x_{i},c\right] \right) _{i\in
\omega }$ is norm-Cauchy in $B$, whence it norm-converges to some element of 
$B$, which must be equal to $\left[ x,c\right] $. This shows that $x\in 
\mathrm{\mathrm{Ball}}\left( \mathfrak{D}\left( C\right) \right) $ is the
strict limit of $\left( x_{i}\right) _{i\in \omega }$ in $\mathrm{\mathrm{%
Ball}}\left( \mathfrak{D}\left( C\right) \right) $. For $n\geq 1$, we can
identify $M_{n}\left( \mathfrak{D}\left( C\right) \right) $ with $\mathfrak{D%
}\left( \Delta _{n}\left( C\right) \right) \subseteq M\left( M_{n}(B)\right) 
$, where $\Delta _{n}\left( C\right) \subseteq M_{n}(B)$ is the image of $C$
under the diagonal embedding $\Delta _{n}:B\rightarrow M_{n}(B)$. From the
above remarks and Lemma \ref{Lemma:Banach-pairing} we thus obtain the
following.

\begin{proposition}
\label{Proposition:essential-commutant}Let $B$ be a separable C*-algebra,
and let $C\subseteq M(B)$ be a separable C*-subalgebra. Let $\mathfrak{D}%
\left( C\right) $ be the corresponding essential commutant. Then $\mathfrak{D%
}\left( C\right) $ is a strict C*-algebra containing $B$ as a strictly dense
essential strict ideal where, for every $n\geq 1$, $M_{n}\left( \mathfrak{D}%
\left( C\right) \right) $ is identified with $\mathfrak{D}\left( \Delta
_{n}\left( C\right) \right) $, and $\mathrm{\mathrm{Ball}}\left( \mathfrak{D}%
\left( \Delta _{n}\left( C\right) \right) \right) $ is endowed with the
strict topology described above.
\end{proposition}

Suppose now as above that $B$ is a separable C*-algebra, and $C\subseteq
M(B) $ is a separable C*-subalgebra. Let also $I\subseteq C$ be a closed
two-sided ideal. Define the essential annihilator 
\begin{equation*}
\mathfrak{D}\left( C//I\right) =\left\{ x\in \mathfrak{D}\left( C\right)
:\forall a\in I,xa\in B\right\} \text{,}
\end{equation*}%
which is a closed two-sided ideal of $\mathfrak{D}\left( C\right) $. The
strict topology on $\mathrm{\mathrm{Ball}}\left( \mathfrak{D}\left(
C//I\right) \right) $ is the topology generated by the seminorms%
\begin{equation*}
x\mapsto \mathrm{\max }\left\{ \left\Vert xb\right\Vert ,\left\Vert
bx\right\Vert ,\left\Vert \left[ x,c\right] \right\Vert \right\}
\end{equation*}%
for $b\in B\cup I$ and $c\in C$. A straightforward argument as above gives
the following.

\begin{proposition}
\label{Proposition:essential-annihilator}Let $B$ be a separable C*-algebra,
let $C\subseteq M(B)$ be a separable C*-subalgebra, and $I\subseteq C$ be a
closed two-sided ideal. Let $\mathfrak{D}\left( C\right) $ be the
corresponding essential commutant, and $\mathfrak{D}\left( C//I\right) $ be
the essential annihilator. Then $\mathfrak{D}\left( C//I\right) $ is a
strict ideal of $\mathfrak{D}\left( C\right) $, where for every $n\geq 1$, $%
M_{n}\left( \mathfrak{D}\left( C//I\right) \right) $ is identified with $%
\mathfrak{D}\left( \Delta _{n}\left( C\right) //\Delta _{n}\left( I\right)
\right) $ and $\mathrm{\mathrm{Ball}}\left( \mathfrak{D}\left( \Delta
_{n}\left( C\right) //\Delta _{n}\left( I\right) \right) \right) $ is
endowed with the strict topology described above.
\end{proposition}

\begin{example}
\label{Exampl:Paschke-dual}Suppose that $A$ is a separable unital
C*-algebra, $J$ is a closed two-sided ideal of $A$, and $\rho :A\rightarrow
B\left( H\right) $ is a nondegenerate representation of $A$ that is \emph{%
ample}, in the sense that $\rho (A)\cap K\left( H\right) =\left\{ 0\right\} $%
. We regard $B\left( H\right) $ as the multiplier algebra of $K\left(
H\right) $. The\textrm{\ }\emph{Paschke dual} $\mathfrak{D}_{\rho }(A)$ as
defined in \cite[Definition 5.1.1]{higson_analytic_2000} is the essential
commutant $\mathfrak{D}\left( \rho (A)\right) $ of $\rho (A)$ inside $%
B\left( H\right) $; see also \cite{paschke_theory_1981}. The relative dual
algebra $\mathfrak{D}_{\rho }\left( A//J\right) $ as defined in \cite[%
Definition 5.3.2]{higson_analytic_2000} is the strict ideal $\mathfrak{D}%
\left( \rho (A)//\rho \left( J\right) \right) $ of $\mathfrak{D}_{\rho }(A)=%
\mathfrak{D}\left( \rho (A)\right) $.
\end{example}

\subsection{Homotopy of projections\label{Subsection:strict-pair}}

Suppose that $\mathfrak{A}$ is a strict unital C*-algebra. Recall that a
strict ideal of $\mathfrak{A}$ is a proper norm-closed Borel two-sided ideal 
$\mathfrak{J}$ of $\mathfrak{A}$ that is also a strict C*-algebra and such
that the inclusion map $\mathfrak{J}\rightarrow \mathfrak{A}$ is a strict
*-homomorphism.

\begin{definition}
A strict (unital) C*-pair is a pair $\left( \mathfrak{A},\mathfrak{J}\right) 
$ where $\mathfrak{A}$ is a strict (unital) C*-algebra and $\mathfrak{J}$ is
a strict ideal of $\mathfrak{A}$.
\end{definition}

We regard strict unital C*-pairs as objects of a category, where a morphism
from $\left( \mathfrak{A},\mathfrak{I}\right) $ to $\left( \mathfrak{B},%
\mathfrak{J}\right) $ is a strict unital *-homomorphism $\varphi :\mathfrak{A%
}\rightarrow \mathfrak{B}$ that maps $\mathfrak{I}$ to $\mathfrak{J}$.

Every strict unital C*-pair $\left( \mathfrak{A},\mathfrak{J}\right) $
determines a quotient unital C*-algebra $\mathfrak{A/J}$. If $\mathfrak{A/I}$
and $\mathfrak{B/J}$ are two unital C*-algebras obtained in this way, then
we say that a unital *-homomorphism $\varphi :\mathfrak{A/I}\rightarrow 
\mathfrak{B/J}$ is \emph{definable }if it has a Borel lift \emph{\ }(or a
Borel representation\emph{\ }in the terminology of \cite%
{farah_automorphisms_2011,ghasemi_isomorphisms_2015}). This is a Borel
function $f:\mathfrak{A}\rightarrow \mathfrak{B}$ (which is not necessarily
a *-homomorphism) such that $\varphi \left( a+\mathfrak{I}\right) =f\left(
a\right) +\mathfrak{J}$ for every $a\in A$. The notion of definable unital
*-homomorphisms determines a category, whose objects are strict unital
C*-pairs and whose morphisms are the definable unital *-homomorphism. When
the strict unital C*-pair $\left( \mathfrak{A},\mathfrak{J}\right) $ is
considered as the object of this category, we call it a unital C*-algebra
with a strict cover, and denote it by $\mathfrak{A/J}$, as we think of it as
a unital C*-algebra explicitly presented as the quotient of a strict unital
C*-algebra by a strict ideal. The category of unital C*-algebras with a
strict cover thus has unital C*-algebras with strict cover as objects and
definable unital *-homomorphisms as morphisms. The notion of a unital
C*-algebra with a strict cover is the analogue in the context of C*-algebras
to the notion of group with a Polish cover considered in \cite%
{bergfalk_ulam_2020}; see Remark \ref{Remark:Polish-cover}.

Notice that every strict unital *-homomorphism $\left( \mathfrak{A},%
\mathfrak{I}\right) \rightarrow \left( \mathfrak{B},\mathfrak{J}\right) $
between strict unital\ C*-pairs induces a definable unital *-homomorphism $%
\mathfrak{A/J}\rightarrow \mathfrak{B/J}$ between the corresponding unital
C*-algebras with a strict cover. This allows one to regard the category of
strict unital C*-pairs as a subcategory of the category of unital
C*-algebras with a strict cover. These categories have the same objects, but
different morphisms.

If $\left( \mathfrak{A},\mathfrak{J}\right) $ is a strict unital C*-pair and 
$a,b\in \mathfrak{A}$, we write $a\equiv b\mathrm{\ \mathrm{mod}}\ \mathfrak{%
J}$ if $a-b\in \mathfrak{J}$. If $a\in M_{n}\left( \mathfrak{A}\right) $ and 
$b\in M_{k}\left( \mathfrak{A}\right) $, then we set%
\begin{equation*}
a\oplus b=%
\begin{bmatrix}
a & 0 \\ 
0 & b%
\end{bmatrix}%
\in M_{n+k}\left( \mathfrak{A}\right) \text{.}
\end{equation*}%
We let $1_{n}$ be the identity element of $M_{n}\left( \mathfrak{A}\right) $
and $0_{n}$ be the zero element of $M_{n}\left( \mathfrak{A}\right) $.

Suppose that $\left( \mathfrak{A},\mathfrak{J}\right) $ is a strict unital
C*-pair. A positive element of \textrm{Ball}$\left( \mathfrak{A}\right) $ is
a projection $\mathrm{\mathrm{mod}}\ \mathfrak{J}$ if $p^{2}\equiv p\mathrm{%
\ \mathrm{mod}}\ \mathfrak{J}$ or, equivalently, $p+\mathfrak{J}$ is a
projection in $\mathfrak{A/J}$. Define the set $\mathrm{Proj}\left( 
\mathfrak{A/J}\right) \subseteq \mathrm{\mathrm{Ball}}\left( \mathfrak{A}%
\right) $ to be the Borel set of projections $\mathrm{\mathrm{mod}}\ 
\mathfrak{J}$ in $\mathfrak{A}$. The Borel structure on $\mathrm{Proj}\left( 
\mathfrak{A/J}\right) $ is induced by the Polish topology defined by
declaring a net $\left( p_{i}\right) _{i\in I}$ to converge to $p$ if and
only if $p_{i}\rightarrow p$ strictly in \textrm{Ball}$\left( \mathfrak{A}%
\right) $ and $p_{i}^{2}-p_{i}\rightarrow p^{2}-p$ strictly in $2$\textrm{%
Ball}$\left( \mathfrak{J}\right) $. (Recall that the strict topology on $%
\mathrm{\mathrm{Ball}}\left( \mathfrak{J}\right) $ might be different from
the topology induced by the strict topology on $\mathrm{\mathrm{Ball}}\left( 
\mathfrak{A}\right) $.)

We also say that an element $u$ of \textrm{Ball}$\left( \mathfrak{A}\right) $
is a unitary $\mathrm{\mathrm{mod}}\ \mathfrak{J}$ if $uu^{\ast }\equiv 1%
\mathrm{\ \mathrm{mod}}\ \mathfrak{J}$ and $u^{\ast }u\equiv 1\mathrm{\ 
\mathrm{mod}}\ \mathfrak{J}$ or, equivalently, $u+\mathfrak{J}$ is a unitary
in $\mathfrak{A/J}$. We let $U\left( \mathfrak{A/J}\right) $ be the Borel
set of unitaries $\mathrm{\mathrm{\mathrm{mod}}\ }\mathfrak{J}$ in $%
\mathfrak{A}$. The Borel structure on $U\left( \mathfrak{A/J}\right) $ is
induced by the Polish topology defined by declaring a net $\left(
u_{i}\right) _{i\in I}$ to converge to $u$ if and only if $u_{i}\rightarrow
u $ strictly in $\mathrm{\mathrm{Ball}}\left( \mathfrak{A}\right) $, $%
u_{i}u_{i}^{\ast }-1\rightarrow uu^{\ast }-1$ strictly in $2\mathrm{\mathrm{%
Ball}}\left( \mathfrak{J}\right) $, and $u_{i}^{\ast }u_{i}-1\rightarrow
u^{\ast }u-1$ strictly in $2\mathrm{\mathrm{Ball}}\left( \mathfrak{J}\right) 
$.

More generally, an element $v$ of $\mathrm{\mathrm{Ball}}\left( \mathfrak{A}%
\right) $ is called a \emph{partial unitary} $\mathrm{\mathrm{mod}}\ 
\mathfrak{J}$ if $uu^{\ast }\equiv uu^{\ast }$ $\mathrm{\ \mathrm{mod}}\ 
\mathfrak{J}$ and $uu^{\ast }$ is a $\mathrm{\mathrm{mod}}\ \mathfrak{J}$
projection or, equivalently, if $v+\mathfrak{J}$ is a partial unitary in $%
\mathfrak{A/J}$ as in \cite[8.2.12]{rordam_introduction_2000}. We let $%
\mathrm{PU}\left( \mathfrak{A/J}\right) $ be the Borel set of $\mathrm{%
\mathrm{mod}}\ \mathfrak{J}$ partial unitaries in $\mathfrak{A}$. In a
similar fashion one can define the Borel set \textrm{PI}$\left( \mathfrak{A/J%
}\right) $ of $\mathrm{\mathrm{\mathrm{\mathrm{mod}}\ }}\mathfrak{J}$
partial isometries in $\mathfrak{A}$, consisting of those $v\in \mathrm{%
\mathrm{Ball}}\left( \mathfrak{A}\right) $ such that $vv^{\ast }$ and $%
v^{\ast }v$ are $\mathrm{\mathrm{mod}}\ \mathfrak{J}$ projections.

In the rest of this section we record some lemmas about unitaries and
projections modulo a strict ideal in a strict unital C*-algebra. The content
of these lemmas can be summarized as the assertion that a homotopy between
projections and unitaries in a unital C*-algebra with a strict cover is
witnessed by unitary elements in the path-component of the identity of the
unitary group that can be chosen in a Borel fashion. The proofs follow
standard arguments from the literature on $\mathrm{K}$-theory for
C*-algebras; see \cite%
{rordam_introduction_2000,higson_analytic_2000,blackadar_theory_1998,wegge-olsen_theory_1993}%
.

Given elements $y_{1},\ldots ,y_{n}$ of $\mathrm{\mathrm{\mathrm{\mathrm{Ball%
}}}}\left( \mathfrak{A}\right) $, subject to a certain relation $P\left(
y_{1},\ldots ,y_{n}\right) $, we say that an element $z\in \mathrm{\mathrm{%
Ball}}\left( \mathfrak{A}\right) $ satisfying a relation $R\left(
y_{1},\ldots ,y_{n},z\right) $ \emph{can be chosen in a Borel fashion }(from 
$y_{1},\ldots ,y_{n}$) if there is a Borel function $\left( y_{1},\ldots
,y_{n}\right) \mapsto z\left( y_{1},\ldots ,y\right) $ that assign to each $%
n $-tuple $\left( y_{1},\ldots ,y_{n}\right) $ in $\mathrm{\mathrm{\mathrm{%
\mathrm{Ball}}}}\left( \mathfrak{A}\right) $ satisfying $P$ an element $%
z\left( y_{1},\ldots ,y_{n}\right) $ in $\mathrm{\mathrm{\mathrm{\mathrm{Ball%
}}}}\left( \mathfrak{A}\right) $ such that $\left( y_{1},\ldots
,y_{n},z\left( y_{1},\ldots ,y\right) \right) $ satisfies $R$. In other
words, the set of tuples $\left( y_{1},\ldots ,y_{n},z\right) \in \mathrm{%
\mathrm{Ball}}\left( \mathfrak{A}\right) ^{n}\times \mathrm{\mathrm{Ball}}%
\left( \mathfrak{A}\right) $ such that $\left( y_{1},\ldots ,y_{n}\right) $
satisfies $P$ and $\left( y_{1},\ldots ,y_{n},z\right) $ satisfies $R$ has a
Borel uniformization \cite[Section 18.A]{kechris_classical_1995}.

Suppose that $A$ is a unital C*-algebra. Let $\mathrm{Proj}(A)$ be the set
of projections in $A$. Two projections $p,q$ in $A$ are:

\begin{itemize}
\item Murray--von Neumann equivalent if there exists $v\in A$ such that $%
v^{\ast }v=p$ and $vv^{\ast }=q$, in which case we write $p\sim _{\mathrm{MvN%
}}q$;

\item unitary equivalent if there exists $u\in U(A)$ such that $u^{\ast
}qu=p $;

\item homotopic if there is a\emph{\ norm-continuous} path $\left(
p_{t}\right) _{t\in \left[ 0,1\right] }$ in $\mathrm{Proj}\left(
M_{n}(A)\right) $ with $p_{0}=p$ and $p_{1}=q$.
\end{itemize}

\begin{lemma}
\label{Lemma:log}Suppose that $\left( \mathfrak{A},\mathfrak{J}\right) $ is
a strict unital C*-pair, $d\in \mathbb{N}$, and $u\in d\mathrm{\mathrm{Ball}}%
\left( \mathfrak{A}\right) $ satisfies $\left\Vert u-1+\mathfrak{J}%
\right\Vert \leq 1/2$. Then one can choose in a Borel way a self-adjoint
element $y\in \mathrm{\mathrm{\mathrm{Ball}}}\left( \mathfrak{A}\right) $
such that $e^{iy}\equiv u\mathrm{\ \mathrm{mod}}\ \mathfrak{J}$.
\end{lemma}

\begin{proof}
Consider $u-1+\mathfrak{J}\in \mathfrak{A/J}$, and observe that there exists 
$a\in \mathfrak{A}$ such that $\left\Vert a\right\Vert \leq 1/2$ and $a+%
\mathfrak{J}=u-1+\mathfrak{J}$, which can be chosen in a Borel way by strict
continuity of the continuous functional calculus. Hence, setting $\tilde{u}%
:=a+1\in \left( d+1\right) \mathrm{\mathrm{Ball}}\left( \mathfrak{A}\right) $%
, we have that $\tilde{u}\equiv u\mathrm{\mathrm{\ \mathrm{mod}}\ }\mathfrak{%
J}$ and $\left\Vert \tilde{u}-1\right\Vert \leq 1/2$. Thus, after replacing $%
d$ with $d+1$ and $u$ with $\tilde{u}$, we can assume that $\left\Vert
u-1\right\Vert \leq 1/2$.

Let $\log :D\rightarrow \mathbb{C}$ be an holomorphic branch of the
logarithm defined on $\left\{ z\in \mathbb{C}:\left\vert z-1\right\vert
<1\right\} $. Considering the holomorphic functional calculus, one can
define the element $\mathrm{\log }(u)\in \mathfrak{A}$. As%
\begin{equation*}
\mathrm{\log }(z)=\sum_{n=0}^{\infty }\frac{\left( 1-z\right) ^{n}}{n}
\end{equation*}%
is the uniformly convergent power series expansion in $\left\{ z\in \mathbb{C%
}:\left\vert z\right\vert \leq 1/2\right\} $, we have that%
\begin{equation*}
\mathrm{\log }(u)=\sum_{n=0}^{\infty }\frac{\left( 1-u\right) ^{n}}{n}\text{.%
}
\end{equation*}%
In particular, $\left\Vert \mathrm{\log }(u)\right\Vert \leq 1$. Define%
\begin{equation*}
y:=\frac{\mathrm{\log }(u)+\log (u)^{\ast }}{2}\in \mathrm{\mathrm{Ball}}%
\left( \mathfrak{A}_{\mathrm{sa}}\right) \text{.}
\end{equation*}%
Then we have that $y\equiv \log (u)\mathrm{\ \mathrm{mod}}\ \mathfrak{J}$
satisfies $\exp \left( iy\right) \equiv u\mathrm{\ \mathrm{mod}}\ \mathfrak{J%
}$.
\end{proof}

\begin{corollary}
\label{Corollary:homotopy-unitaries}Suppose that $\left( \mathfrak{A},%
\mathfrak{J}\right) $ is a strict unital C*-pair, and $u,w\in \mathfrak{A}$
are $\mathrm{\mathrm{mod}}\ \mathfrak{J}$ unitaries. Then there following
assertions are equivalent:

\begin{enumerate}
\item there is a norm-continuous path from $u+\mathfrak{J}$ to $w+\mathfrak{J%
}$ in $\mathfrak{A/J}$;

\item there exists $\ell \geq 1$ and $y_{1},\ldots ,y_{\ell }\in \mathrm{%
\mathrm{Ball}}\left( \mathfrak{A}_{\mathrm{sa}}\right) $ such that $%
e^{iy_{1}}\cdots e^{iy_{\ell }}u\equiv w\mathrm{\ \mathrm{mod}}\ \mathfrak{J}
$.
\end{enumerate}
\end{corollary}

\begin{lemma}
\label{Lemma:relative-path-unitary}Suppose that $\left( \mathfrak{A},%
\mathfrak{J}\right) $ is a strict unital C*-pair, and $p,q,x\in \mathrm{%
\mathrm{Ball}}\left( \mathfrak{A}\right) $ are such that $p,q$ are $\mathrm{%
\mathrm{mod}}\ \mathfrak{J}$ projections, $x^{\ast }x\equiv p\mathrm{\ 
\mathrm{mod}}\ \mathfrak{J}$, and $xx^{\ast }\equiv q\mathrm{\ \mathrm{mod}}%
\ \mathfrak{J}$. Then one can choose $Y_{1},\ldots ,Y_{\ell }\in \mathrm{%
\mathrm{Ball}}\left( M_{2}\left( \mathfrak{A}\right) _{\mathrm{sa}}\right) $
in a Borel fashion from $p,q,x$ such that, setting $U:=e^{iY_{1}}\cdots
e^{iY_{\ell }}\mathrm{,}$ one has that 
\begin{equation*}
U^{\ast }\left( q\oplus 0\right) U\equiv \left( p\oplus 0\right) \mathrm{\ 
\mathrm{mod}}\ M_{2}\left( \mathfrak{J}\right)
\end{equation*}%
and%
\begin{equation*}
\left( q\oplus 0\right) U\left( p\oplus 0\right) \equiv x\oplus 0\mathrm{\ 
\mathrm{mod}}\ M_{2}\left( \mathfrak{J}\right) \text{,}
\end{equation*}%
where $\ell \geq 1$ does not depend on $\left( \mathfrak{A},\mathfrak{J}%
\right) $ and $p,q,x$.
\end{lemma}

\begin{proof}
Consider the $\mathrm{\mathrm{mod}}\ M_{2}\left( \mathfrak{J}\right) $
unitary%
\begin{equation*}
X:=%
\begin{bmatrix}
x & 1-q \\ 
1-p & x^{\ast }%
\end{bmatrix}%
\in M_{2}\left( \mathfrak{A}\right) \text{.}
\end{equation*}%
Notice that $X$ satisfies%
\begin{equation*}
X^{\ast }\left( q\oplus 0_{d}\right) X\equiv \left( p\oplus 0_{d}\right) 
\mathrm{\ \mathrm{mod}}\ M_{2}\left( \mathfrak{J}\right)
\end{equation*}%
and%
\begin{equation*}
\left( q\oplus 0_{d}\right) X\left( p\oplus 0_{d}\right) \equiv x\oplus 0_{d}%
\mathrm{\ \mathrm{mod}}\ M_{2}\left( \mathfrak{J}\right) \text{.}
\end{equation*}%
Consider the norm-continuous path of $\mathrm{\mathrm{mod}}\ M_{2}\left( 
\mathfrak{J}\right) $ unitaries%
\begin{equation*}
X_{t}:=%
\begin{bmatrix}
\mathrm{cos}\left( \frac{\pi t}{2}\right) x & 1-\left( 1-\mathrm{sin}\left( 
\frac{\pi t}{2}\right) \right) q \\ 
\left( 1-\mathrm{sin}\left( \frac{\pi t}{2}\right) \right) p-1 & \mathrm{cos}%
\left( \frac{\pi t}{2}\right) x^{\ast }%
\end{bmatrix}%
\end{equation*}%
for $t\in \lbrack 0,1]$. Notice that the modulus of continuity of $\left(
X_{t}\right) _{t\in \left[ 0,1\right] }$ does not depend on $\left( 
\mathfrak{A},\mathfrak{J}\right) $ and $\left( p,q,x\right) $. Fix $\ell
\geq 1$ such that if $t,s\in \left[ 0,1\right] $ satisfy $\left\vert
s-t\right\vert \leq 1/\ell $, then%
\begin{equation*}
\left\Vert X_{t}-X_{s}\right\Vert \leq 1/2\text{.}
\end{equation*}%
Thus, for $i\in \left\{ 1,2,\ldots ,\ell \right\} $ we have that%
\begin{equation*}
\left\Vert X_{i/\ell }-X_{(i+1)/\ell }\right\Vert \leq 1/2\text{.}
\end{equation*}%
By Lemma \ref{Lemma:log} we can choose in a Borel fashion $Y_{1}\in \mathrm{%
\mathrm{Ball}}\left( \mathfrak{A}_{\mathrm{sa}}\right) $ such that $%
e^{iY_{1}}\equiv X_{1/\ell }\mathrm{\ \mathrm{mod}}\ M_{2}\left( \mathfrak{J}%
\right) $. Thus 
\begin{equation*}
R_{1}:=\mathrm{\exp }\left( iY_{1}\right) -X_{1/\ell }\in M_{2}\left( 
\mathfrak{J}\right) .
\end{equation*}%
Consider now $X_{2/\ell }$ and the fact that 
\begin{equation*}
\left\Vert X_{2/\ell }-X_{1/\ell }\right\Vert \leq 1/2\text{.}
\end{equation*}%
Thus%
\begin{equation*}
\left\Vert \mathrm{\exp }\left( iY_{1}\right) -\left( X_{2/\ell
}+R_{1}\right) \right\Vert \leq 1/2\text{.}
\end{equation*}%
and%
\begin{equation*}
\left\Vert 1-\mathrm{\exp }\left( -iY_{1}\right) \left( X_{2/\ell
}+R_{1}\right) \right\Vert \leq 1/2\text{.}
\end{equation*}%
Thus by Lemma \ref{Lemma:log} one can choose in a Borel fashion $Y_{2}\in 
\mathrm{\mathrm{Ball}}\left( \mathfrak{A}_{\mathrm{sa}}\right) $ such that 
\begin{equation*}
\exp \left( iY_{2}\right) \equiv \mathrm{\exp }\left( -iY_{1}\right) \left(
X_{2/\ell }+R_{1}\right) \equiv \exp \left( -iY_{1}\right) X_{2/\ell }%
\mathrm{\ \mathrm{mod}}\ M_{2}\left( \mathfrak{J}\right)
\end{equation*}%
and hence%
\begin{equation*}
\exp \left( iY_{1}\right) \exp \left( iY_{2}\right) \equiv X_{2/\ell }%
\mathrm{\ \mathrm{mod}}\ M_{2}\left( \mathfrak{J}\right) \text{.}
\end{equation*}%
Proceeding recursively in this way, one can choose $Y_{1},\ldots ,Y_{\ell
}\in \mathrm{\mathrm{Ball}}\left( \mathfrak{A}_{\mathrm{sa}}\right) $ in a
Borel fashion such that%
\begin{equation*}
\mathrm{\exp }\left( iY_{1}\right) \cdots \mathrm{\exp }\left( iY_{\ell
}\right) \equiv X\mathrm{\ \mathrm{mod}}\ M_{2}\left( \mathfrak{J}\right) 
\text{.}
\end{equation*}%
Then we have that, setting $U:=\mathrm{\exp }\left( iY_{1}\right) \cdots 
\mathrm{\exp }\left( iY_{\ell }\right) $,%
\begin{equation*}
U^{\ast }\left( q\oplus 0_{d}\right) U\equiv X^{\ast }\left( q\oplus
0_{d}\right) X\equiv p\oplus 0_{d}\mathrm{\ \mathrm{mod}}\ M_{2}\left( 
\mathfrak{J}\right)
\end{equation*}%
and%
\begin{equation*}
\left( q\oplus 0_{d}\right) U\left( p\oplus 0_{d}\right) \equiv \left(
q\oplus 0_{d}\right) X\left( p\oplus 0_{d}\right) \equiv x\oplus 0_{d}%
\mathrm{\ \mathrm{mod}}\ M_{2}\left( \mathfrak{J}\right) \text{.}
\end{equation*}%
\ This concludes the proof.
\end{proof}

\begin{lemma}
\label{Lemma:homotopy-path-unitary}Suppose that $\left( \mathfrak{A},%
\mathfrak{J}\right) $ is a strict unital C*-pair, and $p,q\in \mathfrak{A}_{%
\mathrm{sa}}$ are $\mathrm{\mathrm{mod}}\ \mathfrak{J}$ projections such
that $\left\Vert p-q\right\Vert \leq 1/2$. Then one can choose $y_{1},\ldots
,y_{\ell }\in \mathrm{\mathrm{Ball}}\left( \mathfrak{A}_{\mathrm{sa}}\right) 
$ in a Borel fashion from $p,q$ such that, setting $u:=e^{iy_{1}}\cdots
e^{iy_{n}}$, one has that $u^{\ast }qu\equiv p\mathrm{\ \mathrm{mod}}\ 
\mathfrak{J}$, where $\ell \geq 1$ does not depend on $\left( \mathfrak{A},%
\mathfrak{J}\right) $ and $p,q$.
\end{lemma}

\begin{proof}
As in the proof of \cite[Proposition 2.2.4]{rordam_introduction_2000},
consider the norm-continuous path of $\mathrm{\mathrm{\mathrm{mod}}\ }%
\mathfrak{J}$ projections $a_{t}:=\left( 1-t\right) p+tq$ for $t\in \left[
0,1\right] $. Let also $K=[-1/4,1/4]\cup \lbrack 3/4,5/4]\subseteq \mathbb{R}
$, and $f:K\rightarrow \mathbb{C}$ be the continuous function that is $0$ on 
$[-1/4,1/4]$ and $1$ on $[3/4,5/4]$. Then $p_{t}:=f\left( a_{t}\right) $ for 
$t\in \left[ 0,1\right] $ is a norm-continuous path of $\mathrm{\mathrm{mod}}%
\ \mathfrak{J}$ projections from $p$ to $q$. Notice that the uniform
continuity moduli of $t\mapsto a_{t}$ and $t\mapsto p_{t}$ do not depend on $%
\left( \mathfrak{A},\mathfrak{J}\right) $ and $p,q$.

Thus, there exists $k\in \omega $ (that does depend on $\left( \mathfrak{A},%
\mathfrak{J}\right) $ and $p,q$) such that, for every $t,s\in \left[ 0,1%
\right] $ such that $\left\vert t-s\right\vert \leq 1/k$, one has that $%
\left\Vert p_{t}-p_{s}\right\Vert \leq 1/6$. Thus, $p_{0}=p,p_{1/k},p_{2/k},%
\ldots ,p_{1}=q$ are $\mathrm{\mathrm{mod}}\ \mathfrak{J}$ projections (that
depend in a Borel way from $p,q$ by strict continuity of the continuous
functional calculus) such that $\left\Vert p_{(i+1)/k}-p_{i/k}\right\Vert
\leq 1/6$ for $i\in \left\{ 0,1,\ldots ,k-1\right\} $ and $(p_{\frac{i\left(
1-s\right) +\left( 1+i\right) s}{k}})_{s\in \lbrack 0,1]}$ is a
norm-continuous path from $p_{i/k}$ to $p_{(i+1)/k}$ (whose modulus of
continuity does not depend on $\mathfrak{A}$ and $p,q\in \mathfrak{A}$)
satisfying 
\begin{equation*}
\left\Vert p_{\frac{i\left( 1-s\right) +\left( 1+i\right) s}{k}%
}-p_{i/k}\right\Vert \leq 1/6
\end{equation*}%
for $s\in \left[ 0,1\right] $.

Thus, we can assume without loss of generality that $\left\Vert
p_{t}-p\right\Vert \leq 1/6$ for every $t\in \left[ 0,1\right] $. We now
proceed as in the proof of \cite[Proposition 2.17]{nest_excision_2017}.
Define%
\begin{equation*}
x_{t}:=\left( 2p-1\right) \left( p_{t}-p\right) +1
\end{equation*}%
By definition, we have that $x_{0}=1$. Notice that%
\begin{equation*}
x_{t}\equiv pp_{t}+\left( p-1\right) \left( p_{t}-1\right) \mathrm{\ \mathrm{%
mod}}\ \mathfrak{J}
\end{equation*}%
\begin{equation*}
px_{t}\equiv pp_{t}\equiv x_{t}p_{t}\mathrm{\ \mathrm{mod}}\ \mathfrak{J}
\end{equation*}%
\begin{equation*}
p_{t}x_{t}^{\ast }x_{t}\equiv x_{t}^{\ast }px_{t}\equiv x_{t}^{\ast
}x_{t}p_{t}\mathrm{\ \mathrm{mod}}\ \mathfrak{J}
\end{equation*}%
and%
\begin{equation*}
px_{t}x_{t}^{\ast }\equiv x_{t}p_{t}x_{t}^{\ast }\equiv x_{t}x_{t}^{\ast }p%
\mathrm{\ \mathrm{mod}}\ \mathfrak{J}\text{.}
\end{equation*}%
This implies that 
\begin{equation*}
p\left\vert x_{t}^{\ast }\right\vert \equiv \left\vert x_{t}^{\ast
}\right\vert p\mathrm{\ \mathrm{mod}}\ \mathfrak{J}
\end{equation*}%
and%
\begin{equation*}
p_{t}\left\vert x_{t}\right\vert \equiv \left\vert x_{t}\right\vert p_{t}%
\mathrm{\ \mathrm{mod}}\ \mathfrak{J}
\end{equation*}%
for $t\in \left[ 0,1\right] $. We have that 
\begin{equation*}
\left\Vert x_{t}-1\right\Vert =\left\Vert \left( 2p-1\right) \left(
p_{t}-p\right) \right\Vert \leq \left\Vert 2p-1\right\Vert \left\Vert
p_{t}-p\right\Vert \leq 1/2\text{.}
\end{equation*}%
Thus, $x_{t}$ is invertible. Let $x_{t}:=u_{t}\left\vert x_{t}\right\vert $
be its polar decomposition, where $u_{t}$ is a unitary. Then we have that $%
pu_{t}\equiv u_{t}p_{t}\mathrm{\ \mathrm{mod}}\ \mathfrak{J}$. Indeed,%
\begin{equation*}
pu_{t}\equiv px_{t}\left\vert x_{t}\right\vert ^{-1}\equiv
x_{t}p_{t}\left\vert x_{t}\right\vert ^{-1}\equiv x_{t}\left\vert
x_{t}\right\vert ^{-1}p_{t}\equiv u_{t}p_{t}\mathrm{\ \mathrm{mod}}\ 
\mathfrak{J}\text{.}
\end{equation*}%
Thus%
\begin{equation*}
u_{t}^{\ast }pu_{t}\equiv p_{t}\mathrm{\ \mathrm{mod}}\ \mathfrak{J}
\end{equation*}%
for $t\in \left[ 0,1\right] $ and in particular $u_{1}^{\ast }pu_{1}\equiv q%
\mathrm{\ \mathrm{mod}}\ \mathfrak{J}$.

Notice that $\left( u_{t}\right) _{t\in \left[ 0,1\right] }$ is a
norm-continuous path, whose modulus of continuity does not depend on $\left( 
\mathfrak{A},\mathfrak{J}\right) $ and $p,q$. Therefore, there exists $k\geq
1$ (which do not depend on $\left( \mathfrak{A},\mathfrak{J}\right) $ and $%
p,q$) such that, whenever $s,t\in \left[ 0,1\right] $ satisfy $\left\vert
s-t\right\vert \leq 1/k$, we have $\left\Vert u_{t}-u_{s}\right\Vert \leq
1/2 $. By Lemma \ref{Lemma:log} one can then choose in a Borel way $%
y_{1},\ldots ,y_{k}\in \mathrm{\mathrm{Ball}}\left( \mathfrak{A}_{\mathrm{sa}%
}\right) $ such that, setting $u:=\exp \left( iy_{1}\right) \cdots \exp
\left( iy_{k}\right) $, then $u\equiv u_{1}\mathrm{\ \mathrm{mod}}\ 
\mathfrak{J}$ and hence $u^{\ast }pu\equiv q\mathrm{\ \mathrm{mod}}\ 
\mathfrak{J}$. This concludes the proof.
\end{proof}

\begin{lemma}
\label{Lemma:orthogonal-projections}Suppose that $\left( \mathfrak{A},%
\mathfrak{J}\right) $ is a strict unital C*-pair, and $p,q\in \mathfrak{A}$
are $\mathrm{\mathrm{\mathrm{mod}}\ }\mathfrak{J}$ projections that satisfy $%
pq\equiv qp\equiv 0\mathrm{\ \mathrm{mod}}\ \mathfrak{J}$. Then one can
choose $Y_{1},\ldots ,Y_{\ell }\in \mathrm{\mathrm{Ball}}\left( M_{2}\left( 
\mathfrak{A}\right) _{\mathrm{sa}}\right) $ in a Borel fashion from $p,q$
such that, setting $U:=e^{iY_{1}}\cdots e^{iY_{\ell }}$, one has that $%
U^{\ast }\left( p\oplus q\right) U\equiv \left( p+q\right) \oplus 0\mathrm{\ 
\mathrm{mod}}\ M_{2}\left( \mathfrak{J}\right) $, where $\ell \geq 1$ does
not depend on $\left( \mathfrak{A},\mathfrak{J}\right) $ and $p,q$.
\end{lemma}

\begin{proof}
Consider the path%
\begin{equation*}
r_{t}:=%
\begin{bmatrix}
1 & 0 \\ 
0 & 0%
\end{bmatrix}%
p+%
\begin{bmatrix}
\mathrm{cos}^{2}\left( \frac{\pi t}{2}\right) & \mathrm{cos}\left( \frac{\pi
t}{2}\right) \mathrm{sin}\left( \frac{\pi t}{2}\right) \\ 
\mathrm{cos}\left( \frac{\pi t}{2}\right) \mathrm{sin}\left( \frac{\pi t}{2}%
\right) & \mathrm{sin}^{2}\left( \frac{\pi t}{2}\right)%
\end{bmatrix}%
q
\end{equation*}%
for $t\in \left[ 0,1\right] $. This is a norm-continuous path of $\mathrm{%
\mathrm{mod}}\ M_{2}\left( \mathfrak{J}\right) $ projections in $M_{2}\left( 
\mathfrak{A}\right) $ from $\left( p+q\right) \oplus 0$ to $p\oplus q$,
whose modulus of continuity does not depend on $\mathfrak{A}$ and $p,q$.
Therefore, the conclusion follows from Lemma \ref%
{Lemma:homotopy-path-unitary}.
\end{proof}

\begin{lemma}
\label{Lemma:homotopy-orthogonal-unitaries}Suppose that $\left( \mathfrak{A},%
\mathfrak{J}\right) $ is a strict unital C*-pair, and $u,v\in \mathfrak{A}$
are $\mathrm{\mathrm{mod}}\ \mathfrak{J}$ unitaries.\ Then one can choose $%
y_{1},\ldots ,y_{\ell }\in \mathrm{\mathrm{Ball}}\left( M_{2}\left( 
\mathfrak{A}\right) _{\mathrm{sa}}\right) $ in a Borel fashion from $u$ and $%
v$ such that $\left( u\oplus v\right) \equiv e^{iy_{1}}\cdots e^{iy_{\ell
}}\left( uv\oplus 1\right) \mathrm{\ \mathrm{mod}}\ M_{2}\left( \mathfrak{J}%
\right) $, where $\ell \geq 1$ does not depend on $\mathfrak{A},\mathfrak{J}$
and $u,v$.
\end{lemma}

\begin{proof}
Fix a unitary path $\left( W_{t}\right) _{t\in \left[ 0,1\right] }$ in $%
U\left( M_{2}\left( \mathbb{C}\right) \right) $ from $1$ to%
\begin{equation*}
\begin{bmatrix}
0 & 1 \\ 
1 & 0%
\end{bmatrix}%
\text{.}
\end{equation*}%
Fix $\ell \geq 1$ such that, for $s,t\in \left[ 0,1\right] $, $\left\Vert
W_{s}-W_{t}\right\Vert \leq 1/2$. Then%
\begin{equation*}
u_{t}:=\left( u\oplus 1\right) W_{t}^{\ast }\left( v\oplus 1\right)
W_{t}\left( uv\oplus 1\right) ^{\ast }
\end{equation*}%
is a path of $\mathrm{\mathrm{mod}}\ M_{2}\left( \mathfrak{J}\right) $
unitaries from $1$ to $\left( u\oplus v\right) \left( uv\oplus 1\right)
^{\ast }$ to $1$. Then, as in the proof of Lemma \ref%
{Lemma:relative-path-unitary}, using Lemma \ref{Lemma:log} one can
recursively choose in a Borel fashion $y_{1},\ldots ,y_{\ell }\in \mathrm{%
\mathrm{Ball}}\left( M_{2}\left( \mathfrak{A}\right) _{\mathrm{sa}}\right) $
such that $e^{iy_{k}}\equiv u_{k/\ell }^{\ast }u_{(k+1)/\ell }\mathrm{\ 
\mathrm{mod}}\ M_{2}\left( \mathfrak{J}\right) $ for $k\in \left\{
0,1,\ldots ,\ell -1\right\} $ and hence $\left( u\oplus v\right) \left(
uv\oplus 1\right) ^{\ast }\equiv e^{iy_{1}}\cdots e^{iy_{\ell }}\mathrm{\ 
\mathrm{mod}}\ M_{2}\left( \mathfrak{J}\right) $.
\end{proof}

\subsection{The Definable Arveson Extension Theorem}

In the rest of this section, we present definable versions of some
fundamental results in operator algebras, to be used in the development of
definable $\mathrm{K}$-homology. Suppose that $H$ is a separable Hilbert
space. We regard $B\left( H\right) $ as the multiplier algebra of the
C*-algebra $K\left( H\right) $ of compact operators on $H$. The
corresponding strict topology on $\mathrm{\mathrm{\mathrm{\mathrm{Ball}}}}%
\left( B\left( H\right) \right) $ is the strong-* topology. Consistently, we
consider $B\left( H\right) $ as a standard Borel space with respect to the
induced standard Borel structure. If $Z$ is a separable Banach space, we
consider $L\left( Z,B\left( H\right) \right) $ as a strict Banach space,
where $\mathrm{\mathrm{Ball}}\left( L\left( Z,B\left( H\right) \right)
\right) $ is endowed with the topology of pointwise strong-* convergence. We
denote by $U\left( H\right) $ the unitary group of $B\left( H\right) $,
which is a Polish group when endowed with the strong-* topology.

Suppose that $A$ is a separable unital\ C*-algebra, and $X\subseteq A$ is an
operator system \cite[Chapter 2]{paulsen_completely_2002}. Let $H$ be a
separable Hilbert space. Arveson's Extension Theorem asserts that every
contractive completely positive (ccp) map $\phi :X\rightarrow B\left(
H\right) $ \cite[Section 1.5]{brown_algebras_2008} admits a contractive
completely positive \emph{extension} $\hat{\phi}:A\rightarrow B\left(
H\right) $ \cite[Theorem 7.5]{paulsen_completely_2002}. We observe now that $%
\hat{\phi}$ can be chosen in a Borel way from $\phi $. Notice that the space 
$\mathrm{CCP}\left( X,B\left( H\right) \right) $ of contractive completely
positive maps is closed (hence, compact) in $\mathrm{\mathrm{Ball}}\left(
L\left( X,B\left( H\right) \right) \right) $ endowed with the topology of
pointwise \emph{weak*} convergence.

\begin{lemma}
\label{Lemma:definable-Arveson}Suppose that $A$ is a separable unital\
C*-algebra, and $X\subseteq A$ is an operator system. Let $H$ be a separable
Hilbert space. Then there exists a \emph{Borel} function $\mathrm{CCP}\left(
X,B\left( H\right) \right) \rightarrow \mathrm{CCP}\left( A,B\left( H\right)
\right) $, $\phi \mapsto \hat{\phi}$ such that $\hat{\phi}$ is an extension
of $\phi $.
\end{lemma}

Towards the proof of Lemma \ref{Lemma:definable-Arveson}, we recall the
following particular case of the selection theorem for relations with
compact sections \cite[Theorem 28.8]{kechris_classical_1995}.

\begin{lemma}
\label{Lemma:select-compact-square}Suppose that $X,Y$ are compact metrizable
spaces, and $A\subseteq X\times Y$ is a Borel subset such that, for every $%
x\in X$, the vertical section%
\begin{equation*}
A_{x}=\left\{ y\in Y:\left( x,y\right) \in A\right\}
\end{equation*}%
is a closed nonempty set. Then there exists a Borel function $f:X\rightarrow
Y$ such that $\left( x,f(x)\right) \in A$ for every $x\in X$.
\end{lemma}

Using this selection theorem, Lemma \ref{Lemma:definable-Arveson} follows
immediately from the Arveson Extension Theorem.

\begin{proof}[Proof of Lemma \protect\ref{Lemma:definable-Arveson}]
We consider $\mathrm{CCP}\left( X,B\left( H\right) \right) $ as a compact
metrizable space, endowed with the topology of pointwise weak* convergence.
Consider the Borel set $A\subseteq \mathrm{CCP}\left( X,B\left( H\right)
\right) \times \mathrm{CCP}\left( A,B\left( H\right) \right) $ of pairs $%
\left( \phi ,\psi \right) $ such that $\psi |_{X}=\phi $. Then by the
Arveson Extension Theorem, the vertical sections of $A$ are nonempty, and
clearly closed. Thus, by Lemma \ref{Lemma:select-compact-square} there
exists a Borel function $f:\mathrm{CCP}\left( X,B\left( H\right) \right)
\rightarrow \mathrm{CCP}\left( A,B\left( H\right) \right) $ such that $%
f\left( \phi \right) |_{X}=\phi $ for every $\phi \in \mathrm{CCP}\left(
X,B\left( H\right) \right) $.
\end{proof}

\subsection{The Definable Stinespring Dilation Theorem}

Suppose that $A$ is a separable unital C*-algebra, and $H$ is a separable
Hilbert space. Stinespring's Dilation Theorem asserts that, for every
contractive completely positive map $\phi :A\rightarrow B\left( H\right) $,
there exists a linear map $V_{\phi }:H\rightarrow H$ with $\left\Vert
V_{\phi }\right\Vert ^{2}=\left\Vert \phi \right\Vert $ and a nondegenerate
representation $\pi _{\phi }$ of $A$ on $H$ such that $\phi \left( a\right)
=V_{\phi }^{\ast }\pi _{\phi }\left( a\right) V_{\phi }$ for every $a\in A$.
Notice that the set $\mathrm{Rep}\left( A,H\right) $ of nondegenerate
representations of $A$ on $H$ is a $G_{\delta }$ subset of $\mathrm{\mathrm{%
Ball}}\left( L\left( A,B\left( H\right) \right) \right) $, whence Polish
with the subspace topology, where $\mathrm{\mathrm{Ball}}\left( B\left(
H\right) \right) $ is endowed with the strong-* topology. It follows from
the proof of the Stinespring Dilation Theorem, where $V$ and $\pi $ are
explicitly defined in terms of $\phi $, that they can be chosen in a Borel
way from $\phi $; see \cite[Theorem II.6.9.7]{blackadar_operator_2006}

\begin{lemma}
\label{Lemma:definable-Stinespring}Suppose that $A$ is a separable unital
C*-algebra, and $H$ is a separable Hilbert space. Then there exists a Borel
function $\mathrm{CCP}\left( A,B\left( H\right) \right) \rightarrow \mathrm{%
\mathrm{Ball}}\left( B\left( H\right) \right) \times \mathrm{Rep}\left(
A,B\left( H\right) \right) $, $\phi \mapsto \left( V_{\phi },\pi _{\phi
}\right) $ such that $\phi \left( a\right) =V_{\phi }^{\ast }\pi _{\phi
}\left( a\right) V_{\phi }$ for $a\in A$ and $\left\Vert V_{\phi
}\right\Vert ^{2}=\left\Vert \phi \right\Vert $ for every contractive
completely positive map $\phi :A\rightarrow B\left( H\right) $.
\end{lemma}

\subsection{The Definable Voiculescu Theorem}

Suppose that $A$ is a separable unital C*-algebra, and $\rho ,\rho ^{\prime
}:A\rightarrow B\left( H\right) $ are two maps. If $U\in U\left( H\right) $
is a unitary operator, write $\rho ^{\prime }\thickapprox _{U}\rho $ if $%
\rho ^{\prime }\left( a\right) \equiv U^{\ast }\rho \left( a\right) U\mathrm{%
\ \mathrm{mod}}\ K\left( H\right) $ for every $a\in A$. If $V:H\rightarrow H$
is an isometry, write $\rho ^{\prime }\lesssim _{V}\rho $ if $\rho ^{\prime
}\left( a\right) \equiv V^{\ast }\rho \left( a\right) V\mathrm{\ \mathrm{mod}%
}\ K\left( H\right) $ for every $a\in A$. A nondegenerate representation $%
\rho $ of $A$ on $B\left( H\right) $ is \emph{ample }if, for every $a\in
B\left( H\right) $, $\rho \left( a\right) \in K\left( H\right) \Rightarrow
a=0$. Notice that the set \textrm{ARep}$\left( A,H\right) $ of ample
representations of $A$ on $B\left( H\right) $ is a $G_{\delta }$ subset of 
\textrm{Ball}$\left( L\left( A,B\left( H\right) \right) \right) $.
Similarly, the set $\mathrm{Iso}\left( H\right) $ of isometries $%
H\rightarrow H$ is a $G_{\delta }$ subset of $\mathrm{\mathrm{Ball}}\left(
B\left( H\right) \right) $. A formulation of Voiculescu's Theorem asserts
that if $\rho :A\rightarrow B\left( H\right) $ is an ample representation,
and $\sigma :A\rightarrow B\left( H\right) $ is a unital completely positive
(ucp) map, then there exists an isometry $V:H\rightarrow H$ such that $%
\sigma \lesssim _{V}\rho $; see \cite[Theorem 3.4.3, Theorem 3.4.6, Theorem
3.4.7]{higson_analytic_2000}. We will observe that one can select $V$ in a
Borel fashion from $\rho $ and $\sigma $.

\begin{lemma}
\label{Lemma:Voiculescu}Let $A$ be a separable unital C*-algebra, and $H$ a
separable Hilbert space. There exists a Borel function $\mathrm{\mathrm{UCP}}%
\left( A,B\left( H\right) \right) \times \mathrm{ARep}\left( A,H\right)
\rightarrow \mathrm{Iso}\left( H\right) $, $\left( \sigma ,\rho \right)
\mapsto V_{\sigma ,\rho }$ such that $\sigma \lesssim _{V_{\sigma ,\rho
}}\rho $.
\end{lemma}

Towards obtaining a proof of Lemma \ref{Lemma:Voiculescu}, we argue as in
the proof of Voiculescu's theorem as expounded in \cite[Chapter 3]%
{higson_analytic_2000}. First, one considers the case of ucp maps $%
A\rightarrow B\left( H\right) $ where $H$ is finite-dimensional. The
following can be seen a definable version of \cite[Proposition 3.6.7]%
{higson_analytic_2000}. Notice that the set $\mathrm{Proj}\left( H\right) $
of orthogonal projections $H\rightarrow H$ is closed subset of $\mathrm{%
\mathrm{Ball}}\left( B\left( H\right) \right) $. Let $\mathrm{Proj}_{\mathrm{%
fd}}\left( H\right) $ be the Borel subset of finite-dimensional projections.
The following lemma is a consequence of \cite[Proposition 3.6.7]%
{higson_analytic_2000} itself and the Luzin--Novikov Uniformization Theorem
for Borel relations with countable sections \cite[Theorem 18.10]%
{kechris_classical_1995}.

\begin{lemma}
\label{Lemma:Glimm}Fix a finite-dimensional subspace $H_{0}$ of $H$, and
regard $B\left( H_{0}\right) $ as a C*-subalgebra of $B\left( H\right) $.
For every finite subset $F$ of $A$ and $\varepsilon >0$, there exists a
Borel map $\mathrm{\mathrm{UCP}}\left( A,B\left( H_{0}\right) \right) \times 
\mathrm{ARep}\left( A,H\right) \times \mathrm{Proj}_{\mathrm{fd}}\left(
H\right) \rightarrow \mathrm{\mathrm{Ball}}\left( H\right) $, $\left( \sigma
,\rho ,P\right) \mapsto V$ such that $\mathrm{Ran}(V)$ is orthogonal to $%
P\left( H\right) $ and $\left\Vert \sigma \left( a\right) -V^{\ast }\rho
\left( a\right) V\right\Vert <\varepsilon $ for $a\in F$.
\end{lemma}

One then uses Lemma \ref{Lemma:Glimm} to establish Lemma \ref%
{Lemma:Voiculescu} in the case of \emph{block-diagonal }maps. Recall that $%
\sigma $ is block-diagonal with respect to $\left( P_{n}\right) _{n\in
\omega }$ if $\left( P_{n}\right) _{n\in \omega }$ is a sequence of pairwise
orthogonal finite-rank projections $P_{n}\in B\left( H\right) $ such that $%
\sum_{n}P_{n}=I$ and $\sigma \left( a\right) =\sum_{n}P_{n}\sigma \left(
a\right) P_{n}$ for every $a\in A$ (where the convergence is in the strong-*
topology). Consider the set $\mathrm{BlockUCP}\left( A,B\left( H\right)
\right) $ of pairs $\left( \sigma ,\left( P_{n}\right) _{n\in \omega
}\right) \in \mathrm{\mathrm{UCP}}\left( A,B\left( H\right) \right) \times 
\mathrm{Proj}_{\mathrm{fd}}\left( H\right) ^{\omega }$ such that $\sigma $
is block-diagonal with respect to $\left( P_{n}\right) _{n\in \omega }$. The
proof of \cite[Lemma 3.5.2]{higson_analytic_2000} shows the following.

\begin{lemma}
\label{Lemma:block-diagonal0}There exists a Borel function $\mathrm{BlockUCP}%
\left( A,B\left( H\right) \right) \times \mathrm{ARep}\left( A,H\right)
\rightarrow \mathrm{Iso}\left( H\right) $,$\left( \sigma ,\left(
P_{n}\right) _{n\in \omega },\rho \right) \mapsto V$ such that $\sigma
\lesssim _{V}\rho $.
\end{lemma}

Finally, one shows that the general case of Voiculescu's theorem can be
reduced to the block-diagonal case, as in \cite[Theorem 3.5.5]%
{higson_analytic_2000}.

\begin{lemma}
\label{Lemma:block-diagonal}There exists a Borel function $\mathrm{\mathrm{%
UCP}}\left( A,B\left( H\right) \right) \rightarrow \mathrm{BlockUCP}\left(
A,B\left( H\right) \right) \times \mathrm{Iso}\left( H\right) $, $\sigma
\mapsto \left( \sigma ^{\prime },\left( P_{n}\right) _{n\in \omega
},V^{\prime }\right) $ such that $\sigma \lesssim _{V^{\prime }}\sigma
^{\prime }$.
\end{lemma}

Lemma \ref{Lemma:Voiculescu} is then obtained by combining Lemma \ref%
{Lemma:block-diagonal0} and Lemma \ref{Lemma:block-diagonal}.

As a consequence of the definable Voiculescu Theorem, one obtains the
following; see \cite[Theorem 3.4.6]{higson_analytic_2000}.

\begin{lemma}
\label{Lemma:trivial-representations}Let $A$ be a separable unital
C*-algebra, and $H$ a separable Hilbert space. There exist:

\begin{itemize}
\item a Borel map $\mathrm{Rep}\left( A,H\right) \times \mathrm{ARep}\left(
A,H\right) \rightarrow U\left( H\right) $, $\left( \rho ^{\prime },\rho
\right) \mapsto U_{\rho ^{\prime },\rho }$ such that $\rho ^{\prime }\oplus
\rho \thickapprox _{U_{\rho ^{\prime },\rho }}\rho $;

\item a Borel map $\mathrm{ARep}\left( A,H\right) \times \mathrm{ARep}\left(
A,H\right) \rightarrow U\left( H\right) $, $\left( \rho ^{\prime },\rho
\right) \mapsto W_{\rho ^{\prime },\rho }$ such that $\rho \thickapprox
_{W_{\rho ^{\prime },\rho }}\rho ^{\prime }$.
\end{itemize}
\end{lemma}

\subsection{Spectrum}

Suppose now that $\mathfrak{A}$ is a strict unital C*-algebra, and $J$ is a 
\emph{norm-separable }closed two-sided ideal of $\mathfrak{A}$. One can
consider the quotient C*-algebra $\mathfrak{A}/J$ and, for $a\in \mathfrak{A}
$, the spectrum $\sigma _{\mathfrak{A}/J}\left( a\right) $ of $a+J$ in $%
\mathfrak{A}/J$. We also let the resolvent $\rho _{\mathfrak{A}/J}\left(
a\right) $ be the complement in $\mathbb{C}$ of $\sigma _{\mathfrak{A}%
/J}\left( a\right) $. The following lemma is analogous to \cite[Theorem 3.16]%
{ando_weyl_2015}.

\begin{lemma}
\label{Lemma:spectrum-quotient}Suppose that $\mathfrak{A}$ is a strict
C*-algebra, and $J$ a norm-separable closed two-sided ideal of $\mathfrak{A}$%
. Suppose that every invertible self-adjoint element of $\mathfrak{A}/J$
lifts to an invertible self-adjoint element of $\mathfrak{A}$. If $a\in 
\mathfrak{A}_{\mathrm{sa}}$, and $J_{0}$ is a countable dense subset of $%
J\cap \mathfrak{A}_{\mathrm{sa}}$, then 
\begin{equation*}
\sigma _{\mathfrak{A}/J}\left( a\right) =\bigcap_{d\in J_{0}}\sigma \left(
a+d\right) \text{.}
\end{equation*}
\end{lemma}

\begin{proof}
It suffices to prove that $\rho _{\mathfrak{A}/J}\left( a\right) $ is the
union of $\rho \left( a+d\right) $ for $d\in J_{0}$. Clearly, $\rho \left(
a+d\right) \subseteq \rho _{\mathfrak{A}/J}\left( a\right) $ for every $d\in
J_{0}$, so it suffices to prove the other inclusion. Suppose that $\lambda
\in \rho _{\mathfrak{A}/J}\left( a\right) \cap \mathbb{R}$. We want to show
that $\lambda \in \rho \left( a+d\right) $ for some $d\in J_{0}$. After
replacing $a$ with $a-\lambda $, it suffices to consider the case when $%
\lambda =0$. In this case, $a+J$ is invertible in $\mathfrak{A}/J$.
Therefore, by assumption there exists $d\in J\cap \mathfrak{A}_{\mathrm{sa}}$
such that $a+d$ is invertible in $\mathfrak{A}$. Since the set of invertible
elements of $\mathfrak{A}$ is norm-open, there exists $d_{0}\in J_{0}$ such
that $a+d_{0}$ is invertible in $\mathfrak{A}$, and hence $0\in \rho \left(
a+d_{0}\right) $.
\end{proof}

\begin{lemma}
\label{Lemma:Borel-spectrum-quotient}Suppose that $\mathfrak{A}$ is a strict
unital C*-algebra, and $J$ a norm-separable closed two-sided ideal of $%
\mathfrak{A}$. Suppose that every invertible self-adjoint element of $%
\mathfrak{A}/J$ lifts to an invertible self-adjoint element of $\mathfrak{A}$%
. Then the function $\mathfrak{A}_{\mathrm{sa}}\rightarrow \mathrm{Closed}%
\left( \mathbb{R}\right) $, $a\mapsto \sigma _{\mathfrak{A}/J}\left(
a\right) $ is Borel.
\end{lemma}

\begin{proof}
Fix a countable norm-dense subset $J_{0}$ of $J\cap \mathfrak{A}_{\mathrm{sa}%
}$. Then by the previous lemma we have that, for $a\in \mathfrak{A}_{\mathrm{%
sa}}$,%
\begin{equation*}
\sigma _{\mathfrak{A}/J}\left( a\right) =\bigcap_{d\in J_{0}}\sigma \left(
a+d\right) \text{.}
\end{equation*}%
As the function $\mathrm{Closed}\left( \mathbb{R}\right) ^{\omega
}\rightarrow \mathrm{Closed}\left( \mathbb{R}\right) $, $\left( F_{n}\right)
\mapsto \bigcap_{n\in \omega }F_{n}$ is Borel, this concludes the proof.
\end{proof}

\begin{lemma}
\label{Lemma:Borel-spectrum-B(H)}The function $B\left( H\right) _{\mathrm{sa}%
}\rightarrow K\left( \mathbb{R}\right) $, $T\mapsto \sigma _{\mathrm{ess}%
}\left( T\right) =\sigma _{Q\left( H\right) }\left( T\right) $ is Borel.
\end{lemma}

\begin{proof}
An operator $T\in B\left( H\right) $ induces an invertible element of $%
Q\left( H\right) $ if and only if it is Fredholm. If $T$ is Fredholm and
self-adjoint, then it has index $0$, and $0$ is an isolated point of the
spectrum of $T$ that is an eigenvalue with finite multiplicity. Thus, if $P$
is the finite-rank projection onto the eigenspace of $0$ for $T$, then we
have that $T+P$ is invertible and self-adjoint and induces the same element
of $Q\left( H\right) $ as $T$. This shows that every invertible self-adjoint
element of $Q\left( H\right) $ lifts to an invertible self-adjoint element
of $B\left( H\right) $. Therefore, the conclusion follows from Proposition %
\ref{Lemma:Borel-spectrum-quotient}.
\end{proof}

Suppose that $T\in \mathrm{\mathrm{Ball}}\left( B\left( H\right) \right) $
is a $\mathrm{\mathrm{\ \mathrm{mod}}\ }K\left( H\right) $ projective.
Recall that this means that $T$ is a positive operator satisfying $%
T^{2}\equiv T\mathrm{\ \mathrm{mod}}\ K\left( H\right) $. Then it is
well-known that there exists a projection $P\equiv T\mathrm{\ \mathrm{mod}}\
K\left( H\right) $. We observe that one can choose such a $P$ in a Borel
fashion from $T$; see \cite[Lemma 3.1]{andruchow_note_2020}.

\begin{lemma}
\label{Lemma:essential-projection}Consider the Borel set $\mathrm{Proj}%
\left( B\left( H\right) /K\left( H\right) \right) $ of $\mathrm{\mathrm{mod}}%
\ K\left( H\right) $ projections in $B\left( H\right) $. Then there is a
Borel function $\mathrm{Proj}\left( B\left( H\right) /K\left( H\right)
\right) \rightarrow \mathrm{\mathrm{Proj}}\left( B\left( H\right) \right) $, 
$T\mapsto P_{T}$ such that $T\equiv P_{T}\mathrm{\ \mathrm{mod}}\ K\left(
H\right) $ for every $T\in \mathrm{Proj}\left( B\left( H\right) /K\left(
H\right) \right) $.
\end{lemma}

\begin{proof}
Suppose that $T\in \mathrm{Proj}\left( B\left( H\right) /K\left( H\right)
\right) $. Observe $\sigma _{\mathrm{ess}}\left( T\right) \subseteq \left\{
0,1\right\} $. In particular, $\sigma _{\mathrm{ess}}\left( T\right) $ is
countable, with only accumulation points $0$ and $1$. From Lemma \ref%
{Lemma:Borel-spectrum-B(H)}, the maps $T\mapsto \sigma _{\mathrm{ess}}\left(
T\right) $ and $T\mapsto \sigma \left( T\right) $ are Borel. If $\sigma _{%
\mathrm{ess}}\left( T\right) =\left\{ 0\right\} $ then one can set $P_{T}=0$%
. If $\sigma _{\mathrm{ess}}\left( T\right) =\left\{ 1\right\} $, one can
set $P_{T}=1$.

Let us consider the case when $\left\{ 0,1\right\} =\sigma _{\mathrm{ess}%
}(T) $. By \cite[Theorem 12.13]{kechris_classical_1995} there exists a Borel
map $\mathrm{Proj}\left( B\left( H\right) /K\left( H\right) \right)
\rightarrow \left[ 0,1\right] ^{\omega }$, $T\mapsto \left( t_{n}\right) $
such that $\left( t_{n}\right) $ is an increasing enumeration of $\sigma
\left( T\right) $. One can then choose in a Borel way $n_{0}\in \omega $
such that $t_{n_{0}}<t_{n_{0}+1}$ and then a continuous function $f:\left[
0,1\right] \rightarrow \left[ 0,1\right] $ such that 
\begin{equation*}
f\left( t_{i}\right) =\left\{ 
\begin{array}{ll}
0 & \text{if }i\leq n_{0}\text{,} \\ 
1 & \text{if }i\geq n_{0}+1\text{.}%
\end{array}%
\right.
\end{equation*}%
One can then set $P_{T}=f\left( T\right) $.
\end{proof}

\subsection{Polar decompositions\label{Subsection:polar}}

We now observe that the polar decomposition of an operator is given by a
Borel function. We will use the following version of the selection theorem
for relations with compact sections from \cite[Theorem 28.8]%
{kechris_classical_1995}.

\begin{lemma}
\label{Lemma:select-compact-sections}Suppose that $X$ is a standard Borel
space, $Y$ is a compact metrizable space, and $A\subseteq X\times Y$ is a
Borel subset such that, for every $x\in X$, the vertical section%
\begin{equation*}
A_{x}=\left\{ y\in Y:\left( x,y\right) \in A\right\}
\end{equation*}%
is a closed nonempty set. Then the assignment $X\rightarrow \mathrm{Closed}%
\left( Y\right) $, $x\mapsto A_{x}$, is Borel, where $\mathrm{Closed}\left(
Y\right) $ is endowed with the Effros Borel structure.
\end{lemma}

As an application, we obtain the following. Let $H$ be a separable Hilbert
space. We consider the unit ball $\mathrm{\mathrm{Ball}}\left( H\right) $ of 
$H$ as a compact metrizable space endowed with the weak topology. We also
consider $\mathrm{Closed}\left( \mathrm{\mathrm{Ball}}\left( H\right)
\right) $ as a standard Borel space, endowed with the Effros Borel structure.

\begin{lemma}
\label{Lemma:select-Kernel}The function $B\left( H\right) \rightarrow 
\mathrm{Closed}\left( \mathrm{\mathrm{Ball}}\left( H\right) \right) $, $%
T\mapsto \mathrm{\mathrm{Ker}}\left( T\right) \cap \mathrm{Ball}\left(
H\right) $, is Borel.
\end{lemma}

\begin{proof}
By Lemma \ref{Lemma:select-compact-sections}, it suffices to show that the
set%
\begin{equation*}
A=\left\{ \left( T,x\right) \in B\left( H\right) \times \mathrm{\mathrm{Ball}%
}\left( H\right) :Tx=0\right\}
\end{equation*}%
is Borel. Fix a countable norm-dense subset $\left\{ x_{n}:n\in \omega
\right\} $ of $\mathrm{\mathrm{Ball}}\left( H\right) $. Then we have that,
if $\left( T,x\right) \in \mathrm{\mathrm{Ball}}\left( B\left( H\right)
\right) \times \mathrm{\mathrm{Ball}}\left( H\right) $, then $\left(
T,x\right) \in A$ if and only if $\forall k\in \omega $ $\exists n\in \omega 
$ such that $\left\Vert x-x_{n}\right\Vert <2^{-k}$ and $\left\Vert
Tx_{k}\right\Vert <2^{-k}$. Since the norm on $\mathrm{\mathrm{Ball}}\left(
H\right) $ is weakly lower-semicontinuous, this shows that $A$ is Borel.
\end{proof}

Recall that, for an operator $T\in B\left( H\right) $, one sets $\left\vert
T\right\vert :=\left( T^{\ast }T\right) ^{1/2}$. By strong-* continuity on
bounded sets of continuous functional calculus, the function $T\mapsto
\left\vert T\right\vert $ is Borel. Furthermore, there exists a unique
partial isometry $U$ with $\mathrm{\mathrm{Ker}}\left( U\right) =\mathrm{%
\mathrm{\mathrm{Ker}}}\left( T\right) $ such that $T=U\left\vert
T\right\vert $ \cite[Theorem 3.2.17]{pedersen_analysis_1989}. The
decomposition $T=U\left\vert T\right\vert $ is then called the \emph{polar
decomposition} of $T$.

\begin{lemma}
\label{Lemma:polar}The function $B\left( H\right) \rightarrow B\left(
H\right) $, $T\mapsto U$ that assigns to an operator the partial isometry $U$
in the polar decomposition of $T$ is Borel.
\end{lemma}

\begin{proof}
It suffices to notice that is graph, which is the set of pairs $\left(
T,U\right) $ such that $U$ is a partial isometry with $\mathrm{\mathrm{Ker}}%
\left( U\right) =\mathrm{\mathrm{\mathrm{Ker}}}\left( T\right) $ and $%
T=U\left\vert T\right\vert $, is Borel by Lemma \ref{Lemma:select-Kernel}.
\end{proof}

Consider the Borel set $U\left( B\left( H\right) /K\left( H\right) \right) $
of $\mathrm{\mathrm{mod}}\ K\left( H\right) $\emph{\ unitaries }in $B\left(
H\right) $. Thus, $T\in U\left( B\left( H\right) /K\left( H\right) \right) $
if and only if $T^{\ast }T\equiv TT^{\ast }\equiv I\mathrm{\ \mathrm{mod}}\
K\left( H\right) $. If $U$ is the partial isometry in the polar
decomposition of $T$, then $U\equiv T\mathrm{\ \mathrm{mod}}\ K\left(
H\right) $ and $U$ is an essential unitary. In fact, one can easily define
(in a Borel fashion from $T$) an isometry or co-isometry $V$ such that $%
T\equiv V\mathrm{\ \mathrm{mod}}\ K\left( H\right) $. One has that $T$ is in
particular a Fredholm operator. Its index is defined by%
\begin{equation*}
\mathrm{index}\left( T\right) =\mathrm{\mathrm{rank}}\left( 1-V^{\ast
}V\right) -\mathrm{\mathrm{rank}}\left( 1-VV^{\ast }\right) \text{.}
\end{equation*}%
Thus, $\mathrm{index}\left( T\right) $ is a Borel function of $T\in U\left(
B\left( H\right) /K\left( H\right) \right) $.

More generally, consider the Borel set of pairs $\left( P,T\right) \in 
\mathrm{\mathrm{Ball}}\left( B\left( H\right) \right) ^{2}$ such that $P$ is
a projection, $PT=TP=T$ and $TT^{\ast }\equiv T^{\ast }T\equiv P$.\ If $V$
is the partial isometry in the polar decomposition of $T$, then $V\equiv T%
\mathrm{\ \mathrm{mod}}\ K\left( H\right) $ and the index of $PTP$ regarded
as a Fredholm operator on $PH$ is given by the Borel function%
\begin{equation*}
\mathrm{index}\left( PTP\right) =\mathrm{\mathrm{rank}}\left( P-V^{\ast
}V\right) -\mathrm{\mathrm{rank}}\left( P-VV^{\ast }\right) \text{.}
\end{equation*}

\section{$\mathrm{K}$-theory of unital C*-algebras with a strict cover\label%
{Section:K-theory}}

In this section we explain how the $\mathrm{K}_{0}$ and $\mathrm{K}_{1}$
groups of a unital C*-algebra with a strict cover can be regarded as
semidefinable groups. We also recall the definition of the index map and the
exponential map between the $\mathrm{K}_{0}$ and $\mathrm{K}_{1}$ groups,
and observe that they are definable homomorphisms. Finally, we consider the
six-term exact sequence associated with a strict unital C*-pair, and observe
that the connective maps are all definable group homomorphisms.

\subsection{$\mathrm{K}_{0}$-group}

Suppose that $\mathfrak{A/J}$ is a unital C*-algebra with a strict cover.
Recall that $\mathrm{Proj}\left( \mathfrak{A/J}\right) $ denotes the Polish
space of $\mathrm{\mathrm{\mathrm{mod}}\ }\mathfrak{J}$ projections in $%
\mathfrak{A}$. Similarly, for $n\geq 1$ we have that $\mathrm{Proj}\left(
M_{n}\left( \mathfrak{A/J}\right) \right) :=\mathrm{Proj}\left( M_{n}\left( 
\mathfrak{A}\right) /M_{n}\left( \mathfrak{J}\right) \right) $ is a Polish
space. We say that an element of $\mathrm{Proj}\left( M_{n}\left( \mathfrak{%
A/J}\right) \right) $ for some $n\geq 1$ is a $\mathrm{\mathrm{\mathrm{mod}}%
\ }\mathfrak{J}$ projection \emph{over} $\mathfrak{A}$. We define \textrm{Z}$%
_{0}\left( \mathfrak{A/J}\right) $ be the set of \emph{pairs of} $\mathrm{%
\mathrm{mod}}\ \mathfrak{J}$ projections over\emph{\ }$\mathfrak{A}$, which
is the disjoint union of \textrm{Z}$_{0}^{\left( n\right) }\left( \mathfrak{%
A/J}\right) :=\mathrm{Proj}\left( M_{n}\left( \mathfrak{A/J}\right) \right)
\times \mathrm{Proj}\left( M_{n}\left( \mathfrak{A/J}\right) \right) $ for $%
n\geq 1$ endowed with the induced standard Borel structure. Two $\mathrm{%
\mathrm{mod}}\ \mathfrak{J}$ projections $p,q$ in $\mathfrak{A}$ are
Murray--von Neumann equivalent (respectively, unitary equivalent, and
homotopic) $\mathrm{\mathrm{mod}}\ \mathfrak{J}$ if and only if $p+\mathfrak{%
J}$ and $q+\mathfrak{J}$ are Murray--von Neumann equivalent (respectively,
unitary equivalent, and homotopic) in $\mathfrak{A/J}$.

The $\mathrm{K}_{0}$-group $\mathrm{K}_{0}\left( \mathfrak{A/J}\right) $ of $%
\mathfrak{A/J}$---see \cite[Chapter 4]{higson_analytic_2000}---is defined as
a quotient of $\mathrm{Z}_{0}\left( \mathfrak{A/J}\right) $ by an
equivalence relation $\mathrm{B}_{0}\left( \mathfrak{A/J}\right) $, defined
as follows. For $\left( p,p^{\prime }\right) ,\left( q,q^{\prime }\right)
\in \mathrm{Z}_{0}\left( \mathfrak{A/J}\right) $, $\left( p,p^{\prime
}\right) \mathrm{B}_{0}\left( \mathfrak{A/J}\right) \left( q,q^{\prime
}\right) $ if and only if there exist $m,n\in \omega $ and $r\in \mathrm{Proj%
}\left( M_{m}\left( \mathfrak{A/J}\right) \right) $ such that $p\oplus
q^{\prime }\oplus r\oplus 0_{n}$ and $q\oplus p^{\prime }\oplus r\oplus
0_{n} $ are Murray--von Neumann equivalent $\mathrm{\mathrm{mod}}\ \mathfrak{%
J}$. By Lemma \ref{Lemma:orthogonal-projections}, we have the following
equivalent description of $\mathrm{B}_{0}\left( \mathfrak{A/J}\right) $.

\begin{lemma}
\label{Lemma:equivalent-B0}Suppose that $\mathfrak{A/J}$ is a unital
C*-algebra with a strict cover, and $\left( p,p^{\prime }\right) ,\left(
q,q^{\prime }\right) \in \mathrm{Z}_{0}\left( \mathfrak{A/J}\right) $ where $%
p,p^{\prime }\in M_{d}\left( \mathfrak{A/J}\right) $ and $q,q^{\prime }\in
M_{k}\left( \mathfrak{A/J}\right) $. Then $\left( p,p^{\prime }\right) 
\mathrm{B}_{0}\left( \mathfrak{A/J}\right) \left( q,q^{\prime }\right) $ if
and only if there exist $m,n\in \omega $ and 
\begin{equation*}
y_{1},\ldots ,y_{\ell }\in \mathrm{\mathrm{Ball}}\left( M_{d+k+m+n}\left( 
\mathfrak{A}\right) _{\mathrm{sa}}\right)
\end{equation*}
such that, setting $u:=e^{iy_{1}}\cdots e^{iy_{\ell }}$, one has that 
\begin{equation*}
u\left( p\oplus q^{\prime }\oplus 1_{m}\oplus 0_{n}\right) u^{\ast }\equiv
q\oplus p^{\prime }\oplus 1_{m}\oplus 0_{n}\mathrm{\ \mathrm{mod}}\ 
\mathfrak{J}\text{,}
\end{equation*}%
where $\ell \geq 1$ does not depend on $\mathfrak{A/J}$ and $\left(
p,p^{\prime }\right) ,\left( q,q^{\prime }\right) \in \mathrm{Z}_{0}\left( 
\mathfrak{A/J}\right) $.
\end{lemma}

The (commutative) group operation on $\mathrm{K}_{0}\left( \mathfrak{A/J}%
\right) $ is induced by the Borel function on $\mathrm{Z}_{0}\left( 
\mathfrak{A/J}\right) $, $\left( \left( p,p^{\prime }\right) ,\left(
q,q^{\prime }\right) \right) \mapsto \left( p\oplus q,p^{\prime }\oplus
q^{\prime }\right) $. The neutral element of $\mathrm{K}_{0}\left( \mathfrak{%
A/J}\right) $ corresponds to $\left( 0,0\right) \in \mathrm{Z}_{0}\left( 
\mathfrak{A/J}\right) $. The function that maps an element to its additive
inverse is induced by the Borel function on $\mathrm{Z}_{0}\left( \mathfrak{%
A/J}\right) $ given by $\left( p,p^{\prime }\right) \mapsto \left( p^{\prime
},p\right) $. Thus, $\mathrm{K}_{0}\left( \mathfrak{A/J}\right) $ is in fact
a semidefinable group.

If $\mathfrak{A/I}$ and $\mathfrak{B/J}$ are unital C*-algebras with a
strict cover, and $\varphi :\mathfrak{A/I}\rightarrow \mathfrak{B/J}$ is a
definable unital *-homomorphism, then the induced group homomorphism \textrm{%
K}$_{0}\left( \mathfrak{A/I}\right) \rightarrow \mathrm{K}_{0}\left( 
\mathfrak{B/J}\right) $ is also definable. Thus, the assignment $\mathfrak{%
A/J}\rightarrow \mathrm{K}_{0}\left( \mathfrak{A/J}\right) $ gives a functor
from the category of unital C*-algebras with a strict cover to the category
of semidefinable abelian groups.

Suppose that $\left( \mathfrak{A},\mathfrak{J}\right) $ is a strict unital
C*-pair. We denote by $\mathfrak{J}^{+}$ the unitization of $\mathfrak{J}$,
which can be identified with the C*-subalgebra $\mathfrak{J}^{+}=\mathrm{%
\mathrm{span}}\left\{ \mathfrak{J},1\right\} \subseteq \mathfrak{A}$. Since $%
\mathfrak{J}$ is a proper ideal of $\mathfrak{A}$, we can write every
element of $\mathfrak{J}^{+}$ uniquely as $a+\lambda 1$ where $a\in 
\mathfrak{J}$ and $\lambda \in \mathbb{C}$. More generally, every element of 
$M_{n}\left( \mathfrak{J}^{+}\right) $ can be written uniquely as $a+\alpha
1 $ where $a\in M_{n}\left( \mathfrak{J}\right) $ and $\alpha \in
M_{n}\left( \mathbb{C}\right) $. As the map $M_{n}\left( \mathfrak{J}%
^{+}\right) \rightarrow M_{n}\left( \mathbb{C}\right) $, $a+\alpha 1\mapsto
\alpha $ is a unital *-homomorphism, we have that $\left\Vert \alpha
\right\Vert \leq \left\Vert a+\alpha 1\right\Vert $ and hence $\left\Vert
a\right\Vert \leq 2\left\Vert a+\alpha 1\right\Vert $ for $a+\alpha 1\in
M_{n}\left( \mathfrak{J}^{+}\right) $.

We define $\mathrm{Proj}(M_{n}(\mathfrak{J}^{+}))$ to be the set of
projections in $M_{n}(\mathfrak{J}^{+})$, which we regard as a Borel subset
of $2\mathrm{\mathrm{Ball}}\left( M_{n}\left( \mathfrak{J}\right) \right)
\times \mathrm{\mathrm{Ball}}\left( M_{n}\left( \mathbb{C}\right) \right) $.
Similarly, the unitary group $U(M_{n}(\mathfrak{J}^{+}))$ of $M_{n}(%
\mathfrak{J}^{+})$ is regarded as a Borel subset of $2\mathrm{\mathrm{Ball}}%
\left( M_{n}\left( \mathfrak{J}\right) \right) \times \mathrm{\mathrm{Ball}}%
\left( M_{n}\left( \mathbb{C}\right) \right) $. Define also $\mathrm{Z}%
_{0}^{\left( n\right) }\left( \mathfrak{J}\right) $ to be the Borel subset
of $\mathrm{Proj}(M_{n}(\mathfrak{J}^{+}))\times \mathrm{Proj}(M_{n}(%
\mathfrak{J}^{+}))$ consisting of pairs $\left( p,p^{\prime }\right) $ such
that $p\equiv p^{\prime }\mathrm{\ \mathrm{mod}}\ M_{n}\left( \mathfrak{J}%
\right) $. Finally, let $\mathrm{Z}_{0}\left( \mathfrak{J}\right) $ to be
the disjoint union of $\mathrm{Z}_{0}^{\left( n\right) }\left( \mathfrak{J}%
\right) $ for $n\geq 1$.

The $\mathrm{K}_{0}$-group $\mathrm{K}_{0}\left( \mathfrak{J}\right) $ of $%
\mathfrak{J}$---see \cite[Definition 4.2.1]{higson_analytic_2000}---is
defined as a quotient of $\mathrm{Z}_{0}\left( \mathfrak{J}\right) $ by an
equivalence relation $\mathrm{B}_{0}\left( \mathfrak{J}\right) $, defined as
follows. One has that, for $\left( p,p^{\prime }\right) ,\left( q,q^{\prime
}\right) \in \mathrm{Z}_{0}\left( \mathfrak{J}\right) $, $\left( p,p^{\prime
}\right) \mathrm{B}_{0}\left( \mathfrak{J}\right) \left( q,q^{\prime
}\right) $ if and only if there exist $m,n\in \omega $ and $x\in \mathrm{Proj%
}(M_{m}(\mathfrak{J}^{+}))$ such that $p\oplus q^{\prime }\oplus x\oplus
0_{n}$ and $q\oplus p^{\prime }\oplus x\oplus 0_{n^{\prime }}$ are
Murray--von Neumann equivalent. For $\left( p,p^{\prime }\right) \in \mathrm{%
Z}_{0}\left( \mathfrak{J}\right) $, we let $\left[ p\right] -\left[
p^{\prime }\right] $ be the corresponding element of $\mathrm{K}_{0}\left( 
\mathfrak{J}\right) $. The (commutative) group operation on $\mathrm{K}%
_{0}\left( \mathfrak{J}\right) $ is induced by the Borel function on $%
\mathrm{Z}_{0}\left( \mathfrak{J}\right) $, $\left( \left( p,p^{\prime
}\right) ,\left( q,q^{\prime }\right) \right) \mapsto \left( p\oplus
q,p^{\prime }\oplus q^{\prime }\right) $. The neutral element of $\mathrm{K}%
_{0}\left( \mathfrak{J}\right) $ corresponds to $\left( 0,0\right) \in 
\mathrm{Z}_{0}\left( \mathfrak{J}\right) $. The function that maps an
element of $\mathrm{K}_{0}\left( \mathfrak{J}\right) $ to its additive
inverse is induced by the Borel function on $\mathrm{Z}_{0}\left( \mathfrak{J%
}\right) $ given by $\left( p,p^{\prime }\right) \mapsto \left( p^{\prime
},p\right) $. Thus, $\mathrm{K}_{0}\left( \mathfrak{J}\right) $ is a
semidefinable group.

If $\left( \mathfrak{A},\mathfrak{I}\right) $ are $\left( \mathfrak{B},%
\mathfrak{J}\right) $ are strict C*-pairs, and $\varphi :\left( \mathfrak{A},%
\mathfrak{I}\right) \rightarrow \left( \mathfrak{B},\mathfrak{J}\right) $ is
a strict *-homomorphism, then it induces a strict *-homomorphism $\varphi |_{%
\mathfrak{I}}:\mathfrak{I}\rightarrow \mathfrak{J}$. In turn, this induces a
definable group homomorphism $\mathrm{K}_{0}\left( \mathfrak{J}\right)
\rightarrow \mathrm{K}_{0}\left( \mathfrak{I}\right) $. This gives a functor 
$\left( \mathfrak{A},\mathfrak{I}\right) \mapsto \mathrm{K}_{0}\left( 
\mathfrak{I}\right) $ from strict unital C*-pairs to semidefinable groups.

\begin{lemma}
Suppose that $\left( \mathfrak{A},\mathfrak{J}\right) $ is a strict unital
C*-pair. Then there is a Borel map $\mathrm{Z}_{0}\left( \mathfrak{J}\right)
\rightarrow \mathrm{Z}_{0}\left( \mathfrak{J}\right) $, $\left( P,P^{\prime
}\right) \mapsto \left( p,p^{\prime }\right) $ such that $\left[ P\right] -%
\left[ P^{\prime }\right] =\left[ p\right] -\left[ p^{\prime }\right] $ and $%
p^{\prime }\in M_{n}\left( \mathbb{C}\right) $.
\end{lemma}

\begin{proof}
Suppose that $\left( P,P^{\prime }\right) \in \mathrm{Z}_{0}^{\left(
d\right) }\left( \mathfrak{J}\right) $. By definition, we have that for some 
$x,x^{\prime }\in M_{d}\left( \mathfrak{J}\right) $ and $\alpha \in
M_{d}\left( \mathbb{C}\right) $, $P=x+\alpha $ and $P^{\prime }=x^{\prime
}+\alpha $.\ Thus, we can define%
\begin{equation*}
p:=%
\begin{bmatrix}
P & 0 \\ 
0 & 1_{d}-P^{\prime }%
\end{bmatrix}%
\in M_{2d}\left( \mathfrak{J}^{+}\right)
\end{equation*}%
and%
\begin{equation*}
p^{\prime }:=%
\begin{bmatrix}
\alpha & 0 \\ 
0 & 1_{d}-\alpha%
\end{bmatrix}%
\in M_{2d}\left( \mathbb{C}\right) \text{.}
\end{equation*}%
Then we have that%
\begin{equation*}
\left[ p\right] -\left[ p^{\prime }\right] =\left[ P\right] +\left[
1_{d}-P^{\prime }\right] +\left[ \alpha \right] -\left[ 1_{d}-\alpha \right]
=\left[ P\right] -\left[ P^{\prime }\right] \text{.}
\end{equation*}%
This concludes the proof.
\end{proof}

\subsection{Relative $\mathrm{K}_{0}$-group}

Suppose now that $\left( \mathfrak{A},\mathfrak{J}\right) $ is a strict
unital C*-pair. For $n\geq 1$, define $\mathrm{Z}_{0}^{\left( n\right)
}\left( \mathfrak{A},\mathfrak{A/J}\right) $ to be the Borel set of triples $%
\left( p,q,x\right) \in \mathrm{\mathrm{Ball}}\left( \mathfrak{A}\right)
^{3} $ where $p,q$ are projections and $x\in \mathrm{\mathrm{Ball}}\left(
M_{n}\left( \mathfrak{A}\right) \right) $ satisfies $x^{\ast }x\equiv p%
\mathrm{\ \mathrm{mod}}\ M_{n}\left( \mathfrak{J}\right) $ and $xx^{\ast
}\equiv q\mathrm{\ \mathrm{mod}}\ M_{n}\left( \mathfrak{J}\right) $. Define $%
\mathrm{Z}_{0}\left( \mathfrak{A},\mathfrak{A/J}\right) $ to be the disjoint
union of $\mathrm{K}_{0}^{\left( n\right) }\left( \mathfrak{A},\mathfrak{A/J}%
\right) $ for $n\geq 1$ endowed with the induced standard Borel structure.
The elements of $\mathrm{Z}_{0}\left( \mathfrak{A},\mathfrak{A/J}\right) $
are called relative $\mathrm{K}$-cycles for $\left( \mathfrak{A},\mathfrak{%
A/J}\right) $; see \cite[Definition 4.3.1]{higson_analytic_2000}. If $\left(
p,q,x\right) \in \mathrm{Z}_{0}^{\left( n\right) }\left( \mathfrak{A},%
\mathfrak{A/J}\right) $, then we say that $\left( p,q,x\right) $ is a
relative $\mathrm{K}$-cycle of dimension $n$. A relative $\mathrm{K}$-cycle $%
\left( p,q,x\right) $ for $\left( \mathfrak{A},\mathfrak{A/J}\right) $ is 
\emph{degenerate }if $x^{\ast }x=p$ and $xx^{\ast }=q$. Two relative $%
\mathrm{K}$-cycles $\left( p,q,x\right) $ and $\left( p^{\prime },q^{\prime
},x^{\prime }\right) $ of dimension $n$ are homotopic if there exists a 
\emph{norm-continuous }path $\left( \left( p_{t},q_{t},x_{t}\right) \right)
_{t\in \left[ 0,1\right] }$ of relative $\mathrm{K}$-cycles for $\left( 
\mathfrak{A},\mathfrak{A/J}\right) $ of dimension $n$ with $\left(
p,q,x\right) =\left( p_{0},q_{0},x_{0}\right) $ and $\left( p^{\prime
},q^{\prime },x^{\prime }\right) =\left( p_{1},q_{1},x_{1}\right) $.

Notice that if $\left( p,q,x\right) $ is a relative $\mathrm{K}$-cycle of
dimension $d$, and $\left( u_{t}\right) _{t\in \left[ 0,1\right] }$ is a
path of unitaries in $M_{d}\left( \mathfrak{A}\right) $ starting at $1$, then%
\begin{equation*}
\left( u_{t}^{\ast }pu_{t},u_{t}^{\ast }qu_{t},u_{t}^{\ast }xu_{t}\right)
\end{equation*}%
and%
\begin{equation*}
\left( p,u_{t}^{\ast }qu_{t},u_{t}^{\ast }x\right)
\end{equation*}%
are norm-continuous paths of relative $\mathrm{K}$-cycles starting at $%
\left( p,q,x\right) $. If $p\equiv q\equiv x\mathrm{\ \mathrm{mod}}\
M_{d}\left( \mathfrak{J}\right) $,$\mathrm{\ }$then%
\begin{equation*}
\left( p,q,tp+\left( 1-t\right) q\right)
\end{equation*}%
is a norm-continuous path of relative cycles from $\left( p,q,x\right) $ to $%
\left( p,q,p\right) $. We have the following lemma; see \cite[Proposition 3.4%
]{nest_excision_2017}.

\begin{lemma}
\label{Lemma:relative-K-cycles}Suppose that $\left( p,q,x\right) $ is a
relative cycle of dimension $n$ for $\left( \mathfrak{A},\mathfrak{A/J}%
\right) $. Then $r_{0}:=\left( p\oplus q,q\oplus p,x\oplus x^{\ast }\right) $
is homotopic to the degenerate cycle $\left( p\oplus q,p\oplus q,p\oplus
q\right) $.
\end{lemma}

The relative $\mathrm{K}_{0}$-group $\mathrm{K}_{0}\left( \mathfrak{A},%
\mathfrak{A/J}\right) $ is defined to be the quotient of $\mathrm{Z}%
_{0}\left( \mathfrak{A},\mathfrak{A/J}\right) $ by the equivalence relation $%
\mathrm{B}_{0}\left( \mathfrak{A},\mathfrak{A/J}\right) $ defined as
follows. For $\left( p,q,x\right) ,\left( p^{\prime },q^{\prime },x^{\prime
}\right) $, set $\left( p,q,x\right) \mathrm{B}_{0}\left( \mathfrak{A},%
\mathfrak{A/J}\right) \left( p^{\prime },q^{\prime },x^{\prime }\right) $ if
and only if there exists a \emph{degenerate }relative $\mathrm{K}$-cycles $%
\left( p_{0},q_{0},x_{0}\right) ,\left( p_{0}^{\prime },q_{0}^{\prime
},x_{0}^{\prime }\right) $ such that $\left( p\oplus p_{0},q\oplus
q_{0},x\oplus x_{0}\right) $ and $\left( p^{\prime }\oplus p_{0}^{\prime
},q^{\prime }\oplus q_{0}^{\prime },x^{\prime }\oplus x_{0}^{\prime }\right) 
$ are of the same dimension and homotopic. The group operations on $\mathrm{K%
}_{0}\left( \mathfrak{A},\mathfrak{A/J}\right) $ are induced by the Borel
maps 
\begin{equation*}
\left( \left( p,q,x\right) ,\left( p^{\prime },q^{\prime },x^{\prime
}\right) \right) \mapsto \left( p\oplus p^{\prime },q\oplus q^{\prime
},x\oplus x^{\prime }\right)
\end{equation*}%
and%
\begin{equation*}
\left( p,q,x\right) \mapsto \left( q,p,x^{\ast }\right) \text{.}
\end{equation*}%
It follows from Lemma \ref{Lemma:relative-K-cycles} that $\mathrm{K}%
_{0}\left( \mathfrak{A},\mathfrak{A/J}\right) $ is indeed a group. We let $%
\left[ p,q,x\right] $ be the element of $\mathrm{K}_{0}\left( \mathfrak{A},%
\mathfrak{A/J}\right) $ represented by the relative $\mathrm{K}$-cycle $%
\left( p,q,x\right) $. The trivial element of $\mathrm{K}_{0}\left( 
\mathfrak{A},\mathfrak{A/J}\right) $ is equal to $\left[ p,q,x\right] $
where $\left( p,q,x\right) $ is any degenerate relative $\mathrm{K}$-cycle.
Let $\mathfrak{J}^{+}$ be the unitization of $\mathfrak{J}$, which we
identify with $\mathrm{\mathrm{span}}\left( \mathfrak{J},1\right) \subseteq 
\mathfrak{A}$.

\begin{lemma}
\label{Lemma:relative-cycle-unitization}There is a Borel function $\mathrm{Z}%
_{0}\left( \mathfrak{A},\mathfrak{A/J}\right) \rightarrow \mathrm{Z}%
_{0}\left( \mathfrak{A},\mathfrak{A/J}\right) $, $\left( P,Q,X\right)
\mapsto \left( p,q,p\right) $ such that $p\in M_{n}\left( \mathbb{C}\right) $%
, $q\in M_{n}\left( \mathfrak{J}^{+}\right) $, $p\equiv q\in M_{n}\left( 
\mathfrak{J}\right) $, and $[P,Q,X]=[p,q,p]\in \mathrm{K}_{0}\left( 
\mathfrak{A},\mathfrak{A/J}\right) $.
\end{lemma}

\begin{proof}
Notice that $\left( 1-P,1-P,1-P\right) $ is a degenerate relative $\mathrm{K}
$-cycle of dimension $d$. Consider then 
\begin{equation*}
\left( P\oplus \left( 1-P\right) ,Q\oplus \left( 1-P\right) ,X\oplus \left(
1-P\right) \right) .
\end{equation*}%
By Lemma \ref{Lemma:homotopy-orthogonal-unitaries} one can choose $%
Y_{1},\ldots ,Y_{\ell }\in \mathrm{\mathrm{Ball}}\left( M_{2d}\left( 
\mathfrak{A}\right) _{\mathrm{sa}}\right) $ in a Borel way such that,
setting $U:=e^{iY_{1}}\cdots e^{iY_{\ell }}$, one has that $U\left( P\oplus
\left( 1-P\right) U^{\ast }\right) =1_{d}\oplus 0_{d}$, where $\ell \geq 1$
does not depend on $\mathfrak{A},\mathfrak{J}$ and $\left( P,Q,X\right) $.
Thus, after replacing $\left( P,Q,X\right) $ with 
\begin{equation*}
\left( U\left( P\oplus \left( 1-P\right) \right) U^{\ast },U\left( Q\oplus
\left( 1-P\right) \right) U^{\ast },U\left( X\oplus \left( 1-P\right)
\right) U^{\ast }\right) ,
\end{equation*}%
we can assume without loss of generality that $P=1_{d}\oplus 0_{d}$.

By Lemma \ref{Lemma:relative-path-unitary} one can choose $Y_{1},\ldots
,Y_{\ell }\in \mathrm{\mathrm{Ball}}\left( M_{4d}\left( \mathfrak{A}\right)
_{\mathrm{sa}}\right) $ in a Borel fashion from $\left( P,Q,X\right) $ such
that, setting $U:=e^{iY_{1}}\cdots e^{iY_{\ell }}$, one has that 
\begin{equation*}
U^{\ast }\left( Q\oplus 0_{2d}\right) U\equiv \left( P\oplus 0_{2d}\right) 
\mathrm{\ \mathrm{mod}}\ M_{2d}\left( \mathfrak{J}\right) \text{.}
\end{equation*}%
Thus, after replacing $\left( P,Q,X\right) $ with $\left( P\oplus
0_{2d},U^{\ast }\left( Q\oplus 0_{2d}\right) U,U^{\ast }\left( X\oplus
0_{2d}\right) \right) $ we can assume without loss of generality that $%
P=1_{d}\oplus 0_{3d}\in M_{4d}\left( \mathbb{C}\right) $ and $Q\in
M_{4d}\left( \mathfrak{A}\right) $ satisfy $P\equiv Q\mathrm{\ \mathrm{mod}}%
\ M_{4d}\left( \mathfrak{J}\right) $ and hence $Q\in M_{4n}\left( \mathfrak{J%
}^{+}\right) $.

In this case, we have that $\left[ P,Q,X\right] =\left[ P,Q,P\right] $,
since $\left( P_{t},Q_{t},tP+\left( 1-t\right) X\right) _{t\in \left[ 0,1%
\right] }$ is a norm-continuous path of relative $\mathrm{K}$-cycles from $%
\left( P,Q,X\right) $ to $\left( P,Q,P\right) $. This concludes the proof.
\end{proof}

\begin{proposition}
\label{Proposition:excision-K-theory}Suppose that $\left( \mathfrak{A},%
\mathfrak{J}\right) $ is a strict C*-pair. Then $\mathrm{K}_{0}\left( 
\mathfrak{A},\mathfrak{A/J}\right) $ is a definable group. The assignment $%
\mathrm{K}_{0}\left( \mathfrak{J}\right) \rightarrow \mathrm{K}_{0}\left( 
\mathfrak{A},\mathfrak{A/J}\right) $, $[P]-\left[ Q\right] \mapsto \lbrack
P,Q,P]$ is a natural definable isomorphism, called the\emph{\ excision
isomorphism}.
\end{proposition}

\begin{proof}
By \cite[Theorem 3.9]{nest_excision_2017}, the excision homomorphism $%
\mathrm{K}_{0}\left( \mathfrak{J}\right) \rightarrow \mathrm{K}_{0}\left( 
\mathfrak{A},\mathfrak{A/J}\right) $ is bijective; see also \cite[Theorem
4.3.8]{higson_analytic_2000}. Clearly, it is induced by a Borel function $%
\mathrm{Z}_{0}\left( \mathfrak{J}\right) \rightarrow \mathrm{Z}_{0}\left( 
\mathfrak{A},\mathfrak{A/J}\right) $. By Lemma \ref%
{Lemma:relative-cycle-unitization} the inverse homomorphism $\mathrm{K}%
_{0}\left( \mathfrak{A},\mathfrak{A/J}\right) \rightarrow \mathrm{K}%
_{0}\left( \mathfrak{J}\right) $ is also induced by a Borel function $%
\mathrm{Z}_{0}\left( \mathfrak{A},\mathfrak{A/J}\right) \rightarrow \mathrm{Z%
}_{0}\left( \mathfrak{J}\right) $. Thus, $\mathrm{K}_{0}\left( \mathfrak{A},%
\mathfrak{A/J}\right) $ is a definable group, and the excision isomorphism
is a definable isomorphism.
\end{proof}

There is a natural definable homomorphism $\mathrm{K}_{0}\left( \mathfrak{A},%
\mathfrak{A/J}\right) \rightarrow \mathrm{K}_{0}\left( \mathfrak{A}\right) $
that is induced by the Borel map \textrm{Z}$_{0}\left( \mathfrak{A},%
\mathfrak{A/J}\right) \rightarrow \mathrm{Z}_{0}\left( \mathfrak{A}\right) $ 
$\left( p,q,x\right) \mapsto \left( p,q\right) $. We also have a natural
definable homomorphism $\mathrm{K}_{0}\left( \mathfrak{A}\right) \rightarrow 
\mathrm{K}_{0}\left( \mathfrak{A/J}\right) $ induced by the Borel map $%
\mathrm{Z}_{0}\left( \mathfrak{A}\right) \rightarrow \mathrm{Z}_{0}\left( 
\mathfrak{A/J}\right) $. We have the following result; see \cite[Proposition
4.3.5]{higson_analytic_2000}.

\begin{proposition}
\label{Proposition:exact-K-theory-relative}Suppose that $\left( \mathfrak{A},%
\mathfrak{J}\right) $ is a strict unital C*-pair. The (natural) sequence of
definable groups and definable group homomorphisms 
\begin{equation*}
\mathrm{K}_{0}\left( \mathfrak{A},\mathfrak{A/J}\right) \rightarrow \mathrm{K%
}_{0}\left( \mathfrak{A}\right) \rightarrow \mathrm{K}_{0}\left( \mathfrak{%
A/J}\right)
\end{equation*}%
is exact.
\end{proposition}

Combining the excision isomorphism $\mathrm{K}_{0}\left( \mathfrak{J}\right)
\rightarrow \mathrm{K}_{0}\left( \mathfrak{A},\mathfrak{A/J}\right) $ with
the natural definable homomorphism $\mathrm{K}_{0}\left( \mathfrak{A},%
\mathfrak{A/J}\right) \rightarrow \mathrm{K}_{0}\left( \mathfrak{A}\right) $%
, we obtain a natural definable group homomorphism $\mathrm{K}_{0}\left( 
\mathfrak{J}\right) \rightarrow \mathrm{K}_{0}\left( \mathfrak{A}\right) $.
This is defined by mapping $\left( p,q\right) \in \mathrm{Z}_{0}^{\left(
n\right) }\left( \mathfrak{J}\right) $ to $\left( p,q\right) $ regarded as
an element of $\mathrm{Z}_{0}^{\left( n\right) }\left( \mathfrak{A}\right) $%
. Combining Proposition \ref{Proposition:excision-K-theory} with Proposition %
\ref{Proposition:exact-K-theory-relative} we have the following.

\begin{corollary}
\label{Corollary:exact-K-theory}Suppose that $\mathfrak{A}$ is a unital
strict C*-algebra and $\mathfrak{J}$ is a proper strict ideal of $\mathfrak{A%
}$. Then the natural sequence 
\begin{equation*}
\mathrm{K}_{0}\left( \mathfrak{J}\right) \rightarrow \mathrm{K}_{0}\left( 
\mathfrak{A}\right) \rightarrow \mathrm{K}_{0}\left( \mathfrak{A/J}\right)
\end{equation*}%
is exact.
\end{corollary}

\subsection{$\mathrm{K}_{1}$ group}

Suppose that $\left( \mathfrak{A},\mathfrak{J}\right) $ is a strict unital
C*-pair. We can then consider the Borel set $U\left( \mathfrak{A/J}\right) $
of elements of $\mathrm{\mathrm{Ball}}\left( \mathfrak{A}\right) $ that are
unitaries $\mathrm{\mathrm{mod}}\ \mathfrak{J}$. We then let $\mathrm{Z}%
_{1}\left( \mathfrak{A/J}\right) $ to be the disjoint union of $U\left(
M_{n}\left( \mathfrak{A}\right) /M_{n}\left( \mathfrak{J}\right) \right) $
for $n\geq 1$. The equivalence relation \textrm{B}$_{1}\left( \mathfrak{A/J}%
\right) $ on $\mathrm{Z}_{1}\left( \mathfrak{A/J}\right) $ is defined by
setting $u$\textrm{B}$_{1}\left( \mathfrak{A/J}\right) u^{\prime }$ for $%
u\in U\left( M_{n}\left( \mathfrak{A}\right) /M_{n}\left( \mathfrak{J}%
\right) \right) $ and $u^{\prime }\in U\left( M_{n^{\prime }}\left( 
\mathfrak{A}\right) /M_{n^{\prime }}\left( \mathfrak{J}\right) \right) $ if
and only if there exist $k,k^{\prime }\in \omega $ with $n+k=n^{\prime
}+k^{\prime }$ and such that there is a \emph{norm-continuous} path from $%
u\oplus 1_{k}+M_{n+k}\left( \mathfrak{J}\right) $ to $u^{\prime }\oplus
1_{k^{\prime }}+M_{n+k}\left( \mathfrak{J}\right) $ in the unitary group of
the quotient unital C*-algebra $M_{n+k}\left( \mathfrak{A}\right)
/M_{n+k}\left( \mathfrak{J}\right) $. This equivalence relation is analytic
by Corollary \ref{Corollary:homotopy-unitaries}. The definable $\mathrm{K}%
_{1}$-group $\mathrm{K}_{1}\left( \mathfrak{A/J}\right) $ is then the
semidefinable group obtained as quotient $\mathrm{Z}_{1}\left( \mathfrak{A/J}%
\right) /\mathrm{B}_{1}\left( \mathfrak{A/J}\right) $ with group operations
defined as above. This defines a functor $\mathfrak{A/J}\mapsto \mathrm{K}%
_{1}\left( \mathfrak{A/J}\right) $ from unital C*-algebras with strict cover
to semidefinable groups.

Given a strict unital C*-pair $\left( \mathfrak{A},\mathfrak{J}\right) $, we
also consider the definable $\mathrm{K}_{1}$-group $\mathrm{K}_{1}\left( 
\mathfrak{J}\right) $ of $\mathfrak{J}$. As above, we identify the
unitization $\mathfrak{J}^{+}$ of $\mathfrak{J}$ with $\mathrm{\mathrm{span}}%
\left\{ \mathfrak{J},1\right\} \subseteq \mathfrak{A}$. For $n\geq 1$ we let 
$U\left( M_{d}\left( \mathfrak{J}^{+}\right) \right) $ be the unitary group
of $M_{d}\left( \mathfrak{J}^{+}\right) $. Recall that every element of $%
M_{d}\left( \mathfrak{J}^{+}\right) $ can be written uniquely as $x+\alpha 1$
where $x\in M_{d}\left( \mathfrak{J}\right) $ and $\alpha \in M_{d}\left( 
\mathbb{C}\right) $. We consider $U\left( M_{d}\left( \mathfrak{J}%
^{+}\right) \right) $ as a Borel subset of $\mathrm{\mathrm{Ball}}\left(
M_{d}\left( \mathfrak{J}\right) \right) \times \mathrm{\mathrm{Ball}}\left(
M_{d}\left( \left( \mathbb{C}\right) \right) \right) $. We then set $\mathrm{%
Z}_{1}\left( \mathfrak{J}\right) $ to be the disjoint union of $U\left(
M_{d}\left( \mathfrak{J}^{+}\right) \right) $ for $d\geq 1$, and let $%
\mathrm{B}_{1}\left( \mathfrak{J}\right) $ be the (analytic) equivalence
relation on $\mathrm{Z}_{1}\left( \mathfrak{J}\right) $ obtained by setting $%
u$\textrm{B}$_{1}\left( \mathfrak{J}\right) u^{\prime }$ for $u\in U\left(
M_{n}\left( \mathfrak{J}^{+}\right) \right) $ and $u^{\prime }\in U\left(
M_{n^{\prime }}\left( \mathfrak{J}^{+}\right) \right) $ if and only if there
exist $k,k^{\prime }\in \omega $ with $n+k=n^{\prime }+k^{\prime }$ and such
that there is a \emph{norm-continuous }path from $u\oplus 1_{k}$ to $%
u^{\prime }\oplus 1_{k^{\prime }}$ in $U\left( M_{n+k}\left( \mathfrak{J}%
^{+}\right) \right) $. The definable $\mathrm{K}_{1}$-group $\mathrm{K}%
_{1}\left( \mathfrak{J}\right) $ is then the semidefinable group obtained as
quotient $\mathrm{Z}_{1}\left( \mathfrak{J}\right) /\mathrm{B}_{1}\left( 
\mathfrak{J}\right) $ with group operations defined as above.

Suppose that $\left( \mathfrak{A},\mathfrak{J}\right) $ is a strict unital
C*-pair. We have natural definable group homomorphisms%
\begin{equation*}
\mathrm{K}_{1}\left( \mathfrak{J}\right) \rightarrow \mathrm{K}_{1}\left( 
\mathfrak{A}\right) \rightarrow \mathrm{K}_{1}\left( \mathfrak{A/J}\right) 
\text{.}
\end{equation*}%
The definable group homomorphism $\mathrm{K}_{1}\left( \mathfrak{J}\right)
\rightarrow \mathrm{K}_{1}\left( \mathfrak{A}\right) $ is induced by the
inclusion $\mathfrak{J}^{+}\subseteq \mathfrak{A}$, which gives an inclusion 
$U\left( M_{d}\left( \mathfrak{J}^{+}\right) \right) \rightarrow U\left(
M_{d}\left( \mathfrak{A}\right) \right) $ for every $d\geq 1$. The definable
group homomorphism $\mathrm{K}_{1}\left( \mathfrak{A}\right) \rightarrow 
\mathrm{K}_{1}\left( \mathfrak{A/J}\right) $ is also induced by the
inclusion maps $U\left( M_{d}\left( \mathfrak{A}\right) \right) \rightarrow
U\left( M_{d}\left( \mathfrak{A}\right) /M_{d}\left( \mathfrak{J}\right)
\right) $ for $d\geq 1$. We have the following result, which can be easily
verified directly, and also follows from Corollary \ref%
{Corollary:exact-K-theory} via the Bott isomorphism theorem \cite[Theorem
4.9.1]{higson_analytic_2000}.

\begin{proposition}
\label{Proposition:exact-K1}Suppose that $\left( \mathfrak{A},\mathfrak{J}%
\right) $ is a strict unital C*-pair. The sequence of natural definable
homomorphisms%
\begin{equation*}
\mathrm{K}_{1}\left( \mathfrak{J}\right) \rightarrow \mathrm{K}_{1}\left( 
\mathfrak{A}\right) \rightarrow \mathrm{K}_{1}\left( \mathfrak{A/J}\right)
\end{equation*}%
is exact.
\end{proposition}

\subsection{The six-term exact sequence}

Suppose that $\left( \mathfrak{A},\mathfrak{J}\right) $ is a strict unital
C*-pair. One can define a natural definable group homomorphism $\partial
_{1}:\mathrm{K}_{1}\left( \mathfrak{A/J}\right) \rightarrow \mathrm{K}%
_{0}\left( \mathfrak{J}\right) $ called the \emph{index map}, as follows. An
element of $\mathrm{K}_{1}\left( \mathfrak{A/J}\right) $ is of the form $%
\left[ u\right] $ where $u\in U\left( M_{d}\left( \mathfrak{A}\right)
/M_{d}\left( \mathfrak{J}\right) \right) $ for some $d\geq 1$. Then define%
\begin{equation*}
P:=%
\begin{bmatrix}
uu^{\ast } & u\left( 1-u^{\ast }u\right) ^{1/2} \\ 
u^{\ast }\left( 1-uu^{\ast }\right) ^{1/2} & 1-u^{\ast }u%
\end{bmatrix}%
\in M_{2d}\left( \mathfrak{J}^{+}\right)
\end{equation*}%
and%
\begin{equation*}
Q:=1_{d}\oplus 0_{d}\in M_{2d}\left( \mathfrak{J}^{+}\right) \text{.}
\end{equation*}%
Then $P,Q$ are projections such that $P\equiv Q\mathrm{\ \mathrm{mod}}\
M_{2d}\left( \mathfrak{J}\right) $ and hence $\left( P,Q\right) \in \mathrm{Z%
}_{0}\left( \mathfrak{J}\right) $. One then defines $\partial _{1}\left( %
\left[ u\right] \right) =\left[ P\right] -\left[ Q\right] $; see \cite[%
Proposition 4.8.10]{higson_analytic_2000}. As $\left( P,Q\right) $ is
obtained in a Borel fashion from $u$, the boundary map $\partial _{1}:%
\mathrm{K}_{1}\left( \mathfrak{A/J}\right) \rightarrow \mathrm{K}_{0}\left( 
\mathfrak{J}\right) $ is definable.

Equivalently, one can define $\partial _{1}$ as follows. Given an element $%
\left[ u\right] $ of $\mathrm{K}_{1}\left( \mathfrak{A/J}\right) $ for some $%
u\in U\left( M_{d}\left( \mathfrak{A}\right) /M_{d}\left( \mathfrak{J}%
\right) \right) $. Consider the partial isometry $v\in M_{2d}\left( 
\mathfrak{A}\right) $ defined by%
\begin{equation*}
v=%
\begin{bmatrix}
u & 0 \\ 
\left( 1-u^{\ast }u\right) ^{1/2} & 0%
\end{bmatrix}%
\end{equation*}%
and observe that $v\equiv u\oplus 0\mathrm{\ \mathrm{mod}}\ \mathfrak{J}$ 
\cite[Lemma 9.2.1]{rordam_introduction_2000}. Then $1_{2d}-v^{\ast }v$ and $%
1_{2d}-vv^{\ast }$ are projections in $M_{2d}\left( \mathfrak{J}^{+}\right) $
such that $1_{2d}-v^{\ast }v\equiv 1_{2d}-vv^{\ast }\equiv 0_{d}\oplus
1_{d}\in M_{d}\left( \mathbb{C}\right) $. Therefore, $\left( 1_{2d}-v^{\ast
}v,1_{2d}-vv^{\ast }\right) \in \mathrm{Z}_{0}\left( \mathfrak{J}\right) $.
One has that $\partial _{1}[u]=[1_{2d}-v^{\ast }v]-[1_{2d}-vv^{\ast }]\in 
\mathrm{K}_{0}\left( \mathfrak{J}\right) $; see \cite[Proposition 9.2.3]%
{rordam_introduction_2000}. Then we have the following; see \cite[Lemma
9.3.1 and Lemma 9.3.2]{rordam_introduction_2000}.

\begin{proposition}
\label{Proposition:exact-boundary1}Suppose that $\mathfrak{A}$ is a strict
unital C*-algebra and $\mathfrak{J}$ is a strict ideal of $\mathfrak{A}$.
Then the sequence%
\begin{equation*}
\mathrm{K}_{1}\left( \mathfrak{A}\right) \rightarrow \mathrm{K}_{1}\left( 
\mathfrak{A/J}\right) \overset{\partial _{1}}{\rightarrow }\mathrm{K}%
_{0}\left( \mathfrak{J}\right) \rightarrow \mathrm{K}_{0}\left( \mathfrak{A}%
\right)
\end{equation*}%
is exact.
\end{proposition}

Suppose that $\left( \mathfrak{A},\mathfrak{J}\right) $ is a strict unital
C*-pair. One can consider a natural definable homomorphism $\partial _{0}:%
\mathrm{K}_{0}\left( \mathfrak{A/J}\right) \rightarrow \mathrm{K}_{1}\left( 
\mathfrak{J}\right) $ called the \emph{exponential map}. This is defined as
follows. Consider an element of $\mathrm{K}_{0}\left( \mathfrak{A/J}\right) $
of the form $[p]-[q]$ for some $p,q\in \mathrm{Proj}\left( M_{n}\left( 
\mathfrak{A}\right) \mathfrak{/}M_{n}\left( \mathfrak{J}\right) \right) $.\
Then we have that $\exp \left( 2\pi ip\right) $ and $\exp \left( 2\pi
iq\right) $ are unitary elements of $M_{n}\left( \mathfrak{J}^{+}\right) $
such that $\exp \left( 2\pi ip\right) \equiv \exp \left( 2\pi iq\right) 
\mathrm{\ \mathrm{mod}}\ M_{n}\left( \mathfrak{J}\right) $. Then one has
that $\partial _{0}\left( [p]-[q]\right) =[\exp \left( 2\pi ip\right)
]-[\exp \left( 2\pi iq\right) ]\in \mathrm{K}_{1}\left( \mathfrak{J}\right) $%
; see \cite[Proposition 12.2.2]{rordam_introduction_2000} and \cite[Section
4.9]{higson_analytic_2000}. From Proposition \ref%
{Proposition:exact-boundary1} one can obtain via the Bott isomorphism
theorem \cite[Theorem 4.9.1]{higson_analytic_2000} the following.

\begin{proposition}
\label{Proposition:exact-boundary0}Suppose that $\left( \mathfrak{A},%
\mathfrak{J}\right) $ is a strict unital C*-pair. Then the sequence%
\begin{equation*}
\mathrm{K}_{0}\left( \mathfrak{A}\right) \rightarrow \mathrm{K}_{0}\left( 
\mathfrak{A/J}\right) \overset{\partial _{0}}{\rightarrow }\mathrm{K}%
_{1}\left( \mathfrak{J}\right) \rightarrow \mathrm{K}_{1}\left( \mathfrak{A}%
\right)
\end{equation*}%
is exact.
\end{proposition}

Suppose that $\left( \mathfrak{A},\mathfrak{J}\right) $ is a strict unital
C*-pair. Then as discussed above we have exact sequences 
\begin{equation*}
\mathrm{K}_{0}\left( \mathfrak{J}\right) \rightarrow \mathrm{K}_{0}\left( 
\mathfrak{A}\right) \rightarrow \mathrm{K}_{0}\left( \mathfrak{A/J}\right)
\end{equation*}%
and 
\begin{equation*}
\mathrm{K}_{1}\left( \mathfrak{J}\right) \rightarrow \mathrm{K}_{1}\left( 
\mathfrak{A}\right) \rightarrow \mathrm{K}_{1}\left( \mathfrak{A/J}\right) 
\text{.}
\end{equation*}%
These are joined together by the index and exponential maps. From
Proposition \ref{Proposition:exact-boundary0}, Proposition \ref%
{Proposition:exact-boundary1}, Corollary \ref{Corollary:exact-K-theory} and
Proposition \ref{Proposition:exact-K1}, one obtains the six-term exact
sequence
\begin{center}
\begin{tikzcd}
\mathrm{K}_{1}\left( \mathfrak{J}\right) \arrow[r]  &   \mathrm{K}_{1}\left( \mathfrak{A}\right) \arrow[r]  &  \mathrm{K}_{1}\left(\mathfrak{A/J}\right) \arrow[d,"\partial _0"]  \\ 
\mathrm{K}_{0}\left( \mathfrak{A/J}\right) \arrow[u, "\partial _1"]  &  \mathrm{K}_{0}\left( \mathfrak{A}\right) \arrow[l] &  \mathrm{K}_{0}\left(\mathfrak{J}\right) \arrow[l]
\end{tikzcd}
\end{center}
for the strict unital C*-pair $\left( \mathfrak{A},\mathfrak{J}\right) $,
where the vertical arrows are the index map and the exponential map; see 
\cite[Theorem 12.1.2]{rordam_introduction_2000}.

\section{Definable $\mathrm{K}$-homology of separable C*-algebras\label%
{Section:Ext}}

In this section we recall the definition of the \textrm{Ext }invariant for
separable unital C*-algebras, and its description due to Paschke in terms of
the $\mathrm{K}$-theory of Paschke dual algebras as defined in \cite%
{higson_algebra_1995,higson_analytic_2000} or, equivalently, of commutants
in the Calkin algebra. Following \cite[Chapter 3]{higson_analytic_2000}, we
consider the group \textrm{Ext}$\left( -\right) ^{-1}$ defined in terms of 
\emph{unital }semi-split extensions. In the case of separable unital \emph{%
nuclear }C*-algebras, every unital extension is semi-split, and the group 
\textrm{Ext}$\left( -\right) ^{-1}$ coincides with the group \textrm{Ext}$%
\left( -\right) $ defined in terms of unital extensions. Using Paschke's $%
\mathrm{K}$-theoretical description of \textrm{Ext }from \cite%
{paschke_theory_1981}, we show that $\mathrm{Ext}\left( -\right) ^{-1}$
yields a contravariant functor from separable unital C*-algebras to the
category of \emph{definable }groups.

We also recall the definition of the $\mathrm{K}$-homology groups of
separable C*-algebras as in \cite[Chapter 5]{higson_analytic_2000}. Using
their description in terms of $\mathrm{Ext}$, we conclude that they can be
endowed with the structure of \emph{definable groups}, in such a way that
the assignments $A\mapsto \mathrm{K}^{0}\left( A\right) $ and $A\mapsto 
\mathrm{K}^{1}\left( A\right) $ are functors from the category of separable
C*-algebras to the category of definable groups.

We define a \emph{separable C*-pair }to be a pair $\left( A,I\right) $ where 
$A$ is a separable C*-algebra and $I$ is a closed two-sided ideal of $A$. A
morphism $\left( A,I\right) \rightarrow \left( B,J\right) $ between
separable C*-pairs is a *-homomorphisms $A\rightarrow B$ that maps $I$ to $J$%
. Recall that a C*-algebra $A$ is \emph{nuclear }if the identity map of $A$
is the pointwise limit of contractive completely positive maps that factor
through finite-dimensional C*-algebras; see \cite[Section 3.3]%
{higson_analytic_2000}. We say that a separable C*-pair $\left( A,I\right) $
is nuclear if $A$ is nuclear. In this section, we will also introduce the
relative definable $\mathrm{K}$-homology groups, and the six-term exact
sequence in $\mathrm{K}$-homology associated with a separable nuclear
C*-pair.

\subsection{C*-algebra extensions and the \textrm{Ext group\label%
{Subsection:definable-Ext}}}

Let $H$ be a separable Hilbert space, and $B\left( H\right) $ be the algebra
of bounded linear operators on $H$. We let $K\left( H\right) \subseteq
B\left( H\right) $ be the closed ideal of compact operators, and $Q\left(
H\right) $ be the \emph{Calkin algebra}, which is the quotient of $B\left(
H\right) $ by $K\left( H\right) $. Let $\pi :B\left( H\right) \rightarrow
Q\left( H\right) $ be the quotient map.

If $U\in U\left( H\right) $ is a unitary operator, then $U$ defines an
automorphism \textrm{Ad}$\left( U\right) :B\left( H\right) \rightarrow
B\left( H\right) $ given by $T\mapsto U^{\ast }TU$. As $K\left( H\right) $
is $\mathrm{Ad}\left( U\right) $-invariant, we have an induced automorphism
of $Q\left( H\right) $, still denoted by \textrm{Ad}$\left( U\right) $.

Suppose that $A$ is a unital, separable C*-algebra. A unital\emph{\ extension%
} of $A$ (by $K\left( H\right) $) is a unital *-homomorphism $\varphi
:A\rightarrow Q\left( H\right) $. A unital extension of $A$ is \emph{%
injective} or \emph{essential} if it is an injective *-homomorphism $%
A\rightarrow Q\left( H\right) $. Two extensions $\varphi ,\varphi ^{\prime
}:A\rightarrow Q\left( H\right) $ are\emph{\ equivalent }if there exists $%
U\in U\left( H\right) $ such that \textrm{Ad}$\left( U\right) \circ \varphi
^{\prime }=\varphi $. An injective, unital extension $\varphi :A\rightarrow
Q\left( H\right) $ is \emph{semi-split }(or \emph{weakly nuclear }in the
terminology of \cite{elliott_abstract_2001}) if there exists a unital
completely positive (ucp) map $\sigma :A\rightarrow B\left( H\right) $ such
that $\varphi =\pi \circ \sigma $ \cite[Theorem 3.1.5]{higson_analytic_2000}%
. An injective unital extension $\varphi :A\rightarrow Q\left( H\right) $ is 
\emph{split }or \emph{trivial }if there exists a unital *-homomorphism $%
\tilde{\varphi}:A\rightarrow B\left( H\right) $ such that $\varphi =\pi
\circ \tilde{\varphi}$.

Every unital, essential extension $\varphi :A\rightarrow Q\left( H\right) $
determines an exact sequence%
\begin{equation*}
0\rightarrow K\left( H\right) \rightarrow E_{\varphi }\rightarrow
A\rightarrow 0
\end{equation*}%
where 
\begin{equation*}
E_{\varphi }=\left\{ \left( x,y\right) \in A\oplus B\left( H\right) :\varphi
(x)=\pi (y)\right\} \text{.}
\end{equation*}%
and $K\left( H\right) $ is an essential ideal of $E_{\varphi }$. The
extension is split if and only if the map $E_{\varphi }\rightarrow A$ is a
split epimorphism in the category of unital C*-algebras and unital
*-homomorphisms.

Conversely, given an exact sequence%
\begin{equation*}
0\rightarrow K\left( H\right) \rightarrow E\overset{p}{\rightarrow }%
A\rightarrow 0
\end{equation*}%
where $p:E\rightarrow A$ is a unital *-homomorphism and $K\left( H\right) $
is an essential ideal of $E$, one can define an essential unital extension $%
\varphi :A\rightarrow Q\left( H\right) $ as follows. Consider $K\left(
H\right) \subseteq E\subseteq B\left( H\right) $, then define $\varphi
\left( a\right) =\pi \left( \widehat{a}\right) \in Q\left( H\right) $ for $%
a\in A$ where $\widehat{a}\in E$ is such that $p\left( \widehat{a}\right) =a$%
. Again, we have that $\varphi $ is trivial if and only if $p:E\rightarrow A$
is a split epimorphism.

Let $A$ be a separable, unital C*-algebra. One defines \textrm{Ext}$\left(
A\right) $ to be the set of unitary equivalence classes of unital, injective
extensions of $A$ by $K\left( H\right) $; see \cite[Definition 2.7.1]%
{higson_analytic_2000}, and \textrm{Ext}$_{\text{nuc}}\left( A\right) =%
\mathrm{Ext}\left( A\right) ^{-1}$ to be the subset of unitary equivalence
classes of unital, injective \emph{semi-split }(or weakly nuclear)\emph{\ }%
extensions of $A$ by $K\left( H\right) $ \cite[15.7.2]{blackadar_theory_1998}%
.

One can define a \emph{commutative monoid operation }on $\mathrm{Ext}\left(
A\right) $. The (additively denoted) operation on \textrm{Ext}$\left(
A\right) $ is induced by the map $\left( \varphi ,\varphi ^{\prime }\right)
\mapsto \mathrm{Ad}\left( V\right) \circ (\varphi \oplus \varphi ^{\prime })$
where $V:H\rightarrow H\oplus H$ is a surjective linear isometry; see \cite[%
Proposition 2.7.2]{higson_analytic_2000}. By Voiculescu's Theorem \cite[%
Theorem 3.4.3]{higson_analytic_2000}, one has that the neutral element of $%
\mathrm{Ext}\left( A\right) $ is the set of split extensions, which form a
single unitary equivalence class \cite[Theorem 3.4.7]{higson_analytic_2000}.
Furthermore, the set \textrm{Ext}$\left( A\right) ^{-1}$ is equal to the set
of elements of \textrm{Ext}$\left( A\right) $ that have an additive inverse,
whence it forms a group \cite[Definition 2.7.6]{higson_analytic_2000}. When $%
A$ is a \emph{nuclear} unital separable C*-algebra, by the Choi--Effros
lifting theorem \cite[Theorem 3.3.6]{higson_analytic_2000}, one has that
every extension of $A$ is semi-split, and \textrm{Ext}$\left( A\right) =%
\mathrm{Ext}\left( A\right) ^{-1}$. In particular, in this case \textrm{Ext}$%
\left( A\right) $ is itself a group.

Let $A$ be a separable unital C*-algebra. We regard \textrm{Ext}$\left(
A\right) ^{-1}$ as a definable group, as follows. Fix a separable Hilbert
space $H$. Let us say that a ucp map $\phi :A\rightarrow B\left( H\right) $
is \emph{ample }if $\left\Vert \left( \pi \circ \phi \right) (x)\right\Vert
=\left\Vert x\right\Vert $ for every $x\in A$. Notice that the set $\mathrm{A%
\mathrm{UCP}}\left( A,B\left( H\right) \right) $ of ample ucp maps $%
A\rightarrow B\left( H\right) $ is a $G_{\delta }$ subset of the space $%
\mathrm{\mathrm{Ball}}\left( L\left( A,B\left( H\right) \right) \right) $ of
bounded linear maps of norm at most $1$ endowed with the topology of
pointwise strong-* convergence. Thus, $\mathrm{A\mathrm{UCP}}\left(
A,B\left( H\right) \right) $ is a Polish space.

Let $\mathcal{E}\left( A\right) \subseteq \mathrm{A\mathrm{UCP}}\left(
A,B\left( H\right) \right) $ be the Borel set of ample ucp maps $\varphi
:A\rightarrow B\left( H\right) $ such that%
\begin{equation*}
\varphi \left( xy\right) \equiv \varphi (x)\varphi (y)\mathrm{\ \mathrm{mod}}%
\ K\left( H\right)
\end{equation*}%
for $x,y\in A$. An injective, unital semi-split extension of $A$ by
definition has a ucp lift, which is an element of $\mathcal{E}\left(
A\right) $, and conversely every element of $\mathcal{E}\left( A\right) $
gives rise to an injective, unital semi-split extension of $A$. Thus, we can
regard $\mathcal{E}\left( A\right) $ as the space of representatives of
injective, unital semi-split extensions of $A$.\ We define a Polish topology
on $\mathcal{E}\left( A\right) $ that induces the Borel structure on $%
\mathcal{E}\left( A\right) $ by declaring a net $\left( \varphi _{i}\right)
_{i\in I}$ in $\mathcal{E}\left( A\right) $ to converge to $\varphi $ if and
only if, for every $x,y\in X$, $\left( \varphi _{i}(x)\right) _{i\in \omega
} $ strong-* converges to $\varphi (x)$, and $\left( \varphi _{i}\left(
xy\right) -\varphi _{i}(x)\varphi _{i}(y)\right) _{i\in \omega }$
norm-converges to $\varphi \left( xy\right) -\varphi (x)\varphi (y)$.

Two elements $\varphi ,\varphi ^{\prime }$ of $\mathcal{E}\left( A\right) $
represent the same element $\left[ \varphi \right] $ of \textrm{Ext}$\left(
A\right) ^{-1}$ if and only if there exists $U\in U\left( H\right) $ such
that $U^{\ast }\varphi \left( a\right) U\equiv \varphi ^{\prime }\left(
a\right) \mathrm{\ \mathrm{mod}}\ K\left( H\right) $ for every $a\in A$.
This defines an analytic equivalence relation $\thickapprox $ on $\mathcal{E}%
\left( A\right) $. We can thus regard \textrm{Ext}$\left( A\right) ^{-1}$ as
the semidefinable set $\mathcal{E}\left( A\right) \left/ \thickapprox
\right. $.

We now observe that the group operations on \textrm{Ext}$\left( A\right)
^{-1}$ are definable, and thus this turns $\mathrm{Ext}\left( A\right) ^{-1}$
into a semidefinable group. We will later show in Proposition \ref%
{Proposition:K-Ext} that in fact $\mathrm{Ext}\left( A\right) ^{-1}$ is a 
\emph{definable }group.

\begin{proposition}
Let $A$ be a separable unital C*-algebra. The addition operation $\left(
x,y\right) \mapsto x+y$ and the additive inverse operation $x\mapsto -x$ on $%
\mathrm{Ext}\left( A\right) ^{-1}$ are definable functions. Thus, $\mathcal{E%
}\left( A\right) \left/ \thickapprox \right. =\mathrm{Ext}\left( A\right)
^{-1}$ is a semidefinable group.
\end{proposition}

\begin{proof}
Fix a representation $A\subseteq B\left( H\right) $ such that $A\cap K\left(
H\right) =\left\{ 0\right\} $. The assertion for addition is clear, as the
Borel map $\mathcal{E}\left( A\right) \times \mathcal{E}\left( A\right)
\rightarrow \mathcal{E}\left( A\right) $, $\left( \varphi ,\varphi ^{\prime
}\right) \mapsto \mathrm{Ad}\left( W\right) \circ \left( \varphi \oplus
\varphi ^{\prime }\right) $ is a lift for the addition operation, where $W$
is a fixed surjective linear isometry $H\rightarrow H\oplus H$.

In order to obtain a lift for the function $\mathrm{Ext}\left( A\right)
^{-1}\rightarrow \mathrm{Ext}\left( A\right) ^{-1}$, $x\mapsto -x$, one can
use the definable Stinespring Dilation Theorem (Lemma \ref%
{Lemma:definable-Stinespring}). Thus, if $\varphi \in \mathcal{E}\left(
A\right) $, and $\pi $ and $V$ are the nondegenerate representation $\pi $
of $A$ and the isometry $V:H\rightarrow H$ obtained from $\varphi $ in a
Borel fashion as in Lemma \ref{Lemma:definable-Stinespring}, then defining
the projection $P:=I-VV^{\ast }\in B\left( H\right) $ and $\varphi ^{\prime
}:A\rightarrow B\left( H\right) $, $a\mapsto W^{\ast }\left( P\pi \left(
a\right) P\oplus a\right) W$, one has that $\varphi ^{\prime }\in \mathcal{E}%
\left( A\right) $ represents $-\left[ \varphi \right] $, where as above $%
W:H\rightarrow H\oplus H$ is a fixed surjective linear isometry; see also 
\cite[Theorem 3.4.7]{higson_analytic_2000}.
\end{proof}

Suppose that $A,B\subseteq B\left( H\right) $ are separable unital
C*-algebras. A unital *-homomorphism $\alpha :A\rightarrow B$ induces a 
\emph{definable }group homomorphism $\mathrm{Ext}\left( B\right)
^{-1}\rightarrow \mathrm{Ext}\left( A\right) ^{-1}$, as follows. If $\varphi
\in \mathcal{E}\left( B\right) $ is a representative for an injective,
unital extension, then one can consider $\alpha ^{\ast }\left( \varphi
\right) \in \mathcal{E}\left( A\right) $ defined by $a\mapsto W^{\ast
}\left( \left( \varphi \circ \alpha \right) \left( a\right) \oplus a\right)
W $ where $W:H\rightarrow H\oplus H$ is a fixed surjective linear isometry.
This defines a Borel function $\alpha ^{\ast }:\mathcal{E}\left( B\right)
\rightarrow \mathcal{E}\left( A\right) $, which induces a definable group
homomorphism $\alpha ^{\ast }:\mathrm{Ext}\left( B\right) ^{-1}\rightarrow 
\mathrm{Ext}\left( A\right) ^{-1}$. Thus, $\mathrm{Ext}\left( -\right) ^{-1}$
is a contravariant functor from the category of separable unital C*-algebras
to the category of semidefinable groups.

\subsection{$\mathrm{K}_{0}$-group and the Voiculescu property}

Let $\mathfrak{A}$ be a strict unital C*-algebra. Recall that two
projections $p,q\in \mathfrak{A}$ are Murray--von Neumann (MvN) equivalent
if there exists $v\in \mathfrak{A}$ such that $v^{\ast }v=p$ and $vv^{\ast
}=q$, in which case we write $p\sim _{\mathrm{MvN}}q$. We say that a
projection $p\in \mathfrak{A}$ is ample if $p\oplus 0$ is Murray--von
Neumann equivalent to $p\oplus 1$, and co-ample if $1-p$ is ample.

\begin{definition}
\label{Definition:V-property}Let $\mathfrak{A}$ be a strict unital
C*-algebra. We say that $\mathfrak{A}$ satisfies the\emph{\ Voiculescu
property} if the set of ample projections in $\mathfrak{A}$ is a Borel
subset of $\mathrm{Ball}\left( \mathfrak{A}\right) $ containing $1$, and
there exist strict unital *-isomorphisms $\Phi _{k,n}:M_{n}\left( \mathfrak{A%
}\right) \rightarrow M_{k}\left( \mathfrak{A}\right) $ for $n,k\geq 1$ such
that, for $n,k,m,n_{0},k_{0},n_{1},k_{1}\geq 1$:

\begin{enumerate}
\item for $n>k$, $\Phi _{k,n}\left( p\oplus 0_{n-k}\right) \sim _{\mathrm{MvN%
}}p$ for every projection $p\in M_{k}\left( \mathfrak{A}\right) $;

\item $\Phi _{n,n}=\mathrm{id}_{M_{n}\left( \mathfrak{A}\right) }$,

\item $\Phi _{k,m}\circ \Phi _{m,n}$ is unitarily equivalent to $\Phi _{k,n}$%
;

\item $\Phi _{k_{0},n_{0}}\oplus \Phi _{k_{1},n_{1}}$ is unitarily
equivalent to $\Phi _{k_{0}+k_{1},n_{0}+n_{1}}|_{M_{n_{0}}\left( \mathfrak{A}%
\right) \oplus M_{n_{1}}\left( \mathfrak{A}\right) }$.
\end{enumerate}

We then define $\mathrm{Z}_{0}^{\mathrm{A}}\left( \mathfrak{A}\right) $ to
be the Borel set of projections in $\mathfrak{A}$ that are both ample and
co-ample.
\end{definition}

\begin{remark}
\label{Remark:V-property}Let $\mathfrak{A}$ be a strict C*-algebra that
satisfies the Voiculescu property.\ Then for $n>k\geq 1$ and $p\in
M_{n}\left( \mathfrak{A}\right) $ we have $\Phi _{k,n}\left( p\right) \oplus
0_{n-k}\sim _{\mathrm{MvN}}p$. Indeed, by (1) we have that $\Phi
_{k,n}\left( \Phi _{k,n}\left( p\right) \oplus 0_{n-k}\right) \sim _{\mathrm{%
MvN}}\Phi _{k,n}\left( p\right) $. Therefore, $\Phi _{k,n}\left( p\right)
\oplus 0_{n-k}\sim _{\mathrm{MvN}}p$.
\end{remark}

\begin{lemma}
Suppose that $\mathfrak{A}$ is a strict unital C*-algebra that satisfies the
Voiculescu property. Then $\left[ 1\right] $ is the neutral element of $%
\mathrm{K}_{0}\left( \mathfrak{A}\right) $.
\end{lemma}

\begin{proof}
Recall that, by Lemma \ref{Lemma:equivalent-B0}, given projections $p,q$
over $\mathfrak{A}$, we have that $[p]=[q]$ if and only if there exist $%
m,n,n^{\prime }\in \omega $ such that $p\oplus 1_{m}\oplus 0_{n}$ and $%
q\oplus 1_{m}\oplus 0_{n^{\prime }}$ are Murray--von Neumann equivalent.
Suppose that $p$ is a projection over $\mathfrak{A}$. As 
\begin{equation*}
1\oplus 0\sim _{\mathrm{MvN}}1\oplus 1\text{,}
\end{equation*}%
we have that%
\begin{equation*}
\left( p\oplus 0\right) \oplus 1\sim _{\mathrm{MvN}}\left( p\oplus 1\right)
\oplus 1\text{.}
\end{equation*}%
Therefore, $p$ and $p\oplus 1$ represent the same element of $\mathrm{K}%
_{0}\left( \mathfrak{A}\right) $. Therefore, we have that%
\begin{equation*}
\left[ p\right] +\left[ 1\right] =\left[ p\oplus 1\right] =\left[ p\right] 
\text{.}
\end{equation*}%
This shows that $\left[ 1\right] $ is the neutral element of $\mathrm{K}%
_{0}\left( \mathfrak{A}\right) $.
\end{proof}

\begin{lemma}
Suppose that $\mathfrak{A}$ is a strict unital C*-algebra that satisfies the
Voiculescu property. If $p,q\in M_{n}\left( \mathfrak{A}\right) $, then $%
\Phi _{1,2n+2}\left( p\oplus \left( 1-q\right) \oplus 1\oplus 0\right) \in 
\mathrm{Z}_{0}^{\mathrm{A}}\left( \mathfrak{A}\right) $ and $\left[ p\right]
-\left[ q\right] =\left[ \Phi _{1,2n+2}\left( p\oplus \left( 1-q\right)
\oplus 1\oplus 0\right) \right] $ in $\mathrm{K}_{0}\left( \mathfrak{A}%
\right) $.
\end{lemma}

\begin{proof}
If $p\in M_{n}\left( \mathfrak{A}\right) $ and $q\in M_{m}\left( \mathfrak{A}%
\right) $ are projections over $\mathfrak{A}$, then we have that%
\begin{eqnarray*}
\left[ p\right] -\left[ q\right] &=&\left[ p\right] +\left[ 1\right] -\left[
q\right] =\left[ p\right] +\left[ 1-q\right] \\
&=&\left[ p\right] +\left[ 1-q\right] +\left[ 1\right] \\
&=&\left[ p\oplus \left( 1-q\right) \oplus 1\oplus 0\right] \text{.}
\end{eqnarray*}%
As, by Remark \ref{Remark:V-property}, 
\begin{equation*}
\Phi _{1,2n+2}\left( p\oplus \left( 1-q\right) \oplus 1\oplus 0\right)
\oplus 0_{2n+1}\sim _{\mathrm{MvN}}p\oplus \left( 1-q\right) \oplus 1\oplus 0
\end{equation*}%
we have%
\begin{equation*}
\left[ p\right] -\left[ q\right] =\left[ p\oplus \left( 1-q\right) \oplus
1\oplus 0\right] =\left[ \Phi _{1,2n+2}\left( p\oplus \left( 1-q\right)
\oplus 1\oplus 0\right) \right] \text{.}
\end{equation*}%
We now show that $\Phi _{1,2n+2}\left( p\oplus \left( 1-q\right) \oplus
1\oplus 0\right) $ is ample and co-ample. Set $r:=p\oplus \left( 1-q\right) $%
. We have by (4) of Definition \ref{Definition:V-property} and since $1$ is
ample,%
\begin{align*}
\Phi _{1,2n+2}\left( r\oplus 1\oplus 0\right) \oplus 1& =\Phi
_{1,2n+2}\left( r\oplus 1\oplus 0\right) \oplus \Phi _{1,1}\left( 1\right) \\
& \sim _{\mathrm{MvN}}{}\Phi _{2,2n+3}\left( r\oplus 1\oplus 0\oplus 1\right)
\\
& \sim _{\mathrm{MvN}}{}\Phi _{2,2n+3}\left( r\oplus 1\oplus 0\oplus 0\right)
\\
& \sim _{\mathrm{MvN}}{}\Phi _{1,2n+2}\left( r\oplus 1\oplus 0\right) \oplus
\Phi _{1,1}\left( 0\right) \\
& \sim _{\mathrm{MvN}}{}\Phi _{1,2n+2}\left( r\oplus 1\oplus 0\right) \oplus
0\text{.}
\end{align*}%
This shows that $\Phi _{1,2n+2}\left( r\oplus 1\oplus 0\right) $ is ample.
Considering that%
\begin{equation*}
1-\Phi _{1,2n+2}\left( r\oplus 1\oplus 0\right) =\Phi _{1,2n+2}(\left(
1-r\right) \oplus 0\oplus 1)\sim _{\mathrm{MvN}}\Phi _{1,2n+2}(\left(
1-r\right) \oplus 1\oplus 0)
\end{equation*}%
we have by the above that $1-\Phi _{1,2n+2}\left( r\oplus 1\oplus 0\right) $
is also ample, and hence $\Phi _{1,2n+2}\left( r\oplus 1\oplus 0\right) $ is
co-ample.
\end{proof}

\begin{lemma}
Let $\mathfrak{A}$ be a strict unital C*-algebra that satisfies the
Voiculescu property. If $p,q\in \mathrm{Z}_{0}^{\mathrm{A}}\left( \mathfrak{A%
}\right) $ are ample and co-ample projections, then the following assertions
are equivalent:

\begin{enumerate}
\item $p,q$ represent the same element of $\mathrm{K}_{0}\left( \mathfrak{A}%
\right) $;

\item $p,q$ are Murray--von Neumann equivalent;

\item $p,q$ are unitary equivalent.
\end{enumerate}
\end{lemma}

\begin{proof}
The implications (3)$\Rightarrow $(2)$\Rightarrow $(1) hold in general.

(1)$\Rightarrow $(3) Suppose that $p,q\in \mathrm{Z}_{0}^{\mathrm{A}}\left( 
\mathfrak{A}\right) $ are such that $\left[ p\right] =\left[ q\right] $.
Then there exist $n,k\in \omega $ such that 
\begin{equation*}
p\oplus 1_{n}\oplus 0_{k}\sim _{\mathrm{MvN}}q\oplus 1_{n}\oplus 0_{k}\text{.%
}
\end{equation*}%
Since $p,q$ are ample, we have%
\begin{equation*}
p\oplus 0_{n+k}\sim _{\mathrm{MvN}}p\oplus 1_{n}\oplus 0_{k}\sim _{\mathrm{%
MvN}}q\oplus 1_{n}\oplus 0_{k}\sim _{\mathrm{MvN}}q\oplus 0_{n+k}\text{.}
\end{equation*}%
Therefore, we have 
\begin{equation*}
p\sim _{\mathrm{MvN}}\Phi _{n+k+1}\left( p\oplus 0_{n+k}\right) \sim _{%
\mathrm{MvN}}\Phi _{n+k+1}\left( q\oplus 0_{n+k}\right) \sim _{\mathrm{MvN}}q%
\text{.}
\end{equation*}%
Using the fact that $p,q$ are co-ample, the same argument applied to $1-p$
and $1-q$ shows that $1-p\sim _{\mathrm{MvN}}1-q$. Thus, $p,q$ are unitarily
equivalent.
\end{proof}

\begin{lemma}
Let $\mathfrak{A}$ be a strict C*-algebra that satisfies the Voiculescu
property. If $p,q\in \mathrm{Z}_{0}^{\mathrm{A}}\left( \mathfrak{A}\right) $%
, then $\Phi _{1,2}\left( p\oplus q\right) \in \mathrm{Z}_{0}^{\mathrm{A}%
}\left( \mathfrak{A}\right) $.
\end{lemma}

\begin{proof}
We need to show that $\Phi _{1,2}\left( p\oplus q\right) $ is ample and
co-ample. We have that%
\begin{align*}
\Phi _{1,2}\left( p\oplus q\right) \oplus 1 &=\Phi _{1,2}\left( p\oplus
q\right) \oplus \Phi _{1,1}\left( 1\right) \\
&=\Phi _{2,3}\left( p\oplus q\oplus 1\right) \\
&\sim _{\mathrm{MvN}}\Phi _{2,3}\left( p\oplus q\oplus 0\right) \\
&\sim _{\mathrm{MvN}}\Phi _{1,2}\left( p\oplus q\right) \oplus \Phi
_{1,1}\left( 0\right) \\
&\sim _{\mathrm{MvN}}\Phi _{1,2}\left( p\oplus q\right) \oplus 0\text{.}
\end{align*}%
This shows that $\Phi _{1,2}\left( p\oplus q\right) $ is ample. The same
argument applied to $1-\left( p\oplus q\right) =\left( 1-p\right) \oplus
\left( 1-q\right) $ shows that $\Phi _{1,2}\left( p\oplus q\right) $ is
co-ample.
\end{proof}

Suppose that $\mathfrak{A}$ is a strict unital C*-algebra satisfying
Voiculescu's property. Consider the unitary group $U\left( \mathfrak{A}%
\right) $, which is a strictly closed subset of $\mathrm{\mathrm{Ball}}%
\left( \mathfrak{A}\right) $ and hence a Polish group when endowed with the
strict topology, and the standard Borel space \textrm{Z}$_{0}^{\mathrm{A}%
}\left( \mathfrak{A}\right) $ of projections in $\mathfrak{A}$ that are both
ample and co-ample, which by assumption is a Borel subset of $\mathfrak{A}$
invariant under unitary conjugation. We can consider the Borel action $%
U\left( \mathfrak{A}\right) \curvearrowright \mathrm{Z}_{0}^{\mathrm{A}%
}\left( \mathfrak{A}\right) $ by conjugation. We let \textrm{B}$_{0}^{%
\mathrm{A}}\left( \mathfrak{A}\right) $ be the corresponding orbit
equivalence relation, and $\mathrm{K}_{0}^{\mathrm{A}}\left( \mathfrak{A}%
\right) :=\mathrm{Z}_{0}^{\mathrm{A}}\left( \mathfrak{A}\right) \left/ 
\mathrm{B}_{0}^{\mathrm{A}}\left( \mathfrak{A}\right) \right. $ be the
corresponding semidefinable set. For $p\in \mathrm{Z}_{0}^{\mathrm{A}}\left( 
\mathfrak{A}\right) $, we let $\left[ p\right] _{\mathrm{B}_{0}^{\mathrm{A}%
}\left( \mathfrak{A}\right) }$ be the $\mathrm{B}_{0}^{\mathrm{A}}\left( 
\mathfrak{A}\right) $-class of $p$. The Borel functions $\left( p,q\right)
\mapsto \Phi _{1,2}\left( p\oplus q\right) $ and $p\mapsto 1-p$ induce a
semidefinable group structure on $\mathrm{K}_{0}^{\mathrm{A}}\left( 
\mathfrak{A}\right) $ with trivial element $\left[ \Phi _{1,2}\left( 1\oplus
0\right) \right] _{\mathrm{B}_{0}^{\mathrm{A}}\left( \mathfrak{A}\right) }$.
Since $\mathrm{B}_{0}^{\mathrm{A}}\left( \mathfrak{A}\right) $ is the orbit
equivalence relation associated with a Borel action of a Polish group on $%
\mathrm{B}_{0}^{\mathrm{A}}\left( \mathfrak{A}\right) $, we have that $%
\mathrm{K}_{0}^{\mathrm{A}}\left( \mathfrak{A}\right) $ is in fact a \emph{%
definable group }by Corollary \ref{Corollary:definable-group}.

The following proposition is an immediate consequence of the lemmas above.

\begin{proposition}
\label{Proposition:V-Def}Suppose that $\mathfrak{A}$ is a strict C*-algebra
satisfying Voiculescu's property. Adopt the notation from Definition \ref%
{Definition:V-property}. Then $\mathrm{K}_{0}^{\mathrm{A}}\left( \mathfrak{A}%
\right) $ and $\mathrm{K}_{0}\left( \mathfrak{A}\right) $ are definably
isomorphic definable groups.
\end{proposition}

\begin{proof}
By the above remarks, $\mathrm{K}_{0}^{\mathrm{A}}\left( \mathfrak{A}\right) 
$ is a definable group. Furthermore, the Borel functions $\mathrm{Z}_{0}^{%
\mathrm{A}}\left( \mathfrak{A}\right) \rightarrow \mathrm{Z}_{0}\left( 
\mathfrak{A}\right) $, $p\mapsto \left( p,0\right) $ and $\mathrm{Z}%
_{0}\left( \mathfrak{A}\right) \rightarrow \mathrm{Z}_{0}^{\mathrm{A}}\left( 
\mathfrak{A}\right) $, $\left( p,q\right) \mapsto \Phi _{1,2n+2}\left(
p\oplus \left( 1-q\right) \oplus 1\oplus 0\right) $ induce mutually inverse
definable isomorphisms between $\mathrm{K}_{0}^{\mathrm{A}}\left( \mathfrak{A%
}\right) $ and $\mathrm{K}_{0}\left( \mathfrak{A}\right) $. By Lemma \ref%
{Lemma:iso-semi}, this shows that $\mathrm{K}_{0}\left( \mathfrak{A}\right) $
is also a definable group, definably isomorphic to $\mathrm{K}_{0}^{\mathrm{A%
}}\left( \mathfrak{A}\right) $.
\end{proof}

Suppose that $A$ is a separable, unital C*-algebra, and $\rho :A\rightarrow
B\left( H\right) $ is a nondegenerate ample representation. Define the
corresponding Paschke dual $\mathfrak{D}_{\rho }\left( A\right) $ as in
Example \ref{Exampl:Paschke-dual} to be the algebra%
\begin{equation*}
\mathfrak{D}_{\rho }\left( A\right) =\left\{ T\in B\left( H\right) :\forall
a\in B\left( H\right) ,T\rho \left( a\right) \equiv \rho \left( a\right) T%
\mathrm{\ \mathrm{mod}}\ K\left( H\right) \right\} \text{.}
\end{equation*}%
Then, $\mathfrak{D}_{\rho }\left( A\right) $ is a strict unital C*-algebra,
with respect to the strict topology on $\mathrm{\mathrm{Ball}}\left( 
\mathfrak{D}_{\rho }\left( a\right) \right) $ induced by the seminorms%
\begin{equation*}
T\mapsto \max \{\left\Vert TS\right\Vert ,\left\Vert ST\right\Vert
,\left\Vert T\rho \left( a\right) -\rho \left( a\right) T\right\Vert \}
\end{equation*}%
for $S\in K\left( H\right) $ and $a\in A$. We now observe that, as a
consequence of Voiculescu's theorem, $\mathfrak{D}_{\rho }\left( A\right) $
satisfies the Voiculescu property; see Definition \ref{Definition:V-property}%
.

Let $\rho ^{n}:A\rightarrow B\left( H^{n}\right) $ be the $n$-fold direct
sum of $\rho $. Notice that, under the usual identification of $B\left(
H^{n}\right) $ with $M_{n}\left( B\left( H\right) \right) $, $\mathfrak{D}%
_{\rho ^{n}}\left( A\right) $ corresponds to $M_{n}\left( \mathfrak{D}_{\rho
}\left( A\right) \right) $. For $k,n\geq 1$, as both $\rho ^{k}$ and $\rho
^{n}$ are ample representations of $A$, by Voiculescu's theorem there exists
a surjective isometry $V_{k,n}:H^{k}\rightarrow H^{n}$ such that \textrm{Ad}$%
\left( V\right) :B\left( H^{n}\right) \rightarrow B\left( H^{k}\right) $
satisfies $\left( \mathrm{Ad}\left( V\right) \circ \rho ^{n}\right) \left(
a\right) \equiv \rho ^{k}\left( a\right) \mathrm{\ \mathrm{mod}}\ K\left(
H\right) $ for every $a\in A$. This implies that \textrm{Ad}$\left( V\right) 
$ induces a strict *-isomorphism $\Phi _{k,n}:=\mathrm{Ad}\left( V\right)
:M_{n}\left( \mathfrak{D}_{\rho }\left( A\right) \right) \rightarrow
M_{k}\left( \mathfrak{D}_{\rho }\left( A\right) \right) $. By Voiculescu's
theorem, $\Phi _{k,n}$ does not depend, up to unitary equivalence, from the
choice of the surjective isometry $V_{k,n}:H^{k}\rightarrow H^{n}$. Thus, we
have that $\mathfrak{D}_{\rho }\left( A\right) $ satisfies (2), (3), and (4)
of Definition \ref{Definition:V-property}.

Every projection $P\in \mathfrak{D}_{\rho }\left( A\right) $ defines a
unital extension $\varphi _{P}:A\rightarrow B\left( PH\right) $, $a\mapsto
\pi \left( P\rho \left( a\right) |_{PH}\right) $. By \cite[Lemma 5.1.2]%
{higson_analytic_2000}, we have the following.

\begin{lemma}
\label{Lemma:Toeplitz}Suppose that $P,P_{1},P_{2}\in \mathfrak{D}_{\rho
}\left( A\right) $ are projections. The following assertions are equivalent:

\begin{enumerate}
\item $\varphi _{P_{1}},\varphi _{P_{2}}$ are equivalent extensions;

\item $P_{1},P_{2}$ are Murray--von Neumann equivalent.
\end{enumerate}

Furthermore, the following assertions are equivalent:

\begin{enumerate}
\item $P$ is ample;

\item $\varphi _{P}$ is injective.
\end{enumerate}
\end{lemma}

From Lemma \ref{Lemma:Toeplitz} it is easy to deduce the following.

\begin{proposition}
\label{Proposition:P-V}Suppose that $A$ is a separable, unital C*-algebra, $%
\rho :A\rightarrow B\left( H\right) $ is a nondegenerate ample
representation of $A$, and $\mathfrak{D}_{\rho }\left( A\right) $ is the
corresponding Paschke dual. Then $\mathfrak{D}_{\rho }\left( A\right) $
satisfies Voiculescu's property.
\end{proposition}

\begin{proof}
By Lemma \ref{Lemma:Toeplitz}, a projection $P\in \mathfrak{D}_{\rho }\left(
A\right) $ is ample if and only if $\varphi _{P}$ is injective. This is
equivalent to the assertion that, for every self-adjoint $a\in A$, every $%
S\in K\left( H\right) $, and every $\varepsilon >0$,%
\begin{equation*}
\left\Vert P\rho \left( a\right) -S\right\Vert >\left\Vert a\right\Vert
-\varepsilon \text{.}
\end{equation*}%
By strict lower semicontinuity of the norm in $\mathfrak{D}_{\rho }\left(
A\right) $, this is an open condition. This shows that the set of ample
projections is a $G_{\delta }$ set.

Since $\rho $ is an ample representation, and $\varphi _{I}=\rho $, we have
that $I\in \mathfrak{D}_{\rho }\left( \mathfrak{A}\right) $ is ample.

Finally, we need to verify (1) of Definition \ref{Definition:V-property}.
For $n>k$, and projection $P\in M_{k}\left( \mathfrak{D}_{\rho }\left(
A\right) \right) =\mathfrak{D}_{\rho ^{k}}\left( A\right) $ we have 
\begin{equation*}
Q:=\Phi _{k,n}\left( P\oplus 0_{n-k}\right) =V_{k,n}^{\ast }\left( P\oplus
0_{n-k}\right) V_{k,n}\text{.}
\end{equation*}%
Thus, $\varphi _{Q}$ is equivalent to $\varphi _{P\oplus 0_{n-k}}=\varphi
_{P}$. Hence, by Lemma \ref{Lemma:Toeplitz}, $Q$ and $P$ are Murray--von
Neumann equivalent.
\end{proof}

As a consequence of Proposition \ref{Proposition:P-V} and Proposition \ref%
{Proposition:V-Def} we have the following.

\begin{proposition}
\label{Proposition:P-Def}Suppose that $A$ is a separable C*-algebra, and $%
\rho :A\rightarrow B\left( H\right) $ is a nondegenerate ample
representation of $A$. Then $\mathrm{K}_{0}\left( \mathfrak{D}_{\rho }\left(
A\right) \right) $ and $\mathrm{K}_{0}^{\mathrm{A}}\left( \mathfrak{D}_{\rho
}\left( A\right) \right) $ are definably isomorphic definable groups.
\end{proposition}

Suppose that $A,B$ are separable unital C*-algebras. Recall that, if $\rho
,\rho ^{\prime }$ are linear maps from $A$ to $B\left( H\right) $, and $%
V:H\rightarrow H$ is an isometry, then we write $\rho ^{\prime }\lesssim
_{V}\rho $ if $\rho ^{\prime }\left( a\right) \equiv V^{\ast }\rho \left(
a\right) V\mathrm{\ \mathrm{mod}}\ K\left( H\right) $ for every $a\in A$.
Suppose that $A,B$ are separable unital C*-algebras, $\alpha :A\rightarrow B$
is a unital *-homomorphism. Let $\rho _{A},\rho _{B}$ be ample
representations of $A,B$ on a Hilbert space $H$. An isometry $V_{\alpha
}:H\rightarrow H$ \emph{covers }$\alpha $ if $\rho _{A}\lesssim _{V_{\alpha
}}\rho _{B}\circ \alpha $.

We have that for every unital *-homomorphism $\alpha :A\rightarrow B$ there
exists an isometry $V_{\alpha }:H\rightarrow H$ that covers $\alpha $ \cite[%
Lemma 5.2.3]{higson_analytic_2000}. Furthermore, $\mathrm{Ad}\left(
V_{\alpha }\right) $ induces a strict unital *-homomorphism $\mathrm{Ad}%
\left( V_{\alpha }\right) :\mathfrak{D}_{\rho _{B}}\left( B\right)
\rightarrow \mathfrak{D}_{\rho _{A}}\left( A\right) $. In turn, $\mathrm{Ad}%
\left( V_{\alpha }\right) :\mathfrak{D}_{\rho _{B}}\left( B\right)
\rightarrow \mathfrak{D}_{\rho _{A}}\left( A\right) $ induces a definable
group homomorphism $\mathrm{K}_{0}\left( \mathfrak{D}_{\rho _{B}}\left(
B\right) \right) \rightarrow \mathrm{K}_{0}\left( \mathfrak{D}_{\rho
_{A}}\left( A\right) \right) $. This definable group isomorphism only
depends on $\alpha $, and not on the choice of the isometry $V_{\alpha }$
that covers $\alpha $ \cite[Lemma 5.2.4]{higson_analytic_2000}. This gives a
contravariant functor $A\mapsto \mathrm{K}_{0}\left( \mathfrak{D}_{\rho
_{A}}\left( A\right) \right) $ from the category of separable unital
C*-algebras to the category of definable groups. Similarly, one can regard $%
A\mapsto \mathrm{K}_{0}^{\mathrm{A}}\left( \mathfrak{D}_{\rho _{A}}\left(
A\right) \right) $ as a contravariant functor, naturally isomorphic to $%
\mathrm{K}_{0}\left( \mathfrak{D}_{\rho _{A}}\left( A\right) \right) $.

Using Proposition \ref{Proposition:P-Def} we can show the following.

\begin{proposition}
\label{Proposition:K-Ext}Suppose that $A$ is a separable, unital C*-algebra,
and $\rho :A\rightarrow B\left( H\right) $ is a nondegenerate ample
representation of $A$. Then $\mathrm{Ext}\left( A\right) ^{-1}$ is a
definable group, which is naturally definably isomorphic to $\mathrm{K}%
_{0}\left( \mathfrak{D}_{\rho }\left( A\right) \right) $.
\end{proposition}

\begin{proof}
Consider the corresponding Paschke dual $\mathfrak{D}_{\rho }\left( A\right) 
$. Recall that $\mathcal{E}\left( A\right) $ denotes the Polish space of
representatives of injective, unital, semi-split extensions of $A$ by $%
K\left( H\right) $, which are ample ucp maps $\varphi :A\rightarrow B\left(
H\right) $ satisfying $\varphi \left( xy\right) \equiv \varphi (x)\varphi (y)%
\mathrm{\ \mathrm{mod}}\ K\left( H\right) $ for $x,y\in A$. An ample and
co-ample projection $P\in \mathfrak{D}_{\rho }\left( A\right) $ determines
an extension $\varphi _{P}\in \mathcal{E}\left( A\right) $ of $A$, defined
as follows. Choose a linear isometry $V:H\rightarrow H$ such that $VV^{\ast
}=P$, and define $\varphi _{P}\left( a\right) =V^{\ast }\rho \left( a\right)
V$. (Notice that $V$ can be chosen in a Borel fashion from $P$.) The Borel
function $P\mapsto \varphi _{P}$ induces a definable group isomorphism 
\begin{equation*}
\gamma :\mathrm{K}_{0}^{\mathrm{A}}\left( \mathfrak{D}_{\rho }\left(
A\right) \right) \rightarrow \mathrm{Ext}\left( A\right) ^{-1};
\end{equation*}%
see \cite[Proposition 5.1.6]{higson_analytic_2000}.

We claim that the inverse group homomorphism $\gamma ^{-1}:\mathrm{Ext}%
\left( A\right) ^{-1}\rightarrow \mathrm{K}_{0}^{\mathrm{A}}\left( \mathfrak{%
D}_{\rho }\left( A\right) \right) $ is definable as well. Indeed, if $%
\varphi \in \mathcal{E}\left( A\right) $ then, by the Definable Voiculescu
Theorem (Lemma \ref{Lemma:Voiculescu}), one can choose in a Borel way an
isometry $V_{\varphi }:H\rightarrow H$ such that $\varphi \lesssim
_{V_{\varphi }}\rho $. Thus, we have that%
\begin{equation*}
\varphi \left( a\right) \equiv V_{\varphi }^{\ast }\rho \left( a\right)
V_{\varphi }\mathrm{\ \mathrm{mod}}\ K\left( H\right)
\end{equation*}%
for every $a\in A$. If $P:=V_{\varphi }V_{\varphi }^{\ast }$ then we have
that $P$ is a projection in $\mathfrak{D}_{\rho }\left( A\right) $ such that 
$\varphi _{P}$ is equivalent to $\varphi $. As $P$ is not necessarily ample
and co-ample, one can replace $P$ with $\Phi _{1,3}\left( P\oplus 1\oplus
0\right) $ to obtain an ample and co-ample projection $P_{\varphi }\in 
\mathrm{Z}_{0}\left( \mathfrak{D}_{\rho }\left( A\right) \right) $ such that 
$\varphi _{P_{\varphi }}$ is equivalent to $\varphi $. Thus the Borel
function $\varphi \mapsto P_{\varphi }$ is a lift of the inverse map $\gamma
^{-1}:\mathrm{Ext}\left( A\right) ^{-1}\rightarrow \mathrm{K}_{0}^{\mathrm{A}%
}\left( \mathfrak{D}_{\rho }\left( A\right) \right) $. This shows that $%
\gamma ^{-1}$ is also definable. Therefore, $\gamma $ is a natural
isomorphism in the category of semidefinable groups.

As $\mathrm{K}_{0}^{\mathrm{A}}\left( \mathfrak{D}_{\rho }\left( A\right)
\right) $ is in fact a \emph{definable} group, this implies that $\mathrm{Ext%
}\left( A\right) ^{-1}$ is a definable group. Since $\mathrm{K}_{0}^{\mathrm{%
A}}\left( \mathfrak{D}_{\rho }\left( A\right) \right) $ is naturally
definably isomorphic to $\mathrm{K}_{0}\left( \mathfrak{D}_{\rho }\left(
A\right) \right) $, we have that $\mathrm{Ext}\left( A\right) ^{-1}$ is
naturally definably isomorphic to $\mathrm{K}_{0}\left( \mathfrak{D}_{\rho
}\left( A\right) \right) $ as well.
\end{proof}

\subsection{Definable $\mathrm{K}$-theory of commutants in the Calkin
algebra \label{Subsection:K0-commutants}}

Suppose that $A$ is a unital separable C*-algebra, and $\rho $ is an ample
representation of $A$ on the infinite-dimensional separable Hilbert space $H$%
. Then $\rho $ induces an ample representation $\rho ^{+}$ of the
unitization $A^{+}$ on $H\oplus H$, defined by $\rho ^{+}\left( a\right)
=\rho \left( a\right) \oplus 0$ for $a\in A$. Recall that the Paschke dual
algebra is the strict unital C*-algebra%
\begin{equation*}
\mathfrak{D}_{\rho }\left( A\right) :=\left\{ T\in B\left( H\right) :\forall
a\in A,T\rho \left( a\right) \equiv \rho \left( a\right) T\mathrm{\ \mathrm{%
mod}}\ K\left( H\right) \right\} \text{.}
\end{equation*}%
We also have the Paschke dual algebra 
\begin{equation*}
\mathfrak{D}_{\rho ^{+}}\left( A^{+}\right) =\left\{ T\in B\left( H\oplus
H\right) :\forall a\in A,T\rho ^{+}\left( a\right) \equiv \rho ^{+}\left(
a\right) T\mathrm{\ \mathrm{mod}}\ K\left( H\oplus H\right) \right\} \text{.}
\end{equation*}%
Notice that%
\begin{equation*}
\mathfrak{D}_{\rho ^{+}}\left( A^{+}\right) =%
\begin{bmatrix}
\mathfrak{D}_{\rho }\left( A\right) & K\left( H\right) \\ 
K\left( H\right) & B\left( H\right)%
\end{bmatrix}%
\text{;}
\end{equation*}%
see \cite[Section 5.2]{higson_analytic_2000}.

Define $\mathfrak{J}$ to be the strict ideal%
\begin{equation*}
\begin{bmatrix}
K\left( H\right) & K\left( H\right) \\ 
K\left( H\right) & B\left( H\right)%
\end{bmatrix}%
\end{equation*}%
of $\mathfrak{D}_{\rho ^{+}}\left( A^{+}\right) $. Let also $\mathfrak{D}%
_{\rho ^{+}}\left( A//A\right) $ be the strict ideal%
\begin{equation*}
\left\{ T\in \mathfrak{D}_{\rho ^{+}}\left( A^{+}\right) :\forall a\in
A,T\rho ^{+}\left( a\right) \equiv 0\mathrm{\ \mathrm{mod}}\ K\left( H\oplus
H\right) \right\}
\end{equation*}%
of $\mathfrak{D}_{\rho ^{+}}\left( A^{+}\right) $; see Proposition \ref%
{Proposition:essential-annihilator}.

\begin{lemma}
\label{Lemma:trivial-K}The C*-algebras $\mathfrak{J}$ and $\mathfrak{D}%
_{\rho ^{+}}\left( A//A\right) $ defined above have trivial $\mathrm{K}$%
-theory.
\end{lemma}

\begin{proof}
The assertion about $\mathfrak{J}$ follows by considering the six-term exact
sequence in $\mathrm{K}$-theory associated with the pair $\left( \mathfrak{J}%
,M_{2}\left( K\left( H\right) \right) \right) $; see also \cite[Exercise
4.10.9]{higson_analytic_2000}. The assertion about $\mathfrak{D}_{\rho
^{+}}\left( A//A\right) $ is \cite[Lemma 5.4.1]{higson_analytic_2000}.
\end{proof}

\begin{lemma}
\label{Lemma:K-commutants}Suppose that $A$ is a separable unital C*-algebra.
For $i\in \left\{ 0,1\right\} $:

\begin{enumerate}
\item the definable group homomorphism $\mathrm{K}_{i}\left( \mathfrak{D}%
_{\rho ^{+}}\left( A^{+}\right) \right) \rightarrow \mathrm{K}_{i}\left( 
\mathfrak{D}_{\rho ^{+}}\left( A^{+}\right) /\mathfrak{J}\right) $ is an
isomorphism in the category of semidefinable groups;

\item the strict *-homomorphism $\varphi :\mathfrak{D}_{\rho }\left(
A\right) \rightarrow \mathfrak{D}_{\rho ^{+}}\left( A^{+}\right) $, $%
x\mapsto x\oplus 0$ induces an isomorphism 
\begin{equation*}
\mathrm{K}_{i}\left( \mathfrak{D}_{\rho }\left( A\right) /K\left( H\right)
\right) \rightarrow \mathrm{K}_{i}\left( \mathfrak{D}_{\rho ^{+}}\left(
A^{+}\right) /\mathfrak{J}\right)
\end{equation*}
in the category of semidefinable groups.

\item The map $\mathrm{K}_{0}\left( \mathfrak{D}_{\rho }\left( A\right)
\right) \rightarrow \mathrm{K}_{0}\left( \mathfrak{D}_{\rho }\left( A\right)
/K\left( H\right) \right) $ is an isomorphism in the category of
semidefinable groups;

\item The subgroup $G$ of $\mathrm{K}_{1}\left( \mathfrak{D}_{\rho }\left(
A\right) /K\left( H\right) \right) $, consisting of the kernel of the
(surjective) index map 
\begin{equation*}
\partial _{0}:\mathrm{K}_{1}\left( \mathfrak{D}_{\rho }\left( A\right)
/K\left( H\right) \right) \rightarrow \mathrm{K}_{0}\left( K\left( H\right)
\right) \cong \mathbb{Z}
\end{equation*}%
is Borel, and the definable group homomorphism $\mathrm{K}_{1}\left( 
\mathfrak{D}_{\rho }\left( A\right) \right) \rightarrow \mathrm{K}_{1}\left( 
\mathfrak{D}_{\rho }\left( A\right) /K\left( H\right) \right) $ induces an
isomorphism $\mathrm{K}_{1}\left( \mathfrak{D}_{\rho }\left( A\right)
\right) \rightarrow G$ in the category of semidefinable groups.
\end{enumerate}
\end{lemma}

\begin{proof}
(1) Since $\mathfrak{J}$ has trivial $\mathrm{K}$-theory, the group
homomorphism $\mathrm{K}_{i}\left( \mathfrak{D}_{\rho ^{+}}\left(
A^{+}\right) \right) \rightarrow \mathrm{K}_{i}\left( \mathfrak{D}_{\rho
^{+}}\left( A^{+}\right) /\mathfrak{J}\right) $ is an isomorphism. We need
to prove that the inverse group homomorphism $\mathrm{K}_{i}\left( \mathfrak{%
D}_{\rho ^{+}}\left( A^{+}\right) /\mathfrak{J}\right) \rightarrow \mathrm{K}%
_{i}\left( \mathfrak{D}_{\rho ^{+}}\left( A^{+}\right) \right) $ is
definable.

Consider first the case $i=0$. Consider $p\in \mathrm{Z}_{0}\left( \mathfrak{%
D}_{\rho ^{+}}\left( A^{+}\right) /\mathfrak{J}\right) $. Thus, $p\in 
\mathrm{Proj}\left( M_{d}\left( \mathfrak{D}_{\rho ^{+}}\left( A\right)
\right) /M_{d}\left( \mathfrak{J}\right) \right) $ for some $d\geq 1$. After
replacing $\rho $ with $\rho ^{d}$ we can assume that $d=1$. Thus, $p\in 
\mathfrak{D}_{\rho ^{+}}\left( A\right) $ is a $\mathrm{\mathrm{mod}}\ 
\mathfrak{J}$ projection. This implies that%
\begin{equation*}
p=%
\begin{bmatrix}
p_{11} & p_{12} \\ 
p_{21} & p_{22}%
\end{bmatrix}%
\end{equation*}%
where $p_{11}\in \mathfrak{D}_{\rho }\left( A\right) $ is a $\mathrm{\mathrm{%
mod}}\ K\left( H\right) $ projection. Then by Lemma \ref%
{Lemma:essential-projection}, one can choose in a Borel fashion from $p$ a
projection $q\in \mathfrak{D}_{\rho }\left( A\right) $ such that $q\equiv
p_{11}\mathrm{\ \mathrm{mod}}\ K\left( H\right) $ and hence $q\oplus 0\equiv
p\mathrm{\ \mathrm{mod}}\ \mathfrak{J}$.

We now consider the case when $i=1$. Consider $q\in \mathrm{Z}_{1}\left( 
\mathfrak{D}_{\rho ^{+}}\left( A^{+}\right) /\mathfrak{J}\right) $. Thus, $%
q\in U\left( M_{d}\left( \mathfrak{D}_{\rho ^{+}}\left( A\right) \right)
/M_{d}\left( \mathfrak{J}\right) \right) $ for some $d\geq 1$. After
replacing $\rho $ with $\rho ^{d}$ we can assume that $d=1$. Thus,%
\begin{equation*}
u=%
\begin{bmatrix}
u_{11} & u_{12} \\ 
u_{21} & u_{22}%
\end{bmatrix}%
\end{equation*}%
where $u_{11}\in \mathfrak{D}_{\rho ^{+}}\left( A\right) $ is a $\mathrm{%
\mathrm{mod}}\ K\left( H\right) $ unitary. Let $v\in \mathfrak{D}_{\rho
}\left( A\right) $ be the partial isometry in the polar decomposition of $%
u_{11}$, which depends in a Borel fashion from $u_{11}$ by Lemma \ref%
{Lemma:polar}. Then we have that%
\begin{equation*}
v:=%
\begin{bmatrix}
v & I-vv^{\ast } \\ 
I-v^{\ast }v & v^{\ast }%
\end{bmatrix}%
\in \mathfrak{D}_{\rho ^{+}}\left( A^{+}\right)
\end{equation*}%
is a unitary such that $v\equiv u\mathrm{\ \mathrm{mod}}\ \mathfrak{J}$.

(2) Since $\varphi $ induces a *-isomorphism $\mathfrak{D}_{\rho }\left(
A\right) /K\left( H\right) \rightarrow \mathfrak{D}_{\rho ^{+}}\left(
A^{+}\right) /\mathfrak{J}$, it induces a definable group isomorphism $%
\mathrm{K}_{i}\left( \mathfrak{D}_{\rho }\left( A\right) /K\left( H\right)
\right) \rightarrow \mathrm{K}_{i}\left( \mathfrak{D}_{\rho ^{+}}\left(
A^{+}\right) /\mathfrak{J}\right) $. It is immediate to verify that the
inverse group homomorphism is also definable, as it is induced by the Borel
function%
\begin{equation*}
\begin{bmatrix}
x_{11} & x_{12} \\ 
x_{21} & x_{22}%
\end{bmatrix}%
\mapsto x_{11}\text{.}
\end{equation*}

(3) and (4): Under the isomorphism $\mathrm{K}_{0}\left( K\left( H\right)
\right) \cong \mathbb{Z}$, the definable group homomorphism $\mathrm{K}%
_{1}\left( \mathfrak{D}_{\rho }\left( A\right) /K\left( H\right) \right)
\rightarrow K_{0}\left( K\left( H\right) \right) \cong \mathbb{Z}$ maps each 
$T\in U\left( \mathfrak{D}_{\rho }\left( A\right) /K\left( H\right) \right) $
to its Fredholm index, and in particular it is surjective. As the Fredholm
index is given by a Borel map, and $\mathrm{K}_{1}\left( K\left( H\right)
\right) =\left\{ 0\right\} $, it follows from the six-term exact sequence in 
$\mathrm{K}$-theory associated with $\mathfrak{D}_{\rho }\left( A\right) $
and $K\left( H\right) $ that $\mathrm{K}_{0}\left( \mathfrak{D}_{\rho
}\left( A\right) \right) \rightarrow \mathrm{K}_{0}\left( \mathfrak{D}_{\rho
}\left( A\right) /K\left( H\right) \right) $ is a definable group
isomorphism, and that $\mathrm{K}_{1}\left( \mathfrak{D}_{\rho }\left(
A\right) \right) \rightarrow \mathrm{K}_{1}\left( \mathfrak{D}_{\rho }\left(
A\right) /K\left( H\right) \right) $ is an injective definable group
homomorphism with range equal to $G$. The inverse $\mathrm{K}_{0}\left( 
\mathfrak{D}_{\rho }\left( A\right) /K\left( H\right) \right) \rightarrow 
\mathrm{K}_{0}\left( \mathfrak{D}_{\rho }\left( A\right) \right) $ is
definable by Lemma \ref{Lemma:essential-projection}. The inverse $%
G\rightarrow \mathrm{K}_{1}\left( \mathfrak{D}_{\rho }\left( A\right)
\right) $ is definable by Lemma \ref{Lemma:polar}, considering that given $%
T\in U\left( B\left( H\right) /K\left( H\right) \right) $ such that $\mathrm{%
index}\left( T\right) =0$, then the partial isometry $U$ in the polar
decomposition of $T$ is a unitary such that $U\equiv T\mathrm{\ \mathrm{mod}}%
\ K\left( H\right) $.
\end{proof}

\begin{corollary}
\label{Corollary:K^0-definable}Suppose that $A$ is a separable unital
C*-algebra, and $\rho $ is an ample representation of $A$. Then $\mathrm{Ext}%
\left( A\right) ^{-1}$, $\mathrm{K}_{0}\left( \mathfrak{D}_{\rho }\left(
A\right) \right) $, $\mathrm{K}_{0}\left( \mathfrak{D}_{\rho }\left(
A\right) /K\left( H\right) \right) $, $\mathrm{K}_{0}\left( \mathfrak{D}%
_{\rho ^{+}}\left( A^{+}\right) /\mathfrak{J}\right) $, and $\mathrm{K}%
_{0}\left( \mathfrak{D}_{\rho ^{+}}\left( A^{+}\right) \right) $ are
definably isomorphic definable groups.
\end{corollary}

\begin{proof}
This is a consequence of Lemma \ref{Lemma:K-commutants}, Proposition \ref%
{Proposition:K-Ext}, and Corollary \ref{Corollary:Kechris--MacDonald}.
\end{proof}

Suppose that $A$ is a unital separable C*-algebra. Let $C\left( \mathbb{T}%
,A\right) $ be the unital separable C*-algebra of continuous functions $f:%
\mathbb{T}\rightarrow A$. We identify $A$ with the C*-subalgebra of $C\left( 
\mathbb{T},A\right) $ consisting of constant functions. The \emph{suspension 
}$SA$ of $A$ is the C*-subalgebra $\left\{ f\in C\left( \mathbb{T},A\right)
:f\left( 1\right) =0\right\} $. The unital suspension $\Sigma A$ of $A$ is
the unitization of $SA$, which can be identified with $\left\{ f\in C\left( 
\mathbb{T},A\right) :f\left( 1\right) \in \mathbb{C}1\right\} $.

Consider an ample representation $\rho $ of $C\left( \mathbb{T},A\right) $
on a Hilbert space $H$, and let $\rho _{\Sigma A}$ be its restriction to $%
\Sigma A$ and $\rho _{A}$ be its restriction to $A$. We can then consider
the Paschke dual algebras $\mathfrak{D}_{\rho _{\Sigma A}}\left( \Sigma
A\right) \subseteq B\left( H\right) $ and $\mathfrak{D}_{\rho _{A}}\left(
A\right) \subseteq B\left( H\right) $.

\begin{lemma}
\label{Lemma:Paschke-iso}Suppose that $A$ is a separable unital C*-algebra.
Then $\mathrm{K}_{1}\left( \mathfrak{D}_{\rho _{A}}\left( A\right) /K\left(
H\right) \right) $ is a definable group, definably isomorphic to the
definable group $\mathrm{K}_{0}\left( \mathfrak{D}_{\rho _{\Sigma A}}\left(
\Sigma A\right) /K\left( H\right) \right) $.
\end{lemma}

\begin{proof}
A definable group isomorphism 
\begin{equation*}
\mathrm{K}_{0}\left( \mathfrak{D}_{\rho _{\Sigma A}}\left( \Sigma A\right)
/K\left( H\right) \right) \rightarrow \mathrm{K}_{1}\left( \mathfrak{D}%
_{\rho _{A}}\left( A\right) /K\left( H\right) \right)
\end{equation*}%
is described in \cite[Theorem 6]{paschke_theory_1981}, as follows. Let $U\in
C\left( \mathbb{T},A\right) $ be the function $\lambda \mapsto \lambda 1$.
Let $p$ be a $\mathrm{\mathrm{mod}}\ K\left( H\right) $ projection in $%
\mathfrak{D}_{\rho _{\Sigma A}}\left( \Sigma A\right) $. Then $f\left(
p\right) :=pU+\left( 1-p\right) \in \mathfrak{D}_{\rho _{A}}\left( A\right) $
is a $\mathrm{\mathrm{mod}}\ K\left( H\right) $ unitary. A similar
definition for $\mathrm{\mathrm{mod}}\ K\left( H\right) $ projections over $%
\mathfrak{D}_{\rho _{\Sigma A}}\left( \Sigma A\right) $ defines a Borel
function \textrm{Z}$_{0}\left( \mathfrak{D}_{\rho _{\Sigma A}}\left( \Sigma
A\right) /K\left( H\right) \right) \rightarrow \mathrm{Z}_{1}\left( 
\mathfrak{D}_{\rho _{A}}\left( A\right) /K\left( H\right) \right) $, $%
p\mapsto f\left( p\right) $. It is proved in \cite[Theorem 6]%
{paschke_theory_1981} that this Borel function induces an isomorphism $%
\mathrm{K}_{0}\left( \mathfrak{D}_{\rho _{\Sigma A}}\left( \Sigma A\right)
/K\left( H\right) \right) \rightarrow \mathrm{K}_{1}\left( \mathfrak{D}%
_{\rho _{A}}\left( A\right) /K\left( H\right) \right) $.

By Corollary \ref{Corollary:K^0-definable} we have that $\mathrm{K}%
_{0}\left( \mathfrak{D}_{\rho _{\Sigma A}}\left( \Sigma A\right) /K\left(
H\right) \right) $ is a definable group.\ Thus, by Proposition \ref%
{Corollary:Kechris--MacDonald} we have that $\mathrm{K}_{1}\left( \mathfrak{D%
}_{\rho _{A}}\left( A\right) /K\left( H\right) \right) $ is a definable
group as well.
\end{proof}

\begin{proposition}
\label{Proposition:K^1-definable}Suppose that $A$ is a separable unital
C*-algebra, and $\rho $ is an ample representation of $A$. Then $\mathrm{Ext}%
\left( \Sigma A\right) ^{-1}$, $\mathrm{K}_{1}\left( \mathfrak{D}_{\rho
}\left( A\right) /K\left( H\right) \right) $, $\mathrm{K}_{1}\left( 
\mathfrak{D}_{\rho ^{+}}\left( A^{+}\right) /\mathfrak{J}\right) $, and $%
\mathrm{K}_{1}\left( \mathfrak{D}_{\rho ^{+}}\left( A^{+}\right) \right) $
are definably isomorphic definable groups.
\end{proposition}

\begin{proof}
Let $\rho _{\Sigma A}$ and $\rho _{A}$ be the ample representations of $%
\Sigma A$ and $A$, respectively, as in Lemma \ref{Lemma:Paschke-iso}. Then
by Corollary \ref{Corollary:K^0-definable}, $\mathrm{Ext}\left( \Sigma
A\right) ^{-1}$ and $\mathrm{K}_{0}\left( \mathfrak{D}_{\rho _{\Sigma
A}}\left( \Sigma A\right) /K\left( H\right) \right) $ are definably
isomorphic definable groups. By Lemma \ref{Lemma:Paschke-iso}, $\mathrm{K}%
_{0}\left( \mathfrak{D}_{\rho _{\Sigma A}}\left( \Sigma A\right) /K\left(
H\right) \right) $ and $\mathrm{K}_{1}\left( \mathfrak{D}_{\rho _{A}}\left(
A\right) /K\left( H\right) \right) $ are definably isomorphic definable
groups. By Voiculescu's theorem, $\mathfrak{D}_{\rho _{A}}\left( A\right)
/K\left( H\right) $ and $\mathfrak{D}_{\rho }\left( A\right) /K\left(
H\right) $ are isomorphic in the category of unital C*-algebras with a
strict cover; see Lemma \ref{Lemma:trivial-representations}. In particular, $%
\mathrm{K}_{1}\left( \mathfrak{D}_{\rho _{A}}\left( A\right) /K\left(
H\right) \right) $ and $\mathrm{K}_{1}\left( \mathfrak{D}_{\rho }\left(
A\right) /K\left( H\right) \right) $ are isomorphic in the category of
semidefinable groups. From this and Corollary \ref%
{Corollary:Kechris--MacDonald}, it follows that $\mathrm{K}_{1}\left( 
\mathfrak{D}_{\rho }\left( A\right) /K\left( H\right) \right) $ is a
definable group. Finally, $\mathrm{K}_{1}\left( \mathfrak{D}_{\rho
^{+}}\left( A^{+}\right) /\mathfrak{J}\right) $ and $\mathrm{K}_{1}\left( 
\mathfrak{D}_{\rho ^{+}}\left( A^{+}\right) \right) $ are definable groups,
definably isomorphic to $\mathrm{K}_{1}\left( \mathfrak{D}_{\rho }\left(
A\right) /K\left( H\right) \right) $ by Lemma \ref{Lemma:K-commutants} and
Corollary \ref{Corollary:Kechris--MacDonald} again.
\end{proof}

\subsection{Definable $\mathrm{K}$-homology\label{Subsection:K-homology}}

Suppose that $A$ is a separable C*-algebra. Fix an ample representation $%
\rho ^{+}$ of $A^{+}$, and define $\mathfrak{D}\left( A\right) :=\mathfrak{D}%
_{\rho ^{+}}\left( A^{+}\right) $. The $\mathrm{K}$-homology groups of $A$
are the definable groups%
\begin{equation*}
\mathrm{K}^{1}\left( A\right) :=\mathrm{K}_{0}\left( \mathfrak{D}\left(
A\right) \right) \cong \mathrm{Ext}\left( A^{+}\right) ^{-1}
\end{equation*}%
and%
\begin{equation*}
\mathrm{K}^{0}\left( A\right) :=\mathrm{K}_{1}\left( \mathfrak{D}\left(
A\right) \right) \cong \mathrm{Ext}((SA)^{+})^{-1}\text{;}
\end{equation*}%
see \cite[Definition 5.2.7]{higson_analytic_2000}. By Proposition \ref%
{Proposition:K-Ext}, $\mathrm{K}^{p}(-)$ for $p\in \left\{ 0,1\right\} $ is
a contravariant functor from the category of separable C*-algebras to the
category of definable abelian groups.

When $A$ is a separable \emph{unital }C*-algebra, one can also define the 
\emph{reduced }$\mathrm{K}$-homology groups by considering an ample
representation $\rho $ of $A$ and the corresponding Pachke dual algebra $%
\mathfrak{\tilde{D}}\left( A\right) :=\mathfrak{D}_{\rho }(A)$ and set%
\begin{equation*}
\mathrm{\tilde{K}}^{1}\left( A\right) :=\mathrm{K}_{0}(\mathfrak{\tilde{D}}%
\left( A\right) )\cong \mathrm{Ext}\left( A\right) ^{-1}
\end{equation*}%
and%
\begin{equation*}
\mathrm{\tilde{K}}^{0}\left( A\right) :=\mathrm{K}_{1}(\mathfrak{\tilde{D}}%
\left( A\right) )\text{;}
\end{equation*}%
see \cite[Definition 5.2.1]{higson_analytic_2000}.

Suppose now that $A$ is a separable C*-algebra, and $J$ is a closed
two-sided ideal of $A$.\ Fix as above an ample representation $\rho ^{+}$ of 
$A^{+}$. Define $\mathfrak{D}\left( A\right) :=\mathfrak{D}_{\rho
^{+}}\left( A^{+}\right) $ as above, and set $\mathfrak{D}\left( A//J\right) 
$ to be the strict ideal%
\begin{equation*}
\left\{ T\in \mathfrak{D}\left( A\right) :\forall a\in J,T\rho ^{+}\left(
a\right) \equiv 0\mathrm{\ \mathrm{mod}}\ K\left( H\right) \right\}
\end{equation*}%
of $\mathfrak{D}\left( A\right) $.

\begin{lemma}
\label{Lemma:full-ideal}Suppose that $A$ is a separable C*-algebra, and $%
i\in \left\{ 0,1\right\} $. Then $\mathrm{K}_{i}\left( \mathfrak{D}\left(
A\right) /\mathfrak{D}\left( A//A\right) \right) $ is a definable group,
definably isomorphic to $\mathrm{K}^{1-i}\left( A\right) $.
\end{lemma}

\begin{proof}
By Lemma \ref{Lemma:trivial-K}, $\mathfrak{D}\left( A//A\right) $ has
trivial $\mathrm{K}$-theory. Thus, by the six-term exact sequence in \textrm{%
K}-theory, the definable group homomorphism $\mathrm{K}_{i}\left( \mathfrak{D%
}\left( A^{+}\right) \right) \rightarrow \mathrm{K}_{i}\left( \mathfrak{D}%
\left( A^{+}\right) /\mathfrak{D}\left( A//A\right) \right) $ is an
isomorphism. Since $\mathrm{K}_{i}\left( \mathfrak{D}\left( A^{+}\right)
\right) $ is a definable group by Proposition \ref{Proposition:K^1-definable}
and Corollary \ref{Corollary:K^0-definable}, the conclusion follows from
Corollary \ref{Corollary:Kechris--MacDonald}.
\end{proof}

\begin{lemma}
\label{Lemma:Kasparov-Paschke}Suppose that $A$ is a separable C*-algebra,
and $J$ is a closed two-sided ideal of $A$. Then the inclusion map $%
\mathfrak{D}\left( A\right) \subseteq \mathfrak{D}\left( J^{+}\right) $
induces an isomorphism $\mathfrak{D}\left( A\right) /\mathfrak{D}\left(
A//J\right) \rightarrow \mathfrak{D}\left( J\right) /\mathfrak{D}\left(
J//J\right) $ in the category of separable unital C*-algebras with a strict
cover.
\end{lemma}

\begin{proof}
We identify $A^{+}$ with its image inside $B\left( H\right) $ under $\rho
^{+}$. It follows from the definition that $\mathfrak{D}\left( A//J\right) =%
\mathfrak{D}\left( J//J\right) \cap \mathfrak{D}\left( A\right) $. Thus, the
inclusion map $\mathfrak{D}\left( A\right) \subseteq \mathfrak{D}\left(
J\right) $ induces a definable injective unital *-homomorphism $\mathfrak{D}%
\left( A\right) /\mathfrak{D}\left( A//J\right) \rightarrow \mathfrak{D}%
\left( J\right) /\mathfrak{D}\left( J//J\right) $, which is in fact onto 
\cite[Theorem 5.4.5]{higson_analytic_2000}. It remains to prove that the
inverse unital *-isomorphism $\mathfrak{D}\left( J\right) /\mathfrak{D}%
\left( J//J\right) \rightarrow \mathfrak{D}\left( A\right) /\mathfrak{D}%
\left( A//J\right) $ is also definable. This amounts at noticing that the
proof of \cite[Theorem 5.4.5]{higson_analytic_2000} via Kasparov's Technical
Theorem \cite[Theorem 3.8.1]{higson_analytic_2000} can be used to describe a
Borel lift $\mathfrak{D}\left( J\right) \rightarrow \mathfrak{D}\left(
A\right) $ of the unital *-isomorphism $\mathfrak{D}\left( J\right) /%
\mathfrak{D}\left( J//J\right) \rightarrow \mathfrak{D}\left( A\right) /%
\mathfrak{D}\left( A//J\right) $.

For $T\in \mathfrak{D}\left( J\right) $ let $E\left( T\right) $ be closed
linear span of $\left\{ \left[ a,T\right] ,\left[ a,T^{\ast }\right] :a\in
A\right\} $. Fix a dense sequence $\left( j_{m}\right) $ in $\mathrm{\mathrm{%
Ball}}\left( J\right) $, a dense sequence $\left( a_{m}\right) $ in $\mathrm{%
\mathrm{Ball}}\left( A\right) $, and a dense sequence $\left( b_{m}\right) $
in $\mathrm{\mathrm{Ball}}\left( K\left( H\right) \right) $. Notice that,
for $j\in J$, and $a\in A$, and $T\in \mathfrak{D}\left( J\right) $, we have
that%
\begin{equation*}
j\left[ a,T\right] =jaT-jTa\equiv jaT-Tja\equiv 0\mathrm{\ \mathrm{mod}}\
K\left( H\right) \text{.}
\end{equation*}%
Fix an approximate unit $\left( u_{n}\right) $ for $J$ such that, for $m\leq
n$,%
\begin{equation*}
\left\Vert u_{n}j_{m}-j_{m}\right\Vert \leq 2^{-n}
\end{equation*}%
and 
\begin{equation*}
\left\Vert u_{n}a_{m}-a_{m}u_{n}\right\Vert \leq 2^{-n}.
\end{equation*}%
Fix an approximate unit $\left( w_{n}\right) _{n\in \omega }$ for $K\left(
H\right) $ such that, if we set 
\begin{equation*}
d_{n}:=\left( w_{n}-w_{n-1}\right) ^{1/2}\text{,}
\end{equation*}%
then we have, for $m\leq n$,%
\begin{equation*}
\left\Vert d_{n}b_{m}\right\Vert \leq 2^{-n}
\end{equation*}%
\begin{equation*}
\left\Vert d_{n}j_{m}-j_{m}d_{n}\right\Vert \leq 2^{-n}
\end{equation*}%
\begin{equation*}
\left\Vert d_{n}a_{m}-a_{m}d_{n}\right\Vert \leq 2^{-n}\text{.}
\end{equation*}%
One can see that such an approximate unit for $K\left( H\right) $ exists by
considering a approximate unit for $K\left( H\right) $ that is quasicentral
for $J$ and $A$ \cite[Theorem 3.2.6]{higson_analytic_2000} and then a
suitable subsequence via a diagonal argument.

Fix $T\in \mathfrak{D}\left( J\right) $. Then using the Lusin--Novikov
Selection Theorem \cite[Theorem 18.10]{kechris_classical_1995} and \cite[%
Theorem 3.2.6]{higson_analytic_2000} one can see that one can recursively
define, for $n\in \omega $, $\ell _{n}^{T}\in \omega $, $k_{n,0}^{T},\ldots
,k_{n,\ell _{n}^{T}}^{T}\geq n$, and $t_{n,0}^{T},\ldots ,t_{n,\ell
_{n}^{T}}^{T}\in \left[ 0,1\right] \cap \mathbb{Q}$ that depend in a Borel
fashion from $T$ such that, setting%
\begin{equation*}
w_{n}^{T}:=t_{n,0}^{T}w_{k_{n,0}^{T}}+\cdots +t_{n,\ell
_{n}^{T}}^{T}w_{k_{n,\ell _{n}^{T}}^{T}}
\end{equation*}%
and%
\begin{equation*}
d_{n}^{T}:=\left( w_{n}^{T}-w_{n-1}^{T}\right) ^{1/2}
\end{equation*}%
one has that $w_{n}^{T},d_{n}^{T}\in K\left( H\right) $ depend in a Borel
fashion from $T$ and, for $m_{1},m_{2},m\leq n$,%
\begin{equation*}
\left\Vert d_{n}^{T}\left[ a_{m},T\right] -\left[ a_{m},T\right]
d_{n}^{T}\right\Vert \leq 2^{-n}
\end{equation*}%
\begin{equation*}
\left\Vert d_{n}^{T}\left[ T,j_{m}\right] \right\Vert \leq 2^{-n}
\end{equation*}%
\begin{equation*}
\left\Vert d_{n}^{T}u_{m_{1}}\left[ a_{m_{2}},T\right] \right\Vert \leq
2^{-n}\text{.}
\end{equation*}%
Furthermore, we also have from the choice of $\left( w_{n}\right) $ that,
for $m\leq n$,%
\begin{equation*}
\left\Vert d_{n}^{T}b_{m}\right\Vert \leq 2^{-n}
\end{equation*}%
\begin{equation*}
\left\Vert d_{n}^{T}j_{m}-j_{m}d_{n}^{T}\right\Vert \leq 2^{-n}
\end{equation*}%
\begin{equation*}
\left\Vert d_{n}^{T}a_{m}-a_{m}d_{n}^{T}\right\Vert \leq 2^{-n}\text{.}
\end{equation*}%
As in the proof of Kasparov's Technical Theorem \cite[Theorem 3.8.1]%
{higson_analytic_2000}, one has that%
\begin{equation*}
\sum_{n\in \omega }d_{n}^{T}u_{n}d_{n}^{T}
\end{equation*}%
converges in the strong-* topology to some positive element $X_{T}\in 
\mathrm{\mathrm{Ball}}\left( B\left( H\right) \right) $. Furthermore, we
have that 
\begin{equation*}
\left( 1-X_{T}\right) j\equiv 0\mathrm{\ \mathrm{mod}}\ K\left( H\right)
\end{equation*}%
\begin{equation*}
X_{T}\left[ T,a\right] \equiv 0\mathrm{\ \mathrm{mod}}\ K\left( H\right)
\end{equation*}%
\begin{equation*}
\left[ X_{T},a\right] \equiv 0\mathrm{\ \mathrm{mod}}\ K\left( H\right)
\end{equation*}%
for $j\in J$ and $a\in A$. Thus, $X_{T}T\in \mathfrak{D}\left( A\right) $
and $\left( 1-X_{T}\right) T\in \mathfrak{D}\left( J//J\right) $. Indeed, if 
$a\in A$ then we have that%
\begin{eqnarray*}
\left[ X_{T}T,a\right] &=&X_{T}Ta-aX_{T}T \\
&=&X_{T}Ta-X_{T}aT+X_{T}aT-aX_{T}T \\
&=&X_{T}\left[ T,a\right] +\left[ X_{T},a\right] T\equiv 0\mathrm{\ \mathrm{%
mod}}\ K\left( H\right) \text{.}
\end{eqnarray*}%
If $j\in J$ then we have that%
\begin{equation*}
\left( 1-X_{T}\right) Tj\equiv \left( 1-X_{T}\right) jT\equiv 0\mathrm{\ 
\mathrm{mod}}\ K\left( H\right) \text{.}
\end{equation*}%
We have that the function $\mathfrak{D}\left( J\right) \mapsto K\left(
H\right) $, $T\mapsto \left[ X_{T},a\right] $ is Borel, being the pointwise
limit of Borel functions%
\begin{equation*}
T\mapsto \sum_{k=0}^{n}\left[ d_{k}^{T}u_{k}d_{k}^{T},a\right]
\end{equation*}%
for $n\in \omega $. Thus, the function $\mathfrak{D}\left( J\right) \mapsto
K\left( H\right) $, $T\mapsto \left[ X_{T},a\right] T$ is Borel as well. For
the same reasons, the function $\mathfrak{D}\left( J\right) \mapsto K\left(
H\right) $, $T\mapsto X_{T}\left[ T,a\right] $ is Borel, and hence the
function $\mathfrak{D}\left( J\right) \mapsto K\left( H\right) $, $T\mapsto %
\left[ X_{T}T,a\right] =X_{T}\left[ T,a\right] +\left[ X_{T},a\right] T$ is
Borel. A similar argument shows that the function $\mathfrak{D}\left(
J\right) \rightarrow K\left( H\right) $, $T\mapsto X_{T}Tb$ is Borel for $%
b\in K\left( H\right) $. Therefore, the function $\mathfrak{D}\left(
J\right) \rightarrow \mathfrak{D}\left( A\right) $, $T\mapsto X_{T}T$ is
Borel. Since $T-X_{T}T=\left( 1-X_{T}\right) T\in \mathfrak{D}\left(
J//J\right) $, we have that $T\mapsto X_{T}T$ is a lift of the unital
*-isomorphism $\mathfrak{D}\left( J\right) /\mathfrak{D}\left( J//J\right)
\rightarrow \mathfrak{D}\left( A\right) /\mathfrak{D}\left( A//J\right) $.
This concludes the proof.
\end{proof}

\begin{corollary}
\label{Corollary:relative-K-homology}Suppose that $A$ is a separable
C*-algebra, and $J$ is a closed two-sided ideal of $A$. Fix $i\in \left\{
0,1\right\} $. Then $\mathrm{K}_{i}\left( \mathfrak{D}\left( A^{+}\right) /%
\mathfrak{D}\left( A//J\right) \right) $ is a definable group, definably
isomorphic to $\mathrm{K}_{i}\left( \mathfrak{D}\left( J^{+}\right) /%
\mathfrak{D}\left( J//J\right) \right) $.
\end{corollary}

\begin{proof}
By Lemma \ref{Lemma:full-ideal}, $\mathrm{K}_{i}\left( \mathfrak{D}\left(
J^{+}\right) /\mathfrak{D}\left( J//J\right) \right) $ is a definable group.
By Lemma \ref{Lemma:Kasparov-Paschke}, $\mathrm{K}_{i}\left( \mathfrak{D}%
\left( A^{+}\right) /\mathfrak{D}\left( A//J\right) \right) $ is isomorphic
to $\mathrm{K}_{i}\left( \mathfrak{D}\left( J^{+}\right) /\mathfrak{D}\left(
J//J\right) \right) $ in the category of semidefinable groups. Whence, the
conclusion follows from Lemma \ref{Lemma:iso-semi}.
\end{proof}

Suppose as above that $A$ is a separable C*-algebra, and $J$ is a closed
two-sided ideal of $A$. One defines for $i\in \left\{ 0,1\right\} $ the
relative $\mathrm{K}$-homology groups%
\begin{equation*}
\mathrm{K}^{i}\left( A,A/J\right) :=\mathrm{K}_{1-i}\left( \mathfrak{D}%
\left( A\right) /\mathfrak{D}\left( A//J\right) \right) ;
\end{equation*}%
see \cite[Definition 5.3.4]{higson_analytic_2000}. These are definable
groups by Corollary \ref{Corollary:relative-K-homology}. The assignment $%
\left( A,J\right) \mapsto \mathrm{K}^{i}\left( A,A/J\right) $ gives a
contravariant functor from the category of separable C*-pairs to the
category of definable groups. Here, a separable C*-pair is a pair $\left(
A,I\right) $ where $A$ is a separable C*-algebra and $I$ is a closed
two-sided ideal of $A$. A morphism $\left( A,I\right) \rightarrow \left(
B,J\right) $ of separable C*-pairs is a *-homomorphism $A\rightarrow B$ that
maps $I$ to $J$. If $\alpha :\left( A,I\right) \rightarrow \left( B,J\right) 
$ is a morphism of C*-pairs, and $V:H\rightarrow H$ is an isometry that
covers $\alpha ^{+}:A^{+}\rightarrow B^{+}$, then we have that the
corresponding strict unital *-homomorphism $\mathrm{Ad}\left( V\right) :%
\mathfrak{D}\left( B^{+}\right) \rightarrow \mathfrak{D}\left( A^{+}\right) $
maps $\mathfrak{D}\left( B//J\right) $ to $\mathfrak{D}\left( A//I\right) $.
Thus, it induces a definable unital *-homomorphism $\mathfrak{D}\left(
B\right) /\mathfrak{D}\left( B//J\right) \rightarrow \mathfrak{D}\left(
A\right) /\mathfrak{D}\left( A//I\right) $, and a definable group
homomorphisms $\mathrm{K}^{i}\left( B,B/J\right) \rightarrow \mathrm{K}%
^{i}\left( A,A/I\right) $.

Suppose that $\left( A,J\right) $ is a separable C*-pair. The natural
definable isomorphisms%
\begin{eqnarray*}
\mathrm{K}^{i}\left( A,A/J\right) &=&\mathrm{K}_{1-i}\left( \mathfrak{D}%
\left( A^{+}\right) /\mathfrak{D}\left( A//J\right) \right) \\
&\cong &\mathrm{K}_{1-i}\left( \mathfrak{D}\left( J\right) /\mathfrak{D}%
\left( J//J\right) \right) \cong \mathrm{K}_{1-i}\left( \mathfrak{D}\left(
J\right) \right) =\mathrm{K}^{i}\left( J\right)
\end{eqnarray*}%
from Lemma \ref{Lemma:Kasparov-Paschke} and Lemma \ref{Lemma:K-commutants}
give a natural definable isomorphism $\mathrm{K}^{i}\left( A,A/J\right)
\cong \mathrm{K}^{i}\left( J\right) $ called the \emph{excision isomorphism}%
; see \cite[Theorem 5.4.5]{higson_analytic_2000}.

Suppose that $\left( A,J\right) $ is a separable C*-pair. We say that $%
\left( A,J\right) $ is \emph{semi-split} if the short exact sequence%
\begin{equation*}
0\rightarrow J\rightarrow A\rightarrow A/J\rightarrow 0
\end{equation*}%
is semi-split in the sense of \cite[Definition 5.3.6]{higson_analytic_2000},
namely the quotient map $A^{+}\rightarrow A^{+}/J$ admits a ucp right
inverse. By the Choi--Effros lifting theorem \cite{choi_completely_1976},
every nuclear separable C*-pair is semi-split. Suppose that $\left(
A,J\right) $ is semisplit. If $V:H\rightarrow H$ is a linear isometry that
covers the quotient map $A\rightarrow A/J$, then the unital *-homomorphism $%
\mathrm{Ad}\left( V\right) :\mathfrak{D}\left( A/J\right) \rightarrow 
\mathfrak{D}\left( A//J\right) $ induces a natural definable isomorphism in $%
\mathrm{K}$-theory \cite[Proposition 5.3.7]{higson_analytic_2000}. In this
case, from the six-term exact sequence in $\mathrm{K}$-theory
\begin{center}
\begin{tikzcd}
\mathrm{K}_{1}\left( \mathfrak{D}\left( A//J\right) \right) \arrow[r] & \mathrm{K}_{1}\left( \mathfrak{D}\left( A\right) \right) \arrow[r]  & \mathrm{K}_{1}\left( \mathfrak{D}\left( A\right) \mathfrak{/D}\left(A//J\right) \right) \arrow[d] \\  
\mathrm{K}_{0}\left( \mathfrak{D}\left( A\right) \mathfrak{/D}\left(A//J\right) \right) \arrow[u]  & \mathrm{K}_{0}\left( \mathfrak{D}\left(A\right) \right) \arrow[l]  & \mathrm{K}_{0}\left( \mathfrak{D}\left(A//J\right) \right) \arrow[l] 
\end{tikzcd}
\end{center}
associated with the strict unital C*-pair $\left( \mathfrak{D}\left(
A\right) ,\mathfrak{D}\left( A//J\right) \right) $, one obtains the six-term
exact sequence in $\mathrm{K}$-homology associated with the separable
C*-pair $\left( A,J\right) $
\begin{center}
\begin{tikzcd}
\mathrm{K}^{0}\left( A/J\right) \arrow[r] & \mathrm{K}^{0}\left( A\right) \arrow[r] &  \mathrm{K}^{0}\left( A,A/J\right) \arrow[d] \\ 
\mathrm{K}^{1}\left( A,A/J\right)  \arrow[u] &  \mathrm{K}^{1}\left(A\right)  \arrow[l] &  \mathrm{K}^{1}\left( A/J\right)  \arrow[l]
\end{tikzcd}
\end{center}
as in \cite[Theorem 5.3.10]{higson_analytic_2000}, where the connecting maps
are definable homomorphisms.

\section{The Kasparov and Cuntz pictures of definable $\mathrm{K}$-homology 
\label{Section:Kasparov}}

In this section we recall the notion of graded Hilbert space and of (graded)
Fredholm module for a separable C*-algebra as in \cite[Chapter 8 and
Appendix A]{higson_analytic_2000}. We also recall Kasparov's description of $%
\mathrm{K}$-homology groups in terms of Fredholm modules from \cite%
{kasparov_topological_1975}. We then show that Kasparov's $\mathrm{K}$%
-homology groups can be regarded as definable groups, and are definably
isomorphic to the $\mathrm{K}$-homology groups as defined in the previous
section. We conclude by recalling the Cuntz picture for $\mathrm{K}$%
-homology from \cite{cuntz_new_1987}; see also \cite%
{higson_characterization_1987} and \cite[Chapter 5]{jensen_elements_1991}.
Again, we show that the Cuntz $\mathrm{K}$-homology groups can be seen as
definable groups, and are naturally definably isomorphic to the $\mathrm{K}$%
-homology groups as previously defined.

\subsection{Graded vector spaces and algebras}

Let $V$ be a vector space. A \emph{grading} of $V$ is a decomposition $%
V=V^{+}\oplus V^{-}$ as a direct sum of two subspaces, called the positive
and negative part of $V$. The corresponding grading operator $\gamma _{V}$
is the involution of $V$ whose eigenspaces for $1$ and $-1$ are $V^{+}$ and $%
V^{-}$, respectively. A vector space endowed with a grading is a\emph{\
graded vector space}. The \emph{opposite} of the graded vector space $V$ is
the graded vector space $V^{\mathrm{op}}$ obtained from $V$ by interchanging
the positive and the negative part. An endomorphism $T$ of $V$ is \emph{even}
if $T\left( V^{+}\right) \subseteq V^{+}$ and $T\left( V^{-}\right)
\subseteq V^{-}$ or, equivalently, $\gamma _{V}T=T\gamma _{V}$; it is \emph{%
odd }if $T\left( V^{+}\right) \subseteq V^{-}$ and $T\left( V^{-}\right)
\subseteq V^{+}$ or, equivalently, $\gamma _{V}T=-T\gamma _{V}$.

A \emph{graded Hilbert space} is Hilbert space endowed with a grading whose
positive and negative parts are closed orthogonal subspaces or,
equivalently, the grading operator is a self-adjoint unitary.

A \emph{graded algebra }is a complex algebra that is also a graded vector
space, and such that:%
\begin{equation*}
A^{+}\cdot A^{+}\cup A^{-}\cdot A^{-}\subseteq A^{+}
\end{equation*}%
and%
\begin{equation*}
A^{+}\cdot A^{-}\cup A^{-}\cdot A^{+}\subseteq A^{-}
\end{equation*}%
or, equivalently, the grading operator $\gamma _{A}$ is an algebra
automorphism of $A$. The elements of $A^{+}$ are \emph{even }elements of the
algebra, and the elements of $A^{-}$ are called \emph{odd }elements of the
algebra. An element is \emph{homogeneous} if it is either even or odd. The
degree $\partial a$ of an even element $a$ is $0$, while the degree $%
\partial a$ of an odd element $a$ is $1$. The graded commutator of elements
of $A$ is defined for homogeneous elements by%
\begin{equation*}
\left[ a,a^{\prime }\right] =aa^{\prime }-\left( -1\right) ^{\partial
a\partial a^{\prime }}a^{\prime }a
\end{equation*}%
and extended by linearity.

A\emph{\ graded\ C*-algebra} $A$ is a C*-algebra that is also a graded
algebra and such that $A^{+}$ and $A^{-}$ are closed self-adjoint subspaces
or, equivalently, the grading operator $\gamma _{A}$ is a C*-algebra
automorphism of $A$.

\begin{example}
Suppose that $V$ is a graded vector space. Then the algebra $\mathrm{\mathrm{%
End}}\left( V\right) $ of endomorphisms of $V$ is a graded algebra, with $%
\mathrm{\mathrm{End}}\left( V\right) ^{+}$ equal to the set of even
endomorphisms of $V$, and $\mathrm{\mathrm{End}}\left( V\right) ^{-}$ is the
set of odd endomorphisms of $V$.

If $H$ is a graded Hilbert space, then $B\left( H\right) \subseteq \mathrm{%
\mathrm{\mathrm{End}}}\left( H\right) $ is a graded C*-algebra.
\end{example}

\begin{example}
Fix $n\geq 1$. Define $\mathbb{C}_{n}$ to be the graded complex unital
*-algebra generated by $n$ odd operators $e_{1},\ldots ,e_{n}$ such that,
for distinct $i,j\in \left\{ 1,2,\ldots ,n\right\} $, 
\begin{equation*}
e_{i}e_{j}+e_{j}e_{i}=0\text{,}
\end{equation*}%
\begin{equation*}
e_{j}^{2}=-1\text{,}
\end{equation*}%
\begin{equation*}
e_{j}^{\ast }=-e_{j}\text{.}
\end{equation*}%
As a complex vector space, $\mathbb{C}_{n}$ has dimension $2^{n}$, where
monomials $e_{i_{1}}\cdots e_{i_{k}}$ for $1\leq i_{1}<\cdots <i_{k}\leq n$
and $0\leq k\leq n$ comprise a basis. Declaring these monomials to be
orthogonal defines an inner product on $\mathbb{C}_{n}$. The left regular
representation of $\mathbb{C}_{n}$ on the Hilbert space $\mathbb{C}_{n}$
turns $\mathbb{C}_{n}$ into a graded C*-algebra.
\end{example}

Suppose that $V_{1}$ and $V_{2}$ are graded vector spaces. The graded tensor
product $V:=V_{1}\hat{\otimes}V_{2}$ is the tensor product of $V_{1}$ and $%
V_{2}$ equipped with the grading operator $\gamma _{V}:=\gamma
_{V_{1}}\otimes \gamma _{V_{2}}$. Thus, we have that%
\begin{equation*}
V^{+}=\left( V_{1}^{+}\otimes V_{2}^{+}\right) \oplus \left(
V_{1}^{-}\otimes V_{2}^{-}\right)
\end{equation*}%
\begin{equation*}
V^{-}=\left( V_{1}^{+}\otimes V_{2}^{-}\right) \oplus \left(
V_{1}^{-}\otimes V_{2}^{+}\right) \text{.}
\end{equation*}%
If $A_{1}$ and $A_{2}$ are graded algebras, then the graded tensor product $%
A:=A_{1}\hat{\otimes}A_{2}$ (as graded vector spaces) is a graded algebra
with respect to the multiplication operation defined on homogeneous
elementary tensors by%
\begin{equation*}
\left( a_{1}\hat{\otimes}a_{2}\right) \left( a_{1}^{\prime }\hat{\otimes}%
a_{2}^{\prime }\right) =\left( -1\right) ^{\partial a_{2}\partial
a_{1}^{\prime }}\left( a_{1}a_{1}^{\prime }\hat{\otimes}a_{2}a_{2}^{\prime
}\right) \text{.}
\end{equation*}%
When $V_{1},V_{2}$ are graded vector spaces, then there is canonical
inclusion 
\begin{equation*}
\mathrm{\mathrm{End}}\left( V_{1}\right) \hat{\otimes}\mathrm{\mathrm{End}}%
\left( V_{2}\right) \subseteq \mathrm{\mathrm{End}}\left( V_{1}\hat{\otimes}%
V_{2}\right)
\end{equation*}
obtained by setting, for homogeneous $T_{i}\in \mathrm{\mathrm{End}}\left(
V_{i}\right) $ and $v_{i}\in V_{i}$, 
\begin{equation*}
\left( T_{1}\hat{\otimes}T_{2}\right) \left( v_{1}\hat{\otimes}v_{2}\right)
=\left( -1\right) ^{\partial v_{1}\partial T_{2}}\left( T_{1}v_{1}\otimes
T_{2}v_{2}\right) \text{.}
\end{equation*}
We have that $\mathrm{\mathrm{End}}\left( V_{1}\right) \hat{\otimes}\mathrm{%
\mathrm{End}}\left( V_{2}\right) =\mathrm{\mathrm{End}}\left( V_{1}\hat{%
\otimes}V_{2}\right) $ when $V_{1},V_{2}$ are finite-dimensional.

\begin{example}
There is a canonical isomorphism $\mathbb{C}_{m}\hat{\otimes}\mathbb{C}%
_{n}\cong \mathbb{C}_{m+n}$, obtained by mapping $e_{i}\hat{\otimes}1$ to $%
e_{i}$ and $1\hat{\otimes}e_{j}$ to $e_{m+j}$ for $i\in \left\{ 1,2,\ldots
,m\right\} $ and $j\in \left\{ 1,2,\ldots ,n\right\} $.
\end{example}

Fix $p\geq 0$. A $p$-graded Hilbert space is a graded Hilbert space endowed
with $p$ odd operators $\varepsilon _{1},\ldots ,\varepsilon _{p}$ such that 
$\varepsilon _{i}\varepsilon _{j}+\varepsilon _{j}\varepsilon _{i}=0$, $%
\varepsilon _{j}^{2}=-1$, and $\varepsilon _{j}^{\ast }=-\varepsilon _{j}$
for distinct $i,j\in \left\{ 1,2,\ldots ,p\right\} $. Equivalently, a $p$%
-graded Hilbert space can be thought of as a graded right module over $%
\mathbb{C}_{p}$, where one sets%
\begin{equation*}
xe_{i}:=\varepsilon _{i}(x)
\end{equation*}%
for $i\in \left\{ 1,2,\ldots ,p\right\} $ and $x\in H$. A $0$-graded Hilbert
space is simply a graded Hilbert space. By convention, a $\left( -1\right) $%
-graded Hilbert space is an ungraded Hilbert space. Suppose that $H_{0}$ and 
$H_{1}$ are $p$-graded Hilbert spaces. A $p$-graded bounded linear map $%
H_{0}\rightarrow H_{1}$ is a bounded linear map that is also a right $%
\mathbb{C}_{p}$-module map.

\begin{example}
Suppose that $H,H^{\prime }$ are $p$-graded Hilbert space. Then $H^{\mathrm{%
op}}$ is $p$-graded, where $\varepsilon _{i}^{H^{\mathrm{op}}}=-\varepsilon
_{i}^{H}$ for $1\leq i\leq n$, and $H\oplus H^{\prime }$ is $p$-graded,
where $\varepsilon _{i}^{H\oplus H^{\prime }}=\varepsilon _{i}^{H}\oplus
\varepsilon _{i}^{H^{\prime }}$.
\end{example}

\begin{example}
If $H_{1}$ is $p_{1}$-graded and $H_{2}$ is $p_{2}$-graded, then considering
the isomorphism $\mathbb{C}_{p_{1}}\hat{\otimes}\mathbb{C}_{p_{2}}\cong 
\mathbb{C}_{p_{1}+p_{2}}$, and the inclusion $B\left( H_{1}\hat{\otimes}%
H_{2}\right) \subseteq B\left( H_{1}\right) \hat{\otimes}B\left(
H_{2}\right) $, we have that $H_{1}\hat{\otimes}H_{2}$ is $\left(
p_{1}+p_{2}\right) $-graded.
\end{example}

A straightforward verification shows the following; see \cite[Proposition
A.3.4]{higson_analytic_2000}.

\begin{proposition}
For $p\geq 0$, the categories of $p$-multigraded and $\left( p+2\right) $%
-multigraded Hilbert spaces are equivalent. The categories of Hilbert spaces
and linear maps and $1$-graded Hilbert spaces and \emph{even} $1$-graded
linear maps are equivalent.
\end{proposition}

\subsection{Fredholm modules}

Suppose that $A$ is a separable C*-algebra. We now recall the definition of
Fredholm module over $A$; see \cite[Definition 8.1.1]{higson_analytic_2000}.
For each dimension $d\in \omega \cup \left\{ \aleph _{0}\right\} $ fix a
Hilbert space $H_{d}$ of dimension $d$.

\begin{definition}
Suppose that $A$ is a separable C*-algebra. A \emph{Fredholm module} over $A$
is a triple $\left( F,\rho ,H\right) $ such that:

\begin{itemize}
\item $H$ is a separable Hilbert space $H_{d}$ for some $d\in \omega \cup
\left\{ \aleph _{0}\right\} $;

\item $\rho :A\rightarrow B\left( H\right) $ is a *-homomorphism;

\item $F\in B\left( H\right) $ satisfies $\left( F^{2}-1\right) \rho \left(
a\right) \equiv \left( F-F^{\ast }\right) \rho \left( a\right) \equiv \left[
F,\rho \left( a\right) \right] \equiv 0\mathrm{\ \mathrm{mod}}\ K\left(
H\right) $ for every $a\in A$.
\end{itemize}
\end{definition}

\begin{remark}
\label{Remark:Borel-F-1}Recall that, if $H$ is a separable Hilbert space,
then $B\left( H\right) $ is a standard Borel space with respect to the Borel
structure induced by the strong-* topology on $\mathrm{\mathrm{Ball}}\left(
B\left( H\right) \right) $. Similarly, the Banach space $L\left( A,B\left(
H\right) \right) $ of bounded linear maps $A\rightarrow B\left( H\right) $
is a standard Borel space when endowed with the Borel structure induced by
the topology of pointwise strong-* convergence on $\mathrm{\mathrm{Ball}}%
\left( L\left( A,B\left( H\right) \right) \right) $. The set $\mathrm{F}%
_{-1}\left( A\right) $ of Fredholm modules over $A$ can thus be naturally
regarded as a standard Borel space.
\end{remark}

The definition of \emph{graded }Fredholm module is similar, where one
replaces Hilbert spaces with \emph{graded }Hilbert spaces.

\begin{definition}
Suppose that $A$ is a separable C*-algebra. A \emph{graded} Fredholm module
over $A$ is a triple $\left( F,\rho ,H\right) $ such that:

\begin{itemize}
\item $H$ is a separable graded Hilbert space of the form $\left(
H_{d},\gamma \right) $ for some $d\in \omega \cup \left\{ \aleph
_{0}\right\} $ and some grading operator $\gamma $ on $H_{d}$;

\item $\rho :A\rightarrow B\left( H\right) $ is a *-homomorphism such that,
for every $a\in A$, $\rho \left( a\right) \in B\left( H\right) ^{+}$ is 
\emph{even}, and hence $\rho =\rho ^{+}\oplus \rho ^{-}$ for some
representations $\rho ^{\pm }$ of $A$ on $H^{\pm }$, where we regard $%
B\left( H\right) $ as a graded C*-algebra;

\item $F\in B\left( H\right) $ is an \emph{odd }operator that satisfies, for
every $a\in A$, $\left( F^{2}-1\right) \rho \left( a\right) \equiv \left(
F-F^{\ast }\right) \rho \left( a\right) \equiv \left[ F,\rho \left( a\right) %
\right] \equiv 0\mathrm{\ \mathrm{mod}}\ K\left( H\right) $.
\end{itemize}
\end{definition}

\begin{remark}
\label{Remark:Borel-F0}Again, we have that the set $\mathrm{F}_{0}\left(
A\right) $ of Fredholm modules over $A$ is a standard Borel space.
\end{remark}

The notions of graded and ungraded Fredholm modules can be recognized as
particular instances (for $p=0$ and $p=-1$, respectively) of the notion of $%
p $-graded Fredholm module; see \cite[Definition 8.1.11]%
{higson_analytic_2000}.

\begin{definition}
Suppose that $A$ is a separable C*-algebra. A $p$-multigraded Fredholm
module over $A$ is a triple $\left( F,\rho ,H\right) $ such that:

\begin{itemize}
\item $H$ is a separable $p$-multigraded Hilbert space $H$ of the form $%
\left( H_{d},\gamma ,\varepsilon _{1},\ldots ,\varepsilon _{p}\right) $ for $%
d\in \omega \cup \left\{ \aleph _{0}\right\} $, grading operator $\gamma $
on $H_{d}$, and odd operators $\varepsilon _{1},\ldots ,\varepsilon _{d}$ on 
$\left( H_{d},\gamma \right) $;

\item a *-homomorphism $\rho :A\rightarrow B\left( H\right) $ such that, for
every $a\in A$, $\rho \left( a\right) $ is an \emph{even }$p$-\emph{%
multigraded }operator on $H$;

\item $F\in B\left( H\right) $ is an \emph{odd }$p$-\emph{multigraded}
operator on $H$ such that, for every $a\in A$, $\left( F^{2}-1\right) \rho
\left( a\right) \equiv \left( F-F^{\ast }\right) \rho \left( a\right) \equiv %
\left[ F,\rho \left( a\right) \right] \equiv 0\mathrm{\ \mathrm{mod}}\
K\left( H\right) $.
\end{itemize}
\end{definition}

\begin{remark}
\label{Remark:Borel-Fp}As in the cases $p=0$ and $p=-1$, the set $\mathrm{F}%
_{p}\left( A\right) $ of $p$-multigraded Fredholm modules over $A$ is a
standard Borel space.
\end{remark}

We recall the notion of \emph{degenerate }$p$-multigraded Fredholm module;
see \cite[Definition 8.2.7]{higson_analytic_2000}.

\begin{definition}
Suppose that $A$ is a separable C*-algebra. A $p$-multigraded Fredholm
module $\left( F,\rho ,H\right) $ over $A$ is \emph{degenerate }if $\left(
F^{2}-1\right) \rho \left( a\right) =\left( F-F^{\ast }\right) \rho \left(
a\right) =\left[ F,\rho \left( a\right) \right] $ for every $a\in A$.
\end{definition}

It is clear that the set \textrm{D}$_{p}\left( A\right) $ of degenerate $p$%
-multigraded Fredholm modules is a Borel subset of $\mathrm{F}_{p}\left(
A\right) $.

Given $p$-multigraded Fredholm modules $x=\left( \rho ,H,F\right) $ and $%
x^{\prime }=\left( \rho ^{\prime },H^{\prime },F^{\prime }\right) $ over $A$%
, their \emph{sum} is the $p$-multigraded Fredholm module $x\oplus x^{\prime
}=\left( \rho \oplus \rho ^{\prime },H\oplus H^{\prime },F\oplus F^{\prime
}\right) $. The \emph{opposite} of $x$ is the $p$-multigraded Fredholm
module $x^{\mathrm{op}}=\left( \rho ,H^{\mathrm{op}},-F\right) $. The sum
and opposite define Borel functions $\mathrm{F}_{p}\left( A\right) \times 
\mathrm{F}_{p}\left( A\right) \rightarrow \mathrm{F}_{p}\left( A\right) $, $%
\left( x,x^{\prime }\right) \mapsto x\oplus x^{\prime }$ and $\mathrm{F}%
_{p}\left( A\right) \rightarrow \mathrm{F}_{p}\left( A\right) $, $x\mapsto
x^{\mathrm{op}}$.

Let $A$ be a separable C*-algebra, and fix $p\geq -1$. Suppose that $\left(
\rho ,H,F\right) $ and $\left( \rho ^{\prime },H^{\prime },F^{\prime
}\right) $ are $p$-multigraded Fredholm modules. Then $\left( \rho
,H,F\right) $ and $\left( \rho ,H^{\prime },F^{\prime }\right) $ are:

\begin{itemize}
\item \emph{unitarily equivalent} if there exists an \emph{even} $p$%
-multigraded unitary linear map $U:H^{\prime }\rightarrow H$ such that $%
F^{\prime }=\mathrm{Ad}\left( U\right) \left( F\right) $ and $\rho ^{\prime
}=\mathrm{Ad}\left( U\right) \circ \rho $ \cite[Definition 8.2.1]%
{higson_analytic_2000};

\item \emph{operator homotopic} if $\rho =\rho ^{\prime }$, $H=H^{\prime }$,
and there exists a \emph{norm-continuous} path $\left( F_{t}\right) _{t\in %
\left[ 0,1\right] }$ in $B\left( H\right) $ such that $F_{0}=F$, $%
F_{1}=F^{\prime }$ and, for every $t\in \left[ 0,1\right] $, $\left( \rho
,H,F_{t}\right) $ is a $p$-multigraded Fredholm module over $A$ \cite[%
Definition 8.2.2]{higson_analytic_2000}.
\end{itemize}

The notion of \emph{stable homotopy }is defined in terms of unitary
equivalence and operator homotopy; see \cite[Proposition 8.2.12]%
{higson_analytic_2000}.

\begin{definition}
\label{Definition:stable homotopy}Suppose that $A$ is a separable
C*-algebra, and $p\geq -1$. The relation $\mathrm{B}_{p}\left( A\right) $ of 
\emph{stable homotopy }of $p$-multigraded Fredholm modules over $A$ is the
relation defined by setting $x\mathrm{B}_{p}\left( A\right) x^{\prime }$ if
and only if there exists a degenerate $p$-multigraded Fredholm module $x_{0}$
over $A$ such that $x\oplus x_{0}$ and $x^{\prime }\oplus x_{0}$ are
unitarily equivalent to a pair of operator homotopic $p$-multigraded
Fredholm modules over $A$.
\end{definition}

One has that $\mathrm{B}_{p}\left( A\right) $ is an equivalence relation on $%
\mathrm{F}_{p}\left( A\right) $; see \cite[Proposition 8.2.12]%
{higson_analytic_2000}. Furthermore, $\mathrm{B}_{p}\left( A\right) $ is an 
\emph{analytic }equivalence relation, as it follows easily from the
definition and the fact that the set \textrm{D}$_{p}\left( A\right) $ of
degenerate $p$-multigraded Fredholm modules is a Borel subset of $\mathrm{F}%
_{p}\left( A\right) $, and the set of norm-continuous paths in $B\left(
H\right) $ is a Borel subset of the C*-algebra $C_{\beta }\left(
[0,1],B\left( H\right) \right) =M\left( C\left( [0,1],K\left( H\right)
\right) \right) $ of strictly continuous bounded functions $\left[ 0,1\right]
\rightarrow B\left( H\right) $ by Corollary \ref%
{Corollary:norm-continuous-paths}.

\subsection{The Kasparov $\mathrm{K}$-homology groups}

We now recall the definition of the Kasparov $\mathrm{K}$-homology groups in
terms of Fredholm modules; see \cite[Definition 8.2.5 and Proposition 8.2.12]%
{higson_analytic_2000}.

\begin{definition}
\label{Definition:Kasparov-K-homology}Let $A$ be a separable C*-algebra, and
fix $p\geq -1$. The \emph{Kasparov }$\mathrm{K}$-\emph{homology group }$%
\mathrm{KK}_{-p}\left( A;\mathbb{C}\right) $ is the \emph{semidefinable}
abelian group obtained as the quotient of the standard Borel space $\mathrm{F%
}_{p}\left( A\right) $ by the analytic equivalence relation $\mathrm{B}%
_{p}\left( A\right) $ of stable homotopy of $p$-multigraded Fredholm
modules, where the group operation on $\mathrm{KK}_{-p}\left( A;\mathbb{C}%
\right) $ is induced by the Borel function $\mathrm{F}_{p}\left( A\right)
\times \mathrm{F}_{p}\left( A\right) \rightarrow \mathrm{F}_{p}\left(
A\right) $, $\left( x,x^{\prime }\right) \mapsto x\oplus x^{\prime }$, and
the function $\mathrm{KK}_{-p}\left( A;\mathbb{C}\right) \rightarrow \mathrm{%
KK}_{-p}\left( A;\mathbb{C}\right) $ that maps each element to its additive
inverse is induced by the Borel function $\mathrm{F}_{p}\left( A\right)
\rightarrow \mathrm{F}_{p}\left( A\right) $, $x\mapsto x^{\mathrm{op}}$.
\end{definition}

The fact that $\mathrm{KK}_{-p}\left( A;\mathbb{C}\right) $ is indeed a
group is the content of \cite[Proposition 8.2.10, Corollary 8.2.11,
Proposition 8.2.12]{higson_analytic_2000}. The trivial element of $\mathrm{KK%
}_{-p}\left( A;\mathbb{C}\right) $ is given by the $\mathrm{B}_{p}\left(
A\right) $-class of any degenerate Fredholm module. The assignment $A\mapsto 
\mathrm{KK}_{-p}\left( A;\mathbb{C}\right) $ is easily seen to be a
contravariant functor from separable C*-algebras to semidefinable groups. We
will later show in Proposition \ref{Proposition:definable-Kasparov} that $%
\mathrm{KK}_{-p}\left( A;\mathbb{C}\right) $ is in fact a definable group.

Suppose that $A$ is a separable C*-algebra. Fix a representation $\rho
_{A}:A\rightarrow B\left( H_{A}\right) $ of $A$ that is the restriction to $%
A $ of an ample representation of the unitization $A^{+}$ of $A$. We then
let $\rho _{A}\oplus \rho _{A}$ be the corresponding representation (by even
operators) on the graded Hilbert space $H_{A}\oplus H_{A}$. Consider also
the Paschke dual algebra $\mathfrak{D}\left( A\right) =\mathfrak{D}_{\rho
_{A}}\left( A\right) $ associated with $\rho _{A}$; see Section \ref%
{Subsection:polar}

There is a natural definable group homomorphism $\Phi ^{1}:\mathrm{K}%
_{0}\left( \mathfrak{D}\left( A\right) \right) \rightarrow \mathrm{KK}%
_{1}\left( A,\mathbb{C}\right) $, $\left[ P\right] \mapsto \left[ x_{P}%
\right] $, defined as follows. Given a projection $P$ in $\mathfrak{D}\left(
A\right) $, define $x_{P}$ to be the ungraded Fredholm module $\left( \rho
_{A},H_{A},2P-I\right) $ over $A$; see \cite[Example 8.1.7]%
{higson_analytic_2000}. We also have a natural definable group homomorphism $%
\Phi ^{0}:\mathrm{K}_{1}\left( \mathfrak{D}\left( A\right) \right)
\rightarrow \mathrm{KK}_{0}\left( A,\mathbb{C}\right) $, $\left[ U\right]
\mapsto \left[ x_{U}\right] $, defined as follows. Given a unitary $U$ in $%
\mathfrak{D}\left( A\right) $, define $x_{U}$ to be the graded Fredholm
module $\left( \rho _{A}\oplus \rho _{A},H_{A}\oplus H_{A},F_{U}\right) $
where $H_{A}\oplus H_{A}$ is graded by $I_{H_{A}}\oplus \left(
-I_{H_{A}}\right) $, and%
\begin{equation*}
F_{U}=%
\begin{bmatrix}
0 & U^{\ast } \\ 
U & 0%
\end{bmatrix}%
\text{;}
\end{equation*}%
see \cite[Example 8.1.7]{higson_analytic_2000}.\ Then it is shown in \cite[%
Theorem 8.4.3]{higson_analytic_2000} that $\Phi ^{1}$ and $\Phi ^{0}$ are in
fact group isomorphism. From this, we obtain the following.

\begin{proposition}
\label{Proposition:definable-Kasparov}Suppose that $A$ is a separable
C*-algebra and $i\in \left\{ 0,1\right\} $. Then $\mathrm{KK}_{i}\left( A;%
\mathbb{C}\right) $ is a definable group, naturally definably isomorphic to $%
\mathrm{K}^{i}\left( A\right) $.
\end{proposition}

\begin{proof}
Since $\mathrm{K}_{1-i}\left( \mathfrak{D}\left( A\right) \right) =\mathrm{K}%
^{i}\left( A\right) $ is a definable group and $\Phi ^{i}:\mathrm{K}%
_{1-i}\left( \mathfrak{D}\left( A\right) \right) \rightarrow \mathrm{KK}%
_{i}\left( A;\mathbb{C}\right) $ is a definable group isomorphism, it
follows from Corollary \ref{Corollary:Kechris--MacDonald} that $\mathrm{KK}%
_{i}\left( A;\mathbb{C}\right) $ is a definable group, naturally definably
isomorphic to $\mathrm{K}^{i}\left( A\right) $.
\end{proof}

\begin{remark}
\label{Remark:normalization}Suppose that $A$ is a separable C*-algebra, and $%
p\geq -1$. A $p$-multigraded Fredholm module $\left( \rho ,H,F\right) $ over 
$A$ is self-adjoint if $F$ is self-adjoint, and contractive if $F$ is
contractive \cite[Definition 8.3.1]{higson_analytic_2000}. A self-adjoint,
contractive $p$-multigraded Fredholm module $\left( \rho ,H,F\right) $ is 
\emph{involutive} if $F^{2}=1$ \cite[Definition 8.3.4]{higson_analytic_2000}%
. Kasparov's $\mathrm{K}$-homology groups can be \emph{normalized }by
requiring that the Fredholm modules be involutive. This means that one
obtain the same definable abelian group (up to a natural isomorphism) if one
only considers in the definition of the Kasparov $\mathrm{K}$-homology
groups involutive Fredholm modules, where also stable homotopy is defined in
terms of involutive Fredholm modules; see \cite[Lemma 8.3.5]%
{higson_analytic_2000}.

A graded Fredholm module $\left( \rho ,H,F\right) $ over $A$ is \emph{%
balanced }if there is a separable Hilbert space $H^{\prime }$ such that $%
H=H^{\prime }\oplus H^{\prime }$ is graded by $I_{H^{\prime }}\oplus \left(
-I_{H^{\prime }}\right) $, and $\rho =\rho ^{+}\oplus \rho ^{-}$, where $%
\rho ^{+}$ and $\rho ^{-}$ are the same representation of $A$ on $H^{\prime
} $. Then one has that $\mathrm{KK}_{0}\left( A;\mathbb{C}\right) $ can be
normalized by requring that the graded Fredholm modules be involutive and
balanced \cite[Proposition 8.3.12]{higson_analytic_2000}.
\end{remark}

Suppose that $A$ is a separable C*-algebra. Fix $p\geq 0$. If $x=\left( \rho
,H,F\right) $ is a $p$-multigraded Fredholm module over $A$, then one can
assign to it the $\left( p+2\right) $-multigraded Fredholm module $x^{\prime
}=\left( \rho \oplus \rho ^{\mathrm{op}},H\oplus H^{\mathrm{op}},F\oplus F^{%
\mathrm{op}}\right) $ where $H\oplus H^{\mathrm{op}}$ is $\left( p+2\right) $%
-multigraded by the operators $\varepsilon _{i}\oplus \varepsilon _{i}^{%
\mathrm{op}}$ for $1\leq i\leq p$ together with%
\begin{equation*}
\begin{bmatrix}
0 & I \\ 
-I & 0%
\end{bmatrix}%
\text{ and }%
\begin{bmatrix}
0 & iI \\ 
iI & 0%
\end{bmatrix}%
\text{.}
\end{equation*}%
When $p=-1$ one can define $x^{\prime }$ to be the to be the $1$-graded
Fredholm module $\left( \rho \oplus \rho ,H\oplus H,F\oplus F\right) $ where 
$H\oplus H$ is graded by $I_{H}\oplus \left( -I_{H}\right) $ and $1$%
-multigraded by the odd operator 
\begin{equation*}
\begin{bmatrix}
0 & iI \\ 
iI & 0%
\end{bmatrix}%
\text{.}
\end{equation*}%
This gives for $p\geq -1$ a natural definable group isomorphism $\mathrm{KK}%
_{-p}\left( A;\mathbb{C}\right) \rightarrow \mathrm{KK}_{-p-2}\left( A;%
\mathbb{C}\right) $ \cite[Proposition 8.2.13]{higson_analytic_2000}.\ From
this, Proposition \ref{Corollary:Kechris--MacDonald}, and Proposition \ref%
{Proposition:definable-Kasparov}, we obtain that, for $p\geq -1$, $\mathrm{KK%
}_{-p}\left( A;\mathbb{C}\right) $ is a definable group, naturally definably
isomorphic to $\mathrm{K}_{-p-2}\left( A\right) $.

\subsection{Relative Kasparov $\mathrm{K}$-homology}

The \emph{relative }Kasparov $\mathrm{K}$-homology groups are defined as
above, by replacing Fredholm modules with \emph{relative }Fredholm modules;
see \cite[Definition 8.5.1]{higson_analytic_2000}.

\begin{definition}
Suppose that $\left( A,J\right) $ is a separable C*-pair. An ungraded
Fredholm module over $\left( A,J\right) $ is a triple $\left( \rho
,H,F\right) $ where:

\begin{itemize}
\item $H$ is a separable Hilbert space $H_{d}$ for some $d\in \omega \cup
\left\{ \aleph _{0}\right\} $;

\item $\rho :A\rightarrow B\left( H\right) $ is a *-homomorphism;

\item $F\in B\left( H\right) $ satisfies $\left( F^{2}-1\right) \rho \left(
j\right) \equiv \left( F-F^{\ast }\right) \rho \left( j\right) \equiv \left[
F,\rho \left( a\right) \right] \equiv 0\mathrm{\ \mathrm{mod}}\ K\left(
H\right) $ for every $a\in A$ and $j\in J$.
\end{itemize}

A graded Fredholm module over $\left( A,J\right) $ is a triple $\left( \rho
,H,F\right) $ where:

\begin{itemize}
\item $H$ is a separable graded Hilbert space of the form $\left(
H_{d},\gamma \right) $ for some $d\in \omega \cup \left\{ \aleph
_{0}\right\} $ and some grading operator $\gamma $ on $H_{d}$;

\item $\rho :A\rightarrow B\left( H\right) $ is a *-homomorphism such that,
for every $a\in A$, $\rho \left( a\right) \in B\left( H\right) ^{+}$ is 
\emph{even};

\item $F\in B\left( H\right) $ is an \emph{odd }operator that satisfies $%
\left( F^{2}-1\right) \rho \left( j\right) \equiv \left( F-F^{\ast }\right)
\rho \left( j\right) \equiv \left[ F,\rho \left( a\right) \right] \equiv 0%
\mathrm{\ \mathrm{mod}}\ K\left( H\right) $ for every $a\in A$ and $j\in J$.
\end{itemize}
\end{definition}

As above, one can consider the definable group $\mathrm{KK}_{-1}\left( A,J;%
\mathbb{C}\right) $ whose elements are stable homotopy equivalence classes
of Fredholm modules over $\left( A,J\right) $. Considering \emph{graded}
Fredholm modules over $\left( A,J\right) $ one obtains the definable group $%
\mathrm{KK}_{0}\left( A,J;\mathbb{C}\right) $. These groups are called the
relative Kasparov $\mathrm{K}$-homology groups of the pair $\left(
A,J\right) $, and turn out to be naturally definably isomorphic to the
relative $\mathrm{K}$-homology groups $\mathrm{K}^{1}\left( A,J\right) $ and 
$\mathrm{K}^{0}\left( A,J\right) $; see \cite[Section 8.5]%
{higson_analytic_2000}. More generally, one can define $\mathrm{KK}%
_{-p}\left( A,J;\mathbb{C}\right) $ in terms of $p$-multigraded Fredholm
modules over $\left( A,J\right) $.

In the Kasparov picture, the excision isomorphism $\mathrm{KK}_{-p}\left(
A,J;\mathbb{C}\right) \rightarrow \mathrm{K}_{-p}\left( J;\mathbb{C}\right) $
is induced by the inclusion map from the set of $p$-multigraded Fredholm
modules over $\left( A,J\right) $ into the set of $p$-multigraded Fredholm
modules over $J$. (Notice that a Fredholm module over $\left( A,J\right) $
is, in particular, a Fredholm module over $J$.)

\subsection{\textrm{KK}$_{h}$-cycles and $\mathrm{K}$-homology}

Suppose that $A$ is a separable C*-algebra. Let $H$ be the separable
infinite-dimensional Hilbert space. A $\mathrm{KK}_{h}$-cycle for $A$ is a
pair $\left( \phi _{+},\phi _{-}\right) $ of *-homomorphisms $A\rightarrow
B\left( H\right) $ such that $\phi _{+}\left( a\right) \equiv \phi
_{-}\left( a\right) \mathrm{\ \mathrm{mod}}\ K\left( H\right) $ for every $%
a\in A$; see \cite[Definition 4.1.1]{jensen_elements_1991}. Define $\mathbb{F%
}\left( A;\mathbb{C}\right) $ to be the standard Borel space of $\mathrm{KK}%
_{h}$-cycles for $A$. The standard Borel structure on $\mathbb{F}\left( A;%
\mathbb{C}\right) $ is induced by the Polish topology obtained by setting $%
(\phi _{+}^{\left( i\right) },\phi _{-}^{\left( i\right) })\rightarrow
\left( \phi _{+},\phi _{-}\right) $ if and only if, for every $a\in A$, $%
\phi _{+}^{\left( i\right) }\left( a\right) \rightarrow \phi _{+}\left(
a\right) $ and $\phi _{-}^{\left( i\right) }\rightarrow \phi _{-}\left(
a\right) $ in the strong-* topology, and $(\phi _{+}^{\left( i\right) }-\phi
_{-}^{\left( i\right) })\left( a\right) \rightarrow (\phi _{+}-\phi
_{-})\left( a\right) $ in norm. We regard $\mathbb{F}\left( A;\mathbb{C}%
\right) $ as a Polish space with respect to such a topology.

Define $\sim $ to be the relation of homotopy for elements of the Polish
space $\mathbb{F}\left( A;\mathbb{C}\right) $. Thus, for $x,x^{\prime }\in 
\mathbb{F}\left( A;\mathbb{C}\right) $, $x\sim x^{\prime }$ if and only if
there exists a continuous path $\left( x_{t}\right) _{t\in \left[ 0,1\right]
}$ in $\mathbb{F}\left( A;\mathbb{C}\right) $ such that $x_{0}=x$ and $%
x_{1}=x^{\prime }$; see \cite[Definition 4.1.2]{jensen_elements_1991}. As
discussed in Section \ref{Subsection:multiplier} one can regard a strong-*
continuous path $\left( \phi _{t}\right) _{t\in \left[ 0,1\right] }$ of
*-homomorphisms $A\rightarrow B\left( H\right) $ as an element of the unit
ball of $C_{\beta }\left( [0,1],B\left( H\right) \right) =M\left( C\left(
[0,1],K\left( H\right) \right) \right) $. This allows one to regard the set
of such paths as a Polish space endowed with the strict topology on $\mathrm{%
\mathrm{Ball}}\left( C_{\beta }\left( [0,1],B\left( H\right) \right) \right) 
$, such that norm-continuous paths form a Borel subset by Corollary \ref%
{Corollary:norm-continuous-paths}. It can be deduced from these observations
that the relation $\sim $ of homotopy in $\mathbb{F}\left( A;\mathbb{C}%
\right) $ is an \emph{analytic} equivalence relation.

One lets $\mathrm{KK}_{h}\left( A;\mathbb{C}\right) $ be the semidefinable
set obtained as a quotient of the Polish space $\mathbb{F}\left( A;\mathbb{C}%
\right) $ by the analytic equivalence relation $\sim $ \cite[Definition 4.1.3%
]{jensen_elements_1991}. One has that $\mathrm{KK}_{h}\left( A;\mathbb{C}%
\right) $ is a semidefinable group, where the group operation is induced by
the Borel function $\mathbb{F}\left( A;\mathbb{C}\right) \times \mathbb{F}%
\left( A;\mathbb{C}\right) \rightarrow \mathbb{F}\left( A;\mathbb{C}\right) $%
, $\left( \left( \phi _{+},\phi _{-}\right) ,\left( \psi _{+},\psi
_{-}\right) \right) \mapsto \left( \mathrm{Ad}\left( V\right) \circ \left(
\phi _{+}\oplus \psi _{+}\right) ,\mathrm{Ad}\left( V\right) \circ \left(
\phi _{-}\oplus \psi _{-}\right) \right) $, where $V:H\oplus H\rightarrow H$
is a fixed surjective linear isometry, and the function mapping each element
to its additive inverse is induced by the Borel function $\mathrm{KK}%
_{h}\left( A;\mathbb{C}\right) \rightarrow \mathrm{KK}_{h}\left( A;\mathbb{C}%
\right) $, $\left( \phi _{+},\phi _{-}\right) \mapsto \left( \phi _{-},\phi
_{+}\right) $; see \cite[Proposition 4.1.5]{jensen_elements_1991}. The
trivial element of $\mathrm{KK}_{h}\left( A;\mathbb{C}\right) $ is the
homotopy class of $\left( 0,0\right) $. The assignment $A\mapsto \mathrm{KK}%
_{h}\left( A;\mathbb{C}\right) $ gives a contravariant functor from
separable C*-algebras to semidefinable groups.

Let $A$ be a separable C*-algebra. We now observe that $\mathrm{KK}%
_{h}\left( A;\mathbb{C}\right) $ is in fact a definable group, definably
isomorphic to $\mathrm{KK}_{0}\left( A;\mathbb{C}\right) $ and hence to $%
\mathrm{K}^{0}\left( A\right) $. There is a natural definable isomorphism $%
\Psi :\mathrm{KK}_{h}\left( A;\mathbb{C}\right) \rightarrow \mathrm{KK}%
_{0}\left( A;\mathbb{C}\right) $ defined as follows; see \cite[Theorem 4.1.8]%
{jensen_elements_1991}. Suppose that $\left( \phi _{+},\phi _{-}\right) \in 
\mathbb{F}\left( A;\mathbb{C}\right) $. Then one can consider the graded
Kasparov module over $A$ defined as $\left( \phi _{+}\oplus \phi
_{-},H\oplus H,F\right) $ where $H\oplus H$ is graded by $I_{H}\oplus \left(
-I_{H}\right) $ and%
\begin{equation*}
F=%
\begin{bmatrix}
0 & I_{H} \\ 
I_{H} & 0%
\end{bmatrix}%
\text{.}
\end{equation*}%
Then one sets $\Psi \left( \lbrack \phi _{+},\phi _{-}]\right) =[\phi
_{+}\oplus \phi _{-},H\oplus H,F]$.

We now observe that the inverse function $\Psi ^{-1}:\mathrm{KK}_{0}\left( A;%
\mathbb{C}\right) \rightarrow \mathrm{KK}_{h}\left( A;\mathbb{C}\right) $ is
also definable, as it follows from the proof of \cite[Theorem 4.1.8]%
{jensen_elements_1991}. Let $\left( \rho _{0},H_{0},F_{0}\right) $ be a
graded Kasparov module, which can be assumed to be involutive and balanced
by normalization and where we can assume $H_{0}$ to be infinite-dimensional;
see \cite[Proposition 8.3.12]{higson_analytic_2000}. Then we have that $%
H_{0}=H\oplus H$ is graded by $I_{H}\oplus \left( -I_{H}\right) $ and $\rho
_{0}^{+}=\rho _{0}^{-}$ are the same representation of $A$ on $H$, and%
\begin{equation*}
F_{0}=%
\begin{bmatrix}
0 & u^{\ast } \\ 
u & 0%
\end{bmatrix}%
\end{equation*}%
for some unitary $u\in B\left( H^{\prime }\right) $. Then by \cite[E 2.1.3]%
{jensen_elements_1991}, the Kasparov modules 
\begin{equation*}
\left( \rho _{0},H_{0},F_{0}\right) \text{\quad and\quad }\left( \left( 
\mathrm{Ad}(u)\circ \rho ^{+}\right) \oplus \rho ^{-},H\oplus H,F\right)
\end{equation*}
represent the same element of $\mathrm{KK}_{0}\left( A;\mathbb{C}\right) $,
where as above%
\begin{equation*}
F=%
\begin{bmatrix}
0 & I_{H} \\ 
I_{H} & 0%
\end{bmatrix}%
\text{.}
\end{equation*}%
One has that $\Psi ^{-1}[\rho _{0},H_{0},F_{0}]=[\left( \left( \mathrm{Ad}%
(u)\circ \rho ^{+}\right) ,\rho ^{-}\right) ]$. As the assignment $\left(
\rho _{0},H_{0},F_{0}\right) \mapsto \left( \left( \mathrm{Ad}(u)\circ \rho
^{+}\right) ,\rho ^{-}\right) $ is given by a Borel function, this shows
that the inverse $\Psi ^{-1}:\mathrm{KK}_{0}\left( A;\mathbb{C}\right)
\rightarrow \mathrm{KK}_{h}\left( A;\mathbb{C}\right) $ is definable. We
thus obtain the following.

\begin{proposition}
\label{Proposition:KKh}Let $A$ be a separable C*-algebra. Then $\mathrm{KK}%
_{h}\left( A;\mathbb{C}\right) $ is a definable group, naturally definably
isomorphic to $\mathrm{K}^{0}\left( A\right) $.
\end{proposition}

\begin{proof}
By the above discussion, the natural definable homomorphism $\mathrm{KK}%
_{h}\left( A;\mathbb{C}\right) \rightarrow \mathrm{KK}_{0}\left( A;\mathbb{C}%
\right) $ is an isomorphism in the category of semidefinable groups.
Therefore, $\mathrm{KK}_{h}\left( A;\mathbb{C}\right) $ is also a definable
group, naturally isomorphic to $\mathrm{KK}_{0}\left( A;\mathbb{C}\right) $.
As in turn $\mathrm{KK}_{0}\left( A;\mathbb{C}\right) $ is naturally
definably isomorphic to $\mathrm{K}^{0}\left( A\right) $, the conclusion
follows.
\end{proof}

\subsection{Cuntz's $\mathrm{K}$-homology}

Suppose that $A,B$ are separable C*-algebras. Let $\mathrm{Hom}\left(
A,B\right) $ be the set of *-homomorphisms $A\rightarrow B$.\ Then $\mathrm{%
Hom}\left( A,B\right) $ is a Polish space when endowed with the topology of
pointwise norm-convergence. Two *-homomorphisms $\phi ,\phi ^{\prime
}:A\rightarrow B$ are homotopic, in which case we write $\phi \sim \phi
^{\prime }$, if they belong to the same path-connected component of $\mathrm{%
Hom}\left( A,B\right) $. Thus, two *-homomorphism $\phi ,\phi ^{\prime
}:A\rightarrow B$ satisfy $\phi \sim \phi ^{\prime }$ if and only if there
exists a continuous path $\left( \lambda _{t}\right) _{t\in \left[ 0,1\right]
}$ in \textrm{Hom}$\left( A,B\right) $ such that $\lambda _{0}=\phi $ and $%
\lambda _{1}=\phi ^{\prime }$; see \cite[Definition 1.3.10]%
{jensen_elements_1991}. Such a path $\left( \lambda _{t}\right) _{t\in \left[
0,1\right] }$ can be thought of as a *-homomorphism $\lambda :A\rightarrow
C([0,1],B)$, where $C([0,1],B)\cong C\left( [0,1]\right) \otimes B$ is the
C*-algebra of continuous functions $[0,1]\rightarrow B$. This shows that the
relation $\sim $ of homotopy in $\mathrm{Hom}\left( A,B\right) $ is an
analytic equivalence relation. We let $[A,B]$ be the semidefinable set of
homotopy classes of *-homomorphisms $A\rightarrow B$.

Recall that a separable C*-algebra $B$ is stable if $B\otimes K\left(
H\right) $ is *-isomorphic to $B$. Suppose in the following that $B$ is
stable. Thus we have that $M\left( B\right) \otimes M\left( K\left( H\right)
\right) \subseteq M\left( B\otimes K\left( H\right) \right) \cong M\left(
B\right) $. This implies that one can choose isometries $w_{0},w_{1}\in
M\left( B\right) $ satisfying $w_{0}w_{0}^{\ast }+w_{1}w_{1}^{\ast }=1$ and $%
w_{i}^{\ast }w_{j}=0$ for $i,j\in \left\{ 0,1\right\} $ distinct. (This is
equivalent to the assertion that $w_{0},w_{1}$ generate inside $M\left(
B\right) $ a copy of the Cuntz algebra $\mathcal{O}_{2}$.) One can then
define a *-isomorphism $\theta :M_{2}\left( B\right) \rightarrow B$, $%
x\mapsto w_{0}xw_{0}^{\ast }+w_{1}xw_{1}^{\ast }$. A *-isomorphism of this
form is called \emph{inner}; see \cite[Definition 1.3.8]%
{jensen_elements_1991}. Any two inner *-isomorphisms $\theta ,\theta
^{\prime }:M_{2}\left( B\right) \rightarrow B$ are unitary equivalent,
namely there exists a unitary $u\in M\left( B\right) $ such that $\mathrm{Ad}%
(u)\circ \theta =\theta ^{\prime }$ \cite[Lemma 1.3.9]{jensen_elements_1991}.

Under the assumption that $B$ is stable, one can endow the semidefinable set 
$\left[ A,B\right] $ with the structure of semidefinable semigroup. The
operation on $\left[ A,B\right] $ is induced by the Borel function $\mathrm{%
Hom}\left( A,B\right) ^{2}\rightarrow \mathrm{Hom}\left( A,B\right) $, $%
\left( \phi ,\psi \right) \mapsto \theta \circ \left( \phi \oplus \psi
\right) $, where $\theta $ is a fixed inner *-isomorphism $M_{2}\left(
B\right) \rightarrow B$; see \cite[Lemma 1.3.12]{jensen_elements_1991}.\ The
trivial element in $\left[ A,B\right] $ is the homotopy class of the zero
*-homomorphism. Furthermore, the argument of \cite[E 4.1.4]%
{jensen_elements_1991} shows that $\left[ A,B\right] $ is isomorphic to $%
\left[ K\left( H\right) \otimes A,B\right] $ in the category of
semidefinable semigroups.

Suppose that $A$ is a separable C*-algebra. Define $QA$ to be the separable
C*-algebra $A\ast A$, where $A\ast A$ denotes the \emph{free product }of $A$
with itself. We let $i,\overline{i}$ be the two canonical inclusions of $A$
inside $QA$. Let $qA$ be the closed two-sided ideal of $QA$ generated by the
elements of the form $i\left( a\right) -\overline{i}\left( a\right) $ for $%
a\in A$; see \cite[Definition 5.1.1]{jensen_elements_1991}.

If $B$ is a separable C*-algebra, and $\phi ,\psi :A\rightarrow B$ are
*-homomorphism, then there is a unique *-homomorphism $Q\left( \phi ,\psi
\right) :QA\rightarrow B$ such that $Q\left( \phi ,\psi \right) \circ i=\phi 
$ and $Q\left( \phi ,\psi \right) \circ \overline{i}=\psi $. The restriction
of $Q\left( \phi ,\psi \right) $ to $qA$ is denoted by $q\left( \phi ,\psi
\right) $. One has that the range of $q\left( \phi ,\psi \right) $ is
contained in an ideal $J$ of $B$ if and only if $\phi \left( a\right) \equiv
\psi \left( a\right) \mathrm{\ \mathrm{mod}}\ J$ for every $a\in A$, in
which case $q\left( \phi ,\psi \right) \in \mathrm{Hom}\left( qA,J\right) $.
One has that $qA$ is the kernel of the map $Q\left( \mathrm{id}_{A},\mathrm{%
id}_{A}\right) :QA\rightarrow A$; see \cite[Lemma 5.1.2]%
{jensen_elements_1991}.

If $B$ is a separable C*-algebra, then the semidefinable semigroup $\left[
qA,K\left( H\right) \otimes B\right] $ is in fact a semidefinable group,
where the function that maps each element to its additive inverse is induced
by the Borel function 
\begin{equation*}
\mathrm{Hom}\left( qA,K\left( H\right) \otimes B\right) \rightarrow \mathrm{%
Hom}\left( qA,K\left( H\right) \otimes B\right) ,\phi \mapsto -\phi \text{;}
\end{equation*}%
\cite[Theorem 5.1.6]{jensen_elements_1991}. The proof of \cite[Theorem 5.1.12%
]{jensen_elements_1991} shows that $\left[ qA,K\left( H\right) \otimes B%
\right] $ is isomorphic in the category of semidefinable groups to $\left[
qA,K\left( H\right) \otimes qB\right] $. In turn, $\left[ qA,K\left(
H\right) \otimes qB\right] $ is isomorphic to $\left[ K\left( H\right)
\otimes qA,K\left( H\right) \otimes qB\right] $ in the category of
semidefinable groups by \cite[E 4.1.4]{jensen_elements_1991}.

Observe that, for a fixed C*-algebra $B$, the assignment $A\mapsto \left[
qA,K\left( H\right) \otimes B\right] $ is a contravariant functor from
C*-algebras to semidefinable groups. Suppose that $A$ is a separable
C*-algebra. Then there is a natural definable homomorphism $S:\mathrm{KK}%
_{h}\left( A;\mathbb{C}\right) \rightarrow \left[ qA,K\left( H\right) \right]
$ defined by setting $S\left( [\phi _{+},\phi _{-}]\right) =\left[ \psi %
\right] $ where $\psi =q\left( \phi _{+},\phi _{-}\right) \in \mathrm{Hom}%
\left( qA,K\left( H\right) \right) $. One has that in fact $S$ is a group
isomorphism \cite[Theorem 5.2.4]{jensen_elements_1991}. Therefore, we obtain
from Proposition \ref{Proposition:KKh} and Corollary \ref%
{Corollary:Kechris--MacDonald} the following.

\begin{proposition}
Suppose that $A$ is a separable C*-algebra.\ Then $\left[ qA,K\left(
H\right) \right] $ is a definable group, naturally definably isomorphic to $%
\mathrm{K}^{0}\left( A\right) $.
\end{proposition}

This description of $\mathrm{K}$-homology is called Cuntz's picture, as it
was introduced by Cuntz in \cite{cuntz_new_1987}; see also \cite[Section 17.6%
]{blackadar_theory_1986} and \cite%
{cuntz_generalized_1983,cuntz_theory_1984,zekri_new_1989}. Using the Cuntz
picture, one can easily define the more general Kasparov \textrm{KK}-groups $%
\mathrm{KK}_{0}\left( A,B\right) $ for separable C*-algebras $A,B$, by
setting%
\begin{equation*}
\mathrm{KK}_{0}\left( A;B\right) =\left[ qA,K\left( H\right) \otimes B\right]
\cong \left[ K\left( H\right) \otimes qA,K\left( H\right) \otimes qB\right] 
\text{.}
\end{equation*}%
These are semidefinable groups, although we do not know whether they are
definable groups when $B$ is an arbitrary separable C*-algebra. In
particular, one has that $\mathrm{KK}_{0}\left( A;\mathbb{C}\right) \cong 
\mathrm{K}^{0}\left( A\right) $ and $\mathrm{KK}_{0}\left( A,C_{0}\left( 
\mathbb{R}\right) \right) \cong \mathrm{KK}_{0}\left( SA;\mathbb{C}\right)
\cong \mathrm{K}^{1}\left( A\right) $. The $\mathrm{K}$-theory groups are
also recovered as particular instances of the $\mathrm{KK}$-groups, as $%
\mathrm{KK}_{0}\left( \mathbb{C};A\right) \cong \mathrm{K}^{0}\left(
A\right) $ and $\mathrm{KK}_{0}\left( C_{0}\left( \mathbb{R}\right)
;A\right) \cong \mathrm{KK}_{0}\left( \mathbb{C};SA\right) \cong \mathrm{K}%
^{1}\left( A\right) $.

Given separable C*-algebras $A,B,C$, composition of *-homomorphisms $K\left(
H\right) \otimes qA\rightarrow K\left( H\right) \otimes qB$ and $K\left(
H\right) \otimes qB\rightarrow K\left( H\right) \otimes qC$ induces a
definable bilinear pairing (Kasparov product)%
\begin{equation*}
\mathrm{KK}_{0}\left( A;B\right) \times \mathrm{KK}_{0}\left( B;C\right)
\rightarrow \mathrm{KK}_{0}\left( A;C\right) \text{.}
\end{equation*}%
In particular, $\mathrm{KK}_{0}\left( A;A\right) $ is a (semidefinable)
ring, with identity element $\boldsymbol{1}_{A}$ corresponding to the
identity map of $K\left( H\right) \otimes qA$. The $\mathrm{KK}$-category of
C*-algebras is the category enriched over the category of (semidefinable)
abelian groups that has separable C*-algebras as objects and $\mathrm{KK}$%
-groups as hom-sets. Two separable C*-algebras are $\mathrm{KK}$-equivalent
if they are isomorphic in the $\mathrm{KK}$-category of C*-algebras.

By way of the Kasparov product and the natural isomorphisms $\mathrm{K}%
^{0}\left( A\right) \cong \mathrm{KK}_{0}\left( A;\mathbb{C}\right) $ and $%
\mathrm{K}^{1}\left( A\right) \cong \mathrm{KK}_{0}\left( SA;\mathbb{C}%
\right) $, one can regard $\mathrm{K}$-homology as a contravariant functor
from the \textrm{KK}-category of separable C*-algebras to the category of
definable groups, and $\mathrm{K}$-theory as a covariant functor from the 
\textrm{KK}-category of separable C*-algebras to the category of countable
groups. In particular, $\mathrm{KK}$-equivalent C*-algebras have definably
isomorphic $\mathrm{K}$-homology groups, and isomorphic $\mathrm{K}$-theory
groups.

\section{Properties of definable $\mathrm{K}$-homology\label%
{Section:properties}}

In this section we consider several properties of definable $\mathrm{K}$%
-homology, which can be seen as definable versions of the properties of an
abstract cohomology theory in the sense of \cite%
{schochet_topologicalIII_1984} that is C*-stable in the sense of \cite%
{cuntz_new_1987}.

\subsection{Products}

Suppose that $\left( X_{i}\right) _{i\in \omega }$ is a sequence of
semidefinable sets $X_{i}=\hat{X}_{i}/E_{i}$. Then the product $\prod_{i\in
\omega }X_{i}$ is the semidefinable set $\hat{X}/E$ where $\hat{X}%
=\prod_{i\in \omega }\hat{X}_{i}$ and $E$ is the (analytic) equivalence
relation on $\hat{X}$ defined by setting $\left( x_{i}\right) E\left(
y_{i}\right) $ if and only if $\forall i\in \omega $, $x_{i}E_{i}y_{i}$. If,
for every $i\in \omega $, $G_{i}$ is a semidefinable group, then $%
\prod_{i\in \omega }G_{i}$ is a semidefinable group when endowed with the
product group operation.

Suppose that $\left( A_{i}\right) _{i\in \omega }$ is a sequence of
separable C*-algebra. Define the \emph{direct sum} $\bigoplus_{i\in \omega
}A_{i}$ to be the C*-algebra $A$ consisting of the sequences $\left(
a_{i}\right) _{i\in \omega }\in \prod_{i\in \omega }A_{i}$ such that $%
\left\Vert a_{i}\right\Vert \rightarrow 0$; see \cite[Definition 7.4.1]%
{higson_analytic_2000}. If $B$ is a separable C*-algebra, then the canonical
maps $A_{i}\rightarrow A$ induce an isomorphism of Polish spaces%
\begin{equation*}
\mathrm{Hom}\left( A,B\right) \rightarrow \prod_{i\in \omega }\mathrm{Hom}%
\left( A_{i},B\right) \text{.}
\end{equation*}%
When $A_{i}$ is commutative with spectrum $X_{i}$, then $A$ is commutative
with spectrum the disjoint union of $X_{i}$ for $i\in \omega $. The
following result can be seen as a noncommutative version of the Cluster
Axiom for a homology theory for pointed compact spaces from \cite%
{milnor_steenrod_1995}.

\begin{proposition}
\label{Proposition:product}Suppose that $\left( A_{i}\right) _{i\in \omega }$
is a sequence of separable C*-algebras, and set $A=\bigoplus_{i\in \omega
}A_{i}$. Fix $p\in \left\{ 0,1\right\} $. Then $\prod_{i\in \omega }\mathrm{K%
}^{p}\left( A_{i}\right) $ is a definable group. Furthermore the canonical
maps $A_{i}\rightarrow A$ for $i\in \omega $ induce a natural definable
isomorphism%
\begin{equation*}
\mathrm{K}^{p}\left( A\right) \rightarrow \prod_{i\in \omega }\mathrm{K}%
^{p}\left( A_{i}\right) \text{.}
\end{equation*}
\end{proposition}

\begin{proof}
Since $\mathrm{K}^{p}\left( A\right) $ is a definable group, it suffices to
prove the second assertion. After replacing $A$ with its suspension, it
suffices to consider the case when $p=0$. In this case, we can replace $%
\mathrm{K}^{0}$ with $\mathrm{KK}_{h}$ by Proposition \ref{Proposition:KKh}.
Recall that we let $\mathbb{F}\left( A;\mathbb{C}\right) $ be the space of $%
\mathrm{KK}_{h}$-cycles for $A$. The canonical maps $A_{i}\rightarrow A$
induce an isomorphism of Polish spaces%
\begin{equation*}
\mathbb{F}\left( A;\mathbb{C}\right) \rightarrow \prod_{i\in \omega }\mathbb{%
F}\left( A_{i};\mathbb{C}\right) \text{.}
\end{equation*}%
In turn, this induces a definable isomorphism of the spaces of homotopy
classes.%
\begin{equation*}
\mathrm{KK}_{h}\left( A\right) \rightarrow \prod_{i\in \omega }\mathrm{KK}%
_{h}\left( A_{i}\right) \text{.}
\end{equation*}%
This concludes the proof.
\end{proof}

\subsection{Homotopy-invariance}

Suppose that $A,B$ are separable C*-algebra. Recall that $\mathrm{Hom}\left(
A,B\right) $ is a Polish space when endowed with the topology of pairwise
convergence. Thus, $\alpha ,\beta \in \mathrm{Hom}\left( A,B\right) $ if
there exists a path in $\mathrm{Hom}\left( A,B\right) $ from $\alpha $ to $%
\beta $; see \cite[Definition 4.4.1]{higson_analytic_2000}. This can be
thought of as an element $\gamma $ of $\mathrm{Hom}\left( A,IB\right) $ such
that $\mathrm{ev}_{0}\circ \gamma =\alpha $ and $\mathrm{ev}_{1}\circ \gamma
=\beta $ where $IB=C\left( [0,1],B\right) $ and $\mathrm{ev}%
_{t}:IB\rightarrow B$, $f\mapsto f\left( t\right) $ for $t\in \left[ 0,1%
\right] $. We let $\left[ A,B\right] $ be the semidefinable set of homotopy
classes of *-homomorphisms $A\rightarrow B$. The homotopy category of
C*-algebras has separable C*-algebras as objects and homotopy classes of
*-homomorphisms as morphisms. Two C*-algebras are homotopy equivalent if
they are isomorphic in the homotopy category of C*-algebras \cite[Definition
4.4.7]{higson_analytic_2000}.

\begin{proposition}
For $p\in \left\{ 0,1\right\} $, the $\mathrm{K}$-homology functor $\mathrm{K%
}^{p}\left( -\right) $ from separable C*-algebras is homotopy-invariant.
\end{proposition}

\begin{proof}
As in the case of the proof of Proposition \ref{Proposition:product}, it
suffices to show that the functor $\mathrm{KK}_{h}\left( -\right) $ is
homotopy invariant, which is an immediate consequence of the definition.
\end{proof}

Suppose that $B$ is a separable C*-algebra. Recall that the suspension $SB$
of $B$ can be seen as the C*-subalgebra of $IB$ consisting of $f\in IB$ such
that $f\left( 0\right) =f\left( 1\right) =0$. Then $\left[ A,SB\right] $ is
a semidefinable abelian group, where the group operation is induced by the
Borel function $\left( f,g\right) \mapsto m\left( f,g\right) $ where 
\begin{equation*}
m\left( f,g\right) \left( t\right) =\left\{ 
\begin{array}{ll}
f\left( 2t\right) & t\in \left[ 0,1/2\right] \text{,} \\ 
g\left( 2t-1\right) & t\in \left[ 1/2,1\right] \text{.}%
\end{array}%
\right.
\end{equation*}%
The function that assigns each element to its additive inverse is induced by
the Borel function $f\mapsto \widehat{f}$ where 
\begin{equation*}
\widehat{f}\left( t\right) =f\left( 1-t\right) \text{.}
\end{equation*}%
The trivial element of $\left[ A,SB\right] $ is the homotopy class of $0$.
For $p\in \left\{ 0,1\right\} $, there map $\left[ A,SB\right] \rightarrow 
\mathrm{K}^{p}\left( SB,A\right) $ is a group homomorphism \cite[Proposition
6.3]{schochet_topologicalIII_1984}.

A separable C*-algebra $A$ is \emph{contractible} if it is homotopy
equivalent to the zero C*-algebra; see \cite[Definition 4.4.4]%
{higson_analytic_2000}. By homotopy invariance, $\mathrm{K}^{p}\left(
A\right) =\left\{ 0\right\} $ whenever $A$ is contractible and $p\in \left\{
0,1\right\} $. In particular, if $\left( A,J\right) $ is a separable
semi-split C*-pair such that $A$ is contractible, the boundary homomorphism $%
\mathrm{K}^{p}\left( J\right) \rightarrow \mathrm{K}^{p}\left( A/J\right) $
is a definable isomorphism.

If $A$ is a separable, nuclear C*-algebra, then its cone $CA$ is the
C*-subalgebra of $IA$ consisting of $f\in IA$ such that $f\left( 1\right) =0$%
. This is a contractible C*-algebra \cite[Example 4.4.6]%
{higson_analytic_2000}, and 
\begin{equation*}
0\rightarrow SA\rightarrow CA\rightarrow A\rightarrow 0
\end{equation*}%
is an exact sequence, where $CA\rightarrow A$ is the map $\mathrm{ev}_{0}$.
The boundary homomorphism $\sigma ^{A}:\mathrm{K}^{p}\left( SA\right)
\rightarrow \mathrm{K}^{p+1}\left( A\right) $ is thus an isomorphism; see 
\cite[Theorem 6.5]{schochet_topologicalIII_1984}.

\subsection{Mapping cones}

Suppose that $A,B$ are separable, nuclear C*-algebras, and $f:A\rightarrow B$
is a *-homomorphism. The mapping cone 
\begin{equation*}
Cf=\left\{ \left( x,y\right) \in CB\oplus A:f(y)=\mathrm{ev}_{0}(x)\right\}
\end{equation*}%
of $f$ is obtained as the pullback of $\mathrm{ev}_{0}:CB\rightarrow B$ and $%
f:A\rightarrow B$. As such, it is endowed with canonical *-homomorphisms $%
Cf\rightarrow CB$ and $Cf\rightarrow A$; see \cite[Definition 2.1]%
{schochet_topologicalIII_1984}. We have a natural exact sequence%
\begin{equation*}
0\rightarrow SB\rightarrow Cf\rightarrow A\rightarrow 0
\end{equation*}%
where $SB\rightarrow Cf$, $x\mapsto \left( x,0\right) $. This induces a
boundary homomorphism $\mathrm{K}^{p}\left( SB\right) \rightarrow \mathrm{K}%
^{p+1}\left( A\right) $.

Considering the commutative diagram%
\begin{equation*}
\begin{array}{ccccccccc}
0 & \rightarrow & SB & \rightarrow & Cf & \rightarrow & A & \rightarrow & 0
\\ 
&  & \downarrow &  & \downarrow &  & \downarrow f &  &  \\ 
0 & \rightarrow & SB & \rightarrow & CB & \rightarrow & B & \rightarrow & 0%
\end{array}%
\end{equation*}%
where $SB\rightarrow SB$ is the identity map, and $Cf\rightarrow CB$, $%
\left( x,y\right) \mapsto x$, we obtain by naturality of the six-term exact
sequence in $\mathrm{K}$-homology that the boundary morphism $\mathrm{K}%
^{p}\left( SB\right) \rightarrow \mathrm{K}^{p+1}\left( A\right) $ is equal
to the composition 
\begin{equation*}
\mathrm{K}^{p+1}\left( f\right) \circ \sigma ^{B}:\mathrm{K}^{p}\left(
SB\right) \rightarrow \mathrm{K}^{p+1}\left( B\right) \rightarrow \mathrm{K}%
^{p+1}\left( A\right) \text{.}
\end{equation*}%
The same argument together with the Five Lemma \cite[Proposition 2.72]%
{rotman_introduction_2009} shows that $\mathrm{K}^{p}\left( f\right) $ is an
isomorphism for every $p\in \left\{ 0,1\right\} $ if and only if $\mathrm{K}%
^{p}\left( Cf\right) =\left\{ 0\right\} $ for every $p\in \left\{
0,1\right\} $; see \cite[Theorem 6.5]{schochet_topologicalIII_1984}.

If $f:A\rightarrow B$ is a surjective *-homomorphism with kernel $J$, then
considering the exact sequence%
\begin{equation*}
0\rightarrow J\rightarrow Cf\rightarrow CB\rightarrow 0
\end{equation*}%
one sees that the map $J\rightarrow Cf$ induces an isomorphism $\mathrm{K}%
^{p}\left( J\right) \rightarrow \mathrm{K}^{p}\left( Cf\right) $. Similarly,
if $J$ is an ideal of $A$ and $f:J\rightarrow A$ is the inclusion, then
considering the exact sequence%
\begin{equation*}
0\rightarrow CJ\rightarrow Cf\rightarrow S\left( A/J\right) \rightarrow 0
\end{equation*}%
shows that the map $Cf\rightarrow S\left( A/J\right) $ induces an
isomorphism $\mathrm{K}^{p}\left( Cf\right) \rightarrow \mathrm{K}^{p}\left(
S\left( A/J\right) \right) $; see \cite[Proposition 6.6]%
{schochet_topologicalIII_1984}.

\subsection{Long exact sequence of a triple}

Consider a triple $J\subseteq H\subseteq A$ where $A$ is a separable,
nuclear C*-algebra and $J$ and $H$ are closed two-sided ideals of $A$.\ Then
we have a commutative diagram%
\begin{equation*}
\begin{array}{ccccccccc}
0 & \rightarrow & H & \rightarrow & A & \rightarrow & A/H & \rightarrow & 0
\\ 
&  & \downarrow &  & \downarrow &  & \downarrow &  &  \\ 
0 & \rightarrow & H/J & \rightarrow & A/J & \rightarrow & A/H & \rightarrow
& 0\text{.}%
\end{array}%
\end{equation*}
By naturality of the six-term exact sequence in $\mathrm{K}$-homology, we
have that the boundary map%
\begin{equation*}
\mathrm{K}^{p}\left( H/J\right) \rightarrow \mathrm{K}^{1-p}\left( A/H\right)
\end{equation*}%
is equal to the composition of the map $\mathrm{K}^{p}\left( H/J\right)
\rightarrow \mathrm{K}^{p}\left( H\right) $ induced by the quotient map with
the boundary map $\mathrm{K}^{p}\left( H\right) \rightarrow \mathrm{K}%
^{1-p}\left( A/H\right) $; see \cite[Theorem 6.10]%
{schochet_topologicalIII_1984}.

\subsection{Mayer--Vietoris sequence}

Consider separable, nuclear C*-algebras $P,A_{1},A_{2},B$, and
*-homomorphisms $f_{i}:A_{i}\rightarrow B$ and $g_{i}:P\rightarrow A_{i}$
for $i\in \left\{ 1,2\right\} $. Suppose that $f_{1},f_{2}$ are surjective,
and
\begin{center}
\begin{tikzcd}
P \arrow[r, "g_1"]  \arrow[d, "g_2"] & A_{1} \arrow[d, "f_1"] \\  
A_{2}  \arrow[r, "f_2"] &  B%
\end{tikzcd}
\end{center}
is a pushout diagram. Then there is a six-term exact sequence of definable
group homomorphisms
\begin{center}
\begin{tikzcd}
\mathrm{K}^{0}\left( B\right) \arrow[r] & \mathrm{K}^{0}\left(A_{1}\right) \oplus \mathrm{K}^{0}\left( A_{2}\right)  \arrow[r] & \mathrm{K}^{0}\left( P\right) \arrow[d,"\partial ^0"] \\ 
\mathrm{K}^{1}\left( P\right) \arrow[u, "\partial ^1"] & \mathrm{K}^{1}\left(A_{1}\right) \oplus \mathrm{K}^{1}\left( A_{2}\right) \arrow[l] & \mathrm{K}^{1}\left( B\right) \arrow[l]
\end{tikzcd}
\end{center}
see \cite[Theorem 6.11]{schochet_topologicalIII_1984}. The definable group
homomorphism $\mathrm{K}^{p}\left( B\right) \rightarrow \mathrm{K}^{p}\left(
A_{1}\right) \oplus \mathrm{K}^{p}\left( A_{2}\right) $ is $\left( -\mathrm{K%
}^{p}\left( f_{1}\right) ,\mathrm{K}^{p}\left( f_{2}\right) \right) $, the
definable group homomorphism $\mathrm{K}^{p}\left( A_{1}\right) \oplus 
\mathrm{K}^{p}\left( A_{2}\right) \rightarrow \mathrm{K}^{p}\left( P\right) $
is $\mathrm{K}^{p}\left( g_{1}\right) +\mathrm{K}^{p}\left( g_{2}\right) $.
Furthermore, the definable group homomorphism $\partial ^{p}:\mathrm{K}%
^{p}\left( P\right) \rightarrow \mathrm{K}^{1-p}\left( B\right) $ is defined
as follows. Let $g:P\rightarrow A_{1}\oplus A_{2}$ be defined by $x\mapsto
\left( g_{1}(x),g_{2}(x)\right) $. Consider the corresponding mapping cone $%
Cg$. We can regard $Cg$ as the set of triples $\left( \xi _{1},\xi
_{2},x\right) \in CA_{1}\oplus CA_{2}\oplus P$ such that $\left( \xi
_{1}\left( 0\right) ,\xi _{2}\left( 0\right) \right) =g(x)$. We have a
*-homomorphism $\psi :Cg\rightarrow SB$ defined by setting%
\begin{equation*}
\psi \left( \xi _{1},\xi _{2},x\right) \left( t\right) =\left\{ 
\begin{array}{ll}
f_{1}\left( \xi _{1}\left( 1-2t\right) \right) \text{,} & t\in \left[ 0,1/2%
\right] \text{;} \\ 
f_{2}\left( \xi _{2}\left( 2t-1\right) \right) \text{,} & t\in \left[ 1/2,1%
\right] \text{.}%
\end{array}%
\right.
\end{equation*}%
Then we have a natural short exact sequence%
\begin{equation*}
0\rightarrow CJ_{1}\oplus CJ_{2}\rightarrow Cg\overset{\psi }{\rightarrow }%
SB\rightarrow 0
\end{equation*}%
where $J_{i}=\mathrm{\mathrm{Ker}}\left( f_{i}\right) $ for $i\in \left\{
1,2\right\} $; see \cite[Proposition 4.5]{schochet_topologicalIII_1984}.\
Thus, $\psi $ induces a definable isomorphism $\mathrm{K}^{p}\left(
SB\right) \rightarrow \mathrm{K}^{p}\left( Cg\right) $. The definable group
homomorphism $\partial ^{p}:\mathrm{K}^{p}\left( P\right) \rightarrow 
\mathrm{K}^{1-p}\left( B\right) $ is defined as the composition of definable
homomorphisms%
\begin{equation*}
\mathrm{K}^{p}\left( P\right) \rightarrow \mathrm{K}^{p}\left( Cg\right)
\rightarrow \mathrm{K}^{p}\left( SB\right) \rightarrow \mathrm{K}%
^{1-p}\left( B\right)
\end{equation*}%
where the map $\mathrm{K}^{p}\left( P\right) \rightarrow \mathrm{K}%
^{p}\left( Cg\right) $ is associated with the canonical *-homomorphism $%
Cg\rightarrow P$ as in the definition of mapping cone, the map $\mathrm{K}%
^{p}\left( Cg\right) \rightarrow \mathrm{K}^{p}\left( SB\right) $ is the
inverse of the definable isomorphism $\mathrm{K}^{p}\left( SB\right)
\rightarrow \mathrm{K}^{p}\left( Cg\right) $ induced by $\psi $, and the map 
$\sigma ^{B}:\mathrm{K}^{p}\left( SB\right) \rightarrow \mathrm{K}%
^{1-p}\left( B\right) $ is the suspension isomorphism; see the proof of \cite%
[Theorem 6.11]{schochet_topologicalIII_1984}.

\subsection{The Milnor sequence of an inductive sequence}

A tower of countable abelian groups is a sequence $\boldsymbol{A}=\left(
A^{\left( n\right) },p^{\left( n,n+1\right) }\right) $ of countable abelian
groups and group homomorphism $p^{\left( n,n+1\right) }:A^{\left( n+1\right)
}\rightarrow A^{\left( n\right) }$. Given such a tower we let $p^{\left(
n,n\right) }$ be the identity map of $A^{\left( n\right) }$ and, for $n<m$, $%
p^{\left( n,m\right) }$ be the composition $p^{\left( n,n+1\right) }\circ
\cdots \circ p^{\left( m-1,m\right) }$. Towers of countable groups form a
category. A morphism from $\boldsymbol{A}=\left( A^{\left( n\right)
},p^{\left( n,n+1\right) }\right) $ to $\boldsymbol{B}=\left( B^{\left(
k\right) },p^{\left( k,k+1\right) }\right) $ is represented by a sequence $%
\left( n_{k},f^{\left( k\right) }\right) _{k\in \omega }$ where $\left(
n_{k}\right) $ is an increasing sequence in $\omega $ and $f^{\left(
k\right) }:A^{\left( n_{k}\right) }\rightarrow B^{\left( k\right) }$ is a
group homomorphism. Two such sequences $\left( n_{k},f^{\left( k\right)
}\right) _{k\in \omega }$ and $\left( n_{k}^{\prime },f^{\prime \left(
k\right) }\right) _{k\in \omega }$ represent the same morphism if there
exists an increasing sequence $\left( n_{k}^{\prime \prime }\right) _{k\in
\omega }$ in $\omega $ such that $n_{k}^{\prime \prime }\geq \max \left\{
n_{k},n_{k}^{\prime }\right\} $ and $f^{\left( k\right) }p^{\left(
n_{k},n_{k}^{\prime \prime }\right) }=f^{\left( k\right) }p^{\left(
n_{k}^{\prime },n_{k}^{\prime \prime }\right) }$ for every $k\in \omega $.
The identity morphism and composition of morphisms are defined in the
obvious way.

Given a tower $\boldsymbol{A}$ of countable abelian groups, one lets \textrm{%
lim}$^{1}\boldsymbol{A}$ be the definable group, which is in fact a group
with Polish cover (see Remark \ref{Remark:Polish-cover}), defined as
follows. Consider \textrm{Z}$^{1}\left( \boldsymbol{A}\right) $ to be the
product group%
\begin{equation*}
\prod_{n\in \omega }A^{\left( n\right) }
\end{equation*}%
endowed with the product topology, where each $A^{\left( n\right) }$ is
endowed with the discrete topology. Define $\mathrm{B}^{1}\left( \boldsymbol{%
A}\right) $ to be the Polishable Borel subgroup of $\mathrm{Z}^{1}\left( 
\boldsymbol{A}\right) $ obtained as an image of the continuous group
homomorphism%
\begin{equation*}
\Phi _{\boldsymbol{A}}:\prod_{n\in \omega }A^{\left( n\right) }\rightarrow 
\mathrm{Z}^{1}\left( \boldsymbol{A}\right) \text{, }\left( x_{n}\right)
\mapsto \left( x_{n}-p^{\left( n,n+1\right) }\left( x_{n+1}\right) \right)
_{n\in \omega }\text{.}
\end{equation*}%
Then $\mathrm{lim}^{1}\boldsymbol{A}$ is the corresponding definable group $%
\mathrm{Z}^{1}\left( \boldsymbol{A}\right) /\mathrm{B}^{1}\left( \boldsymbol{%
A}\right) $. The assignment $\boldsymbol{A}\mapsto $\textrm{lim}$^{1}%
\boldsymbol{A}$ is easily seen to be a functor from the category of towers
of countable abelian groups to the category of definable groups; see also 
\cite[Section 5]{bergfalk_ulam_2020}.

Given a tower $\boldsymbol{A}$ of countable abelian groups, we can also
consider the (inverse) limit $\mathrm{lim}\boldsymbol{A}$. This is the
Polish abelian group obtained as the \emph{kernel }of the continuous group
homomorphism $\Phi _{\boldsymbol{A}}$ described above. The assignment $%
\boldsymbol{A}\mapsto \mathrm{lim}\boldsymbol{A}$ is a functor from the
category of towers of countable abelian groups to the category of Polish
abelian groups.

Suppose that $\left( A_{n},\varphi _{n}\right) _{n\in \omega }$ is an
inductive sequence of separable, nuclear C*-algebras, and let $A=\mathrm{%
colim}_{n}\left( A_{n},\varphi _{n}\right) $ be its inductive limit. If $%
\mathrm{K}^{p}\left( A_{n}\right) $ is countable for every $n\in \omega $,
then $\left( \mathrm{K}^{p}\left( A_{n}\right) \right) _{n\in \omega }$ is a
tower of countable abelian groups, where $p^{\left( n,n+1\right) }:\mathrm{K}%
^{p}\left( A_{n+1}\right) \rightarrow \mathrm{K}^{p}\left( A_{n}\right) $ is
induced by $\varphi _{n}:A_{n}\rightarrow A_{n+1}$. The assignment $\left(
A_{n},\varphi _{n}\right) _{n\in \omega }\mapsto \left( \mathrm{K}^{p}\left(
A_{n}\right) \right) _{n\in \omega }$ defines a functor from the category of
inductive sequences of separable C*-algebras with countable $\mathrm{K}$%
-homology groups to the category of towers of countable abelian groups. The
Milnor sequence for $\left( A_{n},\varphi _{n}\right) _{n\in \omega }$
describes $\mathrm{K}^{p}\left( A\right) $ as an extension of groups defined
in terms of $\left( \mathrm{K}^{p}\left( A_{n}\right) \right) _{n\in \omega
} $; see \cite[Theorem 7.1]{schochet_topologicalIII_1984}. The proof is
inspired by Milnor's argument for the corresponding result about Steenrod
homology \cite{milnor_steenrod_1995}; see also \cite{milnor_axiomatic_1962}.
We let $\mathbb{N}$ denote the set of natural numbers not including zero,
and $\omega =\mathbb{N}\cup \left\{ 0\right\} $.

\begin{proposition}
\label{Proposition:Milnor}Suppose that $\left( A_{n},\varphi _{n}\right)
_{n\in \mathbb{N}}$ is an inductive sequence of separable, nuclear
C*-algebras with countable $\mathrm{K}$-homology groups, and $A$ is the
inductive limit of $\left( A_{n},\varphi _{n}\right) _{n\in \omega }$. Then
for $p\in \left\{ 0,1\right\} $ there is a natural short exact sequence of
definable group homomorphisms%
\begin{equation*}
0\rightarrow \mathrm{lim}_{n}^{1}\mathrm{K}^{1-p}\left( A_{n}\right)
\rightarrow \mathrm{K}^{p}\left( A\right) \rightarrow \mathrm{lim}_{n}%
\mathrm{K}^{p}\left( A_{n}\right) \rightarrow 0
\end{equation*}%
where the homomorphism $\mathrm{K}^{p}\left( A\right) \rightarrow \mathrm{lim%
}_{n}\mathrm{K}^{p}\left( A_{n}\right) $ is induced by the canonical maps $%
A_{n}\rightarrow A$.
\end{proposition}

The assertion that the group homomorphisms in Proposition \ref%
{Proposition:Milnor} are definable is a consequence of the proof of \cite[%
Theorem 7.1]{schochet_topologicalIII_1984}. This involves the notion of 
\emph{mapping telescope }$T\left( \boldsymbol{A}\right) $ of an inductive
sequence of $\boldsymbol{A}=\left( A_{n},\varphi _{n}\right) _{n\in \mathbb{N%
}}$ of separable C*-algebras; see \cite[Definition 5.2]%
{schochet_topologicalIII_1984}. Without loss of generality, we can assume
that $A_{0}=\left\{ 0\right\} $. Let $A$ be the corresponding direct limit
and $\varphi _{\left( \infty ,n\right) }:A_{n}\rightarrow A$ be the
canonical maps. For $n<m$ set $\varphi _{\left( m,n\right)
}:A_{n}\rightarrow A_{m}$, $\varphi _{\left( m,n\right) }=\varphi _{n}\circ
\varphi _{n+1}\circ \cdots \circ \varphi _{m-1}$. We also let $\varphi
_{\left( n,n\right) }$ be the identity of $A_{n}$. One fixes an increasing
sequence $\left( t_{n}\right) _{n\in \omega }$ in $[0,1)$ with $t_{0}=0$
converging to $1$. Let $\prod_{n\in \omega }C\left(
[t_{n},t_{n+1}],A_{n+1}\right) $ be the product of $(C\left(
[t_{n},t_{n+1}],A_{n+1}\right) )_{n\in \omega }$ in the category of
C*-algebras. Define then $\tilde{T}\left( \boldsymbol{A}\right) $ to be the
C*-subalgebra of $\prod_{n\in \omega }C\left( [t_{n},t_{n+1}],A_{n+1}\right) 
$ consisting of those elements $\left( \xi _{n}\right) _{n\in \omega }$ such
that, for every $n\in \omega $, $\varphi _{n+1}\left( \xi _{n}\left(
t_{n+1}\right) \right) =\xi _{n+1}\left( t_{n+1}\right) $. An element $%
\left( \xi _{n}\right) _{n\in \omega }$ of $\tilde{T}\left( \boldsymbol{A}%
\right) $ can be seen as a function $\xi :[0,1)\rightarrow \bigcup_{n\in
\omega }A_{n+1}$ where, for $n\in \omega $ and $t\in \lbrack t_{n},t_{n+1})$
one sets $\xi \left( t\right) :=\xi _{n}\left( t\right) $. The function $\xi
^{\infty }:[0,1)\rightarrow A$ defined by $\xi ^{\infty }\left( t\right)
=\varphi _{\left( \infty ,n+1\right) }\left( \xi \left( t\right) \right) $
for $t\in \lbrack t_{n},t_{n+1})$ is then continuous.

The mapping telescope $T\left( \boldsymbol{A}\right) $ consists of the set
of pairs $\left( \xi ,a\right) \in \tilde{T}\left( \boldsymbol{A}\right)
\oplus A$ such that:

\begin{enumerate}
\item for every $\varepsilon >0$ there exists $n_{0}\in \omega $ such that,
for $n\geq m\geq n_{0}$ and for $t\in \left[ t_{n},t_{n+1}\right] $ and $%
s\in \lbrack t_{m},t_{m+1}]$,%
\begin{equation*}
\left\Vert \varphi _{\left( n+1,m+1\right) }\left( \xi _{m}\left( s\right)
\right) -\xi _{n}\left( t\right) \right\Vert <\varepsilon \text{,}
\end{equation*}%
and

\item \textrm{lim}$_{t\rightarrow 1}\xi ^{\infty }\left( t\right) =a$.
\end{enumerate}

Then one has that $T\left( \boldsymbol{A}\right) $ is a contractible
separable C*-algebra; see \cite[Lemma 5.4]{schochet_topologicalIII_1984}.
Define the surjective *-homomorphism $e:T\left( \boldsymbol{A}\right)
\rightarrow A$, $\left( \xi ,a\right) \mapsto a$, and set $J=\mathrm{\mathrm{%
Ker}}\left( e\right) \subseteq T\left( \boldsymbol{A}\right) $. We also have
a map $p:J\rightarrow \bigoplus_{n\in \omega }A_{n+1}$, $\xi \mapsto \left(
\xi _{n}\left( t_{n+1}\right) \right) _{n\in \omega }$; see \cite[Lemma 5.5]%
{schochet_topologicalIII_1984}. As $T\left( \boldsymbol{A}\right) $ is
contractible, the short exact sequence%
\begin{equation*}
0\rightarrow J\rightarrow T\left( \boldsymbol{A}\right) \rightarrow
A\rightarrow 0
\end{equation*}%
gives rise to a definable boundary isomorphism $\partial :\mathrm{K}%
^{1-p}\left( J\right) \rightarrow \mathrm{K}^{p}\left( A\right) $.

For $n\in \omega $ define $M_{n}\subseteq C\left(
[t_{n},t_{n+1}],A_{n+1}\right) \oplus A_{n}$ to be the C*-subalgebra
consisting of $\left( \xi ,a\right) $ such that $\xi \left( t_{n}\right)
=\varphi _{n}\left( a\right) $. The *-homomorphism $M_{n}\rightarrow A_{n}$, 
$\left( \xi ,a\right) \mapsto a$ is a homotopy equivalence with homotopy
inverse $A_{n}\rightarrow M_{n}$, $a\mapsto \left( \xi ,a\right) $ where $%
\xi \left( t\right) =\varphi _{n}\left( a\right) $ for $t\in \left[ 0,1%
\right] $. Then we have that the composition $A_{n}\rightarrow
M_{n}\rightarrow A_{n}$ is the identity, while the composition $%
M_{n}\rightarrow A_{n}\rightarrow M_{n}$ maps $\left( \xi ,a\right) $ to $%
\left( \xi ^{\prime },a\right) $ where $\xi ^{\prime }\left( t\right) =\xi
\left( t_{n}\right) =\varphi _{n}\left( a\right) $ for $t\in \left[ 0,1%
\right] $. This map is homotopic to the identity via the homotopy $\left(
\phi _{t}\right) _{s\in \left[ 0,1\right] }$ defined by $\phi _{s}\left( \xi
,a\right) =\left( \xi _{s},a\right) $ where 
\begin{equation*}
\xi _{s}\left( t_{n}+t\left( t_{n+1}-t_{n}\right) \right) =\xi \left(
t_{n}+st\left( t_{n+1}-t_{n}\right) \right)
\end{equation*}
for $s,t\in \left[ 0,1\right] $.

Define 
\begin{equation*}
D_{1}:=\bigoplus_{n\in \omega }M_{2n+1}
\end{equation*}%
\begin{equation*}
D_{2}:=\bigoplus_{n\in \omega }M_{2n}\text{.}
\end{equation*}%
\begin{equation*}
B:=\bigoplus_{n\in \omega }A_{n}\text{.}
\end{equation*}%
As in \cite[Lemma 5.7]{schochet_topologicalIII_1984}, we have a pullback
diagram
\begin{center}
\begin{tikzcd}
J \arrow[r, "g_1"] \arrow[d, "g_2"] & D_{1} \arrow[d, "f_1"] \\
D_{2} \arrow[r, "f_2"] & B
\end{tikzcd}
\end{center}
where:

\begin{itemize}
\item $g_{1}:J\rightarrow D_{1}$ is defined by 
\begin{equation*}
\left( \xi _{k}\right) _{k\in \omega }\mapsto \left( \eta _{n},b_{n}\right)
_{n\in \omega }
\end{equation*}%
where $\xi _{k}\in C\left( [t_{k},t_{k+1}],A_{k+1}\right) $ for $k\in \omega 
$ and $\left( \eta _{n},b_{n}\right) =\left( \xi _{2n+1},\xi _{2n}\left(
t_{2n+1}\right) \right) \in M_{2n+1}$ for $n\in \omega $;

\item $g_{2}:J\rightarrow D_{2}$ is defined by%
\begin{equation*}
\left( \xi _{k}\right) _{k\in \omega }\mapsto \left( \eta _{n},b_{n}\right)
_{n\in \omega }
\end{equation*}%
where $\xi _{k}\in C\left( [t_{k},t_{k+1}],A_{k+1}\right) $ for $k\in \omega 
$, and $\left( \eta _{n},b_{n}\right) =\left( \xi _{2n},\xi _{2n-1}\left(
t_{2n}\right) \right) \in M_{2n}$;

\item $f_{1}:D_{1}\rightarrow B$ is defined by%
\begin{equation*}
\left( \eta _{n},b_{n}\right) _{n\in \omega }\mapsto \left( c_{n}\right)
_{n\in \omega }
\end{equation*}%
where $\left( \eta _{n},b_{n}\right) \in M_{2n+1}$, $c_{0}=0$, $%
c_{2n+1}=b_{n}$, and $c_{2n+2}=\eta _{n}\left( t_{2n+2}\right) $ for $n\in
\omega $;

\item $f_{2}:D_{2}\rightarrow B$ is defined by%
\begin{equation*}
\left( \eta _{n},b_{n}\right) _{n\in \omega }\mapsto \left( c_{n}\right)
_{n\in \omega }
\end{equation*}%
where $\left( \eta _{n},b_{n}\right) \in M_{2n}$, $c_{2n}=b_{n}$, and $%
c_{2n+1}=\eta _{n}\left( t_{2n+1}\right) $ for $n\in \omega $.
\end{itemize}

We thus have a corresponding Mayer--Vietoris definable six-term exact
sequence
\begin{center}
\begin{tikzcd}
\mathrm{K}^{0}\left( B\right)  \arrow[r]  & \mathrm{K}^{0}\left(D_{1}\right) \oplus \mathrm{K}^{0}\left( D_{2}\right) \arrow[r] & \mathrm{K}^{0}\left( J\right) \arrow[d,"\partial ^0"]  \\
\mathrm{K}^{1}\left( J\right)  \arrow[u,"\partial ^1"]  & \mathrm{K}^{1}\left(D_{1}\right) \oplus \mathrm{K}^{1}\left( D_{2}\right)   \arrow[l] & \mathrm{K}^{1}\left( B\right) \arrow[l]
\end{tikzcd}
\end{center}
associated with it. Combining this with the definable isomorphism $\mathrm{K}%
^{1-p}\left( J\right) \rightarrow \mathrm{K}^{p}\left( A\right) $ as above,
and with the definable isomorphisms%
\begin{equation*}
\mathrm{K}^{p}\left( B\right) \cong \prod_{n\in \omega }\mathrm{K}^{p}\left(
A_{n}\right)
\end{equation*}%
\begin{equation*}
\mathrm{K}^{p}\left( D_{1}\right) \oplus \mathrm{K}^{p}\left( D_{2}\right)
\cong \prod_{n\in \omega }\mathrm{K}^{p}\left( A_{2n}\right) \oplus
\prod_{n\in \omega }\mathrm{K}^{p}\left( A_{2n+1}\right) \cong \prod_{n\in
\omega }\mathrm{K}^{p}\left( A_{n}\right)
\end{equation*}%
obtained from Proposition \ref{Proposition:product} and from the homotopy
equivalences $M_{n}\rightarrow A_{n}$ for $n\in \omega $, one obtains a
definable six-term exact sequence%
\begin{equation*}
\begin{array}{ccccc}
\prod_{n\in \omega }\mathrm{K}^{0}\left( A_{n}\right) & \overset{\Phi ^{0}}{%
\rightarrow } & \prod_{n\in \omega }\mathrm{K}^{0}\left( A_{n}\right) & 
\rightarrow & \mathrm{K}^{1}\left( A\right) \\ 
\uparrow &  &  &  & \downarrow \partial \\ 
\mathrm{K}^{0}\left( A\right) & \leftarrow & \prod_{n\in \omega }\mathrm{K}%
^{1}\left( A_{n}\right) & \overset{\Phi ^{1}}{\leftarrow } & \prod_{n\in
\omega }\mathrm{K}^{1}\left( A_{n}\right)%
\end{array}%
\text{.}
\end{equation*}%
As in the proof of \cite[Theorem 7.1]{schochet_topologicalIII_1984}, the
group homomorphism%
\begin{equation*}
\Phi ^{p}:\prod_{n\in \omega }\mathrm{K}^{p}\left( A_{n}\right) \rightarrow
\prod_{n\in \omega }\mathrm{K}^{p}\left( A_{n}\right)
\end{equation*}%
for $p\in \left\{ 0,1\right\} $ is given by 
\begin{equation*}
\left( x_{n}\right) \mapsto \left( x_{n}-\mathrm{K}^{p}\left( \varphi
_{n}\right) \left( x_{n+1}\right) \right) _{n\in \omega }
\end{equation*}%
whereas the boundary homomorphism $\mathrm{K}^{1}\left( A\right) \rightarrow
\prod_{n\in \omega }\mathrm{K}^{1}\left( A_{n}\right) $ is induced by the
canonical maps $A_{n}\rightarrow A$. Thus, by definition of $\mathrm{lim}$
and $\mathrm{lim}^{1}$ of the tower $\left( \mathrm{K}^{p}\left(
A_{n}\right) \right) _{n\in \omega }$ we have that $\Phi ^{0}$ and $\Phi
^{1} $ yield a definable exact sequence%
\begin{equation*}
0\rightarrow \mathrm{lim}_{n}^{1}\mathrm{K}^{1-p}\left( A_{n}\right)
\rightarrow \mathrm{K}^{p}\left( A\right) \rightarrow \mathrm{lim}_{n}%
\mathrm{K}^{p}\left( A_{n}\right) \rightarrow 0\text{.}
\end{equation*}%
This concludes the proof of Proposition \ref{Proposition:Milnor}.

\subsection{C*-stability}

Suppose that $A$ is a separable C*-algebra, and $H$ is a (not necessarily
infinite-dimensional) separable Hilbert space. If $e\in K\left( H\right) $
is a rank one projection, then we can define a *-homomorphism $%
e_{A}:A\rightarrow K\left( H\right) \otimes A$, $a\mapsto e\otimes a$. In
turn, this induces a definable homomorphism $\mathrm{K}^{p}\left( K\left(
H\right) \otimes A\right) \rightarrow \mathrm{K}^{p}\left( A\right) $. The 
\emph{stability}---or C*-stability \cite{cuntz_new_1987}---property of $%
\mathrm{K}$-homology asserts that such a definable homomorphism $\mathrm{K}%
^{p}\left( K\left( H\right) \otimes A\right) \rightarrow \mathrm{K}%
^{p}\left( A\right) $ is in fact an definable isomorphism; see \cite[Theorem
9.4.1]{higson_analytic_2000}.

\begin{proposition}
\label{Proposition:stability}Suppose that $A$ is a separable C*-algebra, $H$
is a separable Hilbert space, $e\in K\left( H\right) $ is a rank one
projection, and $e_{A}:A\rightarrow K\left( H\right) \otimes A$ is the
*-homomorphism defined by $a\mapsto e\otimes a$. Then the induced map $%
\mathrm{K}^{p}\left( e_{A}\right) :\mathrm{K}^{p}\left( K\left( H\right)
\otimes A\right) \rightarrow \mathrm{K}^{p}\left( A\right) $ is a definable
isomorphism.
\end{proposition}

\begin{proof}
It is easy to see that one can reduce to the case when $H$ is
infinite-dimensional. After replacing $A$ with its stabilization, we can
assume that $p=0$. As $\mathrm{K}^{0}\left( -\right) $ is naturally
isomorphic to $\mathrm{KK}_{h}\left( -;\mathbb{C}\right) $, it suffices to
prove the corresponding statement for $\mathrm{KK}_{h}\left( -;\mathbb{C}%
\right) $. One can then proceed as in \cite[E 4.1.3]{jensen_elements_1991}.
Fix an infinite-dimensional separable Hilbert space $H$, and let $\mathrm{KK}%
_{h}\left( -;\mathbb{C}\right) $ be defined with respect to $H$. Consider
the canonical inclusions $K\left( H\right) \otimes B\left( H\right)
\subseteq B\left( H\right) \otimes B\left( H\right) \subseteq B\left(
H\otimes H\right) $ and the injective *-homomorphism $e_{K\left( H\right)
}:K\left( H\right) \rightarrow K\left( H\right) \otimes K\left( H\right) $, $%
x\mapsto e\otimes x$. Consider also the *-isomorphism $\lambda :K\left(
H\right) \rightarrow K\left( H\right) \otimes K\left( H\right) \cong K\left(
H\otimes H\right) $ defined by setting $\lambda =\mathrm{Ad}\left( V\right)
\circ e_{K\left( H\right) }$ where $V\in B\left( H\otimes H\right) $ is an
isometry with $VV^{\ast }=e\otimes I$. Then $\lambda $ extends to strict
*-isomorphism $\bar{\lambda}=\mathrm{Ad}\left( V\right) \circ \bar{e}%
_{K\left( H\right) }:B\left( H\right) \rightarrow B\left( H\otimes H\right) $%
, where $\bar{e}_{K\left( H\right) }:B\left( H\right) \rightarrow B\left(
H\otimes H\right) $ is the strict extension of $e_{K\left( H\right)
}:K\left( H\right) \rightarrow K\left( H\right) \otimes K\left( H\right)
\cong K\left( H\otimes H\right) $.

One can then consider the definable homomorphism $G:\mathrm{KK}_{h}\left( A;%
\mathbb{C}\right) \rightarrow \mathrm{KK}_{h}\left( K\left( H\right) \otimes
A;\mathbb{C}\right) $ induced by the Borel function 
\begin{equation*}
\mathbb{F}\left( A;\mathbb{C}\right) \rightarrow \mathbb{F}\left( K\left(
H\right) \otimes A;\mathbb{C}\right) \text{, }\left( \phi _{+},\phi
_{-}\right) \mapsto (\bar{\lambda}^{-1}\circ \left( \mathrm{id}_{K\left(
H\right) }\otimes \phi _{+}\right) ,\bar{\lambda}^{-1}\circ \left( \mathrm{id%
}_{K\left( H\right) }\otimes \phi _{-}\right) )\text{.}
\end{equation*}%
Then we have that $\mathrm{KK}_{h}\left( e_{A};\mathbb{C}\right) \circ G:%
\mathrm{KK}_{h}\left( A;\mathbb{C}\right) \rightarrow \mathrm{KK}_{h}\left(
A;\mathbb{C}\right) $ is equal to the identity map. Indeed, $\mathrm{KK}%
_{h}\left( e_{A};\mathbb{C}\right) \circ G$ is induced by the function 
\begin{equation*}
\mathbb{F}\left( A;\mathbb{C}\right) \rightarrow \mathbb{F}\left( A;\mathbb{C%
}\right) \text{, }\left( \phi _{+},\phi _{-}\right) \mapsto (\bar{\lambda}%
^{-1}\circ \left( \mathrm{id}_{K\left( H\right) }\otimes \phi _{+}\right)
\circ e_{A},\bar{\lambda}^{-1}\circ \left( \mathrm{id}_{K\left( H\right)
}\otimes \phi _{-}\right) \circ e_{A})\text{.}
\end{equation*}%
We have that%
\begin{equation*}
(\bar{\lambda}^{-1}\circ \left( \mathrm{id}_{K\left( H\right) }\otimes \phi
_{+}\right) \circ e_{A},\bar{\lambda}^{-1}\circ \left( \mathrm{id}_{K\left(
H\right) }\otimes \phi _{-}\right) \circ e_{A})=\left( \bar{\lambda}%
^{-1}\circ \bar{e}_{K\left( H\right) }\circ \phi _{+},\bar{\lambda}%
^{-1}\circ \bar{e}_{K\left( H\right) }\otimes \phi _{-}\right) \sim \left(
\phi _{+},\phi _{-}\right)
\end{equation*}%
in $\mathbb{F}\left( A;\mathbb{C}\right) $ or, equivalently,%
\begin{equation*}
\left( \bar{e}_{K\left( H\right) }\circ \phi _{+},\bar{e}_{K\left( H\right)
}\otimes \phi _{-}\right) \sim \left( \bar{\lambda}\circ \phi _{+},\bar{%
\lambda}\circ \phi _{-}\right)
\end{equation*}%
in $\mathbb{F}_{H\otimes H}\left( A;\mathbb{C}\right) $, where $\mathbb{F}%
_{H\otimes H}\left( A;\mathbb{C}\right) $ is defined as $\mathbb{F}\left( A;%
\mathbb{C}\right) $ by replacing $H$ with $H\otimes H$. Indeed, by
definition of $\lambda $,%
\begin{equation*}
\left( \bar{\lambda}\circ \phi _{+},\bar{\lambda}\circ \phi _{-}\right)
=\left( \mathrm{Ad}\left( V\right) \circ \bar{e}_{K\left( H\right) }\circ
\phi _{+},\mathrm{Ad}\left( V\right) \circ \bar{e}_{K\left( H\right) }\circ
\phi _{-}\right) \text{.}
\end{equation*}%
By \cite[Lemma 1.3.7]{jensen_elements_1991} there exists a stritly
continuous path $\left( V_{t}\right) _{t\in \left[ 0,1\right] }$ of
isometries in $B\left( H\otimes H\right) $ connecting $I$ to $V$. Thus,%
\begin{equation*}
\left( \mathrm{Ad}\left( V_{t}\right) \circ \bar{e}_{K\left( H\right) }\circ
\phi _{+},\mathrm{Ad}\left( V_{t}\right) \circ \bar{e}_{K\left( H\right)
}\circ \phi _{-}\right)
\end{equation*}%
is a continuous path in $\mathbb{F}_{H\otimes H}\left( A;\mathbb{C}\right) $
connecting $\left( \bar{e}_{K\left( H\right) }\circ \phi _{+},\bar{e}%
_{K\left( H\right) }\otimes \phi _{-}\right) $ to $\left( \bar{\lambda}\circ
\phi _{+},\bar{\lambda}\circ \phi _{-}\right) $. This concludes the proof
that $\mathrm{KK}_{h}\left( e_{A};\mathbb{C}\right) \circ G$ is the identity
of $\mathrm{KK}_{h}\left( A;\mathbb{C}\right) $.

We now show that $G\circ \mathrm{KK}_{h}\left( e_{A};\mathbb{C}\right) $ is
the identity of $\mathrm{KK}_{h}\left( K\left( H\right) \otimes A;\mathbb{C}%
\right) $. We have that $G\circ \mathrm{KK}_{h}\left( A;\mathbb{C}\right) $
is the definable group homomorphism induced by the Borel function%
\begin{equation*}
\mathbb{F}\left( K\left( H\right) \otimes A;\mathbb{C}\right) \rightarrow 
\mathbb{F}\left( K\left( H\right) \otimes A;\mathbb{C}\right) \text{, }%
\left( \psi _{+},\psi _{-}\right) \mapsto \left( \bar{\lambda}^{-1}\circ 
\mathrm{id}_{K\left( H\right) }\otimes \left( \psi _{+}\circ e_{A}\right) ,%
\bar{\lambda}^{-1}\circ \mathrm{id}_{K\left( H\right) }\otimes \left( \psi
_{-}\circ e_{A}\right) \right) \text{.}
\end{equation*}%
We claim that%
\begin{equation*}
\left( \bar{\lambda}^{-1}\circ \mathrm{id}_{K\left( H\right) }\otimes \left(
\psi _{+}\circ e_{A}\right) ,\bar{\lambda}^{-1}\circ \mathrm{id}_{K\left(
H\right) }\otimes \left( \psi _{-}\circ e_{A}\right) \right) \sim \left(
\psi _{+},\psi _{-}\right)
\end{equation*}%
in $\mathbb{F}\left( K\left( H\right) \otimes A;\mathbb{C}\right) $ or,
equivalently%
\begin{eqnarray*}
\left( \mathrm{id}_{K\left( H\right) }\otimes \left( \psi _{+}\circ
e_{A}\right) ,\mathrm{id}_{K\left( H\right) }\otimes \left( \psi _{-}\circ
e_{A}\right) \right) &\sim &\left( \bar{\lambda}\circ \psi _{+},\bar{\lambda}%
^{-1}\circ \psi _{-}\right) \\
&=&\left( \mathrm{Ad}\left( V\right) \circ \bar{e}_{K\left( H\right) }\circ
\psi _{+},\mathrm{Ad}\left( V\right) \circ \bar{e}_{K\left( H\right) }\circ
\psi _{-}\right)
\end{eqnarray*}%
in $\mathbb{F}_{H\otimes H}\left( K\left( H\right) \otimes A;\mathbb{C}%
\right) $. Indeed, define $\sigma _{1},\sigma _{2}:K\left( H\right) \otimes
A\rightarrow K\left( H\right) \otimes K\left( H\right) \otimes A$ be the
(strict) *-homomorphisms given by%
\begin{equation*}
T\otimes a\mapsto T\otimes e\otimes a
\end{equation*}%
and%
\begin{equation*}
T\otimes a\mapsto e\otimes T\otimes a\text{.}
\end{equation*}%
We can consider their strict extensions $\bar{\sigma}_{1},\bar{\sigma}%
_{2}:M\left( K\left( H\right) \otimes A\right) \rightarrow M\left( K\left(
H\right) \otimes K\left( H\right) \otimes A\right) $. Then we have that%
\begin{equation*}
\mathrm{id}_{K\left( H\right) }\otimes \left( \psi _{\pm }\circ e_{A}\right)
=\psi _{\pm }\circ \sigma _{1}:K\left( H\right) \otimes A\rightarrow B\left(
H\otimes H\right)
\end{equation*}%
and%
\begin{equation*}
\bar{e}_{K\left( H\right) }\circ \psi _{\pm }=\left( \mathrm{id}_{K\left(
H\right) }\otimes \psi _{+}\right) \circ \sigma _{2}:K\left( H\right)
\otimes A\rightarrow B\left( H\otimes H\right) \text{.}
\end{equation*}%
We have that $\bar{\sigma}_{1}=\mathrm{Ad}\left( U\otimes 1\right) \circ 
\bar{\sigma}_{2}$ for some unitary $U\in M\left( K\left( H\right) \otimes
K\left( H\right) \otimes A\right) $. Since $M\left( K\left( H\right) \otimes
K\left( H\right) \otimes A\right) $ is connected in the strict topology \cite%
[Lemma 1.3.7]{jensen_elements_1991}, we have that%
\begin{eqnarray*}
\left( \mathrm{id}_{K\left( H\right) }\otimes \left( \psi _{+}\circ
e_{A}\right) ,\mathrm{id}_{K\left( H\right) }\otimes \left( \psi _{-}\circ
e_{A}\right) \right) &=&\left( \left( \mathrm{id}_{K\left( H\right) }\otimes
\psi _{+}\right) \circ \sigma _{1},\left( \mathrm{id}_{K\left( H\right)
}\otimes \psi _{-}\right) \circ \sigma _{1}\right) \\
&\sim &\left( \left( \mathrm{id}_{K\left( H\right) }\otimes \psi _{+}\right)
\circ \sigma _{2},\left( \mathrm{id}_{K\left( H\right) }\otimes \psi
_{-}\right) \circ \sigma _{2}\right) \\
&=&\left( \bar{e}_{K\left( H\right) }\circ \psi _{+},\bar{e}_{K\left(
H\right) }\circ \psi _{-}\right) \\
&\sim &\left( \mathrm{Ad}\left( V\right) \circ \bar{e}_{K\left( H\right)
}\circ \psi _{+},\mathrm{Ad}\left( V\right) \circ \bar{e}_{K\left( H\right)
}\circ \psi _{-}\right) \text{.}
\end{eqnarray*}%
This concludes the proof.
\end{proof}

\subsection{Split exactness}

Suppose that%
\begin{equation*}
0\rightarrow A\overset{i}{\rightarrow }B\overset{p}{\rightarrow }%
C\rightarrow 0
\end{equation*}%
is an exact sequence of definable groups and definable group homomorphisms.
We say that it is \emph{definably} split if $p$ is a split epimorphism in
the category of definable groups, namely there exists a definable group
homomorphism $g:C\rightarrow B$ such that $p\circ g$ is equal to the
identity of $C$. This is equivalent to the assertion that $i:A\rightarrow B$
is a split monomorphism in the category of definable groups, namely there
exists a definable group homomorphism $f:B\rightarrow A$ such that $f\circ i$
is equal to the identity of $A$. In turn, this is equivalent to the
assertion that there exists a definable isomorphism $\gamma :B\rightarrow
A\oplus C$ that makes the diagram
\begin{center}
\begin{tikzcd}
A \arrow[r] \arrow[d,"\mathrm{id}_A"] & B \arrow[d, "\gamma "] \arrow[r] & C \arrow[d, "\mathrm{id}_C"]\\
A \arrow[r] & A\oplus C \arrow[r] & C%
\end{tikzcd}
\end{center}
commute.

If $\left( A,J\right) $ is a separable C*-pair such that the exact sequence%
\begin{equation*}
0\rightarrow J\rightarrow A\rightarrow A/J\rightarrow 0
\end{equation*}%
splits, then $\left( A,J\right) $ is, in particular, semi-split. Thus, there
is a corresponding six-term exact sequence in $\mathrm{K}$-homology. This
reduces to two \emph{definably split }exact sequences of definable groups
and definable group homomorphisms%
\begin{equation*}
0\rightarrow \mathrm{K}^{p}\left( J\right) \rightarrow \mathrm{K}^{p}\left(
A\right) \rightarrow \mathrm{K}^{p}\left( A/J\right) \rightarrow 0
\end{equation*}%
for $p\in \left\{ 0,1\right\} $. This is the split-exactness property of
definable $\mathrm{K}$-homology in the sense of \cite{cuntz_new_1987}.

\section{A definable Universal Coefficient Theorem\label{Section:UCT}}

In this section we consider a definable version of the Universal Coefficient
Theorem for $\mathrm{K}$-homology due to Brown \cite{brown_universal_1984},
later generalized by Rosenberg and Schochet to $\mathrm{KK}$-theory \cite%
{rosenberg_kunneth_1987}. We also consider the \emph{fine structure }of the
definable $\mathrm{K}$-homology groups as in \cite{schochet_uct_1996} in
terms of the notion of filtration for a separable nuclear C*-algebra
introduced therein. As an application, we show that definable $\mathrm{K}$%
-homology is a complete invariant for UHF C*-algebras up to stable
isomorphism, while the same conclusion does not hold for the purely
algebraic $\mathrm{K}$-homology. In this section, we assume all the
C*-algebras to be separable and nuclear.

\subsection{Index pairing for $\mathrm{K}$-homology}

Suppose that $A$ is a separable, nuclear C*-algebra. Fix $p\in \left\{
0,1\right\} $. Then one can define a natural \emph{definable} \emph{index
pairing} $\mathrm{K}^{p}\left( A\right) \times \mathrm{K}_{p}\left( A\right)
\rightarrow \mathbb{Z}$, where $\mathbb{Z}$ and the countable group $\mathrm{%
K}_{p}\left( A\right) $ are regarded as standard Borel spaces with respect
to the trivial Borel structure. Suppose that $A$ is concretely realized as a
C*-subalgebra of $B\left( H\right) $ such that the inclusion map $%
A\rightarrow B\left( H\right) $ is an ample representation of $A$, and let $%
\mathfrak{D}\left( A\right) \subseteq B\left( H\right) $ be the
corresponding Paschke dual algebra. For $p=1$ the pairing is defined by%
\begin{equation*}
\left\langle \left[ P\right] ,\left[ u\right] \right\rangle =\mathrm{Index}%
_{PH^{k}}\left( P^{\oplus k}uP^{\oplus k}\right)
\end{equation*}%
where $k\geq 1$, $u\in U\left( M_{k}\left( A^{+}\right) \right) $ is a
unitary, $P\in \mathfrak{D}\left( A\right) $ is a projection, $P^{\oplus k}$
is the $k$-fold direct sum of $P$, $P^{\oplus k}uP^{\oplus k}\in B\left(
H^{k}\right) $ satisfies%
\begin{equation*}
\left( P^{\oplus k}uP^{\oplus k}\right) ^{\ast }\left( P^{\oplus
k}uP^{\oplus k}\right) \equiv \left( P^{\oplus k}uP^{\oplus k}\right) \left(
P^{\oplus k}uP^{\oplus k}\right) ^{\ast }\equiv P^{\oplus k}\mathrm{\ 
\mathrm{mod}}\ K\left( H^{k}\right) \text{,}
\end{equation*}%
and $\mathrm{Index}_{P^{\oplus k}H^{k}}\left( P^{\oplus k}uP^{\oplus
k}\right) $ is its Fredholm index of $P^{\oplus k}uP^{\oplus k}$ regarded as
a Fredholm operator on $P^{\oplus k}H^{k}$; see \cite[Definition 7.2.1]%
{higson_analytic_2000}. Such a pairing is definable, in the sense that it is
induced by a Borel function $\mathrm{Z}^{1}\left( \mathfrak{D}\left(
A\right) \right) \times \mathrm{K}_{1}\left( A\right) \rightarrow \mathbb{Z}$%
, considering that the Fredholm index is given by a Borel map; see Section %
\ref{Subsection:polar}.

The index pairing $\mathrm{K}^{0}\left( A\right) \times \mathrm{K}_{0}\left(
A\right) \rightarrow \mathbb{Z}$ is defined by%
\begin{equation*}
\left\langle \left[ U\right] ,\left[ p\right] -\left[ q\right] \right\rangle
=\mathrm{Index}_{pH^{k}}\left( pU^{\oplus k}p\right) -\mathrm{Index}%
_{qH^{k}}\left( qU^{\oplus k}q\right) \text{,}
\end{equation*}%
where $k\geq 1$, $p,q\in M_{k}\left( A^{+}\right) $ are projections that
satisfy $p\equiv q\mathrm{\ \mathrm{mod}}\ M_{k}\left( A\right) $, $%
pU^{\oplus k}p\in B\left( H^{k}\right) $ satisfies%
\begin{equation*}
\left( pU^{\oplus k}p\right) ^{\ast }\left( pU^{\oplus k}p\right) \equiv
\left( pU^{\oplus k}p\right) \left( pU^{\oplus k}p\right) ^{\ast }\equiv p%
\mathrm{\ \mathrm{mod}}\ K\left( H^{k}\right) \text{,}
\end{equation*}%
and $\mathrm{Index}_{pH^{k}}\left( pU^{\oplus k}p\right) $ is the Fredholm
index of $pU^{\oplus k}p$ regarded as a Fredholm operator on $pH^{k}$ and
similarly for $qU^{\oplus k}q$; see \cite[Definition 7.2.3]%
{higson_analytic_2000}. Again, this pairing is definable since the Fredholm
index is given by a Borel map.

\subsection{Extensions\textrm{\ }of groups}

Suppose that $C,D$ are countable abelian groups. A ($2$-)cocycle on $C$ with
coefficients in $D$ is a function $c:C\times C\rightarrow D$ such that, for
every $x,y,z\in C$:

\begin{itemize}
\item $c\left( x,y\right) +c\left( x+y,z\right) =c\left( x,y+z\right)
+c\left( y,z\right) $;

\item $c\left( x,y\right) =c\left( y,x\right) $.
\end{itemize}

A cocycle is a coboundary if it is of the form $\left( x,y\right) \mapsto
h(x)+h(y)-h\left( x,y\right) $ for some function $h:C\rightarrow D$. The set 
$\mathrm{Z}\left( C,D\right) $ of cocycles on $C$ with coefficients in $D$
is a closed subgroup of the Polish group $D^{C\times C}$ endowed with the
product topology (where $D$ is endowed with the discrete topology). The set $%
\mathrm{B}\left( C,D\right) $ of coboundaries is a Polishable Borel subgroup
of $\mathrm{Z}\left( C,D\right) $. A \emph{weak }coboundary is a cocycle $c$
such that, for every finite (or, equivalently, for every finitely-generated)
subgroup $S$ of $C$, the restriction of $c$ to $S\times S$ is a coboundary
for $S$. Weak coboundaries form a closed subgroup $\mathrm{B}_{\mathrm{w}%
}\left( C,D\right) $ of $\mathrm{Z}\left( C,D\right) $, which is in fact the
closure of $\mathrm{B}\left( C,D\right) $ inside of $\mathrm{Z}\left(
C,D\right) $.

The group $\mathrm{Ext}\left( C,D\right) $ is the definable group, which is
in fact a group with Polish cover (see Remark \ref{Remark:Polish-cover}),
obtained as the quotient $\mathrm{Z}\left( C,D\right) /\mathrm{B}\left(
C,D\right) $; see \cite[Section 7]{bergfalk_ulam_2020}. The \emph{pure }(or 
\emph{phantom}) Ext group $\mathrm{PExt}\left( C,D\right) $ is the definable
subgroup of $\mathrm{Ext}\left( C,D\right) $ obtained as $\mathrm{B}_{%
\mathrm{w}}\left( C,D\right) /\mathrm{B}\left( C,D\right) $; see \cite%
{schochet_pext_2003,christensen_phantom_1998}. We also define $\mathrm{Ext}_{%
\mathrm{w}}\left( C,D\right) $ to be the Polish group obtained as the
quotient of the Polish group $\mathrm{Z}\left( C,D\right) $ by the closed
subgroup $\mathrm{B}_{\mathrm{w}}\left( C,D\right) $. By definition, we have
a short exact sequence of definable groups%
\begin{equation*}
0\rightarrow \mathrm{PExt}\left( C,D\right) \rightarrow \mathrm{Ext}\left(
C,D\right) \rightarrow \mathrm{Ext}_{\mathrm{w}}\left( C,D\right)
\rightarrow 0\text{.}
\end{equation*}

The terminology is due to the fact that every cocycle on $D$ with
coefficients in $D$ gives rise to an extension of $C$ by $D$, in such a way
that two cocycles differ by a coboundary if and only if the corresponding
extensions are isomorphic. Furthermore, every extension of $C$ by $D$ arises
from a cocycle in this fashion. Explicitly, if%
\begin{equation*}
0\rightarrow D\overset{i}{\rightarrow }E\overset{p}{\rightarrow }%
C\rightarrow 0
\end{equation*}%
is an extension of $C$ by $D$, the corresponding cocycle $c$ is defined as
follows. Fix a right inverse $t:C\rightarrow E$ for the function $%
p:E\rightarrow C$. Then one defines $c\left( x,y\right) :=i^{-1}\left(
t\left( x\right) +t\left( y\right) -t\left( x+y\right) \right) \in D$ for $%
x,y\in C$. Conversely, given a cocycle $c$ on $C$ with coefficients in $D$
one can define an extension as above, where $E=C\times D$ is endowed with
the operation defined by 
\begin{equation*}
\left( x,y\right) +\left( x^{\prime },y^{\prime }\right) =\left( x+x^{\prime
},c\left( x,x^{\prime }\right) +y+y^{\prime }\right) \text{.}
\end{equation*}%
The \emph{weak }coboundaries correspond in this way to extension of $C$ by $%
D $ that are \emph{pure}, i.e. such that $i\left( D\right) $ is a pure
subgroup of $E$; see \cite[Section V.29]{fuchs_infinite_1970}.

If $\left( C_{i}\right) _{i\in \omega }$ is an inductive sequence of\emph{\
finitely-generated }abelian groups and $C=\mathrm{colim}_{i\in \omega }C_{i}$
is the corresponding inductive limit (colimit), then the definable Jensen
theorem asserts that $\mathrm{PExt}\left( C,D\right) $ is naturally
definably isomorphic to $\mathrm{lim}_{i}^{1}\mathrm{Hom}\left(
C_{i},D\right) $, and $\mathrm{Ext}_{\mathrm{w}}\left( C,D\right) $ is
naturally isomorphic as a Polish group to \textrm{lim}$_{i}\mathrm{Ext}%
\left( C_{i},D\right) $; see \cite[Theorem 7.4]{bergfalk_ulam_2020} and \cite%
[Theorem 6.1]{schochet_pext_2003}.

\subsection{The Universal Coefficient Theorem}

Suppose that $A$ is a separable, nuclear C*-algebra. The definable index
pairing $\mathrm{K}^{1}\left( A\right) \times \mathrm{K}_{1}\left( A\right)
\rightarrow \mathbb{Z}$ induces a definable homomorphism%
\begin{equation*}
\mathrm{Index}_{A}:\mathrm{K}^{1}\left( A\right) \rightarrow \mathrm{Hom}%
\left( \mathrm{K}_{1}\left( A\right) ,\mathbb{Z}\right) \text{,}
\end{equation*}%
where we adopt the notation from \cite[Definition 7.2.3]%
{higson_analytic_2000}. Recall that $\mathrm{K}^{1}\left( A\right) $ is
defined as \textrm{Ext}$\left( A^{+}\right) ^{-1}$ where $A^{+}$ is the
unitization of $A$. The definable homomorphism $\mathrm{Index}_{A}:\mathrm{%
Ext}\left( A^{+}\right) ^{-1}\rightarrow \mathrm{Hom}\left( \mathrm{K}%
_{1}\left( A\right) ,\mathbb{Z}\right) $ can be equivalently described as
follows; see \cite{rosenberg_comparing_1981}. Let $\tau $ be an injective
unital extension 
\begin{equation*}
0\rightarrow K\left( H\right) \rightarrow E\rightarrow A^{+}\rightarrow 0%
\text{,}
\end{equation*}%
of $A^{+}$ by $K\left( H\right) $. Then $\tau $ gives rise to a six-term
exact sequence in $\mathrm{K}$-theory
\begin{center}
\begin{tikzcd}
\mathrm{K}_{0}\left( K\left( H\right) \right) =\mathbb{Z} \arrow[r] & \mathrm{K}_{0}\left( E\right) \arrow[r] & \mathrm{K}_{0}\left(A^{+}\right) \arrow[d, "\partial ^0"] \\
\mathrm{K}_{1}\left( A^{+}\right) =\mathrm{K}_{1}\left( A\right) \arrow[u, "\partial ^1"] &  \mathrm{K}_{1}\left( E\right) \arrow[l] & \mathrm{K}_{1}\left( K\left(H\right) \right) =\left\{ 0\right\} \arrow[l]
\end{tikzcd}
\end{center}
The group homomorphism $\mathrm{K}_{1}\left( A\right) \rightarrow \mathbb{Z}$
induced by $\tau $ in the diagram above depends only on the corresponding
element $\left[ \tau \right] $ of $\mathrm{Ext}\left( A^{+}\right) ^{-1}$,
and it is equal to $\mathrm{Index}_{A}\left( [\tau ]\right) $. As in \cite[%
Definition 7.6.7]{higson_analytic_2000}, we let $^{\circ }\mathrm{K}%
^{1}\left( A\right) $ be the definable subgroup of $\mathrm{K}_{1}\left(
A\right) $ obtained as the kernel of the index homomorphism $\mathrm{Index}%
_{A}:\mathrm{K}^{1}\left( A\right) \rightarrow \mathrm{Hom}\left( \mathrm{K}%
_{1}\left( A\right) ,\mathbb{Z}\right) $.

There is also a definable group homomorphism $\varkappa _{A}:{}^{\circ }%
\mathrm{K}^{1}\left( A\right) \rightarrow \mathrm{Ext}\left( \mathrm{K}%
_{0}\left( A\right) ,\mathbb{Z}\right) $, defined as follows. Suppose that $%
\tau $ is an injective unital extension of $A^{+}$ by $K\left( H\right) $ as
above, such that moreover $[\tau ]\in {}^{\circ }\mathrm{K}^{1}\left(
A\right) $. Then the six-term exact sequence above reduces to a short exact
sequence%
\begin{equation*}
0\rightarrow \mathrm{K}_{0}\left( K\left( H\right) \right) =\mathbb{Z}%
\rightarrow \mathrm{K}_{0}\left( E\right) \rightarrow \mathrm{K}_{0}\left(
A^{+}\right) \rightarrow 0
\end{equation*}%
This defines an element of $\mathrm{Ext}\left( \mathrm{K}_{0}\left(
A^{+}\right) ,\mathbb{Z}\right) $, which in turn defines an element of $%
\mathrm{Ext}\left( \mathrm{K}_{0}\left( A\right) ,\mathbb{Z}\right) $ via
the inclusion $\mathrm{K}_{0}\left( A\right) \rightarrow \mathrm{K}%
_{0}\left( A^{+}\right) $. This element $\varkappa _{A}\left( [\tau ]\right) 
$ of $\mathrm{Ext}\left( \mathrm{K}_{0}\left( A\right) ,\mathbb{Z}\right) $
depends only on the class $[\tau ]$ in $\mathrm{Ext}\left( A^{+}\right) ^{-1}
$ of the extension $\tau $. This gives a group homomorphism $\varkappa
_{A}:{}^{\circ }\mathrm{K}^{1}\left( A\right) \rightarrow \mathrm{Ext}\left( 
\mathrm{K}_{0}\left( A\right) ,\mathbb{Z}\right) $, $[\tau ]\mapsto
\varkappa _{A}\left( [\tau ]\right) $, which is easily seen to be definable.
In a similar fashion, by replacing $A$ with its suspension, one can define a
definable group homomorphism 
\begin{equation*}
\mathrm{Index}_{A}:\mathrm{K}^{0}\left( A\right) \rightarrow \mathrm{Hom}%
\left( \mathrm{K}_{0}\left( A\right) ,\mathbb{Z}\right) 
\end{equation*}%
with kernel $^{\circ }\mathrm{K}^{0}\left( A\right) $, and a definable group
homomorphism%
\begin{equation*}
\varkappa _{A}:{}^{\circ }\mathrm{K}^{0}\left( A\right) \rightarrow \mathrm{%
Ext}\left( \mathrm{K}_{1}\left( A\right) ,\mathbb{Z}\right) \text{.}
\end{equation*}

We recall the following definition of a C*-algebra satisfying the Universal
Coefficient Theorem (UCT); see \cite[Definition 4.4]%
{rosenberg_comparing_1981}.

\begin{definition}
\label{Definition:UCT}A separable C*-algebra $A$ is said to satisfy the 
\emph{Universal Coefficient Theorem} (UCT) for $\mathbb{C}$, or the pair $%
\left( A,\mathbb{C}\right) $ satisfies the UCT, if for $p\in \left\{
0,1\right\} $ the group homomorphisms $\mathrm{Index}_{A}:\mathrm{K}%
^{p}\left( A\right) \rightarrow \mathrm{Hom}\left( \mathrm{K}_{p}\left(
A\right) ,\mathbb{Z}\right) $ is surjective, and the group homomorphism $%
\varkappa _{A}:\mathrm{\mathrm{Ker}}\left( \gamma _{A}\right) ={}^{\circ }%
\mathrm{K}^{p}\left( A\right) \rightarrow \mathrm{Ext}\left( \mathrm{K}%
_{1-p}\left( A\right) ,\mathbb{Z}\right) $ is an isomorphism.
\end{definition}

It is proved in \cite{brown_universal_1984} that all the separable nuclear
C*-algebras in the so-called \emph{bootstrap class} satisfy the UCT for $%
\mathbb{C}$; see also \cite{brown_operator_1975}. In fact one can more
generally consider the UCT for $B$, where $B$ is any separable C*-algebra,
defined in terms of Kasparov's \textrm{KK}-groups; see \cite%
{rosenberg_kunneth_1987}. It is unknown whether there exists a separable
nuclear C*-algebra that does \emph{not }satisfy the UCT.

\subsection{Weak and asymptotic $\mathrm{K}$-homology groups}

We now recall the notion of a \emph{filtration }(or $\mathrm{KK}$%
-filtration) for a separable nuclear C*-algebra as in \cite[Definition 1.4]%
{schochet_uct_1996}, and we define the weak and asymptotic $\mathrm{K}$%
-homology groups for C*-algebras with a filtration.

\begin{definition}
\label{Definition:filtration}Suppose that $A$ is a separable, nuclear
C*-algebra. An inductive sequence $\left( A_{n},\eta _{n}\right) _{n\in
\omega }$ of separable, nuclear C*-algebras is a \emph{filtration }of $A$ if:

\begin{itemize}
\item for every $n\in \omega $, $A_{n}$ satisfies the Universal Coefficient
Theorem for $\mathbb{C}$ (as in\ Definition \ref{Definition:UCT});

\item for every $n\in \omega $ and $p\in \left\{ 0,1\right\} $, $\mathrm{K}%
_{p}\left( A_{n}\right) $ is a finitely generated group;

\item $A$ is $\mathrm{KK}$-equivalent to the inductive limit of the sequence 
$\left( A_{n},\eta _{n}\right) _{n\in \omega }$.
\end{itemize}
\end{definition}

\begin{remark}
A slightly more restrictive definition is considered in \cite[Definition 1.4]%
{schochet_uct_1996}, where the C*-algebras $A_{n}$ are supposed to
commutative.
\end{remark}

We let $\mathcal{C}$ be the category that has separable, nuclear C*-algebras
with a filtration as objects, and *-homomorphisms as morphisms.

Suppose that $A$ is a separable, nuclear C*-algebra with a filtration $%
\left( A_{n}\right) _{n\in \omega }$. Then the inductive limit $\mathrm{colim%
}_{n}A_{n}$ of the sequence $\left( A_{n}\right) _{n\in \omega }$ satisfies
the UCT for $\mathbb{C}$ by \cite[Theorem 4.1]{schochet_uct_1996}, whence $A$
satisfies the UCT as well. Thus, the definable group homomorphism $\kappa
_{A}:{}^{\circ }\mathrm{K}^{p}\left( A\right) =\mathrm{\mathrm{Ker}}\left( 
\mathrm{Index}_{A}\right) \rightarrow \mathrm{Ext}\left( \mathrm{K}%
_{p}\left( A\right) ,\mathbb{Z}\right) $ is an isomorphism. After replacing $%
A$ with $\mathrm{colim}_{n}A_{n}$ we can assume that $A=\mathrm{colim}%
_{n}A_{n}$. Notice that, as \textrm{K}$_{0}\left( A_{n}\right) $ and $%
\mathrm{K}_{1}\left( A_{n}\right) $ are finitely-generated and $A_{n}$
satisfies the UCT for $\mathbb{C}$, it follows that $\mathrm{K}^{0}\left(
A_{n}\right) $ and \textrm{K}$^{1}\left( A_{n}\right) $ are \emph{countable}
groups.

We define the \emph{weak }$\mathrm{K}$-homology group $\mathrm{K}_{\mathrm{w}%
}^{p}\left( A\right) $ to be Polish group \textrm{lim}$_{n}\mathrm{K}%
^{p}\left( A_{n}\right) $. The assignment $A\mapsto \mathrm{K}_{\mathrm{w}%
}^{p}\left( A\right) $ defines a homotopy-invariant functor from $\mathcal{C}
$ to the category of Polish groups. The weak $\mathrm{K}$-homology group $%
\mathrm{K}_{\mathrm{w}}^{1}\left( A\right) $ is isomorphic to the group $%
\mathrm{KL}\left( A,\mathbb{C}\right) $ from \cite[Section 4]%
{rordam_classification_1995}; see also \cite[2.4.8]%
{rordam_classification_2002} and \cite[Corollary 3.8]{schochet_uct_1996}. A
description of $\mathrm{K}_{\mathrm{w}}^{1}\left( A\right) $ in terms of the
sum $\underline{K}\left( A\right) $ of all the $\mathrm{K}$-theory groups of 
$A$ in all degrees and all cyclic coefficient groups is obtained in \cite%
{dadarlat_universal_1996}; see also \cite[Theorem 3.10]{schochet_uct_1996}.

We have a canonical surjective definable homomorphism $\mathrm{K}^{p}\left(
A\right) \rightarrow \mathrm{K}_{\mathrm{w}}^{p}\left( A\right) $ as in
Milnor's exact sequence. We define the \emph{asymptotic }$\mathrm{K}$%
-homology group $\mathrm{K}_{\infty }^{p}\left( A\right) $ to be the kernel
of such a definable homomorphism. As $A$ satisfies the UCT for $\mathbb{C}$,
the definable isomorphism $\varkappa _{A}:{}^{\circ }\mathrm{K}^{p}\left(
A\right) \rightarrow \mathrm{Ext}\left( \mathrm{K}_{1-p}\left( A\right) ,%
\mathbb{Z}\right) $ is an isomorphism. Since, for every $n\in \omega $, $%
A_{n}$ satisfies the UCT for $\mathbb{C}$, $\mathrm{K}_{\infty }^{p}\left(
A\right) \subseteq {}^{\circ }\mathrm{K}^{p}\left( A\right) $ is equal to
the inverse image of $\mathrm{PExt}\left( \mathrm{K}_{1-p}\left( A\right) ,%
\mathbb{Z}\right) $ under $\varkappa _{A}$. In particular, this shows that $%
\mathrm{K}_{\infty }^{p}\left( A\right) $ does not depend on the choice of
the filtration for $A$. The assignment $A\mapsto \mathrm{K}_{\infty
}^{p}\left( A\right) $ defines a homotopy-invariant functor from $\mathcal{C}
$ to the category of definable groups. As noticed above, $\mathrm{K}_{\infty
}^{p}\left( A\right) $ is naturally definably isomorphic to $\mathrm{PExt}%
\left( \mathrm{K}_{1-p}\left( A\right) ,\mathbb{Z}\right) $. We also have
that $\mathrm{K}_{1-p}\left( A\right) =\mathrm{colim}_{n}\mathrm{K}%
_{1-p}\left( A_{n}\right) $, and hence $\mathrm{K}_{\infty }^{p}\left(
A\right) $ is definably isomorphic to $\mathrm{lim}_{n}^{1}\mathrm{Hom}%
\left( \mathrm{K}_{1-p}\left( A_{n}\right) ,\mathbb{Z}\right) $ by the
definable Jensen theorem \cite[Theorem 7.4]{bergfalk_ulam_2020}.

\begin{lemma}
Suppose that $A$ is a separable, nuclear C*-algebra, and $\left( A_{n},\eta
_{n}\right) _{n\in \omega }$ is a filtration of $A$. Then the definable
homomorphism $\mathrm{K}^{p}\left( A\right) \rightarrow \mathrm{K}_{\mathrm{w%
}}^{p}\left( A_{n}\right) $ has a definable right inverse $\mathrm{K}_{%
\mathrm{w}}^{p}\left( A_{n}\right) \rightarrow \mathrm{K}^{p}\left( A\right) 
$, which is not necessarily a group homomorphism.
\end{lemma}

\begin{proof}
After replacing $A$ with it suspension, we can assume that $p=1$.
Furthermore, after replacing $A$ with $\mathrm{colim}_{n}A_{n}$, we can
assume that $A=\mathrm{colim}_{n}A_{n}$. Finally, after replacing $A$ with $%
A^{+}$ and $A_{n}$ with $A_{n}^{+}$, we can assume that $A$ and $A_{n}$ for $%
n\in \omega $ are unital, and $\eta _{n}:A_{n}\rightarrow A_{n+1}$ is a
unital *-homomorphism. In this case, we have that $\mathrm{K}^{1}\left(
A\right) =\mathrm{Ext}\left( A\right) ^{-1}$ and $\mathrm{K}^{1}\left(
A_{n}\right) =\mathrm{Ext}\left( A_{n}\right) ^{-1}$ for $n\in \omega $. We
need to show that the definable group homomorphism $\mathrm{Ext}\left(
A\right) ^{-1}\rightarrow \mathrm{lim}_{n}\mathrm{Ext}\left( A_{n}\right)
^{-1}$ has a definable right inverse, which is not necessarily a group
homomorphism. Recall that \textrm{Ext}$\left( A\right) ^{-1}$ is the
quotient of the Polish space $\mathcal{E}\left( A\right) $ of
representatives of injective, unital extensions of $A$ by the equivalence
relation $\thickapprox $ as in Section \ref{Subsection:definable-Ext}.

Fix, for every $\ell \in \omega $ an enumeration $(x_{n}^{(\ell )})_{n\in
\omega }$ of $\mathrm{Ext}\left( A_{\ell }\right) ^{-1}$. For $\ell
_{0}<\ell _{1}$ define the bonding map%
\begin{equation*}
\eta _{\left( \ell _{1},\ell _{0}\right) }=\eta _{\ell _{1}-1}\circ \cdots
\circ \eta _{\ell _{0}}:A_{\ell _{0}}\rightarrow A_{\ell _{1}}\text{,}
\end{equation*}%
and set $\eta _{\left( \ell ,\ell \right) }=\mathrm{id}_{A_{\ell }}$ for $%
\ell \in \omega $. Define%
\begin{equation*}
\eta _{\left( \infty ,\ell \right) }:A_{\ell }\rightarrow A
\end{equation*}%
to be the canonical map. Let also $p^{\left( \ell _{0},\ell _{1}\right) }:%
\mathrm{Ext}\left( A_{\ell _{1}}\right) ^{-1}\rightarrow \mathrm{Ext}\left(
A_{\ell _{0}}\right) ^{-1}$ be the group homomorphism induced by the bonding
map $\eta _{\left( \ell _{1},\ell _{0}\right) }:A_{\ell _{0}}\rightarrow
A_{\ell _{1}}$.\ Then an element of $\mathrm{lim}_{n}\mathrm{Ext}\left(
A_{n}\right) ^{-1}$ is a sequence $(x_{n_{\ell }}^{\left( \ell \right)
})_{\ell \in \omega }$ such that, for $\ell _{0}<\ell _{1}$, $p^{\left( \ell
_{0},\ell _{1}\right) }(x_{n_{\ell _{1}}}^{\ell _{1}})=x_{n_{\ell
_{0}}}^{\ell _{0}}$. For every $\ell ,n\in \omega $ fix $\varphi
_{n}^{\left( \ell \right) }\in \mathcal{E}\left( A_{\ell }\right) $ such
that $[\varphi _{n}^{(\ell )}]=x_{n}^{\left( \ell \right) }$. For every $%
\ell \in \omega $ and $n,m\in \omega $ such that $p^{\left( \ell -1,\ell
\right) }(x_{n}^{\left( \ell \right) })=p^{\left( \ell -1,\ell \right)
}(x_{m}^{\left( \ell \right) })$ fix $U_{n,m}^{\left( \ell \right) }\in
U\left( H\right) $ such that $\mathrm{Ad}(U_{n,m}^{\left( \ell \right)
})\circ \varphi _{n}^{\left( \ell \right) }\circ \eta _{\ell -1}=\varphi
_{m}^{\left( \ell \right) }\circ \eta _{\ell -1}$. If $(x_{n_{\ell }}^{(\ell
)})_{n\in \omega }$ is an element of $\mathrm{lim}_{n}\mathrm{Ext}\left(
A_{n}\right) ^{-1}$, then setting $\psi ^{\left( \ell \right) }:=\mathrm{Ad}%
(U_{n_{\ell },n_{\ell -1}}^{\left( \ell \right) }U_{n_{\ell -1},n_{\ell
-2}}^{\left( \ell -1\right) }\cdots U_{n_{1}n_{0}}^{\left( 1\right) })\circ
\varphi _{n_{\ell }}^{\left( \ell \right) }\in \mathcal{E}\left( A_{\ell
}\right) $, one obtains a sequence $\left( \psi ^{\left( \ell \right)
}\right) _{\ell \in \omega }$ such that $\psi ^{\left( \ell \right) }\circ
\eta _{\ell -1}=\psi ^{\left( \ell -1\right) }$ for every $\ell >0$.
Therefore, setting $\psi =\mathrm{colim}_{\ell }\psi ^{\left( \ell \right)
}:A\rightarrow B\left( H\right) $ defines an element of $\mathcal{E}\left(
A\right) $ such that $\left[ \psi \circ \eta _{\left( \infty ,\ell \right) }%
\right] =x_{n_{\ell }}^{\ell }$ for every $\ell \in \omega $, and hence the
image of $[\psi ]\in \mathrm{Ext}\left( A\right) ^{-1}$ under the definable
homomorphism $\mathrm{Ext}\left( A\right) ^{-1}\rightarrow \mathrm{lim}_{n}%
\mathrm{Ext}\left( A_{n}\right) ^{-1}$ is equal to $\left( x_{n_{\ell
}}^{\ell }\right) _{\ell \in \omega }$. This construction describes a
definable function $\mathrm{lim}_{n}\mathrm{Ext}\left( A_{n}\right)
^{-1}\rightarrow \mathrm{Ext}\left( A\right) ^{-1}$, which is a right
inverse for $\mathrm{Ext}\left( A\right) ^{-1}\rightarrow \mathrm{lim}_{n}%
\mathrm{Ext}\left( A_{n}\right) ^{-1}$. This concludes the proof.
\end{proof}

Suppose that $A$ is a separable, nuclear\emph{\ }C*-algebra with a
filtration $\left( A_{n},\eta _{n}\right) _{n\in \omega }$. The index
homomorphisms 
\begin{equation*}
\mathrm{Index}_{A_{n}}:\mathrm{K}^{p}\left( A_{n}\right) \rightarrow \mathrm{%
Hom}\left( \mathrm{K}_{p}\left( A_{n}\right) ,\mathbb{Z}\right)
\end{equation*}%
for $n\in \omega $ induce a continuous group homomorphism 
\begin{equation*}
\mathrm{K}_{\mathrm{w}}^{p}\left( A\right) \rightarrow \mathrm{Hom}\left( 
\mathrm{K}_{p}\left( A\right) ,\mathbb{Z}\right) =\mathrm{lim}_{n}\mathrm{Hom%
}\left( \mathrm{K}_{p}\left( A_{n}\right) ,\mathbb{Z}\right) \text{.}
\end{equation*}%
Similarly the definable group homomorphisms 
\begin{equation*}
\varkappa _{A_{n}}^{-1}:\mathrm{Ext}\left( \mathrm{K}_{p}\left( A_{n}\right)
,\mathbb{Z}\right) \rightarrow \mathrm{K}^{p}\left( A_{n}\right)
\end{equation*}%
for $n\in \omega $ induce a definable group homomorphism 
\begin{equation*}
\mathrm{lim}_{n}\mathrm{Ext}\left( \mathrm{K}_{p}\left( A_{n}\right) ,%
\mathbb{Z}\right) =\mathrm{Ext}_{\mathrm{w}}\left( \mathrm{K}_{p}\left(
A\right) ,\mathbb{Z}\right) \rightarrow \mathrm{K}_{\mathrm{w}}^{p}\left(
A\right) .
\end{equation*}%
This gives a short exact sequence of definable groups%
\begin{equation*}
0\rightarrow \mathrm{Ext}_{\mathrm{w}}\left( \mathrm{K}_{p}\left( A\right) ,%
\mathbb{Z}\right) \rightarrow \mathrm{K}_{\mathrm{w}}^{p}\left( A\right)
\rightarrow \mathrm{Hom}\left( \mathrm{K}_{p}\left( A\right) ,\mathbb{Z}%
\right) \rightarrow 0\text{.}
\end{equation*}%
By definition of $\mathrm{PExt}$ and $\mathrm{Ext}_{\mathrm{w}}$, we also
have a short exact sequence of definable groups%
\begin{equation*}
0\rightarrow \mathrm{PExt}\left( \mathrm{K}_{p}\left( A\right) ,\mathbb{Z}%
\right) \rightarrow \mathrm{Ext}\left( \mathrm{K}_{p}\left( A\right) ,%
\mathbb{Z}\right) \rightarrow \mathrm{Ext}_{\mathrm{w}}\left( \mathrm{K}%
_{p}\left( A\right) ,\mathbb{Z}\right) \rightarrow 0
\end{equation*}%
where $\mathrm{PExt}\left( \mathrm{K}_{p}\left( A\right) ,\mathbb{Z}\right)
\rightarrow \mathrm{Ext}\left( \mathrm{K}_{p}\left( A\right) ,\mathbb{Z}%
\right) $ is the inclusion map and $\mathrm{Ext}\left( \mathrm{K}_{p}\left(
A\right) ,\mathbb{Z}\right) \rightarrow \mathrm{Ext}_{\mathrm{w}}\left( 
\mathrm{K}_{p}\left( A\right) ,\mathbb{Z}\right) $ is the quotient map.

\begin{proposition}
\label{Proposition:KK}Suppose that $A$ is a separable, nuclear C*-algebra
with a filtration and $p\in \left\{ 0,1\right\} $. If $\mathrm{K}_{p}\left(
A\right) $ is torsion-free, then $\mathrm{K}_{\infty }^{p}\left( A\right) $
is naturally isomorphic to $\mathrm{Ext}\left( \mathrm{K}_{1-p}\left(
A\right) ,\mathbb{Z}\right) $, and $\mathrm{K}_{\mathrm{w}}^{p}\left(
A\right) $ is naturally isomorphic as a Polish group to $\mathrm{Hom}\left( 
\mathrm{K}_{p}\left( A\right) ,\mathbb{Z}\right) $.
\end{proposition}

\begin{proof}
Since $\mathrm{K}_{p}\left( A\right) $ is torsion-free, we have that $%
\mathrm{PExt}\left( \mathrm{K}_{p}\left( A\right) ,\mathbb{Z}\right) =%
\mathrm{Ext}\left( \mathrm{K}_{p}\left( A\right) ,\mathbb{Z}\right) $.
Therefore,%
\begin{equation*}
\mathrm{K}_{\infty }^{p}\left( A\right) \cong \mathrm{PExt}\left( \mathrm{K}%
_{p}\left( A\right) ,\mathbb{Z}\right) =\mathrm{Ext}\left( \mathrm{K}%
_{p}\left( A\right) ,\mathbb{Z}\right) \text{.}
\end{equation*}%
From the exact sequence%
\begin{equation*}
0\rightarrow \mathrm{PExt}\left( \mathrm{K}_{p}\left( A\right) ,\mathbb{Z}%
\right) \rightarrow \mathrm{Ext}\left( \mathrm{K}_{p}\left( A\right) ,%
\mathbb{Z}\right) \rightarrow \mathrm{Ext}_{\mathrm{w}}\left( \mathrm{K}%
_{p}\left( A\right) ,\mathbb{Z}\right) \rightarrow 0
\end{equation*}%
we conclude that%
\begin{equation*}
\mathrm{Ext}_{\mathrm{w}}\left( \mathrm{K}_{p}\left( A\right) ,\mathbb{Z}%
\right) =\left\{ 0\right\} \text{.}
\end{equation*}%
From this and the exact sequence%
\begin{equation*}
0\rightarrow \mathrm{Ext}_{\mathrm{w}}\left( \mathrm{K}_{p}\left( A\right) ,%
\mathbb{Z}\right) \rightarrow \mathrm{K}_{\mathrm{w}}^{p}\left( A\right)
\rightarrow \mathrm{Hom}\left( \mathrm{K}_{p}\left( A\right) ,\mathbb{Z}%
\right) \rightarrow 0
\end{equation*}%
we conclude that%
\begin{equation*}
\mathrm{K}_{\mathrm{w}}^{p}\left( A\right) \cong \mathrm{Hom}\left( \mathrm{K%
}_{p}\left( A\right) ,\mathbb{Z}\right) \text{.}
\end{equation*}%
This concludes the proof.
\end{proof}

\begin{corollary}
\label{Corollary:KK1}Suppose that $A$ is a separable, nuclear C*-algebra
with a filtration and $p\in \left\{ 0,1\right\} $ is such that $\mathrm{K}%
_{p}\left( A\right) $ is a finite-rank torsion-free abelian group and $%
\mathrm{K}_{1-p}\left( A\right) $ is trivial. We can write%
\begin{equation*}
\mathrm{K}_{p}\left( A\right) =\Lambda \oplus \Lambda ^{\prime }
\end{equation*}%
where $\Lambda ^{\prime }$ is finitely-generated and $\Lambda $ has no
nonzero finitely-generated direct summand. Then%
\begin{equation*}
\mathrm{K}^{p}\left( A\right) \cong \mathrm{Hom}\left( \mathrm{K}_{p}\left(
A\right) ,\mathbb{Z}\right) \cong \mathrm{Hom}\left( \Lambda ^{\prime },%
\mathbb{Z}\right)
\end{equation*}%
and%
\begin{equation*}
\mathrm{K}^{1-p}\left( A\right) \cong \mathrm{Ext}\left( \mathrm{K}%
_{p}\left( A\right) ,\mathbb{Z}\right) \cong \mathrm{Ext}\left( \Lambda ,%
\mathbb{Z}\right)
\end{equation*}%
as definable groups.
\end{corollary}

\begin{proof}
After replacing $A$ with $SA$, we can assume that $p=0$. We have that%
\begin{equation*}
\mathrm{K}_{\infty }^{0}\left( A\right) \cong \mathrm{PExt}\left( \mathrm{K}%
_{1}\left( A\right) ,\mathbb{Z}\right) \cong \left\{ 0\right\} \text{.}
\end{equation*}%
Therefore,%
\begin{equation*}
\mathrm{K}^{0}\left( A\right) \cong \mathrm{K}_{\mathrm{w}}^{0}\left(
A\right) \cong \mathrm{Hom}\left( \mathrm{K}_{0}\left( A\right) ,\mathbb{Z}%
\right) \cong \mathrm{Hom}\left( \Lambda ^{\prime },\mathbb{Z}\right) \text{.%
}
\end{equation*}%
Similarly, we have that%
\begin{equation*}
\mathrm{K}_{\mathrm{w}}^{1}\left( A\right) \cong \mathrm{Hom}\left( \mathrm{K%
}_{1}\left( A\right) ,\mathbb{Z}\right) \cong \left\{ 0\right\}
\end{equation*}%
and hence, since $\mathrm{K}_{0}\left( A\right) $ is torsion-free, 
\begin{equation*}
\mathrm{K}^{1}\left( A\right) \cong \mathrm{K}_{\infty }^{1}\left( A\right)
\cong \mathrm{PExt}\left( \mathrm{K}_{0}\left( A\right) ,\mathbb{Z}\right)
\cong \mathrm{Ext}\left( \mathrm{K}_{0}\left( A\right) ,\mathbb{Z}\right)
\cong \mathrm{Ext}\left( \Lambda ,\mathbb{Z}\right) \text{.}
\end{equation*}%
This concludes the proof.
\end{proof}

\begin{corollary}
\label{Corollary:KK2}Suppose that $p\in \left\{ 0,1\right\} $ and $A,B$ are
separable, nuclear C*-algebras with a filtration, such that $\mathrm{K}%
_{p}\left( A\right) $ and $\mathrm{K}_{p}\left( B\right) $ are finite-rank
torsion-free abelian groups, and $\mathrm{K}_{1-p}\left( A\right) $ and $%
\mathrm{K}_{1-p}\left( B\right) $ are trivial. Then the following assertions
are equivalent:

\begin{enumerate}
\item $\mathrm{K}^{i}\left( A\right) $ and $\mathrm{K}^{i}\left( B\right) $
are \emph{definably} isomorphic for $i\in \left\{ 0,1\right\} $;

\item $\mathrm{K}_{p}\left( A\right) $ and $\mathrm{K}_{p}\left( B\right) $
are isomorphic.
\end{enumerate}

If furthermore $\mathrm{K}_{p}\left( A\right) $ and $\mathrm{K}_{p}\left(
B\right) $ have no nonzero finitely-generated direct summand, then the
following assertions are equivalent:

\begin{enumerate}
\item $\mathrm{K}^{1-p}\left( A\right) $ and $\mathrm{K}^{1-p}\left(
B\right) $ are \emph{definably} isomorphic;

\item $\mathrm{K}_{p}\left( A\right) $ and $\mathrm{K}_{p}\left( B\right) $
are isomorphic.
\end{enumerate}
\end{corollary}

\begin{proof}
After passing to the suspension, we can assume that $p=0$. Since $\mathrm{K}%
_{0}\left( A\right) $ and $\mathrm{K}_{0}\left( B\right) $ are finite-rank
torsion-free abelian groups, we can write%
\begin{equation*}
\mathrm{K}_{0}\left( A\right) =\Lambda _{A}\oplus \Lambda _{A}^{\prime }
\end{equation*}%
\begin{equation*}
\mathrm{K}_{0}\left( B\right) =\Lambda _{B}\oplus \Lambda _{B}^{\prime }
\end{equation*}%
where $\Lambda _{A},\Lambda _{B}$ have no nonzero finitely-generated direct
summand, and $\Lambda _{A}^{\prime },\Lambda _{B}^{\prime }$ are
finitely-generated. Then we have that $\mathrm{K}_{0}\left( A\right) \cong 
\mathrm{K}_{0}\left( B\right) $ if and only if $\Lambda _{A}\cong \Lambda
_{B}$ and $\Lambda _{A}^{\prime }\cong \Lambda _{B}^{\prime }$. We have that 
$\Lambda _{A}^{\prime }\cong \Lambda _{B}^{\prime }$ if and only if 
\begin{equation*}
\mathrm{Hom}\left( \Lambda _{A}^{\prime },\mathbb{Z}\right) \cong \mathrm{Hom%
}\left( \Lambda _{B}^{\prime },\mathbb{Z}\right)
\end{equation*}%
Furthermore, by \cite[Corollary 7.6]{bergfalk_ulam_2020}, we have that $%
\Lambda _{A}\cong \Lambda _{B}$ if and only if $\mathrm{Ext}\left( \Lambda
_{A},\mathbb{Z}\right) $ and $\mathrm{Ext}\left( \Lambda _{B},\mathbb{Z}%
\right) $ are definably isomorphic. The conclusion thus follows from
Corollary \ref{Corollary:KK1}.
\end{proof}

We now show that Corollary \ref{Corollary:KK2} does not hold if $\mathrm{K}%
^{p}\left( A\right) $ and $\mathrm{K}^{p}\left( B\right) $ are merely asked
to be isomorphic, rather than \emph{definably }isomorphic; see Theorem \ref%
{Theorem:UHF}.

\subsection{Stable isomorphism of UHF algebras\label{Subsection:UHF}}

Recall that a uniformly hyperfinite (UHF) C*-algebra is an
infinite-dimensional separable unital C*-algebra that is the limit of an
inductive sequence of full matrix algebras \cite[Example III.5.1]%
{davidson_algebras_1996}. Since finite-dimensional C*-algebras are nuclear,
satisfy the UCT for $\mathbb{C}$, and have finitely-generated $\mathrm{K}%
_{0} $ and $\mathrm{K}_{1}$ groups, UHF C*-algebras are nuclear and have a
filtration. If $A$ is a UHF C*-algebra, then $\mathrm{K}_{0}\left( A\right) $
is a rank $1$ torsion-free abelian group that is not isomorphic to $\mathbb{Z%
}$, while $\mathrm{K}_{1}\left( A\right) $ is trivial. Given a rank $1$
torsion-free abelian group $\Lambda $ that is not isomorphic to $\mathbb{Z}$%
, there exists a UHF C*-algebra $A_{\Lambda }$ such that $\mathrm{K}%
_{0}\left( A_{\Lambda }\right) \cong \Lambda $. By Proposition \ref%
{Proposition:KK}, we have that $\mathrm{K}^{1}\left( A_{\Lambda }\right) $
is definably isomorphic to $\mathrm{Ext}\left( \Lambda ,\mathbb{Z}\right) $,
while $\mathrm{K}^{0}\left( A_{\Lambda }\right) $ is trivial.

Recall that a rank $1$ torsion-free abelian group is an abelian group that
is isomorphic to a subgroup of $\mathbb{Q}$. Given a torsion-free group $%
\Lambda $ and a prime number $p$, one defines its $p$-corank $\mathrm{rank}%
^{p}\Lambda $ to be the dimension of $\Lambda /p\Lambda $ as a $\mathbb{Z}/p%
\mathbb{Z}$-vector space. As a particular instance of \cite[Theorem A.7]%
{bergfalk_ulam_2020} we have that, given rank $1$ torsion-free abelian
groups $\Lambda ,\Lambda ^{\prime }$, $\mathrm{Ext}\left( \Lambda ,\mathbb{Z}%
\right) $ and $\mathrm{Ext}\left( \Lambda ^{\prime },\mathbb{Z}\right) $ are
isomorphic as discrete groups if and only if $\mathrm{\mathrm{\mathrm{rank}}}%
^{p}\Lambda =\mathrm{\mathrm{rank}}^{p}\Lambda ^{\prime }$ for every prime $%
p $. It easily follows that there exists an uncountable family $\left(
\Lambda _{i}\right) _{i\in \mathbb{R}}$ of pairwise nonisomorphic rank $1$
torsion-free abelian groups such that $\mathrm{Ext}\left( \Lambda _{i},%
\mathbb{Z}\right) $ and $\mathrm{Ext}\left( \Lambda _{j},\mathbb{Z}\right) $
are isomorphic as discrete groups for $i,j\in \mathbb{R}$.

The following result is an immediate consequence of these observations
together with Corollary \ref{Corollary:KK2} and the Elliott classification
of approximately finite-dimensional (AF) C*-algebras \cite%
{elliott_classification_1976}, or Glimm's classification of UHF C*-algebras 
\cite{glimm_certain_1960}; see also \cite[Chapter 7]%
{rordam_introduction_2000}. Recall that two separable C*-algebras $A,B$ are 
\emph{stably isomorphic }(or, equivalently, Morita-equivalent; see \cite[%
Definition Theorem 5.55]{raeburn_morita_1998}) if $A\otimes K\left( H\right)
\cong B\otimes K\left( H\right) $, where $K\left( H\right) $ is the
C*-algebra of compact operators on the separable infinite-dimensional
Hilbert space.

\begin{theorem}
\label{Theorem:UHF}Definable $\mathrm{K}^{1}$ is a complete invariant for
UHF C*-algebras up to stable isomorphism. In contrast, there exists an un
uncountable family of pairwise non stably isomorphic UHF C*-algebras whose $%
\mathrm{K}^{1}$-groups are isomorphic as discrete groups (but not definably
isomorphic).
\end{theorem}

\begin{proof}
It follows from the classification of AF C*-algebras by $\mathrm{K}$-theory
that the (unordered)\textrm{\ }$\mathrm{K}_{0}$-group is a complete
invariant for UHF C*-algebras up to \emph{stable }isomorphism; see \cite[%
Chapter IV]{davidson_algebras_1996}. From this and Corollary \ref%
{Corollary:KK2}, it follows that the \emph{definable }$\mathrm{K}^{1}$-group
is also a complete invariant for UHF C*-algebras up to stable isomorphism.

If, adopting the notations above, $\left( \Lambda _{i}\right) _{i\in \mathbb{%
R}}$ is an uncountable family of pairwise nonisomorphic rank $1$
torsion-free abelian groups not isomorphic to $\mathbb{Z}$ such that $%
\mathrm{Ext}\left( \Lambda _{i},\mathbb{Z}\right) $ and $\mathrm{Ext}\left(
\Lambda _{j},\mathbb{Z}\right) $ are isomorphic as discrete groups for $%
i,j\in \mathbb{R}$, then $\left( A_{\Lambda _{i}}\right) _{i\in \mathbb{R}}$
is an uncountable family of pairwise non stably isomorphic UHF C*-algebras
whose $\mathrm{K}^{1}$-groups are isomorphic as discrete groups but not
definably isomorphic.
\end{proof}

\section{Definable $\mathrm{K}$-homology of compact metrizable spaces\label%
{Section:spaces}}

In this section, we consider definable $\mathrm{K}$-homology of compact
metrizable spaces, which can be seen as a particular instance of definable $%
\mathrm{K}$-homology when restricted to unital, commutative, separable
C*-algebras. As another application of the definable Universal Coefficient
Theorem, we show that definable $\mathrm{K}$-homology of compact metrizable
spaces is a finer invariant than its purely algebraic version, even when
restricted to connected $1$-dimensional subspaces of $\mathbb{R}^{3}$.

\subsection{$\mathrm{K}$-homology and topological $\mathrm{K}$-theory of
spaces}

The notion (definable) of $\mathrm{K}$-homology for compact metrizable
spaces is obtained as a particular instance of the corresponding notion for
separable C*-algebras, by considering the contravariant functor $X\mapsto
C\left( X\right) $ assigning to a compact metrizable space the separable
unital C*-algebra $C\left( X\right) $ of continuous complex-valued functions
on $X$. Thus, if $X$ is a compact metrizable space, its \emph{definable} $%
\mathrm{K}$-homology groups are given by%
\begin{equation*}
\mathrm{K}_{p}\left( X\right) :=\mathrm{K}^{p}\left( C\left( X\right) \right)
\end{equation*}%
for $p\in \left\{ 0,1\right\} $; see \cite[Chapter 7]{higson_analytic_2000}.
The \emph{reduced }definable $\mathrm{K}$-homology groups are similarly
defined by%
\begin{equation*}
\mathrm{\tilde{K}}_{p}\left( X\right) :=\mathrm{\tilde{K}}^{p}\left( C\left(
X\right) \right) \text{.}
\end{equation*}%
In particular, one sets%
\begin{equation*}
\mathrm{Ext}\left( X\right) :=\mathrm{Ext}\left( C\left( X\right) \right)
\end{equation*}%
and%
\begin{equation*}
\mathrm{\tilde{K}}_{1}\left( X\right) =\mathrm{\tilde{K}}^{1}\left( C\left(
X\right) \right) \text{.}
\end{equation*}%
Notice that, by definition, 
\begin{equation*}
\mathrm{\tilde{K}}_{1}\left( X\right) =\mathrm{\tilde{K}}^{1}\left( C\left(
X\right) \right) =\mathrm{Ext}\left( C\left( X\right) \right) =\mathrm{Ext}%
\left( X\right) \text{.}
\end{equation*}%
Similarly, the \emph{topological }$\mathrm{K}$\emph{-theory} groups of $X$
can be defined in terms of the $\mathrm{K}$-theory of $C\left( X\right) $ by
setting%
\begin{equation*}
\mathrm{K}^{p}\left( X\right) :=\mathrm{K}_{p}\left( C\left( X\right)
\right) \text{;}
\end{equation*}%
see \cite[3.3.7]{rordam_introduction_2000}. Equivalently, the topological $%
\mathrm{K}$-groups can be defined in terms of vector bundles over $X$; see 
\cite[Chapter II]{karoubi_theory_2008} and \cite[Chapter 13]%
{wegge-olsen_theory_1993}. One can also define the reduced $\mathrm{K}$%
-group $\mathrm{\tilde{K}}^{p}\left( X\right) $ to be the quotient of $%
\mathrm{K}^{p}\left( X\right) $ by the subgroup obtained as the image of $%
\mathrm{K}^{p}\left( \left\{ \ast \right\} \right) $ under the homomorphism
induced by the map $X\rightarrow \left\{ \ast \right\} $. (Notice that $%
\mathrm{K}^{1}\left( \left\{ \ast \right\} \right) $ is trivial and $\mathrm{%
K}^{0}\left( \left\{ \ast \right\} \right) \cong \mathbb{Z}$.)

\subsection{The Universal Coefficient Theorem}

Recall that a \emph{compact polyhedron} is a compact metrizable space that
is obtained as the topological realization of a finite simplicial complex;
see \cite[Appendix 1]{mardesic_shape_1982}. (In the following, we assume
that all the polyhedra are compact.) The topological $\mathrm{K}$-groups of
a polyhedron are finitely-generated \cite[Proposition 7.14]%
{higson_analytic_2000}. Furthermore, if $P$ is a polyhedron, then it can be
proved by induction on the number of simplices of the corresponding
simplicial complex that the unital C*-algebra $C\left( P\right) $ satisfies
the UCT for $\mathbb{C}$ \cite{brown_operator_1975,brown_universal_1984}.

If $X$ is a compact metrizable space, then one can write $X$ as the
(inverse) limit of a tower $\left( X_{n}\right) _{n\in \omega }$ of compact
polyhedra \cite[Section I.6]{mardesic_shape_1982}. Such a tower, called a 
\emph{polyhedral resolution }of $X$ in \cite{mardesic_shape_1982}, can be
obtained by considering the topological realizations of the nerves of a
sequence of finite open covers of $X$ that is cofinal in the ordered set of
finite open covers of $X$. If $\left( X_{n}\right) $ is a polyhedral
resolution for $X$, then $\left( C\left( X_{n}\right) \right) _{n\in \omega
} $ is a filtration for $C\left( X\right) $ in the sense of Definition \ref%
{Definition:filtration}. Thus, one can consider the weak $\mathrm{K}$%
-homology group%
\begin{equation*}
\mathrm{K}_{p}^{\mathrm{w}}\left( X\right) :=\mathrm{K}_{\mathrm{w}%
}^{p}\left( C\left( X\right) \right) =\mathrm{lim}_{n}\mathrm{K}_{p}\left(
X_{n}\right)
\end{equation*}%
and the asymptotic $\mathrm{K}$-homology groups%
\begin{equation*}
\mathrm{K}_{p}^{\infty }\left( X\right) :=\mathrm{K}_{\infty }^{p}\left(
C\left( X\right) \right) \cong \mathrm{PExt}\left( \mathrm{K}^{1-p}\left(
X\right) ,\mathbb{Z}\right) \text{.}
\end{equation*}%
We can also consider their \emph{reduced} versions, by letting $\mathrm{%
\tilde{K}}_{p}^{\mathrm{w}}\left( X\right) $ be the kernel of the definable
group homomorphism $\mathrm{K}_{p}^{\mathrm{w}}\left( X\right) \rightarrow 
\mathrm{K}_{p}^{\mathrm{w}}\left( \left\{ \ast \right\} \right) $ induced by
the map $X\rightarrow \left\{ \ast \right\} $, and similarly for $\mathrm{%
\tilde{K}}_{p}^{\infty }\left( X\right) $. It is then easy to see that%
\begin{equation*}
\mathrm{\tilde{K}}_{p}^{\mathrm{w}}\left( X\right) =\mathrm{lim}_{n}\mathrm{%
\tilde{K}}_{p}\left( X_{n}\right)
\end{equation*}%
and%
\begin{equation*}
\mathrm{\tilde{K}}_{p}^{\infty }\left( X\right) \cong \mathrm{K}_{p}^{\infty
}\left( X\right) \text{.}
\end{equation*}%
By definition, we have definable short exact sequences%
\begin{equation*}
0\rightarrow \mathrm{K}_{p}^{\infty }\left( X\right) \rightarrow \mathrm{K}%
_{p}\left( X\right) \rightarrow \mathrm{K}_{p}^{\mathrm{w}}\left( X\right)
\rightarrow 0
\end{equation*}%
and%
\begin{equation*}
0\rightarrow \mathrm{\tilde{K}}_{p}^{\infty }\left( X\right) \rightarrow 
\mathrm{\tilde{K}}_{p}\left( X\right) \rightarrow \mathrm{\tilde{K}}_{p}^{%
\mathrm{w}}\left( X\right) \rightarrow 0\text{.}
\end{equation*}%
As particular instances of Proposition \ref{Proposition:KK}, Corollary \ref%
{Corollary:KK1}, and Corollary \ref{Corollary:KK2} (or, precisely, their
analogues for reduced $\mathrm{K}$-homology), one obtains the following.

\begin{proposition}
\label{Proposition:KK-abelian}Suppose that $X$ is a compact metrizable space
and $p\in \left\{ 0,1\right\} $. If $\mathrm{\tilde{K}}^{p}\left( X\right) $
is torsion-free, then $\mathrm{\tilde{K}}_{p}^{\infty }\left( X\right) $ is
naturally definably isomorphic to $\mathrm{Ext}(\mathrm{\tilde{K}}%
^{1-p}\left( X\right) ,\mathbb{Z)}$, and $\mathrm{\tilde{K}}_{p}^{\mathrm{w}%
}\left( X\right) $ is naturally isomorphic to $\mathrm{Hom}(\mathrm{\tilde{K}%
}^{p}\left( X\right) ,\mathbb{Z)}$.
\end{proposition}

\begin{corollary}
\label{Corollary:KK1-abelian}Suppose that $X$ is a compact metrizable space
and $p\in \left\{ 0,1\right\} $ is such that $\mathrm{\tilde{K}}^{p}\left(
X\right) $ is a finite-rank torsion-free abelian group and $\mathrm{\tilde{K}%
}^{1-p}\left( X\right) $ is trivial. We can write%
\begin{equation*}
\mathrm{\tilde{K}}^{p}\left( X\right) =\Lambda \oplus \Lambda ^{\prime }
\end{equation*}%
where $\Lambda ^{\prime }$ is finitely-generated and $\Lambda $ has no
nonzero finitely-generated direct summand. Then%
\begin{equation*}
\mathrm{\tilde{K}}_{p}^{\mathrm{w}}\left( X\right) \cong \mathrm{Hom}(%
\mathrm{\tilde{K}}^{p}\left( X\right) ,\mathbb{Z)}\cong \mathrm{Hom}\left(
\Lambda ^{\prime },\mathbb{Z}\right)
\end{equation*}%
and%
\begin{equation*}
\mathrm{\tilde{K}}_{1-p}^{\infty }\left( A\right) \cong \mathrm{Ext}(\mathrm{%
\tilde{K}}^{p}\left( X\right) ,\mathbb{Z)}\cong \mathrm{Ext}\left( \Lambda ,%
\mathbb{Z}\right)
\end{equation*}%
as definable groups.
\end{corollary}

\begin{corollary}
\label{Corollary:KK2-abelian}Suppose that $p\in \left\{ 0,1\right\} $, and $%
X,Y$ are compact metrizable spaces, such that $\mathrm{\tilde{K}}^{p}\left(
X\right) $ and $\mathrm{\tilde{K}}^{p}\left( Y\right) $ are finite-rank
torsion-free abelian groups, and $\mathrm{\tilde{K}}^{1-p}\left( X\right) $
and $\mathrm{\tilde{K}}^{1-p}\left( Y\right) $ are trivial. Then the
following assertions are equivalent:

\begin{enumerate}
\item $\mathrm{\tilde{K}}_{i}\left( X\right) $ and $\mathrm{\tilde{K}}%
_{i}\left( Y\right) $ are \emph{definably} isomorphic for $i\in \left\{
0,1\right\} $;

\item $\mathrm{\tilde{K}}^{p}\left( A\right) $ and $\mathrm{\tilde{K}}%
^{p}\left( B\right) $ are isomorphic.
\end{enumerate}

If furthermore $\mathrm{K}^{p}\left( X\right) $ and $\mathrm{K}^{p}\left(
Y\right) $ have no nonzero finitely-generated direct summand, then the
following assertions are equivalent:

\begin{enumerate}
\item $\mathrm{\tilde{K}}_{1-p}\left( X\right) $ and $\mathrm{\tilde{K}}%
_{1-p}\left( Y\right) $ are \emph{definably} isomorphic;

\item $\mathrm{\tilde{K}}^{p}\left( A\right) $ and $\mathrm{\tilde{K}}%
^{p}\left( B\right) $ are isomorphic.
\end{enumerate}
\end{corollary}

\subsection{Solenoids}

A ($1$-dimensional) solenoid is a compact metrizable space $X$ that is
homeomorphic to a $1$-dimensional compact connected abelian group other than 
$\mathbb{T}$. Thus, if $\Lambda $ is a rank $1$ torsion-free abelian group
(or, equivalently, a subgroup of $\mathbb{Q}$) other than $\mathbb{Z}$, then
its Pontryagin dual group $X_{\Lambda }:=\Lambda ^{\ast }$ is a solenoid,
and every solenoid arises in this fashion (up to homeomorphism). When $%
\Lambda =\mathbb{Z}[1/p]$ for some prime number $p$, then the corresponding
solenoid $X_{\Lambda }$ is called the $p$-adic solenoid. A solenoid $X$ can
be realized as a compact subset of $\mathbb{R}^{3}$ (but not of $\mathbb{R}%
^{2}$) \cite[Exercise VIII.E]{eilenberg_foundations_1952}; see also \cite%
{jiang_tame_2011,jiang_embedding_2008,bognar_embedding_1988,bognar_embedding_1988-1}%
. Solenoids were originally considered by Vietoris \cite{vietoris_uber_1927}
and van Danztig \cite{van_dantzig_theorie_1932}. They arise in the context
of dynamical systems, and they provided in the work of Smale the first
examples of attractors of dynamical systems that are \emph{strange} \cite%
{ruelle_what_2006,smale_differentiable_1967,williams_expanding_1974}.

If $\mathbb{T}$ is the circle, then one has that $\mathrm{\tilde{K}}%
^{1}\left( \mathbb{T}\right) =\mathbb{Z}$ and $\mathrm{\tilde{K}}^{0}\left( 
\mathbb{T}\right) =\left\{ 0\right\} $. Furthermore, if $\varphi :\mathbb{T}%
\rightarrow \mathbb{T}$ is a continuous map of degree $n\in \mathbb{Z}$,
then the induced map $\varphi ^{\ast }:\mathrm{\tilde{K}}^{1}\left( \mathbb{T%
}\right) \rightarrow \mathrm{\tilde{K}}^{1}\left( \mathbb{T}\right) $ is
given by $x\mapsto nx$. It follows easily from this that, if $\Lambda $ is a
subgroup of $\mathbb{Q}$, then $\mathrm{\tilde{K}}^{1}\left( X_{\Lambda
}\right) \cong \Lambda $ and $\mathrm{\tilde{K}}^{0}\left( X_{\Lambda
}\right) \cong \left\{ 0\right\} $. Thus, by Proposition \ref%
{Proposition:KK-abelian}, we have that $\mathrm{\tilde{K}}_{0}\left(
X_{\Lambda }\right) \cong \mathrm{Ext}\left( \Lambda ,\mathbb{Z}\right) $
and $\mathrm{\tilde{K}}_{1}\left( X_{\Lambda }\right) \cong \left\{
0\right\} $ as definable groups. (The reduced \textrm{K}-homology of $p$%
-adic solenoids is also computed in \cite[Theorem 6.8]{kaminker_theory_1977}%
.) As in the proof of Theorem \ref{Theorem:UHF}, we have the following.

\begin{theorem}
Definable $\mathrm{\tilde{K}}_{0}$ is a complete invariant for $1$%
-dimensional solenoids up to homeomorphism. In contrast, there exist
uncountably many pairwise non homeomorphic $1$-dimensional solenoids whose $%
\mathrm{\tilde{K}}_{0}$-groups are isomorphic as discrete groups (but not
definably isomorphic).
\end{theorem}

\begin{proof}
If $\Lambda $ is a $1$-dimensional solenoid, then $\mathrm{\tilde{K}}%
^{1}\left( X_{\Lambda }\right) \cong \Lambda $ and $\mathrm{\tilde{K}}%
^{0}\left( X_{\Lambda }\right) \cong \left\{ 0\right\} $.\ It follows from
this and Corollary \ref{Corollary:KK2-abelian} that definable $\mathrm{%
\tilde{K}}_{0}$ is a complete invariant for $1$-dimensional solenoids up to
homeomorphism.

If $\left( \Lambda _{i}\right) _{i\in \mathbb{R}}$ is an uncountable family
of pairwise nonisomorphic rank $1$ torsion-free abelian groups such that $%
\mathrm{Ext}\left( \Lambda _{i},\mathbb{Z}\right)$ and $\mathrm{Ext}\left(
\Lambda _{j},\mathbb{Z}\right) $ are isomorphic as discrete groups for $i,j\in \mathbb{R}$ as in Section \ref%
{Subsection:UHF}, then $\left( X_{\Lambda _{i}}\right) _{i\in \mathbb{R}}$
is an uncountable family of pairwise non homeomorphic solenoids whose $%
\mathrm{\tilde{K}}_{0}$-groups are isomorphic as discrete groups but not
definably isomorphic.
\end{proof}


\bibliographystyle{amsalpha}
\bibliography{cohomology-BE2}
\printindex

\end{document}